\documentclass[english,12pt]{smfbook}
\usepackage[english,francais]{babel}
\usepackage[T1]{fontenc}
\usepackage[utf8]{inputenc}
\usepackage{lmodern}
\usepackage{graphicx}
\usepackage{amssymb}
\usepackage{amsmath}
\usepackage{amsfonts}
\usepackage{amscd}
\usepackage{amsthm}
\usepackage{multicol}

\usepackage{etoolbox}
\patchcmd{\thebibliography}{*}{}{}{}
\pretocmd\thebibliography{\csname c@secnumdepth\endcsname=-2 }{}{}
\patchcmd{\theindex}{*}{}{}{}
\pretocmd\theindex{\csname c@secnumdepth\endcsname=-2 }{}{}

\usepackage{mathrsfs}
\usepackage{a4wide}
\usepackage{mathtools}
\usepackage{xy}
\usepackage{ucs}
\usepackage{relsize}
\usepackage{enumerate}

\input{xy}
\xyoption{all}

\usepackage{hyperref}

\numberwithin{equation}{chapter}

\theoremstyle{plain}
\newtheorem{theorem}{Theorem}[subsection]
\newtheorem{lemma}[theorem]{Lemma}
\newtheorem{prop}[theorem]{Proposition}
\newtheorem{cor}[theorem]{Corollary}
\newtheorem{theoreme}{Theorem}

\newtheorem{corollaire}[theoreme]{Corollary}
\newtheorem{hypothese}{Setting}

\theoremstyle{definition}
\newtheorem{definition}[theorem]{Definition}

\newtheorem{fact}[theorem]{Fact}
\newtheorem{example}[theorem]{Example}
\newtheorem{pbm}{Problem}

\theoremstyle{remark}
\newtheorem{remark}[theorem]{Remark}
\newtheorem{notation}[theorem]{Notation}

\newcommand{\Z}{\mathbf{Z}}
\newcommand{\R}{\mathbf{R}}
\newcommand{\C}{\mathbf{C}}
\newcommand{\Q}{\mathbf{Q}}
\newcommand{\N}{\mathbf{N}}
\newcommand{\F}{\mathbf{F}}

\newcommand{\expm}{\mathscr{E}xp\mathscr{M}_k}

\newcommand{\expp}{\mathscr{E}xp\mathscr{M}}
\newcommand{\kvar}{\mathrm{KVar}}
\newcommand{\svar}{\mathrm{KVar}^+}
\newcommand{\evar}{\mathrm{KExpVar}}
\newcommand{\sevar}{\mathrm{KExpVar}^+}
\newcommand{\LL}{\mathbf{L}}
\newcommand{\M}{\mathscr{M}}
\newcommand{\loc}{\mathrm{loc}}

\newcommand{\OO}{\mathcal{O}}
\newcommand{\ord}{\mathrm{ord}}
\renewcommand{\t}{\mathbf{t}}
\newcommand{\four}{\mathscr{F}}
\newcommand{\dx}{\mathrm{d}}
\newcommand{\Sch}{\mathscr{S}}

\newcommand{\A}{\mathbf{A}}
\newcommand{\Proj}{\mathbf{P}}
\renewcommand{\div}{\mathrm{div}}
\newcommand{\spec}{\mathrm{Spec}\,}

\newcommand{\Sym}{\mathfrak{S}}
\newcommand{\Res}{\mathrm{Res}}
\newcommand{\res}{\mathrm{res}}
\newcommand{\Cl}{\mathrm{Cl}}
\newcommand{\an}{\mathrm{an}}
\newcommand{\Ad}{\mathbb{A}}
\renewcommand{\ker}{\mathrm{Ker}}
\newcommand{\pt}{\mathrm{pt}}
\newcommand{\aff}{\A_k^{(M,N)}}
\newcommand{\card}{\mathrm{Card}}

\newcommand{\arr}{\longrightarrow}
\newcommand{\scr}{\mathscr}
\newcommand{\id}{\mathrm{id}}
\newcommand{\pr}{\mathrm{pr}}
\newcommand{\I}{\mathcal{I}}
\renewcommand{\bar}{\overline}

\newcommand{\n}{\mathbf{n}}
\newcommand{\m}{\mathbf{m}}

\newcommand{\y}{\mathbf{y}}
\renewcommand{\1}{\mathbf{1}}
\newcommand{\T}{\mathbf{T}}

\newcommand{\be}{\beta}
\newcommand{\al}{\alpha}
\renewcommand{\i}{\iota}
\renewcommand{\phi}{\varphi}
\newcommand{\G}{\mathbf{G}}
\newcommand{\eps}{\epsilon}

\renewcommand{\a}{\mathfrak{a}}
\renewcommand{\b}{\mathfrak{b}}
\newcommand{\ce}{\mathfrak{c}}

\newcommand{\calP}{\mathcal{P}}
\newcommand{\calQ}{\mathcal{Q}}

\newcommand{\conv}{\oplus}
\newcommand{\convv}{\star}
\newcommand{\totvan}{\widetilde{\phi}^{\ \mathrm{tot}}}
\newcommand{\totnear}{\widetilde{\psi}^{\ \mathrm{tot}}}
\newcommand{\atotvan}{\phi^{\mathrm{tot}}}
\newcommand{\atotnear}{\psi^{\mathrm{tot}}}
\newcommand{\tot}{\mathrm{tot}}
\newcommand{\add}{\mathrm{add}}

\newcommand{\mhm}{\mathrm{MHM}}
\newcommand{\rat}{\mathrm{rat}}
\newcommand{\bboxtimes}{\mathchoice{%
\raisebox{-2pt}{$\displaystyle{\mathlarger{\mathlarger{\boxtimes}}}$}}
             {\raisebox{-1pt}{$\mathlarger{\mathlarger{\boxtimes}}$}}
            {\raisebox{-0.5pt}{$\mathlarger{\mathlarger{\boxtimes}}$}}
           {\raisebox{-0.2pt}{$\mathlarger{\mathlarger{\boxtimes}}$}}}
\newcommand{\hdg}{\mathrm{Hdg}}
\newcommand{\gr}{\mathrm{Gr}}
\newcommand{\mon}{\mathrm{mon}}
\def\twtimes{{\kern .1em\mathop{\boxtimes}\limits^{\scriptscriptstyle{T}}\kern .1em}}
\newcommand{\bigtwtimes}[1]{\mathchoice{%
\raisebox{-2pt}{$\overset{\mathclap{\scriptstyle{T}}}{\displaystyle{\mathlarger{\mathlarger{\boxtimes}}}} {\vphantom{I}}^{#1} $}}
             {\raisebox{-1pt}{$\overset{\scriptscriptstyle{T}}{\mathlarger{\mathlarger{\boxtimes}}} {\vphantom{I}}^{#1}$}}
            {\raisebox{-0.5pt}{$\overset{\scriptscriptstyle{T}}{\mathlarger{\mathlarger{\boxtimes}} {\vphantom{I}}^{#1}}$}}
           {\raisebox{-0.2pt}{$\overset{\scriptscriptstyle{T}}{\mathlarger{\mathlarger{\boxtimes}}{\vphantom{I}}^{#1}}$}}}
\def\twotimes{{\kern .1em\mathop{\otimes}\limits^{T}\kern .1em}}

\newcommand{\vanhdg}{\phi^{\tot}}

\newcommand{\base}{R}

\title{Motivic Euler products and motivic height zeta functions}
\author{Margaret Bilu}
\address{Courant Institute of Mathematical Sciences\\
 251 Mercer St\\
  New York NY 10012}
\email{bilu\at cims.nyu.edu}
\urladdr{https://cims.nyu.edu/~bilu/}

\begin{document}

\frontmatter

\begin{abstract} 
A motivic height zeta function associated to a family of varieties parametrised by a curve is the generating series of the classes, in the Grothendieck ring of varieties, of moduli spaces of sections of this family with varying degrees. This memoir is devoted to the study of the motivic height zeta function associated to a family of varieties with generic fiber having the structure of an equivariant compactification of a vector group. It is a motivic analogue of Chambert-Loir and Tschinkel's work \cite{CLTi} solving Manin's problem for integral points of bounded height on partial equivariant compactifications of vector groups over number fields, and a generalisation of Chambert-Loir and Loeser's paper \cite{CL}. Our main theorem, stated in the introduction and proved in the final chapter of the text, describes the convergence of this motivic height zeta function with respect to a topology on the Grothendieck ring of varieties coming from the theory of weights in cohomology. We deduce from it the asymptotic behaviour, as the degree goes to infinity, of a positive proportion of the coefficients of the Hodge-Deligne polynomial of the above moduli spaces: in particular, we get an estimate for their dimension and the number of components of maximal dimension.

Our method, based on the Poisson summation formula, is inspired by the one used in \cite{CLTi}, but the motivic setting requires several major adaptations besides the ones already present in \cite{CL}. The main contribution of this text is the introduction of a notion of \textit{motivic Euler product} for series with coefficients in the Grothendieck ring of varieties, analogous to the classical Euler products encountered in number theory. It relies on the construction of geometric objects called symmetric products, which generalise the notion of symmetric power of a variety. These Euler products and symmetric products allow moreover a suitable extension of Hrushovski and Kazhdan's motivic Poisson formula, which is used to rewrite the motivic height zeta function as a sum, over characters, of series with coefficients in a Grothendieck ring of varieties with exponentials. The aforementioned weight topology on this ring is constructed using Saito's theory of mixed Hodge modules, as well as Denef and Loeser's motivic vanishing cycles and Thom-Sebastiani theorem. 
%
%This text is an improved version of the author's PhD thesis \cite{Bilu}. In particular, the latter contained an additional hypothesis on multiplicativity of Euler products which has been removed in the present version. 
\end{abstract}
%
%\subjclass{11G50, 14G05}
%\keywords{Grothendieck rings of varieties, Motivic integration, Heights, Poisson formula, Manin's conjecture, Vanishing cycles, Euler products}
\thanks{The author thanks her PhD advisor, Antoine Chambert-Loir, for his unfailing help and support. She is also indebted to David Bourqui for his thorough reading of this manuscript and his abundant advice on how to improve it, and to Sean Howe, for his help in removing the additional hypothesis on multiplicativity of Euler products in the previous version of this text.}

\maketitle

\tableofcontents

\mainmatter
%\selectlanguage{francais}
\chapter{Introduction}

\section{Manin's problem in the arithmetic setting}

One of the greatest concerns of number theory and algebraic geometry in the previous decades has been to understand the subtle link between the distribution of rational  or integral points on a variety defined over a number field, and some of it geometric invariants. Let $F$ be a number field and $X$ a projective variety over ~$F$, endowed with an ample line bundle~$L$. Such a line bundle defines (up to adding a bounded function) a height function $H:X(F)\to \R_{+}$ such that for any real number~$B$, the set
$$\{x\in X(F), H(x)\leq B\}$$
is finite (this is called \textit{Northcott's property}). We then denote, for any Zariski open subset~$U$ of~$X$, 
$$N_{U,H}(B) := \#\{x\in U(F), H(x)\leq B\}.$$ When the set $U(F)$ is itself infinite, one can ask about the asymptotic behaviour of $N_{U,H}(B)$ as~$B$ goes to $+\infty$. In all known cases, it is of the form
                     $$N_{U,H}(B)\underset{B\to\infty}{\sim} C B^{a}(\log B)^{b-1}$$
                     with $C>0$ and $a\geq 0$ being real numbers, and $b\in \frac{1}{2}\Z$, $b\geq 1$. 
                     
                    A series of conjectures, or rather questions, stated by Manin and his coauthors in \cite{FManT} et \cite{BM} in the end of the 1980s initiated a vast programme aimed at giving a geometric interpretation of the exponents $a$ and $b$, in terms of the classes of the line bundle $L$ and the anti-canonical bundle $K_X$ in the N\'{e}ron-Severi group of $X$, and of the cone of effective classes of divisors on $X$. To get a plausible statement, one of course has to make some restrictions on the variety $X$. In particular, one can only hope to obtain an asymptotic describing the geometry of $X$ properly if the set of rational points~$X(F)$ is Zariski dense in~$X$. Since Lang's conjecture predicts that this should never happen for varieties of general type, Manin-type problems often restrict to Fano varieties, that is, varieties with ample anti-canonical bundle. One sometimes considers a larger class of varieties, called \textit{almost Fano} (we refer to \cite{PeyreBourbaki}, Définition 3.11, for a precise definition) having a (not necessarily ample but) big anti-canonical bundle,  which is a sufficient condition for the height function to satisfy Northcott's property on a nonempty open subset of the variety $X$.  For such varieties, the conjectures take on a particularly simple form if one counts points with respect to the anti-canonical height: 
                     \begin{pbm} \label{mainquest}Let~$X$ be an almost Fano variety defined over a number field~$F$, such that~$X(F)$ is Zariski dense in~$X$. Let $H$ be a height function relative to the anti-canonical bundle on~$X$. Does there exist a dense open subset~$U$ of~$X$ satisfying
                     \begin{equation}\label{asymptotic}N_{U,H}(B)\underset{B\to \infty}{\sim} C_HB(\log B)^{r-1}\end{equation}
                     where $r$ is the Picard rank of~$X$, and $C_H$ a positive constant?
                     \end{pbm}  
It is necessary to autorise restriction to an open subset, in order to take into account the possible accumulation of rational points inside proper closed subsets which overrule the distribution on the complementary open subset. Let us also mention a more precise form of this conjecture, due to Peyre, who in \cite{Peyre} proposed an interpretation for the constant~$C_H$ in terms of volumes of certain adelic spaces. Furthermore, in~\cite{BT98} Batyrev and Tschinkel adjusted Peyre's formula for the constant by adding a cohomological factor. 
                     
                 The first result of this type has been obtained by Schanuel \cite{Schanuel}, long before the conjectures were formulated, in the case where $X = U = \Proj^n_F$, giving moreover an explicit formula for the constant~$C_H$.   In~\cite{FManT}, it is shown that formula~(\ref{asymptotic}) holds for flag varieties (with~$U=X$).  Since then,  problem \ref{mainquest}, sometimes even in Peyre's more precise form, could be answered affirmatively in many special cases, and using many different methods, among which we can quote harmonic analysis, universal torsors and the circle method. However, there also exist some counterexamples like the one by  Batyrev and Tschinkel~\cite{BT98}, which is why we preferred to state the above conjecture in the form of a question. 
                  The main tool, common to most of the existing approaches, is the \textit{height zeta function}  
                    $$\zeta_{U,H}(s) = \sum_{x\in U(F)}H(x)^{-s},$$
                a function defined on a portion of the complex plane, and the convergence properties of which (abscissa of convergence, order of the first pole, coefficient at this pole) are linked, via tauberian theorems, to the exponents and the constant appearing in the asymptotic we are looking for. More precisely, problem~\ref{mainquest} may then essentially be reformulated in the following manner:
                    \begin{pbm}\label{zetaquest} Let~$X$ be an almost Fano variety over a number field~$F$, such that~$X(F)$ is Zariski dense in $X$. Let $H$ be a height function relative to the anti-canonical bundle on~$X$. Does there exist a dense open subset~$U$ of $X$ such that the $\zeta_{U,H}(s)$ converges absolutely for $\mathrm{Re}(s)>1$, and admits a meromorphic continuation on the open set $\{\mathrm{Re}(s) > 1-\delta\}$ for some real number $\delta>0$, with a unique pole of order $r = \mathrm{rg}\ \mathrm{Pic}(X)$ at $s=1$? 
                    \end{pbm}
                   In the same way as for problem~\ref{mainquest}, there is a more precise version of this question, requiring that the coefficient of the main term of~$\zeta_{U,H}(s)$ at 1 should be the constant predicted by Peyre. 
               \section{Manin's problem via harmonic analysis} \label{sect.maninharmonique}     
The aformentioned case of flag varieties was the first one solved using techinques coming from harmonic analysis: the proof relied on the fact that for these varieties, the height zeta function is an Eisenstein series, the analytic properties of which could be determined using results by Langlands.   Then, in the middle of the 1990s, Batyrev and Tschinkel treated the case of toric varieties, the open subset~$U$ being the open orbit of the torus action on the variety. Their argument relies on the Poisson summation formula and the torus action plays a key role there. It could be generalised to many other varieties endowed with an action of an algebraic group with an open orbit, e.g. certain equivariant compactifications of algebraic groups. An \textit{equivariant compactification} of an algebraic group~$G$ is a (projective and smooth) variety~$X$ which has an open subset isomorphic to~$G$ and which is endowed with an action $G\times X\to X$ of $G$ extending the group law $G\times G \to G$. Apart from toric varieties which are exactly the equivariant compactifications of algebraic tori, Manin's problem has been solved for equivariant compactifications of vector groups (\cite{CLT}, \cite{CLTi}), as well as for compactifications of certain non-commutative algebraic groups (\cite{ShTTB03}, \cite{ShTTB07}, \cite{STH}, \cite{STU}, \cite{TTB}, \cite{TT}). 

Let us give an outline of the argument in the case of equivariant compactifications of vector groups, due to Chambert-Loir and Tschinkel, and which will play a central role in this text. Let~$X$ be an equivariant compactification of the group $G = \G^n_a$ for some $n\geq 1$, defined over a number field~$F$. For each place~$v$ of the field~$F$, there is a local height~$H_v:G(F_v)\to \R_+$, the global height~$H$ being given by the formula
                $$H(x) = \prod_{v}H_v(x),$$
                which gives a way of extending~$H$ to the locally compact group~$G(\mathbb{A}_F)$, where~$\Ad_F$ is the group of adeles of~$F$.  The group~$G(\Ad_F)$ is self-dual, and~$G(F)$ may be seen as a discrete subgroup of~$G(\Ad_F)$ with orthogonal identified to itself, so that, when applying the Poisson summation formula to the function $H^{-s}$ (after checking all necessary integrability conditions), one obtains:
                \begin{equation}\label{intro.formulepoisson}\zeta_{H,G}(s) = \sum_{x\in G(F)}H(x)^{-s} = \sum_{\xi\in G(F)}\four(H^{-s})(\xi),\end{equation}
                where $\four$ denotes Fourier transformation. This equality is valid whenever the real part of~$s$ is large enough. Moreover, the function~$H^{-s}$ is invariant modulo a compact subgroup of~$G(\Ad_F)$, so that its Fourier transform is supported inside such a subgroup, which reduces the summation over~$G(F)$ in the right-hand side to a summation over a lattice. The upshot of this procedure is that it rearranges the terms of the height zeta function in a convenient way, and the term corresponding to the trivial character (that is, $\xi=0$), in all cases where this method has worked, and when $H$ is the anti-canonical height, happens to be the only carrier of the first pole with the correct order.  To prove this, one uses again the decomposition into local factors $$\four(H^{-s})(\xi) = \prod_{v}\four(H^{-s}_v)(\xi_v)$$ and one studies the different factors separately. The latter take the form of Igusa-type integrals, for which one gets, for almost all finite places~$v$, an explicit expression by reduction to the residue field. Their convergence properties are then determined using bounds coming from Lang-Weil estimates. At archimedean places, successive integrations by parts yield bounds with polynomial growth in $s$. For the finite number of remaining places, cruder bounds suffice. In the case where the considered height is the anti-canonical height, one concludes from these local estimates that the term corresponding to $\xi =0$ has a pole at~$s=1$, with the coefficient predicted by Peyre, of order the rank of the Picard group, that the other terms have poles with strictly smaller orders, and that the function~$s\mapsto (s-1)^{\mathrm{rg Pic}(X)}\zeta_{H,G}(s)$ extends holomorphically to $\{s\in\C,\ \Re(s) >1-\delta\}$ for some real number~$\delta >0$. 
                
         Though the problem of counting integral solutions of diophantine equations is natural, no one had adressed it from the geometric angle suggested by Manin, before, in their paper~\cite{CLTi}, Chambert-Loir and Tschinkel also proposed a solution to Manin's problem for integral points on \textit{partial} equivariant compactifications of vector groups. Such a partial compactification~$U$ is seen as the complement, in an equivariant compactification~$X$, of a divisor~$D$ which geometrically has strict normal crossings.  As before, it contains a dense open subset~$G$ isomorphic to the additive group~$\G^{n}_{a}$. We fix models of $X,U,D$ above the ring of integers $\OO_F$, a finite set of places~$S$ containing the archimedean places, and one aims at counting the points of~$G(F)$ which are $S$-integral with respect to the chosen model of~$U$. Note that this problem includes the previous one, because in the case where~$D=\varnothing$, if the chosen model for~$D$ is empty as well, then all points of~$G(F)$ are $S$-integral, by projectivity of $X$. For this counting problem, the relevant height isn't the anti-canonical one, but the \textit{log-anticanonical} height, that is, the one associated to the (big) line bundle~$-(K_X + D)$. 
                 Chambert-Loir and Tschinkel then prove that the number~$N(B)$ of $S$-integral points of log-anticanonical height smaller than~$B$ satisfies the asymptotic
                $$N(B)\sim C B(\log B)^{b-1}$$
                where $C$ is real positive constant, and where the exponent~$b$ is given by the formula
                \begin{equation}\label{formuleexposant}b = \mathrm{rg}\ \mathrm{Pic}(U) + \sum_{v\in S}(1 + \dim \scr{C}^{\mathrm{an}}_{F_v}(D)).\end{equation}
                Here, for every~$v$, $\scr{C}^{\mathrm{an}}_{F_v}(D)$ is a simplicial complex encoding the incidence properties of the components of~$D$ at the place~$v$. The term $1 + \dim \scr{C}^{\mathrm{an}}_{F_v}(D)$ corresponds exactly to the maximal number of components of~$D$ over~$F_v$ with intersection having~$F_v$-points. Of course, in the special case~$U = X$, all these terms are zero, and we recover the previous result for rational points.
           \section{Manin's problem over function fields}     
         For the moment, we have always restricted ourselves to the case where~$F$ is a number field. The case where~$F$ is the function field~$k(C)$ of a smooth projective curve~$C$ over a finite field~$k$ of cardinality~$q$ was first mentioned by Batyrev and Manin (\cite{BM}, 3.13). In this setting, the heights used (when suitably normalised) will take their values in the set $q^{\Z}$ of integral powers of $q$, which forbids the existence of an asymptotic like the one predicted by problem~\ref{mainquest}. Nevertheless, problem~\ref{zetaquest} remains valid, up to a slight modification to take into account the fact that the corresponding zeta function will in this case be $\frac{2i\pi}{\log(q)}$-periodic, namely that we need to authorise, in addition to the pole at~1, poles at $1 + \frac{2i\pi m}{\log(q)}$ for every integer $m$. In other words, it will be more suitable here to consider the height zeta function as a function of the variable~$t = q^{-s}$, which yields the following reformulation of problem~\ref{zetaquest}:         
          \begin{pbm}\label{zetafunctionquest} Let $C$ be a smooth projective connected curve over~$\F_q$, with function field denoted by~$F$. Let~$X$ be an almost Fano variety over the field~$F$, such that~$X(F)$ is Zariski dense in~$X$. Let~$H$ be  a height relative to the anti-canonical bundle of~$X$. Does there exist an open dense subset~$U$ of~$X$ such that the series~$\zeta_{U,H}(t)$ converges absolutely on the disc defined by $|t| < q^{-1}$, and extends to a meromorphic function on the disc defined by $|t| < q^{-1 + \delta}$ for some real number $\delta>0$, with a unique pole of order $r = \mathrm{rg}\ \mathrm{Pic}(X)$ at $t = q^{-1}$ ? 
         \end{pbm}
         Let us point out that in general, one can however have other poles of orders~$<r$ on the circle  $|t| = q^{-1}$, in particular for certain split toric varieties. In the functional case, the problem acquires an additional geometric interpretation: if one chooses a model~$\mathcal{X}$ of~$X$ above the curve~$C$, the rational points of an open subset~$U$ of~$X$ correspond to sections~$\sigma: C\to \mathcal{X}$ of the structural morphism $\pi:\mathcal{X}\to C$ such that, denoting $\eta_C$ the generic point of~$C$, one has $\sigma(\eta_C)\in U(F)$. Given a (generically ample, or at least big) line bundle~$\mathcal{L}$ on~$\mathcal{X}$, the height of such a section with respect to~$\mathcal{L}$ is given by~$q^{d}$, where $d$ is the degree of the line bundle~$\sigma^{*}\mathcal{L}$ over~$C$. The height zeta function then takes the form
         $$\zeta_{U,\mathcal{L}}(s) = \sum_{x\in U(F)} H(x)^{-s} = \sum_{d\geq 0}m_dq^{-ds},$$
         with $$m_d = |\{\sigma:C\to \mathcal{X},\ \sigma(\eta_C) \in U(F),\ \deg\sigma^{*}\mathcal{L} = d\}|.$$
         Thus, information about the convergence of the height zeta function will furnish the asymptotic of the number~$m_d$ of points of fixed height~$q^d$, which in this setting will be the analogue of formula~(\ref{asymptotic}).

        Manin's problem on function fields has been studied very scarcely until now: one must nevertheless quote the works of Bourqui \cite{Bou02,Bou03,Bou11}, which completely solve the case of toric varieties (employing harmonic analysis but also the universal torsor method), as well as the papers~\cite{LY} and~\cite{Peyre12} which treat independently the case of generalised flag varieties, Peyre's work containing moreover an interpretation of the constant. Peyre's method is analogous to Franke, Manin and Tschinkel's for flag varieties over number fields in~\cite{FManT}, the role of the results of Langlands being played by those of Morris about Eisenstein series over function fields. 
         
         The recent advances of \textit{motivic integration} finally suggest the following generalisation of Manin's problem over function fields: the sections $\sigma:C\to \mathcal{X}$ such that $\sigma(\eta_C) \in G(F)$ and $ \deg\sigma^{*}\mathcal{L} = d$ have a \textit{moduli space}~$M_d$ which is a quasi-projective $k$-scheme. The interest of the geometric study of such moduli spaces in relation with Manin's conjectures was first pointed out by Batyrev. More precisely, following an idea of Peyre, one can ask not only about the asymptotic cardinality~$m_d$ of $M_d(k)$, but more generally about the properties, when~$d$ is large, of the class of~$M_d$ in the Grothendieck ring of varieties $\kvar_k$ over~$k$. As a group, the latter is defined as the quotient of the free abelian group on the isomorphism classes of varieties over~$k$ by relations of the form $$X - U - Z$$
        for any variety~$X$ and any closed subscheme~$Z$ of~$X$ with open complement~$U$. The ring structure comes from the product of varieties: using brackets to denote classes in~$\kvar_k$, one has $[X][Y] = [X\times_kY]$ for all $k$-varieties $X$ and~$Y$. We denote by $\LL = [\A^1_k]$ the class of the affine line, and we also often consider the localised Grothendieck ring~$\M_{\C} = \kvar_k[\LL^{-1}]$. 
        
      The class of a variety in the Grothendieck ring contains a large amount of geometric information  about this variety: indeed, there exist numerous \textit{motivic measures}, that is, ring morphisms from~$\kvar_k$ to other rings, associating to a class~$[X]$ various geometric invariants of~$X$. Among these, in the case where the field is finite, on can quote the \textit{counting measure}
        $$\begin{array}{ccc} \kvar_{k}&\to& \Z \\
                                  \left[X\right] &\mapsto & \# X(k)\end{array}$$
                                  which recovers the number of rational points of the variety. For any field~$k$, fixing a separable closure~$k^s$ of~$k$ and~$\ell$ a prime number coprime to the characteristic of~$k$, the Euler-Poincaré polynomial (associated to cohomology with coefficients in~$\Q_{\ell}$) also  defines a motivic measure 
                                  $$\begin{array}{ccc} \kvar_{k}&\to& \Z[t] \\
                                  \left[X\right] &\mapsto & EP(X)(t) \end{array}$$%= \sum_{n\geq 0}\sum_{i\geq 0}(-1)^i\dim_{\Q_{\ell}\mathrm{Gr}^{W}_nH^{i}_{\text{ét},c}(X\otimes_kk^s,\Q_{\ell})t^n\end{array}.$$
                                  which for a smooth and projective variety~$X$ is given by
                                  $$EP(X)(t) = \sum_{i=0}^{2\dim X}(-1)^i\dim_{\Q_{\ell}}H^{i}_{\text{ét}}(X\otimes_kk^s,\Q_{\ell})t^i.$$
                                  Another example of the same flavour, and which will be important for us, is the Hodge-Deligne polynomial:
                                 $$\begin{array}{ccc} \kvar_{\C}&\to& \Z[u,v] \\
                                  \left[X\right] &\mapsto & HD(X)(u,v)\end{array}$$
                                  sending the class of a projective and smooth complex variety~$X$ to the polynomial
                                  $$HD(X)(u,v) = \sum_{0\leq p,q\leq \dim X} (-1)^{p+q}h^{p,q}(X)u^pv^q$$
                                  defined from the Hodge numbers $h^{p,q}(X)$ of~$X$. Noting that $HD(\LL) = uv$, one can moreover extend this measure to a ring morphism
                                  $$HD:\M_{\C}\to \Z[u,v,(uv)^{-1}]$$
                                  defined on the localised Grothendieck ring $\M_{\C}$.

                         Now that these definitions have been given, we come to the fundamental observation that we can in fact make sense  of a version of Manin's problem over the function field~$k(C)$ of a curve even when the base field~$k$ is not necessarily finite:  the question would now be to investigate certain geometric invariants of the space~$M_d$ when~$d$ goes to infinity, for example its dimension, or its number of irreducible components of maximal dimension. In this setting, which we will call \textit{motivic}, the notion of height zeta function takes the form of the series 
                               \begin{equation}\label{zetafunctionformula}Z(T) = \sum_{d\geq 0}[M_d]T^d\in \kvar_k[[T]]\end{equation}  with coefficients in the Grothendieck ring of varieties, called the \textit{motivic height zeta function}. Such functions have been studied by Bourqui in~\cite{Bou09} for certain toric varieties. One of the first difficulties appearing when dealing with these is the question of convergence of such a series. The localised Grothendieck ring $\M_k = \kvar_{k}[\LL^{-1}]$ has a natural topology induced by the \textit{dimensional filtration}: for every $n\in \Z$ we define $F_n\M_k$ to be the subgroup of~$\M_k$ generated by classes of the form $[X]\LL^{-m}$ where~$X$ is a $k$-variety such that $\dim(X) - m\leq n$. It is then natural to say that the above series converges at $\LL^{-s}$ if $\dim[M_d] - ds\to -\infty$ when $d\to +\infty$. Unfortunately, because of the coarseness of the notion of dimension, this convergence is not very amenable. In this text, we are going to use a slightly finer topology, coming from the theory of weights in cohomology.  Note that in~\cite{Bou09}, Bourqui nevertheless manages to prove convergence with respect to the dimensional filtration (in the Grothendieck ring of Chow motives), because in the case of the split toric varieties addressed in that paper, the dimensional filtration is sufficient to take into account the annihilation of maximal weights which is necessary for convergence. Later, in~\cite{Bou10}, Bourqui introduces a topology which is similar to ours, and the idea of which appears also in~\cite{Ekedahl}. 
                               \section{Main result}
                              The principal objective of this text is to obtain a motivic analogue of the theorem of Chambert-Loir and Tschinkel on counting integral points of partial equivariant compactifications of vector groups, described in the second half of section~\ref{sect.maninharmonique}. As in the arithmetic case, we are going to use harmonic analysis, which requires the construction of suitable objects in the motivic setting. We start by stating the main theorem.  
                               
\begin{hypothese}\label{hypothese.geom} Let~$C_0$ be a smooth quasi-projective connected curve over an algebraically closed field~$k$ of characteristic zero, let~$C$ be its smooth projective compactification, and let $S = C\setminus C_0$. We denote by $F =k(C)$ the function field of~$C$. We are given a projective irreducible $k$-scheme~$\mathcal{X}$ together with a non-constant morphism $\pi:\mathcal{X}\to C$, a Zariski open subset~$\mathcal{U}$ of~$\mathcal{X}$, and~$\mathcal{L}$ a line bundle over~$\mathcal{X}$. We make the following assumptions on the generic fibres $X = \mathcal{X}_F$, $U = \mathcal{U}_{F}$ and the line bundle $L = \mathcal{L}_{F}$:
\begin{itemize}\item $X$ is smooth, the open subset $U$ of $X$ contains a dense open subset~$G$ isomorphic to $\G^{n}_{a,F}$, and~$U$ and~$X$ are endowed with an action of~$G$ extending the group law of~$G$. In other words, $X$ (resp. $U$) is an equivariant (resp. partial equivariant) compactification of the additive group~$\G^n_{a}$. 
\item the boundary $\partial X = X\setminus U$ is a divisor~$D$ with strict normal crossings. 
\item the line bundle~$L$ on~$X$ is the log-anticanonical line bundle $-K_X(D)$. 
\end{itemize} 
\end{hypothese}

As above, we are interested in the moduli spaces of sections $\sigma: C\to \mathcal{X}$ such that $\sigma(\eta_C)\in \G^{n}_{a}(F)$ (where $\eta_C$ is the generic point of~$C$), but we restrict to those which correspond to $S$-integral points, which amounts to moreover requiring  $\sigma(C_0)\subset \mathcal{U}$. 
%, $\scr{A}_{D}\subset \scr{A}$ le sous-ensemble de $\scr{A}$ tel que $D = \sum_{\al\in\scr{A}_D} D_{\al}$ et qu'on note 
From the geometric point of view, the first condition means that such a section  $\sigma:C\to \mathcal{X}$ exits~$G$ only in a finite number of points of~$C$, called \textit{poles}, and the second one that for all $v\in C_0$, $\sigma(v)$ remains in~$\mathcal{U}$. If we denote by $(D_{\al})_{\al\in \scr{A}}$ the irreducible components of~$X\setminus G$, the log-anticanonical divisor can be written in the form
$\sum_{\al\in \scr{A}} \rho'_{\al}D_{\al}$
for positive integers $\rho'_{\al}$. The line bundle~$\mathcal{L}$ being generically log-anticanonical, it is of the form
$\mathcal{L} = \sum_{\al\in \scr{A}}\rho'_{\al}\mathcal{L}_{\al}$
where for all $\al\in \scr{A}$, the restriction of the line bundle~$\mathcal{L}_{\al}$ to the generic fibre is~$D_{\al}$. For $d\in\Z$ we denote by~$M_d$ the moduli space of sections~$\sigma$ satisfying the above conditions, and such that moreover $\deg \sigma^{*}\mathcal{L} = d$.  
 In view of the description of~$\mathcal{L}$ just given, up to a finite number of places (because~$\mathcal{L}$ may have a finite number of vertical components), only the poles of~$\sigma$ contribute to this degree, the contribution of each pole being the sum of the intersection degrees of~$\sigma$ at this pole with the divisors $\mathcal{L}_{\al}$  (the \textit{order} of the pole with respect to $\mathcal{L}_{\al}$), weighted by the integers $\rho'_{\al}$. One checks that the spaces~$M_d$ are empty for $d\ll 0$, so that one can define the motivic height zeta function by
\begin{equation}\label{def.fonctionzeta}Z(T) = \sum_{d\in \Z}[M_d]T^d \in\kvar_k[[T]][T^{-1}].\end{equation}

In addition to the ``generic'' hypotheses above, we need some assumption on the model~$\mathcal{U}$ for the spaces~$M_d$ to be non-empty. In fact, a ``Hasse principle''-type hypothesis is sufficient:

\begin{hypothese}\label{hypothese.section} We assume that there is no \textit{local obstruction} for the existence of such sections, that is for any closed point~$v\in C_0$ we have $G(F_v)\cap \mathcal{U}(\OO_v) \neq \varnothing$, where~$F_v$ is the completion of~$F$ at the place $v$, and~$\OO_v$ its ring of integers.
\end{hypothese}
To understand this requirement, it is useful to reformulate the condition on the sections in an adelic language. Each section corresponds to a unique point $\sigma(\eta_C)$ of $G(F)$, which by the diagonal embedding defines an element of the set $G(\Ad_F)$ of adelic points of~$G$. Up to a finite number of places, for a closed point $v\in C_0$ which is not a pole of~$\sigma$, the component of~$\sigma$ in~$G(F_v)$ is an element of~$G(\OO_v)$. The $S$-integrality condition $\sigma(C_0)\subset \mathcal{U}$ means that for all $v\in C_0$, the component of~$\sigma$ in~$G(F_v)$ is an element of~$\mathcal{U}(\OO_v)$: the non-emptyness of the intersection  $G(F_v)\cap \mathcal{U}(\OO_v)$ is therefore a necessary condition for the existence of such a section. 

Under these assumptions, we obtain the expected convergence for a topology on the Grothendieck ring~$\M_{\C}$ which will be made more explicit below:
\begin{theoreme}\label{main} We assume $k=\C$, and the notation and hypotheses in Settings \ref{hypothese.geom} et~\ref{hypothese.section}. We denote by~$b$ the integer given by formula (\ref{formuleexposant}). There exists an integer $a\geq 1$ and a real number $\delta >0$ such that the Laurent series $(1-(\LL T)^a)^bZ(T)$ converges for  $|T| < \LL^{-1 + \delta}$ and takes a non-zero effective value at $ T = \LL^{-1}$. 
\end{theoreme}
Thus, the order of the pole of the height zeta function is given by the same formula as in Chambert-Loir and Tschinkel's result mentioned in section~\ref{sect.maninharmonique}. The non-zero effective value at $\LL^{-1}$ announced in the statement is an element of the completion $\widehat{\M}_{\C}$ of $\M_{\C}$ for the topology we consider: it appears as an infinite product of local volumes, and is a motivic analogue of Peyre's constant.

For any $k$-constructible set~$M$, we denote by~$\kappa(M)$ the number of irreducible components of maximal dimension of~$M$. The Hodge-Deligne polynomial~$HD$ extends to a motivic measure
$$\widehat{\M}_{\C}\to \Z[[(uv)^{-1}]][u,v]$$
on the aforementioned completion, and theorem~\ref{main}, via this motivic measure, furnishes a description of the asymptotic behaviour of~$\dim(M_d)$ and $\kappa(M_d)$, with a distinction according to the congruence class of $d$ modulo $a$ imposed by the presence of the exponent~$a$. Moreover, a Lefschetz principle type argument enables us to remove the hypothesis $k=\C$ assumed in the theorem. 
\begin{corollaire}\label{maincor} For all  $p\in\{0,\ldots,a-1\}$, one of the following cases occur when~$d$ goes to infinity in the congruence class of~$p$ modulo~$a$:
\begin{enumerate}[(i)]
\item Either $\limsup \frac{\dim(M_d)}{d}<1$.
\item Or $\dim(M_d) -d$ has a finite limit $d_0$ and  $$\frac{\log(\kappa(M_d))}{\log d}$$ converges to an element of the set $\{0,\ldots,b-1\}$. More generally, for every sufficiently small $\eta>0$ and for sufficiently large $d$ in the congruence class of $p$ modulo $a$, the coefficients of the Hodge-Deligne polynomial $HD(M_d)$ of degrees contained in the interval $$[2(1- \eta)d + 2d_0, 2d + 2d_0]$$ are polynomials in $\frac{d-p}{a}$ of degree at most $b-1$. 
\end{enumerate}
Moreover the second case happens for at least one value of  $p\in\{0,\ldots,a-1\}$.
\end{corollaire}
%\item\label{secondcase} Soit $\dim(M_d) -d$ a une limite finite et  $$\frac{\log(\kappa(M_d))}{\log d}$$ converge vers un élément de $\{0,\ldots,b-1\}$. 
%\end{enumerate}
In other words, for large $d$ in at least one congruence class modulo $a$, the space $M_d$ will be of dimension $d+ d_0$ for some integral constant  $d_0$, and the number of components having precisely this dimension has polynomial growth of degree bounded by $b$. Moreover, more generally, the latter asymptotic is satisfied for a positive proportion of the coefficients of highest degree of the Hodge-Deligne polynomial of $M_d$. 

A condition on congruence classes cannot be avoided in general: for example, if the log-anticanonical divisor is a multiple $(L')^a$ of a class in $\mathrm{Pic}(\mathcal{X})$, then $M_d = \varnothing$ for $d\nmid a$. 

It is important to highlight two important special cases of these results: when~$\mathcal{X} = \mathcal{U}$ we obtain a motivic analogue of Chambert-Loir and Tschinkel's paper~\cite{CLT} for rational points on equivariant compactifications of vector groups. In this case, there is no condition on poles of sections. On the contrary, the case~$U=G_F$ where we only allow sections with poles in the finite set $C\setminus C_0$ has been treated in Chambert-Loir and Loeser's work~\cite{CL}, following the same idea as the proof sketched in section \ref{sect.maninharmonique} for the arithmetic version of Manin's problem.  Because of the restriction on the poles of sections, only a finite number of places contribute to the height.  Moreover, at fixed degree~$d$, the orders of the poles of the sections counted in the space~$M_d$ are bounded in terms of~$d$. Going back to the adelic description above, one sees that the characteristic function of the sections parametrised by~$M_d$ is a function on the adeles~$G(\Ad_F)$ the restriction of which to~$G(F_v)$ is, for almost all~$v$, the characteristic function of $G(\OO_v)$. More precisely, because of the bound on orders of poles and of the equivariance of the compactification, this characteristic function is a \textit{motivic Schwartz-Bruhat function}. The main tool employed by~Chambert-Loir and Loeser is the \textit{motivic Poisson formula} of  Hrushovski and Kazhdan, which is satisfied for this kind of functions and which, when applied to the characteristic function of~$M_d$ in $G(\Ad_F)$ for each~$d$, enables one to rewrite the height zeta function~$Z(T)$ in the form
$$Z(T) = \sum_{\xi\in G(F)} Z(T,\xi),$$
where the $Z(T,\xi)$ are series with coefficients in a Grothendieck ring described below.  This equality is the analogue, in this setting, of identity~(\ref{intro.formulepoisson}). Chambert-Loir and Loeser then study the functions~$Z(T,\xi)$ separately: since, as explained above, only a finite number of places contribute, those are \textit{finite} products of local factors (whereas the decompositions~$H(x) = \prod_{v}H_v(x)$ in Chambert-Loir and Tschinkel's work had an infinite number of factors not equal to 1) which can be rewritten as motivic integrals on the arc space of the variety~$\mathcal{X}$. These integrals take the form of motivic Igusa zeta functions, the study of which goes back to Denef and Loeser~(\cite{DL98}). By a method analogous to Chambert-Loir and Tschinkel's analysis, the reduction to the residue field being replaced by the reduction of the arc space to the special fibre, Chambert-Loir and Loeser prove that each factor is a rational function. Thus, in this situation the motivic height zeta function happens to be rational, and their version of theorem~\ref{main} is stated by describing the denominators. The coefficient at $\LL^{-1}$ is a finite product of motivic volumes in this case. 

Our approach in this text is greatly inspired from Chambert-Loir and Loeser's, but requires nevertheless several major adaptations. First of all, since we impose fewer constraints on the poles of the sections we are counting, those do not live in a fixed finite set any more: the products of local factors which were finite in Chambert-Loir and Loeser's work have therefore no reason to be finite in our setting, and we need to make sense of motivic analogues of the infinite products $H(x) = \prod_{v}H_v(x)$ used by Chambert-Loir and Tschinkel. Moreover, the characteristic function of the sections parametrised by the space~$M_d$ is in our case a complicated adelic function, and a direct application of Hrushovski and Kazhdan's motivic Poisson summation is a priori not possible. Last but not least, the presence of infinite products opens the question of their convergence, which will require the introduction of a suitable topology on the Grothendieck rings involved.

\section{Sketch of proof}
\subsection{Motivic Euler products}
It is well known that the Riemann zeta function
$$\zeta(s) = \sum_{n\geq 1} \frac{1}{n^s}$$
has an \textit{Euler product} decomposition
$$\zeta(s) = \prod_{p}(1 - p^{-s})^{-1} = \prod_{p} \left( 1 + p^{-s} + p^{-2s} + \ldots \right),$$
where the product is over the set of all prime numbers. Each factor, separately, converges for $\Re(s) >0$, but the product converges only for $\Re(s)>1$. This property of Euler product decomposition is true more generally for Dirichlet series
$$\sum_{n\geq 1} a(n)n^{-s} = \prod_{p} (1 + a(p)p^{-s} + a(p^2)p^{-2s} +\ldots)$$
where $a:\N\to \C$ is a multiplicative function. Another, more geometric example is given by the zeta function of a variety~$X$ over a finite field~$\F_q$: 
$$\zeta_X(s) := \exp \left(\sum_{m\geq 1} \frac{|X(\F_{q^m})|}{m}q^{-ms}\right).$$
Denoting by $X_{cl}$ the set of closed points of the variety~$X$, the function $\zeta_X(s)$ can indeed be written in the form of an infinite product 
\begin{equation}\label{produiteulerien}\zeta_X(s) = \prod_{x\in X_{cl}}(1-q^{-s\deg x})^{-1}.\end{equation}
Here, as for the Riemann zeta function, each local factor converges for~$\Re(s)>0$, but the series~$\zeta_X(s)$ converges only for $\Re(s)>\dim X$: taking the product shifts the abscissa of convergence by the dimension of the scheme over the closed points of which the product is taken. The example of Riemann's zeta function can also be understood in this way, if we see the set of prime numbers as the set of closed points of the arithmetic scheme~$\spec(\Z)$ of dimension $1$.

The method of proof described in section \ref{sect.maninharmonique} for  Manin's problem over number fields shows that the two main tools to tackle it are the possibility to decompose a function into an infinite product of local factors, and Fourier analysis. Chapter~\ref{eulerproducts} of this text introduces a notion of \textit{motivic Euler product} which gives a meaning, for any quasi-projective variety~$X$ and any family~$\scr{X} = (X_i)_{i\geq 1}$ of quasi-projective varieties over~$X$ (or, more generally, of classes in a relative Grothendieck ring over~$X$), to a product of the form
\begin{equation}\label{intro.eulerproduct}\prod_{x\in X}\left(1 + X_{1,x}t + X_{2,x}t^2 + \ldots \right),\end{equation}
where each $X_{i,x}$ may be seen as the fibre of~$X_i$ above a point $x\in X$. By analogy with the above examples coming from number theory, one can think of the variable~$t$ as corresponding to~$q^{-s\deg x}$, at least when the field~$k$ is algebraically closed.

To define~(\ref{intro.eulerproduct}), we start by constructing the coefficients of the series that should be the expansion of such a product. When we try to expand~(\ref{intro.eulerproduct}) naïvely, we observe that any contribution to the coefficient of degree~$n$ comes from choosing a certain term in each factor in a way that the sum of the degrees of the chosen terms is~$n$, which induces a certain partition of the integer~$n$. We construct the part of the coefficient of degree~$n$ corresponding to a fixed partition~$\pi$ of~$n$ separately. Writing $\pi=(n_i)_{i\geq 1}$ where~$n_i$ is the number of occurrences of the integer~$i$ in the partition~$\pi$, so that $\sum_{i\geq 1}in_i = n$, we define the \textit{symmetric product} $S^{\pi}\scr{X}$ of the family~$\scr{X}$ in the following way: first of all, since we want to construct the part of the coefficient of degree~$n$ corresponding to partition~$\pi$, we need to choose each term $X_{i,x}t^i$ in exactly~$n_i$ factors, which leads us to consider the product
\begin{equation}\label{intro.product}\prod_{i\geq 1}X_i^{n_i}.\end{equation}
On the other hand, these terms have been chosen in distinct factors, that is, factors corresponding to distinct points $x\in X$. Thus, considering the morphism 
$$\prod_{i\geq 1}X_i^{n_i}\to \prod_{i\geq 1}X^{n_i}$$
induced by the structural morphisms~$X_i\to X$, we have to restrict to points of the product~(\ref{intro.product}) with image in~$\prod_{i\geq 1}X^{n_i}$ having all its coordinates distinct, that is, lying in the complement of the big diagonal. Denoting by
$$\left( \prod_{i\geq 1}X_i^{n_i}\right)_{*,X}$$
the open subset defined in this way, it remains to observe that the factors of the above Euler product do not come in any particular order, which prompts us to take the quotient by the natural permutation action of the product of symmetric groups $\prod_{i\geq 1}\Sym_{n_i}$. We therefore put
$$S^{\pi}\scr{X} = \left( \prod_{i\geq 1}X_i^{n_i}\right)_{*,X}/\prod_{i\geq 1}\Sym_{n_i},$$
which exists as a variety under the above quasi-projectivity assumptions. The Euler product~(\ref{intro.eulerproduct}) will thus be introduced in section~\ref{eulerprod} of chapter~\ref{eulerproducts} first as a~\textit{notation} for the series
$$1 + \sum_{n\geq 1} \left(\sum_{\substack{\pi\ \text{partition}\\ \text{of}\ n}}[S^{\pi}\scr{X}]\right)t^n\in \kvar_{k}[[t]].$$
By showing various properties of the geometric construction we just described, we will indicate how one can do computations with this notion of Euler product. For example, we will prove the multiplicativity property
$$\prod_{x\in X}\left(1 + X_{1,x}t + X_{2,x}t^2 + \ldots \right)$$
$$= \prod_{x\in U}\left(1 + X_{1,x}t + X_{2,x}t^2 + \ldots \right)\prod_{x\in Y}\left(1 + X_{1,x}t + X_{2,x}t^2 + \ldots \right)$$
for any closed subscheme~$Y$ of~$X$ with open complement~$U$, which shows that we indeed have defined something which behaves like a product. 
%Nous définirons les produits eulériens dans un cadre beaucoup plus général

An important example of a series with coefficients in  $\kvar_k$ is Kapranov's zeta function, introduced by Kapranov in~\cite{Kapr}. For a quasi-projective variety~$X$ sur $k$, we denote for every $n\geq 0$ by~$S^nX$ its $n$-th symmetric power  $X^n/\Sym_n$, which is also a quasi-projective variety, and we define
$$Z_X(t) = \sum_{n\geq 0} [S^nX] t^n \in \kvar_k[[t]].$$
It is the motivic analogue of the function~$\zeta_X(s)$ considered above, in the sense that, for finite~$k$ it specialises to the latter via the counting measure. Its decomposition as a motivic Euler product is given by 
$$Z_X(t) = \prod_{x\in X}(1 + t + t^2 + \ldots ) = \prod_{x\in X} \frac{1}{1-t},$$
which is the motivic analogue of the Euler product decomposition~(\ref{produiteulerien}) of $\zeta_X(s)$. Note that writing~$Z_X(t)$ in this way was already within the scope of the notion of \textit{motivic power} due to Gusein-Zade, Luengo and Melle (\cite{gusein}), of which our Euler products are a generalisation. 

\subsection{Hrushovski and Kazhdan's Poisson summation}
Hrushovski and Kazhdan's Poisson summation formula, proved in~\cite{HK}, is a motivic analogue of (a weakened version of) the Poisson formula over the adeles of a number field which intervenes in the proof of Chambert-Loir and Tschinkel's theorem explained in section~\ref{sect.maninharmonique}. First of all, to be able to do Fourier analysis in a motivic setting, we need to work in a larger Grothendieck ring, the Grothendieck ring of varieties with exponentials~$\evar_k$ of the field~$k$. It is given as the quotient of the free abelian group on isomorphism classes of pairs~$(X,f)$ with~$X$ a variety over~$k$ and $f:X\to \A^1$ a morphism, by cut-and-paste relations similar to those of the classical Grothendieck ring~$\kvar_k$, as well as by the additional relation
\begin{equation}\label{relationsuppl}(X\times \A^1,\pr_2)\end{equation}
for every variety~$X$ over~$k$, with $\pr_2:X\times \A^1\to \A^1$ the second projection. One can also define a product on~$\evar_k$ and in the case where the field~$k$ is finite, the counting measure on~$\kvar_k$ extends, for any non-trivial character $\psi:k\to\C^*$ of the field $k$, to a motivic measure
\begin{equation}\label{expmotmeasure}\begin{array}{ccc}\evar_k& \to & \C\\
                       \left[X,f\right]   & \to & \sum_{x\in X(k)}\psi(f(x)),
\end{array}\end{equation}
with relation~(\ref{relationsuppl}) being the one translating the fact that, the character~$\psi$ being non-trivial, we have
$$\sum_{x\in k}\psi(x) = 0,$$
a property which is essential to make Fourier analysis work. The ``motivic'' functions that we will consider will have values in the ring~$\evar_k$, or rather in its localisation~$\expp_k$ obtained by inverting the class $[\A^1,0]$. The natural map $\kvar_k\to \evar_k$ given by $[X]\mapsto [X,0]$ is an injective ring morphism.

Though it holds under much more general hypotheses, the classical Poisson formula on the adeles of a global field~$F$ is often stated for~\textit{Schwartz-Bruhat functions}. These are linear combinations of functions $f:\Ad_F\to \C$ which can be written as a product
$$f = \prod_{v}f_v$$
such that for every place~$v$, $f_v$ is a local Schwartz-Bruhat function $F_v\to \C$ on the completion~$F_v$ of the global field~$F$ at~$v$ (i.e. smooth and rapidly decreasing if~$v$ is archimedean, locally constant and compactly supported if~$v$ is non-archimedean), equal to the characteristic function~$\1_{\OO_v}$ of the ring of integers of~$F_v$ for almost all non-archimedean places. Hrushovski and Kazhdan define a geometric analogue of a Schwartz-Bruhat function on adeles of a function field (where there are only non-archimedean places). 

For a local Schwartz-Bruhat (that is, locally constant and compactly supported) function $f:F\to \C$ defined on a non-archimedan local field~$F$ (for which we denote by~$\OO$ its ring of integers,~$t$ a uniformiser and~$k$ its residue field), by local compactness we can find integers $M,N\geq 0$ such that~$f$ is zero outside of $t^{-M}\OO$, and invariant modulo the subgroup $t^N\OO$. As a consequence, such a function may be seen as a function on the quotient $t^{-M}\OO/t^{N}\OO$, which happens to be a $k$-vector space of dimension $M+N$. Thus, in the motivic setting, a local Schwartz-Bruhat function (of level $(-M,N)$) will be a function defined on an affine space $\A_k^{M+N}$ (denoted by $\A_k^{(-M,N)}$ to keep track of the values of~$M$ and~$N$) and with values in the ring~$\expp_k$.  More precisely, such functions are introduced as elements of the \textit{relative} Grothendieck ring $\expp_{\A_k^{(-M,N)}}$.

This construction can also be performed to produce a motivic analogue of locally constant and compactly supported functions defined on a finite product of local fields (rather than only one local field): these are Hrushovski and Kazhdan's motivic Schwartz-Bruhat functions, and their Poisson summation formula holds for these functions. Thus, this formula is the analogue of the Poisson summation formula over the adeles of a number field~$F$ for classical Schwartz-Bruhat functions.%, which correspond to Schwartz-Bruhat functions $f:\Ad_F\to \C$ with restriction to $F_v$ equal to the characteristic function of $\OO_v$ for almost all~$v$. 

As we have mentioned above, this formula was sufficient for the purposes of Chambert-Loir and Loeser's work~\cite{CL}, because  the height function for the  counting problem they considered satisfied this very restrictive hypothesis. In the general case we find ourselves in, this is no longer the case, the poles of the sections we are counting are not contained in some fixed finite set, and we therefore need to show how we may apply Hrushovski and Kazhdan's Poisson formula in families, making the locus of the poles of sections vary, which is done using the above notion of symmetric product. 

To explain this in a simple special case, let us consider, for every integer~$i\geq 1$ and for any integers $M_{i}, N_i\geq 0$ the variety
$$\A_C^{(-M_i,N_i)} := C\times \A_k^{(-M_i,N_i)}$$
above the curve~$C$. We can define, for any integer~$m\geq 0$ the symmetric product
$$S^{m}((\A_C^{(-M_i,N_i)})_{i\geq 1})$$
of the family $\left(\A_C^{(-M_i,N_i)}\right)_{i\geq 1}$. This symmetric product is naturally endowed with a morphism to the symmetric power~$S^{m}C$. We observe that for every effective zero-cycle $D = \sum_{v}m_vv\in S^{m}C(k)$, the fibre of $S^{m}((\A_C^{(-M_i,N_i)})_{i\geq 1})$ above $D$ is of the form
$$\prod_{v\in C}\A_k^{(-M_{m_v},N_{m_v})},$$
and that it therefore may be seen as the domain of definition of a motivic Schwartz-Bruhat function, with support and invariance controlled by the zero-cycles
$$-\sum_{v}M_{m_v}v\ \ \ \text{et}\ \ \ \sum_{v}N_{m_v} v.$$
Thus, making the zero-cycle~$D$ vary, we can in this way parametrise functions with varying supports (if the $M_i$ are sufficiently large) and varying invariance domains (if the $N_i$ are sufficiently large). The general construction in chapter \ref{poissonformula} is slightly more elaborate because it allows additional parameters, but the main idea is exactly the same. We then show that all the operations of Hrushovski and Kazhdan's theory can be performed in families over the base $S^mC$ for every $m\geq 1$, and check in particular the validity of the Poisson formula in this setting.

\subsection{Weight filtration and convergence}
Let us now go back to the motivic height zeta function~(\ref{def.fonctionzeta}). Thanks to a decomposition of the moduli spaces~$M_d$ according to the values of the poles and the zeroes of the sections they  parametrise, we can apply to $Z(T)$ our generalised Poisson formula from the previous paragraph. This enables us, analogously to what has been explained in section~\ref{sect.maninharmonique}, to rewrite the motivic height zeta function in the form
\begin{equation}\label{zetafunctionpoisson}Z(T) = \sum_{\xi\in k(C)^n} Z(T,\xi)\end{equation}
for some series~$Z(T,\xi)$ with coefficients in~$\expp_{k}$, each having an Euler product decomposition.  We now arrive to the question of the convergence of these Euler produts: as in Chambert-Loir and Tschinkel's work, we wish to prove that the series~$Z(T,0)$ is the only term in the right-hand side of~(\ref{zetafunctionpoisson}) responsible for the first pole of the function~$Z(T)$ at~$\LL^{-1}$, and to extend it meromorphically beyond this pole. We mentioned earlier that the dimensional filtration on the Grothendieck ring of varieties will not give us the expected convergence for the function~$Z(T)$. To understand this, let us go back to the bounds established by Chambert-Loir and Tschinkel: for almost all local factors, their calculations feature differences of the form  $D_{\al}(\mathbf{F}_q) - q^{n-1}$ for each irreducible component~$D_{\al}$ of the divisor at infinity~$D$, $q$ a prime power, and $n$ the dimension of~$X$. To obtain the desired convergence, it is crucial to bound those using the Lang-Weil estimates~\cite{LW}:
\begin{equation}\label{LangWeil.ineq}\left\vert D_{\al}(\mathbf{F}_q) - q^{n-1}\right\vert\leq cq^{n -\frac{3}{2}},\end{equation}
for some constant $c>0$. In the motivic setting, the calculations are completely analogous, and we therefore find ourselves naturally with the same kind of difference, namely $$[D_{\al}] - \LL^{n-1}\in \M_{k}.$$ In general, the dimension of this element of~$\M_k$ is $n-1$ : for example, if $D_{\al}$ is a curve of $g\geq 1$ (in particular, then $n=2$), then the Hodge-Deligne polynomial
$$HD([D_{\al}] - \LL^{n-1}) = (1 + gu + gv + uv) - uv = 1 + gu + gv,$$ is of degree 1, which shows that $[D_{\al}] - \LL^{n-1}$ cannot be a linear combination of elements of dimension $\leq 0$. The dimensional filtration therefore gives a weaker bound than the one obtained via the Lang-Weil estimates in the arithmetic case, which leads us to use a finer topology on the Grothendieck ring of varieties, based on Hodge theory.

Let us temporarily assume that the base field $k$ is the field $\C$ of complex numbers. There exists a motivic measure
$$\chi^{\hdg}: \M_{\C}\to K_0(HS),$$
called the \textit{Hodge realisation}, with values in the Grothendieck ring of Hodge structures, which to a complex variety~$Y$ associates
\begin{equation}\label{defchi}\chi(Y) = \sum_{i=0}^{2\dim Y} (-1)^i[H^{i}_c(Y(\C),\Q)],\end{equation}
where $[H^{i}_c(Y(\C),\Q)]$ is the class in $K_0(HS)$ of the mixed Hodge structure on the $i$-th singular cohomology group with compact supports of~$Y(\C)$. There is a natural increasing weight filtration $(W_{\leq n}K_0(HS))_{n\in \Z}$ on $K_0(HS)$, given by defining
$W_{\leq n}K_0(HS) $ as the subgroup of~$K_0(HS)$ generated by classes of pure Hodge structures of weights~$\leq n$. For an element $\a\in\M_{\C}$, we define its weight by
$$w(\a) = \inf\{n\in \Z, \chi^{\hdg}(\a) \in W_{\leq n}K_0(HS)\}.$$
The formula (\ref{defchi}) giving $\chi^{\hdg}$ then implies directly that for the class of the variety~$Y$,  we have
$w(Y) = 2\dim Y$. More precisely, we have
$$H^{2\dim Y} (Y(\C),\Q) \simeq \Q(-\dim Y)^{\kappa(Y)},$$
where $\kappa(Y)$  is the number of irreducible components of maximal dimension of~$Y$ and $\Q(-\dim Y)$ is the unique pure Hodge structure of weight~$2\dim Y$ and dimension~1. Thus, we observe that in the expression of $$\chi^{\hdg}([D_{\al}] - \LL^{n-1})$$ the terms corresponding to cohomology groups of maximal degree cancel out, and that as a consequence, 
$$w([D_{\al}] - \LL^{n-1}) \leq 2n-3 = 2\left(n-\frac32\right).$$
This bound is the analogue, in the motivic setting, of inequality (\ref{LangWeil.ineq}). 

%on a besoin de tout cela dans le cadre relatif, d'où utilisation des modules de Hodge mixtes
\subsection{Motivic vanishing cycles}
Recall that in the decomposition~(\ref{zetafunctionpoisson}) given by the motivic Poisson formula,  the series~$Z(T,\xi)$ have coefficients in the Grothendieck ring of varieties with exponentials~$\expp_{\C}$ (we still assume~$k=\C$). Thus, we need to extend the weight topology, defined on~$\M_{\C}$ in the previous paragraph, to the ring~$\expp_{\C}$. Moreover, it would be advisable, to preserve  the analogy with the arithmetic case, for such an extended weight function to satisfy the \textit{triangular inequality}: for any complex variety~$X$ with a morphism $f:X\to \A^1$, we would like that
\begin{equation}\label{eq.triangular}w([X,f])\leq w([X]).\end{equation}
Via the motivic measure~(\ref{expmotmeasure}) on the Grothendieck ring of varieties with exponentials, we see that inequality~(\ref{eq.triangular}) is the motivic analogue of the inequality
$$\left\vert\sum_{x\in X(\F_q)}\psi(f(x)) \right\vert\leq \# X(\F_q)$$
coming from the classical triangular inequality,
for a variety~$X$ over~$\F_q$, a morphism $f:X\to \A^1$ and a non-trivial character $\psi:\F_q\to \C^*$. 

The solution to this problem uses the \textit{motivic vanishing cycles} of Denef and Loeser. For a smooth variety~$X$ over a field~$k$ of characteristic zero, and a morphism $f:X\to \A^1_{k}$, the motivic vanishing cycles~$\phi_{f,a}$ of $f$ at~$a\in k$ are an element of the localised Grothendieck ring $\M_{f^{-1}(a)}^{\hat{\mu}}$ of varieties above the fibre $f^{-1}(a)\subset X$, with action of the group~$\hat{\mu}$, which is the projective limit of the groups of $n$-th roots of unity for every~$n\geq 1$. When the fibre $f^{-1}(a)$ is nowhere dense in~$X$, Denef and Loeser give a formula for computing~$\phi_{f,a}$ in terms of the components of the exceptional divisor of a log-resolution of the pair $(X,f^{-1}(a))$. 

 Using the works of Denef and Loeser on this subject, as well as Guibert, Loeser and Merle's paper~\cite{GLM}, Lunts and Schnürer proved in~\cite{LS} that in the case where~$k$ is algebraically closed of characteristic zero, combining the motivic vanishing cycles at all the points of $k$, one could define a ring morphism, called motivic vanishing cycles measure
$$\Phi:\expp_{k}\to (\M_{k}^{\hat{\mu}},\ast)$$
where $\ast$ is a \textit{convolution product}, the definition of which is due to Looijenga, Denef and Loeser. For a class $[X,f]$ with~$X$ smooth and $f:X\to \A^1_{k}$ proper, we have
$$\Phi([X,f]) = \sum_{a\in k}f_!\phi_{f,a}.$$
The main ingredient in this proof is a Thom-Sebastiani theorem for motivic vanishing cycles, due to Denef and Loeser, which is the motivic analogue of the corresponding theorem in the classical theory of vanishing cycles. In chapter~\ref{grothrings} of this text, we extend the definition of the measure~$\Phi$ to~$k$ of characteristic zero but not necessarily algebraically closed, then above an arbitrary $k$-variety~$S$.

Motivic vanishing cycles thus give us a way to go from $\expp_{\C}$ to $\M_{\C}^{\hat{\mu}}$. To extend the weight filtration to the ring $\expp_{\C}$, we then use the fact that the Hodge realisation mentioned above generalises to a morphism
\begin{equation}\label{realhodge}\M_{\C}^{\hat{\mu}}\to K_0(HS^{\mathrm{mon}})\end{equation}
from the Grothendieck group of complex varieties with $\hat{\mu}$-action, to the Grothendieck group of Hodge structures with action of a linear finite order operator (the \textit{monodromy} operator). This morphism becomes a ring morphism when both groups are endowed with appropriate products. Composing this generalised Hodge realisation with the morphism~$\Phi$ then amounts to looking at the Hodge structure on the vanishing cycles (in the classical sense), with the natural action of the monodromy around the given point. We then extend the weight filtration to the group $K_0(HS^{\mathrm{mon}})$, defining $W_{\leq n}K_0(HS^{\mathrm{mon}})$ for every $n\in\Z$ as the subgroup generated by the pure Hodge structures of weight $\leq n$ and trivial monodromy, as well as by the pure Hodge structures of weight $\leq n-1$ with non-trivial monodromy. Using Denef and Loeser's formula, one can see that the triangular inequality is then satisfied.

The last difficulty which comes up is the question of studying convergence of Euler products using the weight filtration, for which we need to use the above results in the relative setting: we will therefore use our motivic vanishing cycles measure above a base variety $S$, as well as the theory of mixed Hodge modules of Saito in chapter~\ref{hodgemodules} of this text.
\chapter{Grothendieck rings of varieties and motivic vanishing cycles}\label{grothrings}

The \textit{Grothendieck group of varieties} was defined by Grothendieck in a letter to Serre dated 16 August 1964, in which he also developed the idea of motives. He wrote:
\selectlanguage{french}
\begin{quote} Soit $k$ un corps, algébriquement clos pour fixer les idées, et soit $\mathrm{L}(k)$ le \og groupe~$\mathbf{K}$ \fg{} défini par les schémas de type fini sur $k$, avec comme relations celles qui proviennent d'un découpage en morceaux (\ldots). 
\end{quote}
\selectlanguage{english}
Though Grothendieck does not mention it in the above quotation, this group is in fact a ring, which we will denote by $\kvar_k$, the product being given by the product of varieties over $k$. The ring $\kvar_k$ became particularly prominent when, in a celebrated lecture in 1995 at Orsay, Kontsevich sketched the basics of \textit{motivic integration}, a theory of integration with values in this ring, and showed how one could use it to prove that birationally equivalent Calabi-Yau varieties had the same Hodge numbers. This gave a generalisation of a result of Batyrev stating equality of Betti numbers under the same conditions. Batyrev's proof relied on reduction to positive characteristic and $p$-adic integration, and Kontsevich's groundbreaking idea consisted in noticing that one could replace the latter by a more geometric theory of integration, with values in the Grothendieck ring of varieties, and thus avoid the former. Moreover, the identities one obtains are valid in a suitable completion of the Grothendieck ring $\kvar_{\C}$ (localised at the class of the affine line), which gives rise to equality of more precise geometric invariants, e.g. Hodge numbers as stated above. 

Since then, several theories of motivic integration have been constructed, and have found a wide range of applications. The first names to quote here are those of Denef and Loeser, who formalised Kontsevich's idea and extended it to singular varieties in \cite{DL99mot}. They also gave numerous applications of motivic integration to singularity theory, generalising to the motivic framework many results obtained previously via $p$-adic integration. In particular, they studied motivic versions of Igusa zeta functions in \cite{DL98}, which provide a notion of motivic nearby fibre and motivic vanishing cycles (\cite{DL98,DL99,DL01,DL02}). Other theories of motivic integration, based on model theory of valued fields, are due to Cluckers-Loeser (\cite{CluckLos, CluckLosexp}) and Hrushovski-Kazhdan (\cite{HKint}): they allow to encompass integrals with parameters, as well as integrals involving additive characters, opening the path to a motivic version of Fourier analysis. The key idea for the latter is to work in larger Grothendieck rings called \textit{Grothendieck rings with exponentials}.

Section \ref{sect.grothrings} of this chapter contains definitions and properties of various Grothendieck rings of varieties, in particular Grothendieck rings with exponentials and Grothendieck rings with $\hat{\mu}$-actions. Section \ref{sect.motvancycles} introduces Denef and Loeser's notion of motivic nearby and vanishing cycles, generalises them to the relative setting and recalls a result due to Guibert, Loeser and Merle from \cite{GLM} which combines various nearby cycles over a fixed variety~$X$ over a field of characteristic zero into a group morphism defined on the Grothendieck ring of varieties over $X$. Section \ref{sect.motvanmeasure} recalls Denef, Loeser and Looijenga's Thom-Sebastiani theorem for motivic vanishing cycles together with the convolution product $\ast$ it involves, and uses it to construct a ring morphism
$$\expp_X\to (\M_X^{\hat{\mu}},*)$$
from the Grothendieck ring of varieties with exponentials over $X$ to the Grothendieck ring of varieties with $\hat{\mu}$-action over $X$ localised at the class of the affine line, endowed with this product. The construction for $X = \spec k$ with $k$ an algebraically closed field of characteristic zero is due to Lunts and Schnürer \cite{LS}. Finally, section \ref{sect.TSexample} contains an explicit computation checking the Thom-Sebastiani theorem for vanishing cycles of the morphism $\A^1\to \A^1$, $x\mapsto x^2$.

% In this chapter, we are going to introduce Grothendieck rings of varieties, as well as Grothendieck rings of varieties with exponentials, which will be the main tool for doing Fourier analysis in the motivic setting. We will then go on to prove some statements about families of generators of these rings. The Grothendieck ring of varieties $\M_S$ over a variety $S$ defined over a field of characteristic zero is naturally a subring of the ring $\expp_S$ with exponentials. In the last section of this chapter,  we will define a motivic measure providing a retraction 

%It is due to works by Denef-Loeser, Guibert-Loeser-Merle and Lunts-Schnürer in the case where 

\section{Grothendieck rings of varieties} \label{sect.grothrings}
References for this section are e.g. \cite{CNS} for most classical definitions and properties of Grothendieck rings, and Hrushovski and Kazdan's \cite{HK}, or Cluckers and Loeser's \cite{CluckLosexp} for Grothendieck rings with exponentials. We follow mostly Chambert-Loir and Loeser's \cite{CL} containing a short introduction to Grothendieck rings with exponentials and their main properties.
\subsection{Grothendieck semirings}
Let $R$ be a noetherian scheme. By a variety over $R$ we mean a $R$-scheme of finite type. For a point $r\in R$, we denote by $\kappa(r)$ its residue field. 

The \emph{Grothendieck monoid of varieties} \index{Grothendieck monoid!of varieties} $\svar_R$ \index{Kvarp@$\svar$}
over $R$ is a commutative monoid defined by generators and relations. Generators are all varieties over $R$, and relations are
$$X\sim Y$$
whenever $X$ and $Y$ are isomorphic as $R$-varieties, 
$$\varnothing \sim 0$$
and
$$X  \sim  Y + U$$
whenever $X$ is a $R$-variety, $Y$ a closed subscheme of $X$ and $U$ its open complement. We will write $[X]$ to denote the class of a variety $X$ in $\svar_R$.  The product $[X][Y] = [X\times_{R} Y]$ endows the monoid~$\svar_{R}$  with a semiring  structure. Two $R$-varieties $X$ and $Y$ have the same class in $\svar_R$ if and only if they are piecewise isomorphic \index{piecewise isomorphic} over $R$, that is, one can partition them into locally closed subsets $X_1,\dots,X_m$ and $Y_1,\ldots,Y_m$ such that for all $i$, $X_i$ and $Y_i$ are isomorphic as $R$-schemes with their induced reduced structures (see \cite{CNS}, Chapter 1, Corollary 1.4.9).

The \emph{Grothendieck monoid of varieties
with exponentials} \index{Grothendieck monoid!of varieties with exponentials}$\sevar_R$ \index{Kexpvarp@$\sevar$}is defined by generators and relations as well. 
Generators are pairs $(X,f)$,
where $X$ is a variety over $R$ 
and $f\colon X\to\A^1=\spec(\Z[T])$ is a morphism. Relations are the following:
$$(X,f)\sim (Y,f\circ u) $$
whenever $X$, $Y$ are $R$-varieties, $f\colon X\to\A^1$
a morphism, and $u\colon Y\to X$ an $R$-isomorphism;
$$(\varnothing,f)\sim 0$$
where $f:\varnothing \to \A^1$ is the empty morphism;
$$ (X,f)\sim (Y,f|_Y)+(U,f|_U) $$
whenever $X$ is an $R$-variety, $f\colon X\to\A^1$ a morphism,
$Y$ a closed  subscheme of~$X$ and $U=X\setminus Y$ its open complement;
$$ (X\times_\Z \A^1,\pr_2)$$
where $X$ is an $R$-variety and $\pr_2$ is the second projection.
We will write~$[X,f]$ to denote the class in $\sevar_R$ 
of a pair~$(X,f)$. The product $[X,f][Y,g] = [X\times_R Y, f\circ \pr_1+g\circ \pr_2]$ endows $\sevar_R$ with a semiring structure, the class $[R\xrightarrow{\id}R]$ being the unit element.

\subsection{Grothendieck rings}\label{subsect.grothrings}
Let $R$ be a noetherian scheme. 

The \emph{Grothendieck group of varieties} \index{Grothendieck group!of varieties} $\kvar_{R}$ \index{Kvar@$\kvar$} is defined by generators and relations. Generators are all varieties over $R$, and relations are
$$X-Y$$
whenever $X$ and $Y$ are isomorphic as $R$-varieties, and
$$X  - Y - U$$
whenever $X$ is a $R$-variety, $Y$ a closed subscheme of $X$ and $U$ its open complement. It is the group associated to the monoid $\svar_R$. Every constructible set~$X$ over~$R$ (that is, every constructible subset of a variety over $R$) has a class~$[X]$ in the group~$\kvar_R$ (see \cite{CNS}, chapter 1, corollaries 1.3.5 and 1.3.6). It will sometimes be denoted by $[X]_R$ when different base schemes are in play. The product $[X][Y] = [X\times_{R} Y]$ endows the group~$\kvar_{R}$  with a ring  structure with unit element the class $1 = [R\xrightarrow{\id}R]$. 
Let $\LL$, or $\LL_R$, \index{L@$\LL$} be the class of the affine line $\A_{R}^1$ in $\kvar_{R}$. We define the Grothendieck ring of varieties \index{Grothendieck ring!of varieties} localised at $\LL$ to be $\M_{R} = \kvar_{R}[\LL^{-1}]$.

The \emph{Grothendieck group of varieties
with exponentials} \index{Grothendieck group!of varieties!with exponentials} $\evar_R$ \index{Kexpvar@$\evar$} is defined by generators and relations as well. 
Generators are pairs $(X,f)$,
where $X$ is a variety over $R$ 
and $f\colon X\to\A^1=\spec(\Z[T])$ is a morphism. Relations are the following:
$$(X,f)-(Y,f\circ u) $$
whenever $X$, $Y$ are $R$-varieties, $f\colon X\to\A^1$
a morphism, and $u\colon Y\to X$ an $R$-isomorphism;
$$ (X,f)-(Y,f|_Y)-(U,f|_U) $$
whenever $X$ is a $R$-variety, $f\colon X\to\A^1$ a morphism,
$Y$ a closed  subscheme of~$X$ and $U=X\setminus Y$ its open complement,
$$ (X\times_\Z \A^1,\pr_2)$$
where $X$ is an $R$-variety and $\pr_2$ is the second projection.
We will write~$[X,f]$ (or $[X,f]_R$ if we want to keep track of the base scheme $R$) to denote the class in $\evar_R$ 
of a pair~$(X,f)$. The product $[X,f][Y,g] = [X\times_R Y, f\circ \pr_1+g\circ \pr_2]$ endows $\evar_R$ with a ring structure.\index{Grothendieck ring!of varieties!with exponentials} We will use the notation $f\oplus g$ for the morphism  $f\circ \pr_1+g\circ \pr_2$. 
We denote by $\LL$, or $\LL_R$,\ \index{L@$\LL$} the class of $[\A^1_R,0]$ in $\evar_R$. As for $\kvar_R$, we may invert $\LL$, which gives us a ring denoted by $\expp_R$.

There are obvious morphisms of semirings
$$\svar_R\to \kvar_R$$ and $$\sevar_R\to \evar_R$$
which identify $\kvar_R$ (resp. $\evar_R$) with the ring obtained from the semiring $\svar_R$ (resp. $\sevar_R$) by adding negatives.  An element in the image of one of these morphisms is said to be \textit{effective}.

There are ring morphisms
$\kvar_R \to \evar_R$ and $\M_R\to \expp_R$ 
sending
the class of~$X$ to the class~$[X,0]$. According to lemma 1.1.3 in \cite{CL} together with lemma \ref{function.equality} below, they are injective. 

Let $X$ be an $R$-variety. A piecewise morphism \index{piecewise morphism} $f:X\to \A^1$ is the datum of a partition $X_1,\ldots,X_m$ of $X$ into locally closed subsets and of morphisms $f_i:X_i\to \A^1$ for every~$i$. Any pair $(X,f)$ consisting of an $R$-variety~$X$ and of a piecewise morphism $f\colon X\to\A^1$
has a class $[X,f]$ in $\evar_R$. 

\begin{remark} The localisation morphism $\kvar_R\to \M_R$ \index{localisation morphism} is not injective in general: it was proved by Borisov in \cite{Borisov} that $\LL$ is a zero-divisor in the Grothendieck ring $\kvar_k$ over a field~$k$ of characteristic zero. Borisov's argument also implies that two varieties $X$ and $Y$ having the same class in the Grothendieck ring of varieties are not necessarily piecewise isomorphic, so that the above morphism $\svar_R\to \kvar_R$ is not injective either. %A préciser
\end{remark}

Let $A\in\{\kvar,\M,\evar,\expp\}$. For noetherian schemes $R$ and $S$ over some noetherian scheme $T$, there is an \textit{exterior product} \index{exterior product} morphism
\begin{equation}\label{exteriorproduct}A_R\otimes_{A_T} A_S \xrightarrow{\boxtimes_T} A_{R\times_T S}\end{equation}
of $A_T$-algebras. For $A = \evar$, it is given by sending a pair $([X,f],[Y,g])$ to the class $[X\times_T Y, f\circ\pr_1 + g\circ\pr_2]$. 

\subsection{Functoriality and interpretation as functions}\label{sect.functoriality}

Let $R$ and $S$ be noetherian schemes.  A morphism $u:R\to S$ induces morphisms $u_{!}$ and~$u^*$ between the corresponding Grothendieck groups. For example, for $\evar$, these morphisms are defined in the following manner: the \textit{proper pushforward} \index{proper pushforward}
$$u_{!}:\evar_R\to \evar_S$$
sends a class $[X,f]_R$ of a $R$-variety $X$ in $\evar_R$ to the class $[X,f]_S$ of the pair $(X,f)$ with $X$ viewed as an $S$-variety through the morphism $u$.  This is a morphism of rings in the case when $u$ is an immersion, but not in general.

In the other direction, there is a morphism of rings
$$u^*:\evar_S\to \evar_R$$
called the \textit{pullback}, \index{pullback} given by sending the class of a pair $(X,f)$ with $X$ a $S$-variety to the class of the pair $(X\times_S R, f\circ\pr_1)$, where $R$ is viewed as an $S$-variety through the morphism~$u$. In particular, $\evar_S$ may be seen as a $\evar_S$-algebra via this morphism. 

Elements of the Grothendieck rings over $R$ may be interpreted as~motivic functions on~$R$. \index{motivic function} When $r\in R$ is a point, and $\a\in\evar_R$, we will denote by $\a(r) = r^{*}(\a)$ (or $\a_r$) the image of $\a$ in $\evar_{\kappa(r)}$ by the morphism $r^*$ induced by $r:\spec (\kappa(r))\to R$. This will be interpreted as evaluation of the motivic function $\a$ at $r$. More generally, the above morphism $u^*$ is just composition with $u$. As for $u_!$, it may be interpreted as ``summation over rational points'' in the fibres of $u$, as explained in the following paragraph.  We recall Lemma~1.1.8 from \cite{CL}, which, with this functional interpretation, means that a motivic function on~$R$ is determined by its values at points of $R$.

\begin{lemma}\label{function.equality} Let $\a\in\kvar_R$ (resp. $\M_R$, $\evar_R$, $\expp_R$). If $\a(r) = 0$ for every $r\in R$, then $\a = 0$.
\end{lemma}

\subsection{Exponential sum \index{exponential sum} notation}\label{exponentialsumnotation}
We start with the following lemma:
\begin{lemma} Let $k$ be a finite field and let $\psi:k\to \C^*$ be a non-trivial additive character. Then there is a motivic measure 
$$\begin{array}{ccc}\evar_k&\to& \C \\
                         \left[X,f\right] & \mapsto & \sum_{x\in X(k)}\psi(f(x))\end{array}$$
          extending the counting measure 
          $$\begin{array}{ccc}\kvar_{k}&\to& \Z \\
          {[X]}&\mapsto& |X(k)|
          \end{array}$$
\end{lemma}
\begin{proof} We may define a group morphism on the free abelian group of generators $(X,f)$ of $\evar_k$, by sending $(X,f)$ to $\sum_{x\in X(k)}\psi(f(x))$. This passes to the quotient modulo the first relation  defining $\evar_k$ because it corresponds to being able to perform changes of variables in such a sum:
$$\sum_{x\in X(k)}\psi(f(x)) = \sum_{y\in Y(k)}\psi(f\circ u(y))$$
whenever $u:Y\to X$ is an isomorphism over $k$. 
It also passes to the quotient by the second relation because the latter corresponds to cutting up the sum into smaller sums:
$$\sum_{x\in X(k)}\psi(f(x)) =  \sum_{x\in Y(k)} \psi(f(x)) + \sum_{x\in U(k)} \psi(f(x))$$
whenever $Y$ is a closed subscheme of $X$ and $U = X\setminus Y$ its open complement. The third relation is also satisfied, and comes from the fact that if $\psi$ is a non-trivial character, then 
$$\sum_{x\in \A^1_k(k)}\psi(x) = 0.$$
Finally, we observe that the product of two classes $[X,f]$ and $[Y,g]$ is sent to
$$\sum_{x\in X}\sum_{y\in Y}\psi(f(x))\psi(g(y)) = \sum_{(x,y)\in X\times Y}\psi(f(x) + g(y))$$
which is exactly the image of the product $[X\times Y,f\conv g]$ of these classes, so this map is a ring morphism.
\end{proof}
 This lemma is meant as a motivation of the fact that, as explained in \cite{CL} 1.1.9, the class of a pair $(X,f)$ in $\evar_k$ must be thought of as an analogue of the exponential sum
$$\sum_{x\in X(k)}\psi(f(x))$$
where $k$ is a finite field and $\psi:k\to\C^*$ is a non-trivial additive character. This is why, when doing calculations with it, we will denote the class $[X,f]\in \evar_k$, even  for a not necessarily finite field $k$, by $$\sum_{x\in X}\psi(f(x)).$$ As noted in the proof of the lemma, when using this notation, the three relations occurring in the definition of $\evar_k$ translate respectively as the possibility of performing change of variables, additivity and the property that
$$\sum_{x\in k}\psi(x) = 0,$$
the latter being essential to make Fourier analysis work.

More generally, let $R$ be a $k$-variety. For any morphism  $h:R\to\A^1$ and any element $\theta\in \evar_R$, we may define
$$\sum_{r\in R}\theta(r)\psi(h(r)) = \theta \cdot [R,h]$$
where the product is taken in $\evar_R$, and its result is viewed in $\evar_k$. In the case when $h = 0$, the map $\theta\mapsto \sum_{r\in R}\theta(r)$ is exactly $u_{!}$ for $u:R\to k$ the structural morphism. In the same manner, if $u:R\to S$ is a morphism, then for any $\theta\in\evar_R$, the motivic function $u_!\theta$ sends $s\in S$ to $\sum_{r\in u^{-1}(s)}\theta(s)$. 
\subsection{Grothendieck rings of varieties with action}\label{sect.grothaction}
Grothendieck rings with actions were introduced by Denef and Loeser in \cite{DL02} to be able to take into account monodromy actions on the motivic nearby fibre. Other references are \cite{DL01} and \cite{GLM}. 

Let $k$ be a field of characteristic zero. We start by giving some general definitions about group actions on varieties. Let $G$ be a finite group acting on a variety $X$ over $k$. We say that the action of $G$ is \textit{good} \index{good action} if every $G$-orbit is contained in an affine open subset of $X$. If $X$ and $Y$ are two varieties with good $G$-action, we denote by $X\times^GY$ \index{Xtimes@$X\times^GY$} the quotient of the product $X\times Y$ by the equivalence relation $(gx,y)\sim (x,gy)$. It exists as a variety, and there is a good $G$-action on $X\times^GY$ induced by the action of $G$ on the first factor of $X\times Y$.

Let $S$ be a variety over $k$ and $X$ a variety over $S$. When we speak of a good action of~$G$ on the $S$-variety $X$ we require it to leave the fibres of the structural morphism $X\to S$ invariant, i.e., the structural morphism must be equivariant if $S$ is equipped with the trivial $G$-action. 

%Let $X$ be a $k$-variety and $p:A\to X$ an affine bundle for the Zariski topology. A good action of $G$ on $A$ is said to be affine if it is a lifting of a good action on $X$ and if its restriction to all fibers is affine. 

For $n\geq 1$, let $\mu_n = \spec(k[x]/(x^n-1))$ \index{mun@$\mu_n$} be the group scheme of $n$-th roots of unity, and let $\hat{\mu}$ \index{muh@$\hat{\mu}$} be the projective limit $\varprojlim\mu_n$ of the projective system with transition morphisms $\mu_{nd}\to \mu_n$, $x\mapsto x^{d}$. For any integer $r\geq 1$, a good $\hat{\mu}^r$-action \index{good $\hat{\mu}^r$-action} on a variety $X$ is an action of $\hat{\mu}^r$ that factors through a good action of $\mu_n^r$ for some integer $n\geq 1$. 

Let $S$ be a variety over $k$ and $r\geq 1$ an integer. The Grothendieck group of varieties with $\hat{\mu}^r$-action \index{Grothendieck group!of varieties!with $\hat{\mu}^r$-action} $\kvar^{\hat{\mu}^r}_S$ \index{Kvarmu@$\kvar^{\hat{\mu}^r}$} is defined in a similar way to $\kvar_S$: generators are pairs $(X,\sigma)$ where $X$ is a variety over $S$ and $\sigma$ a good $\hat{\mu}^r$-action on $X$, and relations are
$$(X,\sigma)-(Y,\tau)$$
whenever $X$ and $Y$ are $S$-varieties with good $\hat{\mu}^r$-actions $\sigma,\tau$ such that there exists an equivariant $S$-isomorphism $u:(X,\sigma)\to (Y,\tau)$,
$$(X,\sigma) - (Y,\sigma_{|Y}) - (U,\sigma_{|U})$$
where $X$ is a variety over $S$ with good $\hat{\mu}^r$-action $\sigma$, $Y$ a closed $\hat{\mu}^r$-invariant subscheme of~$X$, and $U$ its open complement, as well as an additional relation saying that
\begin{equation}\label{actionrelation}(X\times \A^n_k,\sigma) - (X\times \A^n_k,\sigma')\end{equation}
whenever $\sigma$ and $\sigma'$ are two liftings of the same $\hat{\mu}^r$-action on $X$ to an affine action on $X\times \A^n_k$ (that is, a good action the restriction of which to all fibres of the affine bundle $X\times\A^n_k\to X$ is affine). The fibred product with the diagonal $\hat{\mu}^r$-action endows $\kvar_S^{\hat{\mu}^r}$ with a ring structure. 
\begin{remark} We won't use this product much, because in the context of vanishing cycles, products defined via convolution \index{convolution product} are more relevant, see sections \ref{sect.convolution}, \ref{section.grothaffineline}. 
\end{remark}
The class of the pair $(X,\sigma)$ will be denoted $[X,\sigma]$, or even $[X,\hat{\mu}^r]$ or $[X]$ when it is clear from the context what $\sigma$ is.

We denote by $\LL_S$, or $\LL$, \index{L@$\LL$} the class $\A^1_S$ with trivial $\hat{\mu}^r$-action. The last relation (\ref{actionrelation}) says in particular that $\LL$ is equal to the class of the affine space endowed with any affine $\hat{\mu}^r$-action. 
As for $\kvar_S$, we may define the ring $\M_S^{\hat{\mu}^r} = \kvar_S^{\hat{\mu}^r}[\LL^{-1}]$. \index{Mhat@$\M^{\hat{\mu}^r}$}

\begin{remark}[Trivial action and forgetful morphism] There is a ring morphism
\begin{equation}\label{actioninjection}\kvar_S\to \kvar_{S}^{\hat{\mu}^r}\end{equation}
sending the class of a variety $X$ to the class of $X$ endowed with the trivial action. There is also a forgetful ring morphism
$$\kvar_{S}^{\hat{\mu}^r}\to \kvar_{S}$$
which sends a class $[X,\sigma]$ of a variety~$X$ with action~$\sigma$ to the class~$[X]$, well defined because all relations defining $\kvar_S^{\hat{\mu}^r}$ go to zero in $\kvar_S$. The composition of the two gives the identity of $\kvar_S$, which shows that (\ref{actioninjection}) is injective. This remains valid with $\kvar$ replaced by $\M$. 
\end{remark}
\begin{remark} Since $\kvar_S$ is a $\kvar_k$-module, the morphism (\ref{actioninjection}) endows $\kvar_{S}^{\hat{\mu}^r}$ with a $\kvar_k$-module structure. It is given, for any $k$-variety $X$ and any $S$-variety $Y$ with $\hat{\mu}^r$-action $\sigma$ by
$[X][Y,\sigma] = [X\times_k Y, \sigma']$, where $\sigma'$ is acts trivially on $X$ and by $\sigma$ on $Y$. 
\end{remark}
If $u:R\to S$ is a morphism of $k$-varieties, there are, as in \ref{sect.functoriality}, a group morphism \index{proper pushforward}
$$u_!:\kvar_R^{\hat{\mu}^r}\to\kvar_S^{\hat{\mu}^r}$$
and a ring  morphism \index{pullback}
$$u^*:\kvar_S^{\hat{\mu}^r}\to \kvar_{R}^{\hat{\mu}^r},$$
defined similarly. These morphisms exist also when $\kvar$ is replaced with $\M$. 

Exterior products \index{exterior product} in the flavour of (\ref{exteriorproduct}) also exist for Grothendieck rings of varieties with action. For a $k$-variety $T$ and $T$-varieties $R$ and $S$, the one we are going to use is a morphism
$$\kvar_R^{\hat{\mu}^r}\otimes_{\kvar_T}\kvar_S^{\hat{\mu}^r} \xrightarrow{\boxtimes_T}\kvar_{R\times_TS}^{\hat{\mu}^r\times \hat{\mu}^r}$$
of $\kvar_T$-algebras sending a pair $([X,\sigma],[Y,\tau])$ to the class  of $X\times_TY$ endowed with the product action $\sigma\times_T\tau$ given by $(s,t).(x,y) = (\sigma(s)(x),\tau(t)(y))$. 
\subsection{The dimensional filtration}\index{dimensional filtration}\index{filtration!dimensional}
A reference for this section is \cite{CNS}, chapter 1, sections 4.1 and 4.2. 
\begin{definition} Let $S$ be a scheme. The relative dimension \index{relative dimension}\index{dimension!relative} of a variety $X$ over $S$, denoted by $\dim_S(X)$, \index{dim@$\dim_S$} is defined to be
$$\dim_SX := \sup_{s\in S} \dim_{\kappa(s)}X_s.$$
\end{definition}

Let $S$ be a noetherian scheme. The above notion of relative dimension gives rise to a natural filtration on the ring $\kvar_{S}$: we define $F_d\kvar_S$ \index{fdk@$F_d\kvar$} to be the subgroup of $\kvar_S$ generated by classes of varieties of relative dimension $\leq d$. 

We may now define a function
$$\dim_S:\kvar_S\to \Z\cup\{-\infty\}$$ \index{dim@$\dim_S$}
by sending a class $\a$ to $\inf\{d\in\Z,\ \a\in F_d\kvar_S\}$. This function has the following elementary properties:

Let $\a,\a'\in \kvar_S$. Then
\begin{enumerate}[$(i)$]
\item $\dim_S(0) = -\infty$.
\item $\dim_S(\a + \a') \leq \min\{\dim_S(\a), \dim_S(\a')\},$ with equality whenever $\dim_S(\a)\neq\dim_S(\a')$
\item $\dim_S(\a\a')\leq \dim_S(\a) + \dim_S(\a').$
\end{enumerate}
Moreover, using the Euler-Poincaré polynomial, one may prove that for every variety $X$ over $S$, $\dim_S([X]) = \dim_SX$ (see lemma 4.1.3 in \cite{CNS}, chapter 1). 

For every $d\in\Z$, we define $F_d\M_S$ \index{fdm@$F_d\M$} to be the subgroup of $\M_S$ generated by elements of the form $[X]\LL^{-n}$ where $X$ is an $S$-variety, $n\in\Z$ is an integer and $\dim_SX - n\leq d$. This gives us an increasing and exhaustive filtration on the ring $\M_S$. In the same manner as above, we may define a function
$$\dim_S:\M_S\to \Z\cup\{-\infty\}$$\index{dim@$\dim_S$}
satisfying the same properties. 

%\begin{remark} If $\a$ is an element of $\kvar_S$ and $\frac{\a}{1}$ its image in $\M_S$, then we always have $\dim_S\left(\frac{\a}{1}\right)\leq \dim_S\a$, but this equality may be strict, e.g. when $S = k$ is a field of characteristic zero. 
%\end{remark}

The same definition gives rise to dimensional filtrations and functions on Grothendieck rings with actions or with exponentials. Thus, for example, on $\evar_S$, we define $F_d\evar_S$ \index{fde@$F_d\evar$} as the subgroup of $\evar_S$ generated by classes of the form $[X,f]$ where~$X$ is of relative dimension $\leq d$. 

\subsection{Some presentations of Grothendieck groups in characteristic zero}\index{presentation}\index{Grothendieck group!generators}
 We start by recalling a classical result (it follows for example from \cite{Bittner}, Theorem 3.1).
\begin{lemma}\label{smoothpropergensfield} Let $X$ be a variety over a field $k$ of characteristic zero. The Gro\-then\-dieck group $\kvar_X$ is generated by classes of the form~$[Y\xrightarrow{p} X]$ where $p$ is proper and~$Y$ is smooth over $k$. 
\end{lemma}
This may be generalised in the following form:
\begin{lemma}\label{smoothpropergens} Let $S$ be a variety over a field $k$ of characteristic zero, and $X$ a variety over $S$ with structural morphism $u:X\to S$. Then the group $\kvar_{X}$ is generated by classes $[Y\xrightarrow{p} X]$ such that $U =u\circ p(Y)$ is a locally closed subset of $S$, $Y$ is smooth over $U$ and such that the morphism $p\times_S\id_U: Y\times_S U \to X\times_S U$ is proper. 
\end{lemma}
\begin{proof} First of all, $\kvar_{X}$ is generated by classes of quasi-projective morphisms. Let therefore $p:Y\to X$ be a quasi-projective morphism, and define $T$ to be the closure of $u\circ p(Y)$ in $S$. We will argue by induction on $\dim T$. Let $T_1,\ldots,T_n$ be the irreducible components of $T$. For each $i\in\{1,\ldots,n\}$ it suffices to write $[(u\circ p)^{-1}(T_i)\to X]$ as a sum of classes as in the statement of the theorem, and use induction on $\dim T$ to conclude. We may therefore assume that $T$ is irreducible. Compactify $p$ into a proper morphism $\bar{p}:\bar{Y}\to X$, where $\bar{Y}$ is a variety containing $Y$ as a dense open subset. The closure of $u\circ \bar{p}(\bar{Y})$ in $S$ is again $T$. An additional induction on $\dim Y$ (with $\dim T$ fixed), initialised at $\dim Y = \dim T$ (in which case we use the induction hypothesis for $\dim T-1$) then allows us to work with $\bar{p}$ instead of $p$. Let $\eta$ be the generic point of $T$. Consider a resolution of singularities $\widetilde{\bar{Y}_{\eta}}\to \bar{Y}_{\eta}$ above $\kappa(\eta)$. Above some dense open subset $T'$ of $T$, this spreads out to a proper birational morphism $f:\widetilde{\bar{Y}}\to \bar{Y}\times_T T'$ with $\widetilde{\bar{Y}}$ smooth over $T'$. If $\dim T=0$, then $T'=T$, otherwise by induction on $\dim T$, we may replace $\bar{p}:\bar{Y}\to X$ with its restriction $p':\bar{Y}\times_T T'\to X$. Finally, by induction on $\dim Y$ (and on $\dim T$ if $\dim Y = \dim T$), the morphism $p':Y\times_T T'\to X$ may be replaced with the composition $g = p'\circ f:\widetilde{\bar{Y}}\to X$. The morphism $g$ then satisfies the condition in the statement of the theorem. Indeed, the morphism $u\circ g:\widetilde{\bar{Y}}\to T'$ is smooth, and therefore open, so that $U:=u\circ g\left(\widetilde{\bar{Y}}\right)$ is open in its closure and therefore is locally closed is $S$, and~$\widetilde{\bar{Y}}$ is smooth over~$U$. Finally, $g\times_S\id_U$ is the composition of the morphisms $\widetilde{\bar{Y}}\times_{S}U\to \bar{Y}\times_S U$ and $\bar{Y}\times_S U\to X\times_S U$ which are both proper by base change. 
\end{proof}

\section{Motivic vanishing cycles}\label{sect.motvancycles}
Throughout this section, $k$ will be a field of characteristic zero. 
\subsection{Convolution}\label{sect.convolution}
In this section we recall the definition of Looijenga's convolution from \cite{Loo}, in its generalised form due to Guibert, Loeser and Merle (\cite{GLM}), following section 2.3 in \cite{LS}.
\begin{notation} Let $n\geq 1$ be an integer. We denote by $F_0^n$ (resp. $F_1^n$) the Fermat curve with equation $x^n +y^n = 0$ (resp. $x^n + y^n = 1$) inside $\G_m\times \G_m$ (with coordinates $x,y$), with the obvious $\mu_n\times \mu_n$-action. \index{Fermat curve}\index{F0@$F_0^n$}\index{F1@$F_1^n$}
\end{notation}
Let $S$ be a variety over a field $k$ of characteristic zero. There is a $\kvar_k$-linear morphism
$$\Psi:\kvar_{S}^{\hat{\mu}\times \hat{\mu}}\to \kvar_S^{\hat{\mu}}$$\index{Psi@$\Psi$}
given in the following manner: let $Z\xrightarrow{p}S$ be an $S$-variety with a $\hat{\mu}\times \hat{\mu}-$action. Then there is an integer $n$ such that this action factors through $\mu_n\times \mu_n$. One defines
$$\Psi(Z\xrightarrow{p}S) = [Z\times^{\mu_n\times\mu_n}F_0^n\xrightarrow{p\circ\pr_1} S] - [Z\times^{\mu_n\times\mu_n}F_1^n\xrightarrow{p\circ\pr_1} S].$$
This gives an element of $\kvar_S^{\hat{\mu}}$: indeed, as we said in section \ref{sect.grothaction}, for $i=0,1$, the variety $$Z\times^{\mu_n\times\mu_n}F_i^n$$ is endowed with an action of $\mu_n$, given by
$$t.(z,x,y) = ((t,t).z,x,y) = (z, tx, ty).$$
As explained in the discussion after remark 2.13 in \cite{LS}, this construction does not depend on the integer $n$, and therefore $\Psi$ is well-defined. By remark 2.8 in \cite{LS}, for every morphism $f:R\to S$ of $k$-varieties, $\Psi$ commutes with $f_!$ and $f^*$.

 We may define the convolution product \index{convolution product} on $\kvar_{S}^{\hat{\mu}}$ by the $\kvar_k$-linear composition
$$\ast:\kvar_{S}^{\hat{\mu}} \otimes_{\kvar_k}\kvar_S^{\hat{\mu}}\xrightarrow{\boxtimes_S}\kvar_{S}^{\hat{\mu}\times\hat{\mu}}\xrightarrow{\Psi}\kvar_{S}^{\hat{\mu}},$$ \index{ast@$\ast$}
(see definition of the exterior product $\boxtimes_S$ at the end of section \ref{sect.grothaction})  so that for two $S$-varieties $X,Y$ with actions $\sigma$ and $\tau$, we have
$$[X,\sigma] \ast [Y,\tau] = \Psi([X\times_SY,\sigma\times_S \tau]).$$
By proposition 2.12 in \cite{LS} (or proposition 5.2 in \cite{GLM}), for any $k$-variety $S$, the convolution product~$\ast$ endows $\kvar_S^{\hat{\mu}}$ with an associative, commutative $\kvar_k$-algebra structure, with unit element the class of the identity $[S\xrightarrow{\id}S]$ (with trivial $\hat{\mu}$-action).  From now on, the ring structure on $\kvar_S^{\hat{\mu}}$ we are going to consider will be the convolution product $\ast$. Thus, in what follows, both $\kvar_S^{\hat{\mu}}$ and $(\kvar_S^{\hat{\mu}},\ast)$ denote the same thing.

The $n$-fold convolution product of $\LL_S$ with itself is $\LL^n_S$, and therefore localising the $\kvar_k$-algebra $(\kvar_S^{\hat{\mu}},\ast)$ at the multiplicative set $\{1,\LL_S,\LL_S\ast\LL_S,\ldots \}$ yields an $\M_k$-algebra $(\M_S^{\hat{\mu}},\ast)$ with same underlying $\M_k$-module structure as the usual localisation $\M_S^{\hat{\mu}}$ of $\kvar_S^{\hat{\mu}}$. By $\M_k$-linearity, $\Psi$ and~$*$ extend to localised Grothendieck rings. Since $\Psi$ commutes with pullbacks, for any morphism $f:R\to S$ of $k$-varieties, there are pullback morphisms
$$f^*:\kvar_S^{\hat{\mu}}\to \kvar_{R}^{\hat{\mu}}$$
(resp. $f^*:\M_S^{\hat{\mu}}\to \M_{R}^{\hat{\mu}}$) of $\kvar_k$-algebras (resp. of $\M_k$-algebras). 

\begin{remark}\label{asttrivialaction} By remark 2.11 in \cite{LS}, whenever the action $\tau$ on $Y$ is trivial, we have
$$[X,\sigma] \ast [Y,\tau] = [X,\sigma][Y,\tau],$$
where on the right we consider the usual product on $\kvar_S^{\hat{\mu}}$ from section \ref{sect.grothaction}. In particular, there is a ring morphism
$$\kvar_S\to(\kvar^{\hat{\mu}}_S,\ast)$$
sending the class of an $S$-variety $X$ to the class $[X,1]$ where $1$ denotes the trivial action. Such a morphism exists also at the level of localised Grothendieck rings.
\end{remark}
\subsection{Grothendieck rings over the affine line}\label{section.grothaffineline}
Let $S$ be a $k$-variety. By section \ref{sect.convolution}, the Grothendieck group $\kvar_{\A^1_S}^{\hat{\mu}}$ has a natural $(\kvar_S^{\hat{\mu}},\ast)$-algebra structure given by the ring morphism $\epsilon_S^*:(\kvar_{S}^{\hat{\mu}},\ast)\to (\kvar_{\A^1_S}^{\hat{\mu}},\ast)$ where $\epsilon_S:\A^1_S\to S$ is the structural morphism. 

On the other hand, there is a well-defined addition morphism \index{addition morphism}
$$\add:\A^1_S\times_S\A^1_S\to \A^1_S$$\index{add@$\add$}
on the group scheme $\A^1_S$, and this endows $\kvar_{\A^1_S}^{\hat{\mu}}$ with a product $\convv$ given by the composition \index{star@$\convv$} \index{convolution product!over affine line}
$$\convv:\kvar_{\A^1_S}^{\hat{\mu}}\otimes_{\kvar_S}\kvar_{\A^1_S}^{\hat{\mu}}\xrightarrow{\boxtimes_S}\kvar_{\A^1_S\times_S\A^1_S}^{\hat{\mu}\times\hat{\mu}}\xrightarrow{\add_!}\kvar_{\A^1_S}^{\hat{\mu}\times\hat{\mu}}\xrightarrow{\Psi}\kvar_{\A^1_S}^{\hat{\mu}},$$
so that, for $\A^1_S$-varieties $X$ and $Y$ with $\hat{\mu}$-actions $\sigma$ and $\tau$, we have
$$[X\xrightarrow{f}\A^1_S,\sigma]\convv [Y\xrightarrow{g} \A^1_S,\tau] = \Psi([X\times_S Y \xrightarrow{f\conv g} \A^1_S,\sigma\times_S \tau])$$
where $f\conv g = f\circ \pr_1 + g\circ \pr_2$. 

\begin{lemma}\label{iotaSmorphism} Let $\iota_S:S\to \A^1_S$ be the morphism given by the zero-section of the trivial affine bundle $\A^1_S\to S$. Then
$$\begin{array}{rcc}(\iota_S)_!:\ (\kvar_S^{\hat{\mu}},\ast) & \to & (\kvar_{\A^1_S}^{\hat{\mu}},\convv)\\
                   \left[ X\xrightarrow{u} S\right] & \mapsto & [X \xrightarrow{(u,0)} \A^1_S]
                     \end{array}
                    $$
                    is a ring morphism. 
\end{lemma}
\begin{proof} Let $X\xrightarrow{u}S$ and $Y\xrightarrow{v}S$ be two $S$-varieties and let $\sigma$ (resp. $\tau$) be a $\hat{\mu}$-action on $X$ (resp. $Y$). Using the fact that $\Psi$ commutes with proper pushforwards, we get
\begin{eqnarray*}(\iota_S)_!([X\xrightarrow{u}S,\sigma]\ast[Y\xrightarrow{v}S,\tau]) & =& (\iota_S)_!\circ\Psi([X\times_S Y\xrightarrow{u\circ \pr_1}S,\sigma\times_S\tau]) \\
& = & \Psi\circ(\iota_S)_!([X\times_S Y\xrightarrow{u\circ \pr_1}S,\sigma\times_S\tau])\\
& = & \Psi([X\times_S Y\xrightarrow{(u\circ \pr_1,0)}\A^1_S,\sigma\times_S\tau])\\
& = & \Psi\circ\add_!([X\times_SY\xrightarrow{(u\circ\pr_1,0)}\A^1_S\times_S\A^1_S,\sigma\times_S\tau])\\
& = & \Psi\circ \add_!([X\xrightarrow{(u,0)}\A^1_S,\sigma]\boxtimes_S[Y\xrightarrow{(v,0)}\A^1_S,\tau])\\
& = & [X\xrightarrow{(u,0)}\A^1_S,\sigma]\convv [Y\xrightarrow{(v,0)}\A^1_S,\tau]\\
& = & (\iota_S)_!([X\xrightarrow{u}S,\sigma])\convv(\iota_S)_!([Y\xrightarrow{v}\A^1_S,\tau]).
\end{eqnarray*}
\end{proof}
Thus, there is another $\kvar_S^{\hat{\mu}}$-algebra structure on $\kvar_{\A^1_S}^{\hat{\mu}}$, given by $(i_S)_!$, the latter becoming a ring morphism if one replaces the product $\ast$ by $\convv$. 
                    
                    Denote by $\widetilde{\kvar^{\hat{\mu}}_{\A^1_S}}$ the $\kvar_S^{\hat{\mu}}$-module structure on $\kvar_{\A^1_S}^{\hat{\mu}}$ given by $(\iota_S)_!$. 
                     \begin{lemma}\label{identitymorphism} The identity $\id: \kvar_{\A^1_S}^{\hat{\mu}}\to \widetilde{\kvar}_{\A^1_S}^{\hat{\mu}}$ is an isomorphism of $\kvar_S^{\hat{\mu}}$-modules. 
                    \end{lemma}
                    \begin{proof} For all varieties $X,Y$ over $S$ with $\hat{\mu}$-actions $\sigma,\tau$ and any morphism $f:Y\to \A^1_S$, we have (denoting by  $u$ the structural morphism $u:X\to S$ and by 1 the trivial $\hat{\mu}$-action on $\A^1_S$) 
                    \begin{eqnarray*}\epsilon_S^{*}([X\xrightarrow{u}S,\sigma])\ast[Y\xrightarrow{f} \A^1_S,\tau] &=& \Psi([X\times_S \A^1_S\xrightarrow{\pr_2}\A^1_S,\sigma\times_S1]\times_{\A^1_S} [Y\xrightarrow{f} \A^1_S,\tau]) \\
                    & = & \Psi([(X\times_S \A^1_S)\times_{\A^1_S}Y\xrightarrow{f\circ \pr_2}\A^1_S,(\sigma\times_S1)\boxtimes_{\A^1_S}\tau])\\
                    &=& \Psi([X\times_S Y\xrightarrow{f\circ\pr_2}\A^1_S,\sigma\times_S\tau]) \\
                    & = & \Psi\circ\add_!([X\xrightarrow{(u,0)}\A^1_S,\sigma]\boxtimes_S[Y\xrightarrow{f}\A^1_S,\tau])\\
                    & = & [X\xrightarrow{(u,0)}\A^1_S,\sigma]\convv [Y\xrightarrow{f}\A^1_S,\tau]\\
                    &=& (\iota_S)_!([X\xrightarrow{u}S,\sigma])\star[Y\xrightarrow{f}\A^1_S,\tau].\end{eqnarray*}\end{proof}
                    Thus in fact these two $\kvar_S^{\hat{\mu}}$-module structures are the same, and we will denote this $\kvar_{S}^{\hat{\mu}}$-module by~$\kvar_{\A^1_S}^{\hat{\mu}}$. To distinguish between the two $\kvar_S^{\hat{\mu}}$-algebra structures, we will write them $(\kvar_{\A^1_S}^{\hat{\mu}},\ast)$ and $(\kvar_{\A^1_S}^{\hat{\mu}},\convv)$. 
                    
                    \begin{lemma}\label{epsilonmorphism}  The pushforward map $$(\epsilon_S)_!: (\kvar_{\A^1_S}^{\hat{\mu}},\convv)\to (\kvar_{S}^{\hat{\mu}},\ast)$$ is a morphism of $\kvar_{S}^{\hat{\mu}}$-algebras. 
                    \end{lemma}
                    \begin{proof} First of all, we have $(\epsilon_S)_!\circ (\iota_S)_! = \id_{\kvar_S^{\hat{\mu}}}$. Moreover, for all varieties $X,Y$ over $S$ with morphisms $f:X\to \A^1_S$ and $g:Y\to\A^1_S$ and $\hat{\mu}$-actions $\sigma,\tau$, we have, using that $\Psi$ commutes with proper pushforwards:
                    \begin{eqnarray*}(\epsilon_S)_!([X\xrightarrow{f}\A^1_S,\sigma]\convv[Y\xrightarrow{g}\A^1_S,\sigma]) &= &(\epsilon_S)_!\Psi([X\times_S Y \xrightarrow{f\conv g} \A^1_S,\sigma\times_S\tau])\\
                    & = &\Psi([X\times_S Y\to S,\sigma\times_S\tau])\\& =& [X\to S,\sigma]\ast[Y\to S,\tau]\\ &=&
                     (\epsilon_S)_!([X\xrightarrow{f}\A^1_S,\sigma])\ast(\epsilon_S)_!([Y\xrightarrow{g}\A^1_S,\tau]).\end{eqnarray*}
                    \end{proof}
                    \begin{remark}[Trivial actions]\label{convvtrivialaction} From remark 2.15 in \cite{LS}, we see that if $f:X\to\A^1_S$ and $g:Y\to \A^1_S$ are $\A^1_S$-varieties with $\hat{\mu}$-actions $\sigma,\tau$ such that $\tau$ is the trivial action, then 
                    \begin{equation}\label{convvtrivialactionequation}[X\xrightarrow{f}\A^1_S]\convv[Y\xrightarrow{g}\A^1_S] = [X\times_S Y\xrightarrow{f\conv_S g}\A^1_S],\end{equation}
         where $f\conv_S g$ stands for the morphism $\add\circ (f\times_Sg)$.       We denote again by $\convv$ the restriction of $\convv$ to the image of the inclusion $\kvar_{\A^1_S}\to\kvar_{\A^1_S}^{\hat{\mu}}$, given by formula (\ref{convvtrivialactionequation}). Thus, for Grothendieck rings without actions, lemma \ref{iotaSmorphism} boils down to the statement that
                    $$(\iota_S)_!:\kvar_{S}\to (\kvar_{\A^1_S},\convv)$$ 
                    is a ring morphism. 
                    \end{remark}
                    \begin{remark}[Convolution induces product on Grothendieck ring with exponentials] \label{expquotientmorphism} By remark \ref{convvtrivialaction} and by the definition of the Grothendieck ring with exponentials $\evar_S$, the quotient map $q:(\kvar_{\A^1_S},\convv)\to \evar_S$ is a morphism of $\kvar_S$-algebras.   
                    \end{remark}

                     \subsection{Localised Grothendieck rings over the affine line} We have $$\epsilon_S^*(\LL_S) = \LL_{\A^1_S} = [\A^1_S\times_S \A^1_S \xrightarrow{\pr_2} \A^1_S],$$ whereas $$(\iota_S)_!(\LL_S) = [\A^1_S\xrightarrow{(\epsilon_S,0)}\A^1_S],$$ which we denote by $\LL_0\in\kvar_{\A^1_S}^{\hat{\mu}}$. Define $(\widetilde{\M}_{\A^1_S}^{\hat{\mu}},\convv)$ as the $\M_S$-algebra obtained by localising  $(\kvar_{\A^1_S}^{\hat{\mu}},\convv)$ with respect to the multiplicative set $\{1,\LL_0,\LL_0\convv\LL_0,\ldots\}$. Then we have a canonical isomorphism of $\M_S$-algebras
                    $$\M_S\otimes_{\kvar_S}(\kvar^{\hat{\mu}}_{\A^1_S},\convv)\xrightarrow{\sim} (\widetilde{\M}^{\hat{\mu}}_{\A^1_S},\convv).$$
                    Thus, base change  along $\kvar_S\to \M_S$ of lemma \ref{identitymorphism} gives us:
                    \begin{lemma}\label{mkiso} There is a canonical isomorphism $\M_{\A^1_S}^{\hat{\mu}}\xrightarrow{\sim}\widetilde{\M}_{\A^1_S}^{\hat{\mu}}$ of $\M_S$-modules, given by $\frac{a}{\LL_{\A^1_S}^n}\mapsto \frac{a}{\LL_0^n}$ for all $a\in\kvar_{\A^1_S}^{\hat{\mu}}$ and $n\geq 1$. \end{lemma}
                    Moreover, lemma \ref{epsilonmorphism} and remark \ref{expquotientmorphism} remain true with $\kvar$ replaced by $\widetilde{\M}$ and $\evar$ replaced with $\expp$.

 \subsection{Rational series}
Following 2.8 in \cite{GLM}, for any $k$-variety $X$, define
$\M_X^{\hat{\mu}}[[T]]_{\rat}$ to be the $\M_X^{\hat{\mu}}$-subalgebra of $\M_X^{\hat{\mu}}[[T]]$ generated  by rational series of the form $p_{e,i}(T) = \frac{\LL^{e}T^{i}}{1-\LL^{e}T^{i}}$ where $e\in \Z$ and $i>0$.  \index{MX@$\M_X^{\hat{\mu}}[[T]]_{\rat}$}

There is a unique morphism of $\M_X^{\hat{\mu}}$-algebras
$$\lim_{T\to \infty}:\M_X^{\hat{\mu}}[[T]]_{\rat}\longrightarrow \M_X^{\hat{\mu}}$$
such that $$\lim_{T\to \infty}\ p_{e,i}(T) = -1,$$
for any $e\in\Z$ and $i>0$. 

More generally, for an $k$-variety $S$ and for a  variety $X$ over $S$, we define $\M_X^{\hat{\mu}}[[T]]_{\rat,S}$ to be the $\M_S^{\hat{\mu}}$-subalgebra of $\M_X^{\hat{\mu}}[[T]]$ generated  by rational series of the form $p_{e,i}(T) = \frac{\LL^{e}T^{i}}{1-\LL^{e}T^{i}}$ where $e\in \Z$ and $i>0$.
\index{MX@$\M_X^{\hat{\mu}}[[T]]_{\rat,S}$}

There is a unique morphism of $\M_S^{\hat{\mu}}$-algebras
$$\lim_{T\to \infty}:\M_X^{\hat{\mu}}[[T]]_{\rat,S}\longrightarrow \M_X^{\hat{\mu}}$$
such that $$\lim_{T\to \infty}\ p_{e,i}(T) = -1,$$
for any $e\in\Z$ and $i>0$. 

\subsection{Motivic vanishing cycles}\label{sect.vanrecall}
In \cite{DL98,DL99,DL01,DL02} Denef and Loeser defined and studied the notions of \emph{motivic nearby fibre} and \emph{motivic vanishing cycles}. For a smooth connected variety $X$ over$k$ of dimension $d$ and a morphism $f:X\to \A^1_k$, the motivic nearby fibre $\psi_{f}$ of $f$ at $0\in \A^1_k$ is an element of $\M_{X_0(f)}^{\hat{\mu}}$ (where $X_0(f)$ is the fibre of $X$ above~$0$) defined in terms of some motivic zeta function $Z_{f}$. More precisely, denoting by $\scr{L}_n(X)$ the space of $n$-jets of $X$, we define for every $n\geq 1$
$$\scr{X}_n(f) :=\{\gamma\in\scr{L}_n(X)\ |\ f(\gamma) \equiv t^n (\mathrm{mod}\ t^{n+1})\}\in\M_{X_0(f)}^{\hat{\mu}},$$
the $X_0(f)$-variety structure of $\scr{X}_n(f)$ being induced by the truncation morphism $\scr{L}_n(X)\to X$, and the $\hat{\mu}$-action being the one induced by the $\mu_n$-action given by $a.\gamma(t) = \gamma(at)$.  We then put
$$Z_{f}(T) = \sum_{n\geq 1} [\scr{X}_{n}(f)\to X_0(f)]\LL^{-nd}T^{n}\in\M_{X_0(f)}^{\hat{\mu}}[[T]].$$
  One may write $\scr{X}_n(f/k)$, resp. $Z_{f/k}$ if one wants to keep track of the base field $k$. 
  
 Let $X$ be a smooth variety over $k$, not necessarily connected, and $f:X\to \A^1_k$ a morphism. Let~$C$ be the set of connected components of $X$, which are smooth varieties of pure dimension.  Then the above definition extends immediately to the pair $(X,f)$ by  putting:
 $$Z_f(T) = \sum_{Y\in C}\1_{Y}Z_{f_{|Y}}(T) \in \M_{X_0(f)}^{\hat{\mu}}[[T]],$$
 where $\1_{Y}$ denotes the element $[Y\cap X_0(f)\to X_0(f)]\in \M_{X_0(f)}^{\hat{\mu}}$ corresponding to the inclusion of the $Y$-component of $X_0(f)$ into $X_0(f)$, with trivial $\hat{\mu}$ action.
 
  If $f$ is constant, we have $Z_{f}(T) = 0.$ More generally, Denef and Loeser showed in \cite{DL02} that $Z_f$ is a rational function by giving a formula for it in terms of a log-resolution of $(X,X_0(f))$. For this, let~$X$ be a smooth variety over~$k$ of pure dimension~$d$, and $f:X\to \A^1_k$ a  morphism such that $X_0(f)$ is nowhere dense in~$X$. Let $h:X'\to X$ be a log-resolution of the pair $(X,X_0(f))$. Let $(E_i)_{i\in I}$ be the family of irreducible components of $ h^{-1}(X_0(f))$, and for every $i\in I$, let $a_i$ be the multiplicity of $f\circ h$ along $E_i$. For every $J\subset I$ we put $E_J = \cap_{j\in J}E_j$,  $E_J^{\circ} = E_J\setminus \cup_{i\not\in J} E_i$ and $a_J = \gcd_{j\in J}(a_j)$. For every $J\subset I$, one defines an unramified Galois cover $\widetilde{E}_J^{\circ}\to E_J^{\circ}$ by glueing locally constructed covers obtained as follows: around every point of $E^{\circ}_J$, one can find an open subscheme $U$ of $X'$ on which $f\circ h = u \prod_{j\in J}x_j^{a_j}$, where $x_j$ is an equation for $E_j\cap U$ and $u$ is an invertible function on~$U$. One takes the étale cover of $E_J^{\circ}\cap U$ induced by the étale cover $U'\to U$ obtained by taking the $a_J$-th root of $u^{-1}$. There is a natural $\mu_{a_J}$-action on $\widetilde{E_J}^{\circ}$ which induces a $\hat{\mu}$-action in the obvious way. \index{Denef and Loeser's formula}
\begin{theorem}[Denef-Loeser] \label{DLrat}Let $X$ be a smooth $k$-variety of pure dimension $d$, and $f:X\to \A^1_k$ a morphism such that $X_0(f)$ is nowhere dense in $X$. Let $h:X'\to X$ be a log-resolution of the pair $(X,X_0(f))$. Let $(E_i)_{i\in I}$ be the family of irreducible components of $h^{-1}(X_0(f))$. For every $i\in I$, let $a_i$ be the multiplicity of $f\circ h$ along $E_i$ and let $\nu_i-1$ be the multiplicity of the Jacobian ideal of $h$ along $E_i$. Then one has
$$Z_f(T) = \sum_{\varnothing \neq J\subset I} (\LL- 1)^{|J|-1} \left[\widetilde{E_J}^{\circ}\to X_0(f),\hat{\mu}\right]\prod_{j\in J} \frac{\LL^{-\nu_j}T^{a_j}}{1-\LL^{-\nu_j}T^{a_j}}\in\M_{X_0(f)}^{\hat{\mu}}[[T]],$$
where $\widetilde{E_J}^{\circ}$ is the Galois cover defined above. In particular $Z_f(T)\in \M_{X_0(f)}^{\hat{\mu}}[[T]]_{\rat}$.\end{theorem}
\begin{cor}\label{Zfrational} Let $X$ be a smooth variety over $k$ and $f:X\to \A^1_k$ a morphism. Then $Z_f(T)$ is an element of $\M_{X_0(f)}^{\hat{\mu}}[[T]]_{\rat}.$
\end{cor}
\begin{proof} Let $C$ be the set of connected components of $X$. By definition, we have $Z_f = \sum_{Y\in C}\1_{Y}Z_{f_{|Y}}$, and each $Z_{f_{|Y}}$ is 0 if $f_{|Y}$ is constant, or is an element of $\M_{Y_0(f_{|Y})}^{\hat{\mu}}[[T]]_{\rat}$ (which is naturally a $\M_X^{\hat{\mu}}$-subalgebra of $\M^{\hat{\mu}}_{X_0(f)}[[T]]_{\rat}$) by theorem \ref{DLrat}, whence the result.
\end{proof}

 By corollary \ref{Zfrational}, it makes sense to define the \emph{motivic nearby fibre} \index{motivic nearby fibre} $\psi_{f}$ \index{psif@$\psi_f$} of $f$ at $0$ as
$$\psi_{f} = -\lim_{T\to \infty}Z_{f}(T) \in\M_{X_0(f)}^{\hat{\mu}}$$
                         and the \emph{motivic vanishing cycles} \index{motivic vanishing cycles} $\phi_{f}$ of $f$ at $0$ as
$$\phi_{f} := [X_0(f)\xrightarrow{\mathrm{id}} X_0(f)] - \psi_{f}\in\M_{X_0(f)}^{\hat{\mu}}.$$\index{phif@$\phi_f$}
For the vanishing cycles, we use the definition in \cite{LS}, which differs from the one by Denef and Loeser by a sign (which will be important for the construction of our motivic measure, see remark 5.5 in~\cite{LS}).  \index{motivic vanishing cycles!sign}

Under the conditions and with the notations of theorem \ref{DLrat}, we have
$$\psi_f =  \sum_{\varnothing\neq J \subset I}(1-\LL)^{|J|-1}\left[\widetilde{E_J}^{\circ}\to X_0(f),\hat{\mu}\right]\in\M_{X_0(f)}^{\hat{\mu}}.$$\index{psif@$\psi_f$!formula} 

\begin{example}\label{vancyclespecialcases}
Here are some important special cases:
\begin{enumerate}[(I)]
\item\label{constantproperty} Assume $f=a$ is constant.  If $a\neq 0$ then we in fact have $X_0(f) = \varnothing$, so that $\M^{\hat{\mu}}_{X_0(f)}=0$ and $\psi_f = \phi_f = 0$. If $a=0$, then $Z_f = 0$, so $\psi_f = 0$, whereas $\phi_{f} = [X\xrightarrow{\id} X]\in\M_X^{\hat{\mu}}$ (the action being necessarily trivial).
\item In the case when $f$ is non-constant and $X_0(f)$ is smooth and nowhere dense in $X$, then $X\xrightarrow{\mathrm{id}} X$ gives a log-resolution, and we get $\psi_{f} = [X_0(f)\xrightarrow{\id}X_0(f)]$, and $\phi_{f} = 0$. 
\item \label{singproperty} As a consequence, when $f$ is not constant equal to $0$, $\phi_{f}$ lives above $\mathrm{Sing}(f)$, that is, the closed subscheme of $X$ defined by the vanishing of the differential $\dx f\in\Gamma(X,\Omega_{X/k}^1)$. Thus, $\phi_{f}$ may be seen in a canonical way as an element of $\M_{X_0(f)\ \cap\ \mathrm{Sing}(f)}^{\hat{\mu}}$.% (see \cite{LS}, Proposition 3.4, the proof of which does not use the fact that the base field is assumed to be algebraically closed in all of \cite{LS}).
\end{enumerate}
\end{example}
\begin{remark}\label{vandimension} From the formula in theorem \ref{DLrat}, it is clear that, if $X_0(f)$ is nowhere dense in $X$ and if $a_X:X\to k$ is the structural morphism, then $\dim ((a_X)_!\phi_f) \leq \dim X -1$. Without any assumption on $X_0(f)$, we have the weaker inequality $\dim ((a_X)_!\phi_f) \leq \dim X$.
\end{remark}

%In \cite{LS}, Lunts and Schnürer construct a motivic measure in terms of the vanishing cycles at all critical points, not just zero. The following construction is analogous to theirs.

\subsection{Relative motivic vanishing cycles} The previous definitions also make sense in the relative setting. Let $k$ be a field of characteristic zero, $S$ a $k$-variety and $X$ a variety over $S$ of relative dimension $d$, smooth over $S$, together with a morphism $f:X\to \A^1$. Then we may define
$$\scr{X}_n(f/S) := \{\gamma \in \scr{L}_n(X/S)|\ f(\gamma) \equiv t^n(\mathrm{mod}\ t^{n+1})\},$$
with action of $\hat{\mu}$ given in the same manner,
and $Z_{f/S}(T) = \sum_{n\geq 1} [\scr{X}_n(f/S)]\LL^{-nd}T^{n}\in \M_{X_0(f)}^{\hat{\mu}}[[T]].$ Here, $\scr{L}_n(X/S)$ is the $n$-th jet scheme of $X$ relatively to $S$. By the base-change properties for jet schemes (see \cite{CNS}, chapter 2, (2.1.4)), for every $s\in S$ we have 
$$\scr{L}_n(X/S)\times_S\kappa(s) = \scr{L}_n(X\times_S\kappa(s)/\kappa(s))$$ where $\kappa(s)$ is the residue field of $s$, so that the fibre above $s$ of $Z_{f/S}$ is exactly $Z_{f_s/\kappa(s)}(T)\in\M_{X_0(f)_s}^{\hat{\mu}}[[T]]$, where $f_s:X_s\to \A^1$ is the morphism induced by $f$. 
\begin{lemma} The series $Z_{f/S}(T)$ is an element of $\M_{X_0(f)}^{\hat{\mu}}[[T]]_{\rat,S}$.
\end{lemma}
\begin{proof} The proof goes by the classical ``spreading-out'' method. Take a generic point $\eta$ of~$S$: then by corollary  \ref{Zfrational}, the series $Z_{f_{\eta}}$ is an element of $\M_{X_0(f)_{\eta}}^{\hat{\mu}}[[T]]_{\rat}$, and its coefficients can be spread out over some open subset $U$ of $S$. One concludes by Noetherian induction. %More precisely, on every connected component of $X_0(f)_{\eta}$, it is either zero, or given in terms of a resolution as described in theorem \ref{DLrat}. 

 %j'aime pas trop ma rédaction.
\end{proof}

In particular, it makes sense to define the relative versions of the motivic nearby fibre \index{motivic nearby fibre!relative} and motivic vanishing cycles \index{motivic vanishing cycles!relative}, by
$$\psi_{f/S} = -\lim_{T\to \infty} Z_{f/S}(T)\ \ \text{and}\ \ \phi_{f/S}  = [X_0(f)\xrightarrow{\mathrm{id}} X_0(f)] - \psi_{f/S} \in\M_{X_0(f)}^{\hat{\mu}}$$

\subsection{The motivic nearby fibre as a group morphism}
In \cite{GLM}, Guibert, Loeser and Merle construct, for every smooth variety $Y$ together with a function $h:Y\to \A^1_k$ and every dense open subset $U$ of $Y$, an object $\scr{S}_{h,U}\in \M_{Y_0(h)}^{\hat{\mu}}$, such that $\scr{S}_{h,Y} = \psi_h$ is the motivic nearby fibre as defined above and such that these objects fit together into a morphism $\scr{S}_h$ as stated in the following theorem:
\begin{theorem}[\cite{GLM}, Theorem 3.9] \label{glmtheorem} Let $Y$ be a $k$-variety and $h:Y\to \A^1_k$ a morphism. There exists a unique $\M_k$-linear map $\scr{S}_h:\M_Y\to \M_{Y_0(h)}^{\hat{\mu}}$
such that for every proper morphism $p:Z\to Y$ with $Z$ smooth and for every dense open subset $U$ of $Z$, $\scr{S}_h([U\to Y]) = p_!(\scr{S}_{h\circ p,U}).$ \index{S@$\scr{S}_h$, $\scr{S}_{h,U}$}
\end{theorem}
When we want to keep track of the base field $k$, we are going to denote the map in the theorem by $\scr{S}_{h/k}$. We are going to need the following corollary, which adapts this to a relative setting:
\begin{cor}\label{glmfamily} Le $S$ be a $k$-variety and $Y$ a variety over $S$, with $h:Y\to \A^1_k$ a morphism. There exists a unique $\M_S$-linear map $\scr{S}_{h/S}:\M_Y\to \M_{Y_0(h)}^{\hat{\mu}}$ such that for all $s\in S$ the diagram
$$\xymatrix{ \M_Y \ar[r]^{\scr{S}_{h/S}}\ar[d] &\M_{Y_0(h)}^{\hat{\mu}}\ar[d]\\
\M_{Y_s}\ar[r]^{\scr{S}_{h/\kappa(s)}}& \M_{Y_0(h)_s}^{\hat{\mu}}}$$
commutes, where the vertical arrows are induced by the pullback of the inclusion $\{s \}\hookrightarrow S$. \index{S@$\scr{S}_{h/S}$}
\end{cor}
\begin{proof} Uniqueness is immediate by lemma \ref{function.equality}. Denote by $u$ the structural morphism $u:Y\to S$. According to lemma \ref{smoothpropergens}, it suffices to construct $\scr{S}_{h/S}([Z\xrightarrow{p} Y])$ for morphisms of $S$-schemes $p:Z\to Y$ such that $T = u\circ p(Z)$ is locally closed, $u\circ p:Z\to T$ smooth, and $p\times_S\id_T:Z\times_S T\to Y\times_S T$ proper. 
 For such a class, put $$\scr{S}_{h/S}([Z\xrightarrow{p} Y]) = p_{!}\left(\psi_{h\circ p/S}\right)\in \M_{Y_0(h)}^{\hat{\mu}}.$$ The group $\kvar_Y$ is obtained from the free abelian group $A$ on such generators by taking the quotient with respect to some relations, namely elements of the free abelian group on those generators which belong to the kernel of the canonical surjection $A\to \kvar_Y$. Let $R\in A$ be such a relation. For any element $\a\in A$, denote by $\a_s$ its image in $\M_{Y_s}$ through the composition $A \to \kvar_{Y}\to \kvar_{Y_s}\to \M_{Y_s}$. The definition of $\scr{S}_{h/S}$ on elements of $A$ and theorem \ref{glmtheorem} show that, for any element $\a\in A$, $(\scr{S}_{h/S}(\a))_s = \scr{S}_{h/\kappa(s)}(\a_s)$.  In particular, since for every $s\in S$, $R_s = 0$, we have $(\scr{S}_{h/S}(R))_s = 0$, and therefore $\scr{S}_{h/S}(R) = 0$ by lemma \ref{function.equality}. Thus, $\scr{S}_{h/S}$ defines a group morphism $\kvar_Y\to \M_{Y_0(h)}^{\hat{\mu}}$. From the last few lines in the proof of theorem \ref{glmtheorem} in \cite{GLM} (theorem 3.9), it appears that $\scr{S}_{h/\kappa(s)}$ is first constructed on $\kvar_{Y_s}$, and then extended to $\M_{Y_s}$ by $\M_{\kappa(s)}$-linearity. By definition, $\scr{S}_{h/S}$ is compatible with $\scr{S}_{h/\kappa(s)}$ (seen as a morphism with source $\kvar_{Y_s}$) for all $s\in S$. For any $a\in \kvar_S$, and any $x\in \kvar_Y$, the relation $\scr{S}_{h/S}(ax) = a\scr{S}_{h/S}(x)$ is seen to be true by lemma \ref{function.equality}, because for every $s\in S$, $\scr{S}_{h/\kappa(s)}$ is $\M_{\kappa(s)}$-linear. Thus, $\scr{S}_{h/S}$ is $\kvar_S$-linear, and we may extend it by $\M_S$-linearity to a $\M_S$-linear morphism $\M_Y\to \M_{Y_0(h)}^{\hat{\mu}}$, which ensures the commutativity of the diagram in the statement of the theorem.
\end{proof}

\section{The motivic vanishing cycles measure} \label{sect.motvanmeasure}

In \cite{LS}, Lunts and Schnürer defined, for an algebraically closed field $k$ of characteristic zero, a motivic measure $\Phi^{\tot}:(\widetilde{\M}_{\A^1_k},\convv)\to (\widetilde{\M}_{\A_k^1}^{\hat{\mu}},\convv)$, by the formula
\begin{equation}\label{phiformula}\Phi^{\tot} = \sum_{a\in k}(i_a)_!(i_a^* - \scr{S}_{\id- a})\end{equation}
where $i_a:\{a\}\to \A^1_k$ is the inclusion, and $\scr{S}_{\id - a}:\M_{\A^1}\to \M_{\{a\}}^{\hat{\mu}}$ is the morphism from theorem \ref{glmtheorem} applied with $Y =\A^1$ and $h= \id -a$ (this is the measure denoted by $\Phi$ in their paper). Formula (\ref{phiformula}) makes sense because for any class $[X\xrightarrow{f}\A^1]$ with $X$ smooth and $f$ proper, we have, denoting by $X_a$ the fibre of $f$ above $a$ and by $f_a:X_a\to\{a\}$ the constant map induced by $f$ on it, by theorem~\ref{glmtheorem}
\begin{eqnarray*}(i_a)_!(i_a^* - \scr{S}_{\id - a})([X\xrightarrow{f}\A^1]) & = &(i_a)_!\left([X_a \to \{a\}] - (f_a)_!\psi_{f-a}\right) \\
&=& (i_a)_!(f_a)_!([X_a\to X_a] - \psi_{f-a}) \\
&=& f_{!}\phi_{f-a}\end{eqnarray*}
which is zero whenever $a$ is not a critical point of $f$. Thus, the sum is always finite because the Grothendieck ring of varieties is generated by such classes, and because the set of critical points of a morphism $f:X\to \A^1_k$ is finite. The image of a class $[X\xrightarrow{f}\A^1]$ with~$X$ smooth and $f$ proper is
\begin{equation}\label{atotvanformula}\Phi^{\tot}([X\xrightarrow{f}\A^1]) = \sum_{a\in k} f_!\phi_{f-a} =:\atotvan_f ,\end{equation}
the sum of all vanishing cycles of $f$ at all $a\in k$. In other words, $\atotvan_f$ is the element of $\M^{\hat{\mu}}_{\A^1_k}$ corresponding to the motivic function sending $a\in\A^1_k$ to $f_!\phi_{f-a}$

In what follows, we are going to need to construct such a measure in families. Therefore, we will give a definition of $\atotvan_f$ in terms of vanishing cycles relative to the affine line above a base, which behaves well in such a context. For an algebraically closed field $k$ of characteristic zero, this will give an element of $\M_{\A^1_k}^{\hat{\mu}}$ supported above the critical points of~$\A^1_k$, with fibre at every point $a\in k$ given by the vanishing cycles $f_!\phi_{f-a}$, so that we recover formula (\ref{atotvanformula}).

\subsection{Total vanishing cycles}\label{sect.totvandef}
Let $k$ be a field of characteristic zero, $R$ a variety over $k$, and $X$ be a variety over~$R$, smooth over~$R$, and $f:X\to \A_k^1$ a morphism. We apply the construction of the previous paragraph to the variety $X\times \A_k^1$ over $S = \A_R^1$, together with the morphism $g:X\times \A^1_k\to \A^1_k$ given by $g = f\circ \pr_1 - \pr_2$. We have $(X\times \A^1)_0(g) = \Gamma_f$, where $\Gamma_f\subset X\times \A_k^1$ is the graph of the morphism~$f$, which we may identify with $X$ itself through the first projection. 

\begin{notation}
We denote by~$\totvan_{f/R}:= \phi_{g/\A^1_R}$ and~$\totnear_{f/R}:=\psi_{g/\A^1_R}$ the corresponding vanishing cycles and nearby fibre, which are naturally defined as elements of $\M_{X}^{\hat{\mu}}$, and related by the identity 
$$\totvan_{f/R} = [X\xrightarrow{\id}X] - \totnear_{f/R}.$$ \index{phiti@$\totvan_{f/R}$, $\totvan_f$}\index{psiti@$\totnear_{f/R}$, $\totnear_f$} %Considering the natural maps
%$$ X\xrightarrow{f} \A^1_R \xrightarrow{\epsilon} R$$
%where the second one is the structural morphism, 
Denote by $f_R$ the morphism $X\to \A^1_R = R\times \A^1_k$ given by $(u,f)$ where $u:X\to R$ is the structural map. We define~$\atotvan_{f/R}:=(f_R)_!\totvan_{f/R}$ and~$\atotnear_{f/R}:=(f_R)_!\totnear_{f/R}$ their images in $\M_{\A^1_R}^{\hat{\mu}}$, %,  and $(\phi_f)_k:=\epsilon_!\atotvan_{f/R}$ and $(\psi_f)_k:= \epsilon_!\atotnear_{f/R}$ their images in $\M_k$. Note that we have the relations 
 satisfying $$\atotvan_{f/R} = [X\xrightarrow{f_R} \A^1_R] - \atotnear_{f/R}.$$ In the case where $R = k$, we will simply write $\totnear_f, \totvan_f$ etc. \index{phito@$\atotvan_{f/R}$, $\atotvan_f$}\index{psito@$\atotnear_{f/R}$, $\atotnear_f$}
\end{notation}
These objects will be called \textit{total nearby fibre} \index{total nearby fibre} and \textit{total vanishing cycles}, \index{total vanishing cycles} because they take into account the nearby fibre and vanishing cycles of $f$ at all points of $\A^1$: indeed, we see that for any $t\in \A^1$, 
$$\left(\totnear_{f/R}\right)_t = \psi_{g_t/R_{\kappa(t)}} = \psi_{(f-t)/R_{\kappa(t)}}\in\M_{X_t}^{\hat{\mu}},$$
where $f-t$ is the function $X\times_k\kappa(t) \to \A^1_{\kappa(t)}$ given by $x\mapsto f(x) -t$ and $R_{\kappa(t)} = R\times_k\kappa(t)$. A similar remark holds for $\totvan_{f/R}, \atotnear_{f/R}$ and $\atotvan_{f/R}$. The properties of vanishing cycles we recalled above lead to similar properties for total vanishing cycles. 
\begin{remark}\label{constantvan} Let $X$ and $R$ be as above, and assume $f:X\to \A^1_k$ is constant. Then property (\ref{constantproperty}) of example \ref{vancyclespecialcases} together with lemma \ref{function.equality} imply that $\totnear_f = 0$. In particular, we have $\totvan_f = [X\xrightarrow{\id}X]$ and $\atotvan_f = [X\xrightarrow{f_R} \A^1_R]$ (with trivial $\hat{\mu}$-action). 
\end{remark}

\begin{remark}\label{totvandimension} With the above notations, it follows from remark \ref{vandimension} applied above every point of~$\A^1_R$, that
$$\dim_{\A^1_{R}}(\atotvan_{f/R}) \leq \dim_{\A^1_{R}} X.$$
 
\end{remark}
We recall that we denote by $\mathrm{Sing}(f)$ the vanishing locus of the differential form $\dx f$, and we define $\mathrm{Crit}(f)$ to be the scheme-theoretic image $f_R(\mathrm{Sing}(f))\subset \A_R^1$. 
\begin{prop}\label{smoothvanzero} Let $X$ be a smooth $R$-variety and $f:X\to \A^1_k$ a morphism. Then $\totvan_{f/R}$ (resp. $\atotvan_{f/R}$) is canonically an element of $\M_{\mathrm{Sing}(f)}^{\hat{\mu}}$ (resp. $\M_{\mathrm{Crit}(f)}^{\hat{\mu}}$). In particular, if $f$ is smooth, then $\totvan_{f/R} = 0$ and $\atotvan_{f/R} = 0$. 
\end{prop}
\begin{proof} This is a consequence of property (\ref{singproperty}) above applied point by point, together with lemma~\ref{function.equality}.
\end{proof}

\subsection{The Thom-Sebastiani theorem}\label{sect.TS}
The classical theorem proved by Thom and Sebastiani in \cite{TS} is a multiplicativity result for the cohomology of Milnor fibres: for two germs $f:(\C^n,0)\to (\C,0)$ and $g:(\C^n,0)\to (\C,0)$ of holomorphic functions with an isolated critical point at 0, it expresses the reduced cohomology of the Milnor fibre of the germ $f\conv g:(x,y)\mapsto f(x) + g(y)$ as a tensor product of the reduced cohomologies of the Milnor fibres of $f$ and $g$, together with compatibilities of monodromy actions.

An analogue of this for motivic vanishing cycles was first proved by Denef and Loeser in the completed Grothendieck ring of Chow motives in \cite{DL99}. Then Looijenga, who in \cite{Loo} introduced an appropriate convolution operation, and Denef Loeser in \cite{DL01}, showed that essentially the same proof gave an equality in the Grothendieck ring of varieties with $\hat{\mu}$-action.  Finally, in \cite{GLM}, Guibert, Loeser and Merle showed how one may recover the motivic Thom-Sebastiani theorem from a formula involving iterated vanishing cycles. In that paper, the theorem is stated using the generalised convolution operator~$\Psi$ which we defined in section \ref{sect.convolution}.

We recall the motivic Thom-Sebastiani theorem, in the form in which it appears in \cite{GLM}: \index{Thom-Sebastiani!motivic}\index{motivic Thom-Sebastiani}
\begin{theorem}\label{TS} Let $Y_1,Y_2$ be smooth varieties over $k$, with morphisms $g_1:Y_1\to \A^1_k$, $g_2:Y_2\to\A^1_k$. Denote by $i$ the natural inclusion $Y_0:= g_1^{-1}(0)\times g_2^{-1}(0)\to (g_1\conv g_2)^{-1}(0) $. Then
$$i^{*}\left(\phi_{g_1\conv g_2}\right) = \Psi(\phi_{g_1}\boxtimes \phi_{g_2})$$
in $\M_{Y_0}^{\hat{\mu}}$. 
\end{theorem}
This may be globalised in the following manner:
\begin{cor}[Thom-Sebastiani for total vanishing cycles] \label{TStotal} Let $X_1,X_2$ be smooth varieties, with morphisms $f_1:X_1\to \A^1_k$ and $f_2:X_2\to \A^1_k$. 
Then we have the equalities
 \begin{equation}\label{TSequation}\totvan_{f_1\conv f_2}= \Psi(\totvan_{f_1}\boxtimes \totvan_{f_2})\end{equation}
 in $\M_{X_1\times X_2}^{\hat{\mu}}$, and \begin{equation}\label{TSequationA1}\atotvan_{f_1\conv f_2}= \Psi(\atotvan_{f_1}\convv \atotvan_{f_2})\end{equation}
 in $\M_{\A^1_k}^{\hat{\mu}}$. 
\end{cor} \index{motivic Thom-Sebastiani!for $\totvan$ and $\atotvan$}

\begin{proof} Let $g_i:X_i\times \A^1 \to \A^1$ be the functions defined by $g_i = f_i\circ \pr_1 - \pr_2$ for $i=1,2$, and $g:X_1\times X_2\times \A^1\to\A^1$ be the function defined by $g = f_1\circ \pr_1 + f_2\circ \pr_2 - \pr_3$. By definition, the left hand side of the first equality is exactly $\phi_{g/\A^1}$, whereas the right-hand side is given by $\Psi(\phi_{g_1/\A^1}\boxtimes \phi_{g_2/\A^1}).$ %Now recall that $X_1\times X_2$ is naturally a variety over $\A^1\times \A^1$ via the morphism $(f_1,f_2)$. 
We have the commutative diagram
$$\xymatrix{X_1\times X_2 \ar[d]_{(f_1,f_2)}\ar[rd]^{f_1\conv f_2} & \\
           \A^1\times \A^1 \ar[r]^-{+} & \A^1}$$
            Let $t\in \A^1\times \A^1$ with residue field denoted by $K$, and put $t_1 = \pr_1(t)$, $t_2 = \pr_2(t)$ and $s = t_1 + t_2$. Let $i$ be the natural inclusion 
$$i:f_1^{-1}(t_1)\times_kf_2^{-1}(t_2)\to (f_1\conv f_2)^{-1}(s)\times_{\kappa(s)}K.$$
Pulling back via the inclusion of $t$ inside $\A^1\times \A^1$, we see that the left-hand side $\phi_{g/\A^1}$ goes to $i^{*}(\phi_{(f_1 \conv f_2 - s)/K})$, whereas the  right-hand side $\phi_{g_1/\A^1}\boxtimes \phi_{g_2/\A^1}$ goes to $$\Psi(\phi_{(f_1-t_1)/\kappa(t_1)}\boxtimes \phi_{(f_2-t_2)/\kappa(t_2)})$$ because $\Psi$ commutes with pullbacks. These elements are equal in $\M_{f_1^{-1}(t_1)\times f_2^{-1}(t_2)}^{\hat{\mu}}$ by theorem \ref{TS}. Since $t$ was arbitrary, we get the first equality in the statement of the theorem. The second equality is then obtained easily by applying $(f_1\conv f_2)_!$ on both sides and making use of the commutative diagram above. 

\end{proof}

\subsection{Total vanishing cycles as a motivic measure}\label{subsect.totvanmeas}

We are going to prove the following theorem:

\begin{theorem}\label{motmeasureA1} Let $k$ be a field of characteristic zero. There is a unique morphism
$$\Phi^{\tot}: (\kvar_{\A^1_k},\convv) \to (\widetilde{\M}_{\A^1_k}^{\hat{\mu}},\convv)$$
of $\kvar_{k}$-algebras such that $\Phi([X\xrightarrow{f} \A^1_k]) = \atotvan_f$ for any smooth variety $X$ over $k$ and any proper morphism $f:X\to \A^1_k$.
\end{theorem}\index{phitot@$\Phi^{\tot}$}\index{total vanishing cycles measure}
The case where $k$ is algebraically closed was treated by Lunts and Schnürer in \cite{LS}, and our proof goes along the same lines as theirs, the main difference being that we replace their total vanishing cycles $(\phi_f)_{\A^1_k}$ by our total vanishing cycles $\atotvan_f$ which behave better in relative settings. We start by proving the following result, which is the analogue in our setting of Theorem 5.3 in \cite{LS}:

\begin{prop} \label{motmeasureaux} There exists a unique morphism
$\Phi':\M_{\A^1_k} \to \M_{\A^1_k}^{\hat{\mu}}$ of $\M_k$-modules such that $\Phi'([X\xrightarrow{f} \A^1_k]) = \atotvan_f$ for any smooth variety $X$ over $k$ and any proper morphism $f:X\to \A^1_k$.
\end{prop}

\begin{proof} Uniqueness follows from lemma \ref{smoothpropergensfield} so it remains to prove existence. 
Apply corollary~\ref{glmfamily} to $Y = \A^1_k\times \A^1_k$, seen as a variety over $S = \A^1_k$ via the second projection, together with $h = \pr_1 - \pr_2$: we get an $\M_{\A^1}$-linear map
$$\scr{S}_{h/\A^1}:\M_{\A^1_k\times\A^1_k}\to \M_{ \Delta}^{\hat{\mu}},$$
where $\Delta\subset \A^1_k\times \A^1_k$ is the diagonal $h^{-1}(0)$, which is isomorphic to $\A^1_k$ via $\pr_2$. Composing this with the pull-back $\pr_2^*:\M_{\A^1_k}\to \M_{\A^1_k\times \A^1_k}$, 
sending a class $[X\xrightarrow{f}\A_k^1]$ to $[X\times \A_k^1\xrightarrow{(f\circ \pr_1,\pr_2)}\A^1_k\times\A^1_k]$ we get a $\M_k$-linear map $\scr{S}_{h/\A^1}\circ\pr_2^*:\M_{\A^1_k}\to \M_{\A^1_k}^{\hat{\mu}}$. We put, for any $\a\in\M_{\A^1_k}$, 
 $$\Phi'(\a) = \a - \scr{S}_{h/\A^1}\circ \pr_2^*(\a).$$

Let $f:X\to \A^1_k$ be a proper morphism with $X$ smooth. We denote by $p$ the morphism 
$$\pr_2^*(f):X\times \A^1_k\xrightarrow{(f\circ \pr_1,\pr_2)}\A^1_k\times \A^1_k$$ which is again proper by base change. We claim that $\scr{S}_{h/\A^1_k}\circ\,\pr_2^*([X\xrightarrow{f}\A^1])=\atotnear_f$. Indeed, by theorems \ref{glmtheorem} and \ref{glmfamily}, for every $t\in \A^1$, the fibre above $t$ of this element is given by
\begin{eqnarray*}\left(\scr{S}_{h/\A^1_k}([X\times \A^1_k\xrightarrow{p}\A^1_k\times \A^1_k])\right)_t &=& \scr{S}_{h_t/\kappa(t)}([X\times_k\kappa(t)\xrightarrow{p_t} \A^1_{\kappa(t)}])\\
  &=& (p_t)_!(\psi_{(h\circ p)_t/\kappa(t)})\in\M_{\kappa(t)}^{\hat{\mu}}.\end{eqnarray*}
 because $p_t$ is proper and $X_t$ smooth over $\kappa(t)$. On the other hand, we have  $h\circ p = f\circ \pr_1 - \pr_2$, so that for every $t\in \A^1$,  $\psi_{(h\circ p)_t/\kappa(t)} = (\psi_{h\circ p/\A^1_k})_t = (\totnear_f)_t$. Moreover, for every $t$, we have $p_t = f\times_k\kappa(t)$, so that $(p_t)_!(\psi_{(h\circ p)_t/\kappa(t)}) = \left(f_!(\totnear_f)\right)_t = \left(\atotnear_f\right)_t$, which proves the claim. Finally, we may conclude that $\Phi'([X\xrightarrow{f} \A^1_k]) = [X\xrightarrow{f} \A^1_k] - \atotnear_f = \atotvan_f$. \end{proof}

As in remarks 5.4, 5.6 and 5.7 of \cite{LS}, the map $\Phi'$ has the following properties:\begin{lemma}\label{Phiprop} For any $k$-variety $Z$, we have
\begin{enumerate}[(a)]\item \label{affineid} $\Phi'([Z\xrightarrow{0}\A^1_k]) = [Z\xrightarrow{0}\A^1_k]$ (with trivial action), so that in particular $\Phi'(\LL_0) = \LL_0$. %This makes sure that the morphism we will obtain will indeed be a retraction.
\item\label{affinezero}  $\Phi'([Z\times \A^1_k\xrightarrow{\pr_2} \A^1_k]) = 0$, so that in particular $\Phi'(\LL_{\A^1_k}) = 0$. 
\item \label{smoothzero} $\Phi'([Z\xrightarrow{f}\A^1_k]) = 0$ whenever $f$ is a smooth and proper morphism. 
\end{enumerate}
\end{lemma}
\begin{proof} The subgroup of $\kvar_{\A^1_k}$ generated by classes of the form $[Z\xrightarrow{0}\A^1_k]$ is the image of the morphism $(\iota_k)_!:\kvar_k\to \kvar_{\A^1_k}$ from section \ref{section.grothaffineline}. It is therefore generated by such classes where $Z$ is additionally assumed to be smooth and proper, and these are preserved by $\Phi'$ by proposition \ref{motmeasureaux} and remark \ref{constantvan}. This proves (\ref{affineid}). In the same way, the subgroup of $\kvar_{\A^1_k}$ generated by classes of the form $[Z\times \A^1_k\xrightarrow{\pr_2} \A^1_k]$ is the image of the morphism $\epsilon_k^*:\kvar_k\to \kvar_{\A^1_k}$ from section \ref{section.grothaffineline}, and we may again assume $Z$ to be smooth and proper to prove (\ref{affinezero}). The statement then follows from propositions \ref{motmeasureaux} and \ref{smoothvanzero}. As for property (\ref{smoothzero}), if $f:Z\to \A^1_k$ is smooth and proper, then $Z$ is smooth, so this follows directly from propositions \ref{motmeasureaux} and \ref{smoothvanzero}.
\end{proof}
 We then consider the morphism of $\kvar_k$-modules $\Phi^{\tot}$, defined as the composition
$$\kvar_{\A^1_k}\to \M_{\A^1_k} \xrightarrow{\Phi'} \M_{\A^1_k}^{\hat{\mu}} \to \widetilde{\M}_{\A^1_k}^{\hat{\mu}},$$
where the first arrow is the localisation morphism, and the last arrow is the isomorphism of $\M_k$-modules from lemma \ref{mkiso}. The proof of theorem \ref{motmeasureA1} is complete once we have the following:
\begin{prop}\label{motmeasureringmorph} The morphism $\Phi^{\tot}$ is a morphism 
$$\Phi^{\tot}:(\kvar_{\A^1_k},\convv)\to (\widetilde{\M}_{\A^1_k}^{\hat{\mu}},\convv)$$
of $\kvar_k$-algebras.
\end{prop}\index{phitot@$\Phi^{\tot}$}\index{total vanishing cycles measure}
\begin{proof} Part (\ref{affineid}) of lemma \ref{Phiprop} shows that $\Phi^{\tot}$ maps the unit element to the unit element, and that it is compatible with the algebra structure maps. By lemma \ref{smoothpropergensfield}, we may restrict to checking multiplicativity for classes of projective morphisms $f:X\to \A^1_k$ with $X$ a connected quasi-projective $k$-variety which is smooth over $k$.  Let $X,Y$ therefore be connected quasi-projective smooth $k$-varieties, together with projective morphisms $f:X\to \A^1_k$ and $g:Y\to \A^1_k$. Then we know by proposition \ref{motmeasureaux} and corollary \ref{TStotal} that
$$\Phi^{\tot}([X\xrightarrow{f}\A^1_k])\convv \Phi^{\tot}([Y\xrightarrow{g}\A^1_k]) = (\atotvan_f)\convv(\atotvan_g) = \atotvan_{f\conv g}.$$
It remains to show that $\atotvan_{f\conv g} = \Phi^{\tot}([X\times Y \xrightarrow{f\conv g}\A^1_k])$, which does not follow directly from proposition \ref{motmeasureaux} because $f\conv g$ is not proper in general. As in the proof of theorem 5.9 in \cite{LS}, we make use of lemma \ref{compactification} below (and the notation therein) which gives us a compactification $h:Z\to \A^1_k$ of $f\conv g$ for which we may write, 
$$[X\times Y\xrightarrow{f\conv g} \A^1_k] = [Z\xrightarrow{h}\A^1_k]+\sum_{I}(-1)^{|I|} [D_I\xrightarrow{h_I}\A^1_k].$$
Since all $h_{I}:D_{I}\to \A^1_k$ are projective and smooth, their image by $\Phi^{\tot}$, given by their total vanishing cycles, is zero by proposition \ref{smoothvanzero}. On the other hand, $\mathrm{Sing}(h) = \mathrm{Sing}(f\conv g)$, and therefore $$\Phi^{\tot}([X\times Y\xrightarrow{f\conv g} \A^1_k]) = \Phi^{\tot}([Z\xrightarrow{h}\A^1_k]) = \atotvan_h = \atotvan_{f\conv g}.$$ 
\end{proof}
%
%We have constructed a motivic measure $\Phi^{tot}$ on $(\kvar_{\A^1_k},\convv)$, with values in the Grothendieck ring $(\M_{\A^1_k}^{\hat{\mu}},\convv)$ over the affine line. In our setting, we will see it as a measure on the Grothendieck ring with exponentials, and compose it with the pushforward $\epsilon_!$ where $\epsilon:\A^1_k\to \spec k$ is the structural morphism:
%\begin{lemma}
%There is a unique morphism
%$$\Phi_k:\expm \to \M_k^{\hat{\mu}}$$
%of $\M_k$-algebras, called the motivic vanishing cycles measure, such that, for any proper morphism $f:X\to \A^1_k$, with source a smooth $k$-variety $X$, one has
%$$\Phi_k([X\xrightarrow{f}\A^1_k]) = \epsilon_!(\atotvan_f)$$ where $\epsilon:\A^1_k\to k$ is the structural morphism. Moreover, the restriction of $\Phi_k$ to $\M_k$ coincides with the natural inclusion $\M_k\to\M_k^{\hat{\mu}}.$
%\end{lemma}
%\begin{proof} By property (\ref{affinezero}) in lemma 
%\end{proof}
\subsection{The motivic vanishing cycles measure over a base}

We are now going to prove that the above motivic measure may be defined in families. For this purpose, we are going to denote for any field $k$ of characteristic zero by $\Phi^{\tot}_k:\kvar_{\A^1_k}\to \M_{\A^1_k}^{\hat{\mu}}$ the motivic measure from theorem \ref{motmeasureA1} relative to the field $k$.
\begin{theorem} \label{motmeasurebase} Let $k$ be a field of characteristic zero and $S$ a variety over $k$. There is a unique morphism of $\kvar_S$-algebras $\Phi^{\tot}_S: (\kvar_{\A^1_S},\convv)\to (\widetilde{\M}^{\hat{\mu}}_{\A^1_S},\convv)$ such that for every $s\in S$ the diagram
$$\xymatrix{\kvar_{\A^1_S}\ar[d] \ar[r]^-{\Phi^{\tot}_S} & \widetilde{\M}_{\A^1_S}^{\hat{\mu}}\ar[d] \\
            \kvar_{\A^1_{\kappa(s)}}\ar[r]^-{\Phi^{\tot}_{\kappa(s)}} & \widetilde{\M}_{\A^1_{\kappa(s)}}^{\hat{\mu}}
}$$
commutes.
\end{theorem}\index{phitotS@$\Phi^{\tot}_S$}\index{total vanishing cycles measure!relative}
\begin{proof} Uniqueness is immediate by lemma \ref{function.equality}. By lemma \ref{smoothpropergens}, denoting by $u:\A^1_S\to S$ the structural morphism,  it suffices to construct~$\Phi^{\tot}_S$ on classes $[X\xrightarrow{p} \A^1_S]$ such that $T = u\circ p(X)$ is locally closed in $S$, $X$ is smooth over $T$ and $p\times_S\id_T:X\times_ST\to \A^1_T$ is proper. For such a class, denoting by $f$ the composition $X\xrightarrow{p}\A^1_S\to \A^1_k$, we put
$$\Phi_S^{\tot}([X\xrightarrow{p} \A^1_S]) = \atotvan_{f/S}\in\M_{\A^1_S}^{\hat{\mu}}.$$
Then for every $s\in S$, we have 
$$\Phi_S^{\tot}([X\xrightarrow{p} \A^1_S])_s = \atotvan_{f_s/\kappa(s)}.$$
For any $s\in T$, $f_s:X_s\to \A^1_k$ is proper and the fibre $X_s$ is smooth over $\kappa(s)$ by the assumption on $p$. For $s\in S\setminus T$, the fibre $X_s$ is empty and therefore $\Phi_S^{\tot}([X\xrightarrow{p} \A^1_S])_s = 0$ in this case. Thus, by the characterisation of $\Phi^{\tot}_{\kappa(s)}$ in theorem \ref{motmeasureA1}, we may conclude, as in the proof of theorem~\ref{glmfamily}, that~$\Phi_S^{\tot}$ is well-defined as a group morphism $\kvar_{\A^1_S}\to \M_{\A^1_S}^{\hat{\mu}}$ and that the diagram in the statement is commutative. As in the proof of theorem~\ref{glmfamily}, the fact that $\Phi_S^{\tot}$ is a morphism of $\kvar_{S}$-algebras follows from the fact that $\Phi^{\tot}_{\kappa(s)}$ is a morphism of $\kvar_{\kappa(s)}$-algebras for every $s\in S$, using lemma~ \ref{function.equality}.
%To establish the fact that $\Phi^{\tot}$ is a morphism of $\kvar_{S}$-algebras, we need to check, for any $a\in \kvar_S$ and any $x,y\in \kvar_{\A^1_S}$ the equalities
%$\Phi^{\tot}_S(a) = a$, $\Phi^{\tot}_S(x + y) = \Phi^{\tot}_S(x) + \Phi^{\tot}_S(y)$ and $\Phi^{\tot}_S(xy) = \Phi^{\tot}_S(x)\Phi^{\tot}_S(y)$. By lemma \ref{function.equality}, it suffices to check them above every $s\in S$, and there they are true because $\Phi^{\tot}_{\kappa(s)}$ is a morphism of $\kvar_{\kappa(s)}$-algebras.
\end{proof}

\begin{lemma}\label{Phipropbase} Let $S$ be a $k$-variety, and let $X$ be a variety over $S$, with structural morphism $u:X\to S$. Denote by $u_0$ the morphism $X\xrightarrow{0\times_k u}\A^1_k\times_kS\simeq \A^1_S.$ Then
\begin{enumerate}[(a)]\item\label{affineidbase} $\Phi^{\tot}_S([X\xrightarrow{u_0} \A^1_S]) = [X\xrightarrow{u_0} \A^1_S]$ (with trivial action).
\item\label{affinezerobase} $\Phi^{\tot}_S([X\times_S \A^1_S\xrightarrow{\pr_2}\A^1_S]) = 0 $
\end{enumerate}
\end{lemma}
\begin{proof} This follows from lemma \ref{Phiprop} and proposition \ref{motmeasurebase} by lemma \ref{function.equality}. 
\end{proof}

\subsection{A motivic measure on Grothendieck rings of varieties with exponentials}
We come to the final form of the motivic vanishing cycles measure, which will be the one we are going to use. For the moment, we have constructed, for every $k$-variety $S$, a motivic measure $\Phi_S^{\tot}$ defined on the ring $(\kvar_{\A^1_S},\convv)$ and with values in the ring $(\widetilde{\M}_{\A^1_S}^{\hat{\mu}},\convv)$. For our purposes, it will be convenient to view $\Phi_S^{\tot}$ as a motivic measure on the Grothendieck ring with exponentials over $S$, and to compose it with the pushforward map $(\epsilon_S)_!$ where $\epsilon_S:\A^1_S\to S$ is the structural morphism.
\begin{theorem}[Motivic vanishing cycles measure]\label{motmeasuremain} Let $k$ be a field of characteristic zero.
\begin{enumerate}\item There is a unique morphism
$$\Phi_k:\expm \to (\M_k^{\hat{\mu}},\ast)$$
of $\M_k$-algebras, called the motivic vanishing cycles measure, such that, for any proper morphism $f:X\to \A^1_k$, with source a smooth $k$-variety $X$, one has
$$\Phi_k([X\xrightarrow{f}\A^1_k]) = \epsilon_!(\atotvan_f)$$ where $\epsilon:\A^1_k\to k$ is the structural morphism. 
\item Let $S$ be a $k$-variety. There is a unique morphism $\Phi_S:\expp_S\to (\M_S^{\hat{\mu}},\ast)$ of $\M_S$-algebras such that, for any $s\in S$, the diagram
$$\xymatrix{ \expp_S \ar[r]^-{\Phi_S}\ar[d] & \M_S^{\hat{\mu}}\ar[d] \\
              \expp_{\kappa(s)}\ar[r]^-{\Phi_{\kappa(s)} }& \M_{\kappa(s)}^{\hat{\mu}}}$$
              commutes. 
\end{enumerate}
Moreover, for any $k$-variety $S$, the restriction of $\Phi_S$ to $\M_S$ coincides with the natural inclusion $\M_S\to \M_S^{\hat{\mu}}$. \end{theorem}\index{phi@$\Phi_k, \Phi_S$}\index{motivic vanishing cycles measure}
\begin{proof} Property (\ref{affinezerobase}) in lemma \ref{Phipropbase} says that $\Phi_S^{\tot}$ sends the additional relation  $$[X\times_S \A^1_S\xrightarrow{\pr_2}\A^1_S]$$ defining $\evar_{\A^1_S}$ to zero. Using remark \ref{expquotientmorphism}, we see that the morphism $\Phi^{\tot}_S$ induces a morphism of $\kvar_S$-algebras $\evar_S\to (\widetilde{\M}_{\A^1_S}^{\hat{\mu}},\convv)$. By property (\ref{affineidbase}) of lemma (\ref{Phipropbase}), the class $\LL_S$ goes to the invertible element $\LL_0\in\widetilde{\M}_{\A^1_S}^{\hat{\mu}}$ so that $\Phi^{\tot}$ extends by $\M_S$-linearity to a morphism of $\M_S$-algebras $\expp_S\to (\widetilde{\M}_{\A^1_S}^{\hat{\mu}},\convv)$. By the localised version of lemma \ref{epsilonmorphism}, the pushforward $(\epsilon_S)_!$ is a morphism of $\M_S$-algebras $(\widetilde{\M}_{\A^1_S}^{\hat{\mu}},\convv)\to (\M_S,\ast)$, so putting $\Phi_S:=(\epsilon_S)_!\circ \Phi_S^{\tot}$ we get a  morphism of $\M_S$-algebras $\Phi_S:\expp_S\to (\M_S^{\hat{\mu}},\ast)$. Part 1 of the statement then follows immediately from theorem \ref{motmeasureA1}, whereas part 2 comes from proposition \ref{motmeasurebase}. Finally, the statement about the restriction of $\Phi_S$ to~$\M_S$ is seen to be true by property (\ref{affineidbase}) of lemma~\ref{Phipropbase}.
\end{proof}

The following proposition is the motivic version for the dimensional topology of the triangular inequality 
$$\left\vert\sum_{x\in X(\F_q)}\psi(f(x))\right\vert \leq |X(\F_q)|$$
for $X$ a variety over $\F_q$, $f:X\to \A^1_{\F_q}$ a morphism and $\psi:\F_q\to\C^*$ a non-trivial character. 
\begin{prop}[Triangular inequality]\label{triangulardim} Let $S$ be a $k$-variety, $X$ a variety over $S$ and $f:X\to \A^1_k$ a morphism. Then
$$\dim_S(\Phi_S([X,f]))\leq \dim_S X$$
\end{prop}\index{phi@$\Phi_S$!triangular inequality}\index{motivic vanishing cycles measure!triangular inequality}\index{triangular inequality!for $\dim_S$}
\begin{proof} It suffices to prove this above every point $s\in S$, so we may assume $S=\spec k$. Then, up to adding classes of strictly smaller dimension which can be dealt with by induction, we may assume that $X$ is smooth and $f$ is proper. In this case $\Phi([X,f]) = \epsilon_!\atotvan_f$ and the result follows from remark~\ref{totvandimension}.
\end{proof}

\subsection{A compactification lemma}
The following lemma, which we used in the proof of proposition \ref{motmeasureringmorph}, was stated and proved in \cite{LSmat} in the case where the field $k$ is assumed to be algebraically closed of characteristic zero. We show here that it remains true even if the field is no longer assumed to be algebraically closed.

We are going to deduce this from the proof of proposition 6.1 in \cite{LSmat}, by proving that the construction of $Z$ commutes with base change to any extension of the field $k$, and that the statements of the lemma remain true over $k$ if they are true over an extension of $k$. For this, most of the arguments will go by faithfully flat descent (see EGA IV, 2.7.1), using that for any extension $k'$ of $k$, the structural morphism $\spec k'\to \spec k$ is faithfully flat. 

%Following tag 0C3H in the Stacks project, we recall that if $f:X\to Y$ is a morphism of varieties over $k$ which is flat and such that its nonempty fibres are of constant dimension $d$, then the singular locus of $f$ may be seen as the closed subscheme of $X$ cut out by the $d$-th Fitting ideal of $\Omega_{X/S}$ (lemma 30.10.3), and therefore commutes with arbitrary base change (lemma 30.10.1). 
For this, we recall that, if $k'$ is an extension of $k$ (which is still assumed to be of characteristic zero), then (see e.g. tag 0C3H in the Stacks project) :
\begin{enumerate} [(a)]
\item\label{absolutesing} The formation of the singular locus $X^{\mathrm{sing}}$ of a variety $X$ over $k$ commutes with base change to~$k'$.
\item  For a morphism $f:X\to \A^1_k$, the formation of the singular locus of $f$ commutes with base change to~$k'$, i.e: $\mathrm{Sing} f_{k'} = (\mathrm{Sing} f)\times _kk'$. 
\end{enumerate}
%Also, it follows from faithfully flat descent along $\spec K\to \spec k$ that for any of the following properties, if $f_{K}:X_{K}\to Y_{K}$ has it, then $f$ too.
%\begin{enumerate}[(i)]
%\item isomorphism
%\item smooth
%\item projective
%\end{enumerate}
\begin{prop}\label{compactification} Let $k$ be a field of characteristic zero. Let $X$ and $Y$ be smooth varieties and let $f:X\to \A^1_k$ and $g:Y\to \A^1_k$ be projective morphisms. Then there exists a smooth quasi-projective $k$-variety $Z$ with an open embedding $X\times Y \hookrightarrow Z$ and a projective morphism $h:Z\to \A^1_k$ such that the following conditions are satisfied.
\begin{enumerate}[(i)] 
\item The restriction of $h$ to $X\times Y$ is $f\conv g$. 
\item All critical points of $h$ are contained in $X\times Y$, i.e. $\mathrm{Sing}(f\conv g) = \mathrm{Sing}(h)$.
\item The boundary $Z\setminus X\times Y$ is the support of a simple normal crossing divisor with pairwise distinct smooth irreducible components $D_1,\ldots,D_s$. 
\item For every $p$-tuple $I = (i_1,\ldots,i_p)$ of indices (with $p\geq 1$) the morphism
$$h_{I}:D_{I}:=D_{i_1}\cap \ldots \cap D_{i_p} \to \A^1_k$$
induced by $h$ is projective and smooth, so that in particular all $D_{I}$ are smooth quasi-projective $k$-varieties.
\end{enumerate}
\end{prop}
\begin{proof} All references in what follows are to \cite{LSmat} unless otherwise stated. %We assume from beginning on that $X$ and $Y$ are irreducible. 

\textit{Condition} (K): Lunts and Schnürer define a condition on a pair $(U,I)$ where $U$ is a scheme and $I\subset \OO_U$ an ideal sheaf, called condition (K), under which one may associate to $(U,I)$ a monomialisation $c_I(U)\to U$. Since we will change fields, we view (K) as a condition on the triple $(U,I,k)$ where $k$ is the given base field:
\begin{enumerate}[(K)] \item $U$ is a reduced scheme of finite type over $k$, $I$ is not zero on any irreducible component of $U$, and the closed subscheme $V(I)$ defined by $I$ contains the singular locus $U^{\mathrm{sing}}$ of $U$. 
\end{enumerate}
\textbf{Claim:} Let $k'$ be an extension of $k$ and let $U$ be a scheme over $k$. If $(U_{k'},I\otimes_kk',k')$ satisfies condition (K), then so does $(U,I,k)$. 

Indeed, the fact that $U$ is of finite type and reduced follows from the faithful flatness of $\spec k'\to \spec k$. If $I$ is zero on some irreducible component $U$, then $I\otimes_kk'$ would be zero on any irreducible component of $U_{k'}$ lying over this component. Finally, by statement (\ref{absolutesing}) above, the condition that $V(I)$ contains $U^{\mathrm{sing}}$ descends as well. 

\textit{Monomialisation;} The monomialisation procedure recalled in remark 6.3 and used throughout the proof, may be done for any field of characteristic zero, commutes with extension of scalars, as stated in \cite{Kollar}, 3.34.2, 3.35 and 3.36. 

\textit{Compactification of one morphism} Proposition 6.4 is a first compactification result, which produces, from a smooth quasi-projective variety $X$ and a morphism $f:X\to \A^1$, a smooth projective variety $\bar{X}$ with an open embedding $X\hookrightarrow \bar{X}$ and a morphism $\bar{f}:X\to \Proj^1_k$ extending $f$  and such that the fibre at infinity is a strict normal crossing divisor. This may be done over a field that is not necessarily algebraically closed, and the construction commutes with change of fields, since it consists in applying monomialisation and a root extraction morphism. 

\textit{Shifting from addition to projection.} The construction of $Z$ and $h$ uses a morphism $\sigma:\Proj^1\times \A^1\to \Proj^1\times\Proj^1$ compactifying the automorphism $\A^1\times \A^1\xrightarrow{\sim} \A^1\times \A^1$ sending $(x,y)$ to $(x,y-x)$, and therefore transforming the addition map $\A^1\times \A^1\to \A^1$ into the second projection. The image of $\sigma$ is $\A^1\times \A^1 \cup \{(\infty,\infty)\}$, and $\sigma^{-1}(\infty,\infty) = \{\infty\}\times \A^1$ is denoted by $E$.  Starting from compactifications $\overline{X}\xrightarrow{\overline{f}} \Proj^1$ and $\overline{Y}\xrightarrow{\overline{g}}\Proj^1$ of the given morphisms, obtained using proposition 6.4, one considers the pullback diagram
$$\xymatrix{T \ar[r]^{\widehat{\sigma}}\ar[d]_{\theta} &\overline{X}\times \overline{Y} \ar[d]^{\overline{f}\times \overline{g}}\\
\Proj^{1}\times \A^1 \ar[r]^{\sigma} &\Proj^{1}\times \Proj^1}$$

Lunts and Schnürer's proof then goes on with proving that $(T,\theta^{-1}(I_E)\OO_T)$ satisfies condition (K), so that one may form the monomialisation $\gamma:Z\to T$, which provides a morphism
$$h:Z\xrightarrow{\gamma} T \xrightarrow{\theta} \Proj^{1}\times \A^1\xrightarrow{\pr_2} \A^1,$$
which they check satisfies all requested properties. 

By what we said on monomialisation above, the construction of $h$ and $Z$ commutes with change of fields. Thus, by Lunts and Schnürer's result, $h_{\bar{k}}$ and $Z_{\bar{k}}$ do the job over $\bar{k}$, and it suffices to show that the properties they satisfy remain true over $k$. First of all, since $(T,\theta^{-1}(I_E)\OO_T)$ satisfies condition $(K)$ over $\bar{k}$, it does also satisfy it over $k$, so monomialisation may be performed. The fact that $h$ is projective is clear since it is a composition of projective morphisms. Then points $(i)$ and $(iii)$ in the statement of the theorem are satisfied automatically. As for point $(ii)$, Lunts and Schnürer's lemma says that the base change to $\bar{k}$ of the open immersion $\mathrm{Sing}(f\conv g)\to \mathrm{Sing}(h)$ is in fact an isomorphism. By faithfully flat descent, we already have an isomorphism over $k$.
Finally, it remains to check point (iv), or more precisely, the smoothness part of it, the projectivity being immediate. Again, this comes from faithfully flat descent. 
\end{proof}
\begin{remark} When reading carefully Lunts and Schnürer's proof, one notices that in fact it can be made to work without the assumption that $k$ is algebraically closed. Indeed, they construct a diagram 
$$
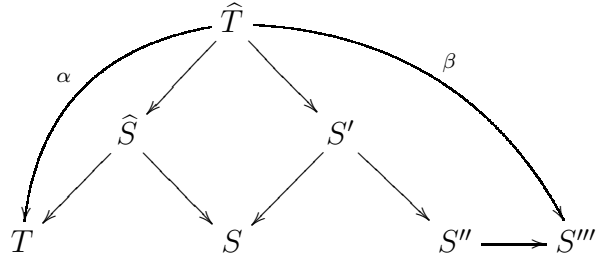
\xymatrix{ & &\widehat{T}\ar[ld]\ar[rd] \ar@/^2pc/[rrrdd]^{\beta} \ar@/_2pc/[lldd]_{\alpha} & & & \\
& \widehat{S}\ar[ld]\ar[rd] & & S'\ar[ld]\ar[rd] & \\
T & & S & & S''\ar[r] & S'''}$$
where all arrows are étale and where $S'''$ is of the form $\A^1_k\times L$ where
$$L = \spec k[x_1,\ldots,x_s,y_1,\ldots,y_t]/(x^{\mu}-y^{\nu})$$
with $x^{\mu} := x_1^{\mu_1}\ldots x_s^{\mu_s}$ and $y^{\nu} := y_1^{\nu_1}\ldots,y_t^{\nu_t}$ where the $\mu_i$ and $\nu_i$ are positive integers. This way one essentially reduces to verifying most properties directly for $L$ or for $S'''$. The only point in the proof when one seems to use the algebraic closedness is when checking that $(L, (x^{\mu}))$ satisfies condition~(K), but as we remarked above, since it is true geometrically it is already true over $k$. 
\end{remark}

\begin{remark}[Compatibility of nearby cycles morphism with sums of proper morphisms]\label{vancyclesnonproper} Let $a\in k$ and let $\scr{S}_{\id -a}:\M_{\A^1_k}\to \M_{k}^{\hat{\mu}}$ be the morphism from theorem \ref{glmtheorem} for $Y=\A^1$ and $g :\A^1\to \A^1$ given by $x\mapsto x-a$. Let $X$ and $Y$ be smooth $k$-varieties and $f:X\to \A^1$, $g:Y\to \A^1$ projective morphisms. Proposition \ref{compactification} shows that, though $f\conv g:X\times Y\to \A^1$ is not necessarily proper, we nevertheless have
$$\scr{S}_{\id -a}([X\times Y \xrightarrow{f\conv g} \A^1]) = (f\conv g)_!\ \psi_{f\conv g -a}.$$
Indeed, using notation from proposition \ref{compactification},  we may write
$$\scr{S}_{\id-a}([X\times Y \xrightarrow{f\conv g} \A^1]) = \scr{S}_{\id-a}([Z\xrightarrow{h}\A^1_k]) + \sum_{I\neq \varnothing}(-1)^{|I|}\scr{S}_{\id -a}([D_I\xrightarrow{h_I}\A^1_k]).$$
Since all $h_I$ are projective and smooth, we have
$$\scr{S}_{\id-a}[D_I\xrightarrow{h_I}\A^1_k] = (h_I)_!\psi_{h_I - a} =0,$$
so that in fact, using that $h$ is projective, we have
$$\scr{S}_{\id-a}([X\times Y \xrightarrow{f\conv g} \A^1]) = \scr{S}_{\id-a}([Z\xrightarrow{h}\A^1_k]) = h_!\psi_{h-a}.$$
On the other hand, since $\mathrm{Sing}(h)=\mathrm{Sing}(f\conv g)$, and $\psi_{h-a}$ is supported on $\mathrm{Sing}(h)$, we have $\psi_{h-a} = \psi_{f\conv g -a}$ (see corollary 3.6 in \cite{LS}), whence the result.
\end{remark}

\section{The Thom-Sebastiani theorem: an explicit example}\label{sect.TSexample}\index{Thom-Sebastiani!example}
To illustrate the contents of section \ref{sect.TS}, we compute in this section both sides of equality (\ref{TSequation}) for $X_1 = X_2 = \A^1_{\C}$ over $k=\C$, with $f_1 = f_2 = (x\mapsto x^2)$.  
\subsection{Computation of left-hand side}\label{LHScomputation}
Here we are dealing with the variety $X = \A^2$ together with the morphism $f:(x,y)\mapsto x^2 + y^2$.  The only critical value is zero, so $\totvan_f$ is just $\phi_f$, seen as an element of $\M_X^{\hat{\mu}}$. Since in $\C$ we have the decomposition $f(x,y) = (x+iy)(x-iy)$, we may apply theorem \ref{DLrat} with $h = \id$, $E_1 = \{x + iy=0\}$, $E_2 = \{x-iy = 0\},$ so that $a_1 = a_2 = 1$, $a_{12} = 1$ (in particular, $\hat{\mu}$-actions are trivial) and
\begin{eqnarray*}\phi_f &= &[X_0(f)\xrightarrow{\id} X_0(f)] - [E_1^{\circ}\to X_0(f)] -[E_2^{\circ}\to X_0(f)] + (\LL-1)[E_1\cap E_2\to X_0(f)]\\
& = & \LL[\{(0,0)\}\to X_0(f)]\in \M_{X_0(f)}^{\hat{\mu}}
\end{eqnarray*}
since $[X_0(f)\xrightarrow{\id} X_0(f)]  = [E_1^{\circ}\to X_0(f)] +[E_2^{\circ}\to X_0(f)] + [E_1\cap E_2\to X_0(f)]$.

Thus, $$\totvan_f = \LL[\{(0,0)\}\to \A^2]\in \M_{\A^2_{\C}}^{\hat{\mu}}.$$

\subsection{Computation of the right-hand side}\label{RHScomputation}

Here we are dealing with $Y = \A^1$ with $g:\A^1\to \A^1$, $x\mapsto x^2$. Again, the only critical value is zero. We may again apply theorem \ref{DLrat} with $h = \id$, the only irreducible component of $Y_0(g)$ being $E = \{0\}$, with multiplicity $a =2$, so that we consider a double cover $\tilde{E}\to E$ with $\mu_2$-action. In other words, $\tilde{E}$ is just a pair of two points, exchanged by the action of the generator of $\mu_2$. By the formula, we have
$$\phi_f = [Y_0(g)\xrightarrow{\id}Y_0(g)] - [\tilde{E}\to Y_0(g)]\in \M_{Y_0(g)}^{\hat{\mu}},$$
so that $\totvan_g = [\{0\}\to\A^1] - [\tilde{E}\to \A^1] \in \M_{\A^1}^{\hat{\mu}}.$

\begin{remark} By example 2.7 in \cite{LS}, whenever we have two $k$-varieties $S_1,S_2$ and $p_i:Z_i\to S_i$ is a variety over $S_i$ with $\hat{\mu}$-action, and the action of $\hat{\mu}$ on $Z_2$ is trivial, then $$\Psi(Z_1\times Z_2\xrightarrow{p_1\times p_2}S_1\times S_2) = [Z_1\times Z_2\to S_1\times S_2].$$
\end{remark}
This shows that 
$$\Psi(\totvan_g\boxtimes \totvan_g) = [\{(0,0)\}\to \A^2] - 2[\{0\}\times \tilde{E}\to \A^2] + \Psi([\tilde{E}\times \tilde{E}\to \A^2]).$$

Note that all the $\A^2$-varieties here are supported above the point $\{(0,0)\}$ of $\A^2$. We are therefore going to stop writing the morphisms to $\A^2$, which implicitly are all taken to be constant equal to $(0,0)$. 

Now we are going to compute $\Psi([\tilde{E}\times \tilde{E}])$. By definition, this is
$$\Psi([\tilde{E}\times \tilde{E}]) = [(\tilde{E}\times \tilde{E})\times^{\mu_2\times \mu_2}F_0^2] - [(\tilde{E}\times \tilde{E})\times^{\mu_2\times \mu_2}F_1^2].$$

Denote the two points of $\tilde{E}$ by $e_{-1}$ and $e_1$. Then the product $\tilde{E}\times \tilde{E}\times F_i^2$ is simply given by four copies of $F_i^2$, corresponding to each pair $(e_i,e_j)$ for $i,j\in\{-1,1\}$. 
Moreover, these copies are all identified via the $\mu_2\times \mu_2$-action: indeed, any element $(\epsilon,\eta)\in\mu_2\times \mu_2$ induces an isomorphism
$$\begin{array}{ccc}\{(e_1,e_1)\}\times F_i^2&\to& \{(e_{\epsilon},e_{\eta})\}\times F_i^2\\
                          (e_1,e_1,x,y)&\mapsto& (e_{\epsilon},e_{\eta},\epsilon x,\eta y)
\end{array}$$
so that $(\tilde{E}\times \tilde{E})\times^{\mu_2\times \mu_2}F_i^2$ is in fact isomorphic to $F_i^2$, endowed with the diagonal $\mu_2$-action. 

\begin{lemma} 
\begin{enumerate} \item The morphism 
$$\begin{array}{ccc} F_0^2&\to& \G_m\times \{-1,1\}\\
                     (x,y)& \mapsto & (x,\frac{y}{ix})\end{array}
                     $$ is an isomorphism (over $\C$), identifying $F_0^2$ with the disjoint union of two copies of $\G_m$. It is equivariant if one endows each copy of $\G_m$ with the obvious $\mu_2$-action by translation.
                     \item The morphism
$$\begin{array}{ccc} F_1^2&\to& \G_m\setminus\{-1,1,i,-i\}\\
                     (x,y)& \mapsto & x + iy\end{array}
                     $$
                     is an isomorphism (over $\C$), identifying $F_1^2$ with $\G_m\setminus \{-1,1,i,-i\}$. It is equivariant if one endows $\G_m\setminus\{-1,1,i,-i\}$ with the action induced by the obvious $\mu_2$-action by translation on~$\G_m$.

\end{enumerate}\end{lemma}
\begin{proof}\begin{enumerate}\item A point $(x,y)$ of $\G_m^2$ is an element of $F_0^2$ if and only if either $x =iy$ or $x = -iy$: an inverse to the map in the statement if therefore given by $(x,\epsilon)\mapsto (x,i\epsilon x)$. The statement about actions follows immediately. 
\item Rewriting the equation of $F_1^2$ in the form $(x + iy)(x-iy) = 1$, we see that $x + iy$ is always non-zero, and that if $x + iy$ is equal to some $a\in \G_m$, then $x-iy$ is equal to $a^{-1}$. This remark allows us to construct an inverse
$$a\mapsto \left(\frac{a + a^{-1}}{2},\frac{a-a^{-1}}{2i}\right), $$
which is well-defined and with image contained in $F_1^2$ whenever $a$ is a complex number outside the set $\{0,1,-1,i,-i\}$. Again, the statement on actions is immediate. 
\end{enumerate}
\end{proof}

Combining the results in the lemma, we have
$$\Psi([\tilde{E}\times \tilde{E}]) = 2[\G_m,\mu_2] - [\G_m\setminus\{-1,1,i,-i\},\mu_2]$$
Denoting by $[\tilde{E},\mu_2]$ the class of a union of two points exchanged by the generator of $\mu_2$, as above, we have 
$$[\G_m\setminus\{-1,1,i,-i\},\mu_2] = [\G_m,\mu_2] - 2[\tilde{E},\mu_2],$$
whence
$$\Psi([\tilde{E}\times \tilde{E}]) = [\G_m,\mu_2] + 2[\tilde{E},\mu_2].$$

Thus, finally, we have, observing that $[\{0\}\times\tilde{E},\mu_2] = [\tilde{E},\mu_2]$,  
$$\Psi(\totvan_g\boxtimes\totvan_g) = 1 + [\G_m,\mu_2] = [\A^1,\mu_2],$$
that is, $\A^1$ with the generator of $\mu_2$ acting through $x\mapsto -x$. By relation (\ref{actionrelation}) in the Grothendieck ring $\M_{\A^2_{\C}}^{\hat{\mu}}$, this is equal to the left-hand side $\LL$ (i.e. $\A^1$ with the trivial action) computed in section \ref{LHScomputation}, whence the result.

\begin{remark} Our computation shows that in  $\M^{\hat{\mu}}_{\A^2_{\C}}$ we have the relation
$$\Psi(\totvan_{x^2}\boxtimes\totvan_{x^2}) = \totvan_{x^2 + y^2}$$
which in our calculation boils down to the equality
$$(1-[\tilde{E},\mu_2])\ast (1-[\tilde{E},\mu_2]) = \LL$$
in $\M_{\C}^{\hat{\mu}}$. Thus, the class $1-[\tilde{E},\mu_2]$ may be seen as a ``square root'' of $\LL$ for the product $\ast$. For obvious dimensional reasons, such a square root does not exist in the ring $\M_{\C}$. 
\end{remark}

\chapter{Motivic Euler products}\label{eulerproducts}
The possibility of writing a function as an Euler product, that is, an infinite product of ``local factors'', is a very important tool in number theory. In particular, the Hasse-Weil zeta function of a variety $X$ over a finite field~$\F_q$, defined by 
$$\zeta_X(t) = \exp\left(\sum_{m\geq 1}\frac{|X(\F_{q^m})|}{m}t^{m}\right)\in \Z[[t]],$$
can be rewritten as a product over the closed points $X_{\text{cl}}$ of $X$ in the following manner:
$$\zeta_X(t) = \prod_{x\in X_{\text{cl}}}\frac{1}{1-t^{\deg x}},$$
the expansion of which gives
\begin{equation}\label{hasseweilexpansion}\zeta_X(t) = \sum_{n\geq 0}|\{\text{effective zero-cycles of degree}\ n\ \text{on}\ X\}|t^n.\end{equation}

The latter expression led Kapranov (\cite{Kapr}) to define a motivic analogue of the Hasse-Weil zeta function: for a variety $X$ over a field $k$, it is given by
$$Z_X(t) = \sum_{n\geq 0}[S^nX]t^n\ \in\kvar_k[[t]],$$\index{Kapranov's zeta function}\index{ZX@$Z_X(t)$, Kapranov's zeta function}
where $S^nX$ is the $n$-th symmetric power of $X$, a variety over $k$ which parametrises precisely effective zero-cycles of degree $n$ on $X$, and $[S^nX]$ is its class in the Grothendieck ring of varieties $\kvar_k$. When $k$ is a finite field, the series $Z_X(t)$ specialises to $\zeta_X(t)$ via the counting measure
$$\begin{array}{ccc}\kvar_k&\to& \Z\\
                  \left[X\right]& \mapsto & |X(k)|
\end{array}$$
as one can see from (\ref{hasseweilexpansion}).  Since then, several mathematicians have been studying the properties of this function and trying to measure the scope of the analogy with the Hasse-Weil zeta function. Kapranov himself showed for example that it was rational for a smooth projective curve having a zero-cycle of degree 1, whereas a result by Larsen and Lunts (\cite{LL}) states that it is not rational for a surface. It is however an open question whether it is rational when regarded as a power series with coefficients in $\M_k= \kvar_k[\LL^{-1}]$, the ring obtained from $\kvar_k$ by inverting the class of the affine line. 

In this chapter we are going to broaden the parallel with the Hasse-Weil zeta function by showing that Kapranov's zeta function can be endowed with an Euler product decomposition. More precisely, we are going to give a way of making sense of expressions of the form
$$\prod_{x\in X}\left(1 + X_{1,x}t + X_{2,x}t^2 + \ldots\right)$$
where $X$ is a variety over $k$, $(X_i)_{i\geq 1}$ is a family of varieties over $X$ (or, more generally, a family of classes in $\kvar_X$), and $X_{i,x}$ must be thought of as the class of the fibre of~$X$ above $x\in X$ in $\kvar_{k(x)}$, where $k(x)$ is the residue field at $x$. The decomposition of Kapranov's zeta function will in particular be covered by this definition, but so will many other power series, and in particular those occurring when studying motivic height zeta functions in chapter \ref{motheightzeta}.

It is important to point out that our construction was inspired by \cite{gusein}, where a first step towards infinite motivic products was made. Indeed, the authors define a notion of motivic ``power'', which is a special case of our construction, recovered when all $X_i$ are of the form $X\times_kA_i$ for a family of varieties $(A_i)_{i\geq 1}$ over $k$, so that all factors are equal to $$(1 + A_1 t + A_2t^2 + \ldots)\in\kvar_k[[t]],$$ and the resulting product can be thought of as this series ``raised to the power $X$''.

Let us sketch the contents of this chapter. The overall idea is to generalise the notion of motivic zeta function in an appropriate way, to define the Euler product notation for all these generalised zeta functions, and then to show that this notation actually does behave like a product. Section~\ref{symmprod} will be devoted to the definition of the coefficients of those zeta functions, which we will call symmetric products. They are a generalisation of the notion of symmetric power. The subsequent sections contain proofs of numerous properties of these products that are necessary to ensure subsequent good behaviour of our Euler products. In particular, in \ref{relativesetting} we show how to iterate this construction, which will enable us to make sense of double products later. We explain there that the iteration makes it necessary to consider symmetric products of families indexed by more general sets than the set of positive integers. In \ref{cuttingintopieces} it is shown how the symmetric product of a family of varieties can be expressed in terms of symmetric products of constructible sets partitioning these varieties, which leads to multiplicative properties for zeta functions. Section~\ref{affinespaces} shows how symmetric products behave if the original varieties are multiplied by some affine spaces. Furthermore, we show that these definitions and properties may be extended to  families of non-effective classes in a Grothendieck ring in \ref{symprodclasses}, and to (classes of) varieties with exponentials in \ref{sect.symprodexp}. Finally, in section \ref{sect.locsymproducts}, we define symmetric products of classes in localised Grothendieck rings.

Section~\ref{eulerprod} defines the Euler product notation and deduces all the properties following from the previous sections that show that we can think of it as a product and do calculations with it. Section \ref{coefficients} shows, thanks to a further slight generalisation of the notion of symmetric product, that inside the product, one can allow a finite number of constant terms not equal to~$1$. 

\begin{notation}\label{partitionnotation} We will denote by $\N$ the set of non-negative integers, and by $\N^{*}$ the set of positive integers. For any set $I$, we define
$$\N^{(I)} = \{(n_i)_{i\in I}\in\N^I,\ n_i = 0\ \text{for almost all $i$}\}$$\index{NI@$\N^{(I)}$}
where ``almost all'' means ``all but a finite number of''. It is a monoid, if we endow it with addition coordinate by coordinate:
$$(n_i)_{i\in I} + (n'_i)_{i\in I} = (n_i + n'_i)_{i\in I}.$$
 Moreover, it has a natural partial order defined by 
 $$(n_i)_{i\in I} \leq (n'_i)_{i\in I} \ \ \text{if and only if}\ \ n_i\leq n'_i \ \text{for all}\ i\in I.$$
 For two elements $\pi = (n_i)_{i\in I}$ and $\pi' = (n'_i)_{i\in I}$ such that $\pi\leq \pi'$, we may also define their difference: $$\pi'-\pi = (n'_i-n_i)_{i\in I}\in\N^{(I)}.$$
 An element $\pi\in\N^{(I)}$ can be thought of as a finite collection of elements of $I$, each element $i$ coming with a multiplicity $n_i$. That's why, especially in the case when $I = \N^*$, such an element will sometimes be written in the form $[a_1,\ldots,a_p]$ where $a_1,\ldots,a_p$ are elements of~$I$, each appearing with the correct multiplicity, so that the integer $p$ is equal to $\sum_{i\in I}n_i$. We denote by $|\pi|$ the integer $\sum_{i\in I}n_i$.
 
 The special case $I=\N^{*}$ will be particularly important: in this case, an element $\pi\in\N^{(I)}$ is called a \textit{partition}. \index{partition}Indeed, since $\pi = (n_i)_{i\geq 1} = [a_1,\ldots,a_p]$ is in this case a collection of positive integers with multiplicities, we may associate to it the number
 $$n = \sum_{i\geq 1}in_i = a_1 + \ldots + a_p$$
 these integers sum to, and $\pi$ is simply a partition of the integer $n$. The elements $a_i$ are called the \textit{parts} of the partition. Note that when denoting partitions of integers in the form $[a_1,\ldots,a_p]$, some authors require the sequence of the $a_i$ to be non-decreasing. In order to simplify the statements of some results below, we prefer to say that the order of the $a_i$ is of no importance: we consider the partitions $[a_1,\ldots,a_p]$ and $[a_{\sigma(1)},\ldots,a_{\sigma(p)}]$ to be the same for any permutation $\sigma\in \Sym_p$. However, when writing concrete partitions, we will often put the integers in increasing order for clarity.
\end{notation}

In this chapter, $R$ will be a variety over a perfect field $k$.
\section{Symmetric products}\label{symmprod}

\subsection{Introduction: Symmetric powers of a variety}\label{sympower}

Let $X$ be a quasi-projective variety over a perfect field $k$. For every non-negative integer $n$, there is a natural action of the symmetric group $\Sym_n$ on the product $X^n$ by permuting the coordinates, and it is a classical result that the quotient $S^nX = X^n / \Sym_n$ exists as a variety (if $X$ is quasi-projective, which we assumed). We will call this variety the $n$-th \textit{symmetric power} \index{symmetric power!of a variety} of $X$. By convention $S^0X$ will be $\spec k$.

The rational points of the variety $S^nX$ correspond to effective zero-cycles of degree $n$ on $X$. Any such zero-cycle $\sum_xn_xx$ determines a partition $$\sum_x\underbrace{(n_x + \ldots + n_x)}_{\deg x\ \text{terms}} = n$$ of $n$, which we will denote by $\pi$. The subset $S^{\pi} X$ of $S^nX$ consisting of the zero-cycles inducing this partition $\pi$ is locally closed in $S^nX$, and can be constructed directly in the following way: for all $i\geq 1$, denote by $n_i$ the number of times the integer $i$ occurs in this partition. Every zero-cycle determining the partition $\pi$ is of the form
$$\sum_{i\geq 1}i(x_{i,1} + \ldots +x_{i,n_i})$$
where the points $x_{i,j}$ are geometric points of $X$, all distinct. 
Consider therefore the product $X^{\sum_{i\geq 1}n_i}$, from which we remove the diagonal $\Delta$, that is, the points having at least two equal coordinates. The product $\prod_{i\geq 1}\Sym_{n_i}$ of symmetric groups has a left action on $\prod_{i\geq 1}X^{n_i} = X^{\sum_{i\geq 1}n_i}$ via
$$(\sigma_i)_{i\geq 1}\cdot (x_{i,1},\ldots x_{i,n_i})_{i\geq 1}\mapsto (x_{i,\sigma_i^{-1}(1)},\ldots,x_{i,\sigma_i^{-1}(n_i)})_{i\geq 1}$$ 
for any $\sigma = (\sigma_i)_{i\geq 1}\in\prod_{i\geq 1}\Sym_{n_i}$. This action restricts to $X^{\sum_{i\geq 1}n_i}\backslash \Delta$, and the quotient will be naturally isomorphic to the above locally closed subset $S^{\pi}X$ . This observation will be the starting point of the construction in the following paragraph.

The variety $S^nX$ can thus be written as a disjoint union of locally closed sets $S^{\pi}X$ with~$\pi$ ranging over all partitions of $n$. In particular, we will denote by $S^n_*X$ the open subset of $S^nX$ corresponding to the partition $[1,\ldots,1]$ of $n$, which parametrises étale zero-cycles of degree $n$ on~$X$.

This construction may be done with $k$ replaced by a $k$-variety $R$, and products replaced by fibred products over $R$. The resulting objects will be varieties over $R$, denoted $S^nX$ and $S^{\pi}X$ as well, or $S^n(X/R)$ and  $S^{\pi}(X/R)$ if we want to keep track of the base variety. For any point $v\in R$, the fibre of $S^{\pi}(X/R)$ above $v$ will be isomorphic to $S^{\pi}(X_v/\kappa(v))$ where $X_v$ is the fibre of $X$ over $v$.

\begin{remark} Though we may define symmetric powers also over non-perfect fields, the above description of points will fail in this case. This justifies our condition on the base variety $R$.
\end{remark}
%\textit{In what follows, we will work with quasi-projective varieties over $R$, that is, quasi-projective schemes of finite presentation over $R$.}

\subsection{Quotients of schemes by finite group actions}\label{appendix}

We gather here some facts about quotients by finite group actions. A detailed account may be found in Chapter 0 of \cite{GIT}. Let $X$ be a scheme endowed with an algebraic action of a finite group $G$. 

\begin{definition} A (categorical) quotient of $X$ by $G$ is a morphism of schemes $\pi: X\rightarrow Y$ with the following two properties:
\begin{itemize}\item $\pi$ is $G$-invariant;
\item $\pi$ is universal with this property: for every scheme $Z$ over $k$ and every $G$-invariant morphism $f:X\rightarrow Z$, there is a unique morphism $h:Y\rightarrow Z$ such that $h\circ\pi = f$.
\end{itemize}
\end{definition}
Because of the universality, a quotient is unique up to canonical isomorphism if it exists. In this case, we write $Y = X/G$. Note that the universal property implies in particular that if $X$ is an $S$-scheme, then so is $X/G$. 

If $X = \spec A$ is affine, of finite type over $S$ (which may be assumed to be affine, equal to $\spec C$ for some ring $C$), then $A$ has a $G$-action, and we may define the subring $A^G$ of $A$ of all $G$-invariant elements of $A$. This induces a morphism
$$\pi: \spec A\arr \spec A^{G}.$$
One can show that this is the quotient of $X$ by $G$.

\begin{definition} We say that the action of $G$ on $X$ is good if  any $x\in X$ has an open affine neighbourhood that is preserved by the $G$-action. 
\end{definition}
For example, if $X$ is quasi-projective over a field $k$, then the action of $G$ is good. 

If the action of $G$ on the variety $X$ is good, then taking an affine cover $(U_i)_i$ of $X$ by such affine subsets, one may construct a quotient $X/G$ by glueing together the quotients $U_i/G$. It follows from \cite{GIT}, theorem 1.10, that this quotient is quasi-projective.

%\begin{remark} The class of quasi-projective varieties being stable under most of the operations we might want to perform, we are going to work with quasi-projective varieties in this paper.
%\end{remark}

\begin{prop}[\cite{Mustata}, Proposition A.8]\label{subquotient} Let $G$ be a finite group acting by algebraic automorphisms on a quasi-projective variety $X$ over $k$. Let $H$ be a subgroup of $G$, and $Y$ an open subset of $X$ such that
\begin{enumerate}\item $Y$ is preserved by the action of $H$ on $X$.
\item If $Hg_1,\ldots,Hg_r$ are the right equivalence classes of $G$ modulo $H$, then 
$$X = \bigcup_{i=1}^rYg_i$$
is a disjoint cover.
\end{enumerate}
Then the natural morphism $Y/H\arr X/G$ is an isomorphism.
\end{prop}

\subsection{Symmetric products of a family of varieties}
\label{definition}
Let $X$ be a variety over $R$, and  let $\mathscr{X}=(X_i)_{i\geq 1}$ be a family of $X$-varieties with structural morphisms $\phi_i:X_i\arr X$. All products in this section are fibred products over~$R$. 

%, and that $X_i$ is quasi-projective over $X$. %The aim of this paragraph is to provide, for any non-negative $n$ and for any partition $\pi = (n_i)_{i\geq 1}$ of $n$, a construction of a variety $S^{\pi}\scr{X}$ parametrising zero-cycles in the disjoint union of the varieties $X_i$, of the form
%$$\sum_{i\geq 1} i(x_{i,1}+\ldots +x_{i,n_i})$$
%where for each $i$, $x_{i,1},\ldots,x_{i,n_i}$ are closed points of $X_i$, mappin

Let $\pi = (n_i)_{i\geq 1}$ be a partition of the integer $n$. The variety $\prod_{i}X_i^{n_i}$ has a morphism $\prod_{i\geq 1}\phi_i^{n_i}$ to $\prod_{i}X^{n_i}$. Denote by $$\left(\prod_{i\geq 1}X^{n_i}\right)_*$$
the open subset of the latter obtained by removing the diagonal, that is, the points having at least two equal coordinates. By base change, we get the open subset $$\prod_{i}X_i^{n_i}\times_{\prod_{i}X^{n_i}}\left(\prod_{i}X^{n_i}\right)_*$$
of elements mapping to $\sum_{i}n_i$-tuples of $X$ which do not belong to the diagonal. This can be summarised by the following cartesian diagram:

$$\xymatrix{
   \left( \prod_{i\geq 1}X_i^{n_i}\right)\times_{\prod_{i\geq 1}X^{n_i}}\left(\prod_{i\geq 1}X^{n_i}\right)_*\ar@{^{(}->}[r] \ar[d] & \prod_{i\geq 1}X_i^{n_i} \ar[d]^{\prod_{i\geq 1}\phi_i^{n_i}}  \\
    \left(\prod_{i\geq 1}X^{n_i} \right)_* \ar@{^{(}->}[r] & \prod_{i\geq 1}X^{n_i}
  }$$
  
 For simplicity, in what follows we will write $\left(\prod_{i\geq 1}X_i^{n_i}\right)_*$ for the variety at the top-left corner of this diagram (when we want to specify that points were removed with respect to coordinates in~$X$, we may write $\left(\prod_{i\geq 1}X^{n_i}\right)_{*,X}$ ). Now the product $\prod_{i\geq 1}\Sym_{n_i}$ of symmetric groups acts naturally on the varieties occurring in the right column of this diagram: each~$\Sym_{n_i}$ acts on the corresponding~$X_i^{n_i}$ and $X^{n_i}$ by permutation of coordinates. It restricts to the varieties in the left column, and is compatible with the vertical maps. Passing to the quotient, the left column gives us a variety which we will denote by $S^{\pi}(\scr{X}/R)$, or simply $S^{\pi}\scr{X}$, with a map to the variety~$S^{\pi}X$ defined in the previous section.
 
  Finally, taking the disjoint union $\cup_{\pi}S^{\pi}\scr{X}$ over all partitions of $n$, we get a variety~$S^{n}\scr{X}$. \index{symmetric product!of varieties}
  
 \begin{remark}\label{diagonal}
  
  The horizontal inclusion maps of the cartesian square $$\xymatrix{\left(\prod_{i\geq 1}X_i^{n_i}\right)_*\ar[r]\ar[d] &  \prod_{i\geq 1}X_i^{n_i}\ar[d]\\
  \left(\prod_{i\geq 1}X^{n_i}\right)_*\ar[r] & \prod_{i\geq 1}X^{n_i}
}
$$ are compatible with taking the quotient, so we get well-defined maps
 
 $$\xymatrix{S^{\pi}\scr{X}\ar[r]\ar[d] &  \prod_{i\geq 1}S^{n_i}X_i\ar[d]\\
  S^{\pi}X\ar[r] & \prod_{i\geq 1}S^{n_i}X
}
$$ 

  The diagonal is a closed subset $\Delta_X$ in $\prod_{i\geq 1}X^{n_i}$. It is stable by the action of $\prod_{i\geq 1}\Sym_{n_i}$, and maps to a closed subset $\Delta_{X,\pi}$ inside $\prod_{i\geq 1}S^{n_i}X$ such that $S^{\pi}X\simeq \left(\prod_{i\geq 1}S^{n_i}X\right) \backslash \Delta_{X,\pi}$, and therefore
  $S^{\pi}\scr{X}$ is exactly the restriction of $\prod_{i\geq 1}S^{n_i}X_i$ to points mapping outside~$\Delta_{X,\pi}$. In other words, the diagonal can be removed before or after passing to the quotient.
\end{remark} 

\begin{notation}\label{allequal} If the family $\scr{X}$ is constant, that is, all $X_i$ are equal to some $X$-variety $Y$, then the resulting symmetric products will be denoted $S^{\pi}_X(Y)$ (resp. $S^n_X(Y)$). In particular, by definition, we have $S^{\pi}_X(X) = S^{\pi}X$. 
\end{notation}

\begin{example}\label{firstex}\begin{enumerate}\item If $\pi$ is the partition $[1,\ldots,1]$, we are going to write $S^{\pi}\scr{X} = S^{n}_*\scr{X}$. It corresponds to the variety $S^n_{*,X}X_1$ parametrising effective zero-cycles on~$X_1$ of degree~$n$ mapping to effective zero-cycles on~$X$ of degree $n$ in which no point occurs with multiplicity strictly greater than one. 
  \item If $\pi$ is the partition $[n]$, $S^{\pi}\scr{X} = X_n$. 
  \item If we take all $X_i$ to be equal to some $X$-variety $Y$, then $S^{\pi}\scr{X} = S^{\pi}_XY$ (see notation \ref{allequal}) corresponds exactly to effective zero-cycles of degree $n$ on $Y$ mapping to zero-cycles on $X$ with partition $\pi$. In particular, if all $X_i$ are equal to $X$, then we get the locally closed subset $S^{\pi}X$ of $S^{n}X$ described in section \ref{sympower}. 
 \item  If $X = \spec R$, then $\left(\prod_{i\geq 1}X^{n_i}\right)_*$ is empty whenever there is more than one factor, that is, except if $n_i = 0$ for all $i\geq 1$ but one (recall the product is over $R$). Since the $n_i$ are subject to the relation $\sum_{i}in_i = n$, this means that $n_i=0$ for $i<n$, and $n_n = 1$. Thus, $S^{\pi}\scr{X}$ is empty for all $\pi$ but $[n]$, and we have $S^n\scr{X} = X_n$. 
 \end{enumerate}
  \end{example}

\section{Iteration of the symmetric product construction}\label{relativesetting}

If $\scr{X} = (X_i)_{i\geq 1}$ is a family of varieties over $X$, and $X$ itself is a variety over some scheme $R$, then the symmetric product construction over $R$ gives rise to a family of varieties
$$S^{\bullet}(\scr{X}/R) = (S^{\pi}(\scr{X}/R))_{\pi\in\N^{(\N^*)}}$$
over $R$, indexed by all partitions $\pi$. 
As it is now, our definition of symmetric products doesn't allow us to carry on and construct a symmetric product of this family. The aim of this section is to generalise our construction in a way that will make this possible. This generalisation is important in itself, as it gives the correct general setting in which symmetric products may be defined.

The idea is to replace families indexed by the set $\N^*$ of positive integers by families indexed by any set $I$. Then the family of their symmetric products will be indexed by the set 
$$\N^{(I)} = \{(n_i)_{i\in I}\in \N^{I},\ n_i = 0\ \text{for almost all}\ i\},$$
 and the family of the symmetric products of those will be indexed by $\N^{\left(\N^{(I)}\right)}.$

%\begin{remark} In theory, the fact that $I$ is countable doesn't seem to be needed for what we do here.  We assume it nevertheless, as it allows us to put $I = \N^{*}$ in many places, which makes the combinatorics of our construction clearer. 
%\end{remark}

\subsection{Symmetric products of varieties indexed by any  set}\label{anyset}
Let $I$ be a set. The construction is completely analogous to the construction of symmetric products in the case where $I$ is just the set of positive integers. Let $X$ be a quasi-projective variety over~$R$, and $\scr{X} = (X_i)_{i\in I}$ a family of quasi-projective $X$-varieties. Fix $\pi= (n_i)_{i\in I}$. The product
$$\prod_{i\in I}X_i^{n_i}$$
has a morphism to $\prod_{i\in I}X^{n_i}$. We consider the open subset 
$$\left(\prod_{i\in I}X_i^{n_i}\right)_{\!\!\!*}\subset \prod_{i\in I}X_i^{n_i}$$
of points lying above the complement of the diagonal of $\prod_{i\in I}X^{n_i}$, that is mapping to points having pairwise distinct coordinates. We have a natural action of the product of symmetric groups $\prod_{i\in I} \Sym_{n_i}$ by permutation of coordinates, and we define
$$S^{\pi}\scr{X} := \left(\prod_{i\in I}X_i^{n_i}\right)_{\!\!\!*}/\prod_{i\in I}\Sym_{n_i},$$
which is a variety because the varieties we started with were quasi-projective.
\index{symmetric product! of varieties}
Note that in particular, in the case $\pi = 0$, we get $S^{0}\scr{X} = R$. 
\begin{remark} The construction is functorial in the sense that if we have two families of $X$-varieties $\scr{X} = (X_i)_{i\in I}$ and $\scr{Y} = (Y_i)_{i\in I}$, and if we are given, for every $i$, a morphism $f_i:X_i\to Y_i$, then the family of morphisms $f =(f_i)_{i\in I}$ induces, for every $\pi\in\N^{(I)}$, a morphism $S^{\pi}f:S^{\pi}\scr{X}\to S^{\pi}\scr{Y}$. 
\end{remark}
\begin{example}\label{generalex} If $X= R$ and $\pi\neq 0$, then as in Example \ref{firstex}, $S^{\pi}\scr{X}$ is empty except if there exists $i_0\in I$ such that $\pi = (n_i)_{i\in I}$ satisfies $n_i=0$ for all $i\neq i_0$, and $n_{i_0} = 1$, and in this case $S^{\pi}\scr{X} = X_{i_0}$. 
\end{example}

%\bigskip
%There is again an interpretation in terms of zero-cycles. Points of $S^{\pi}\scr{X}$ can be seen as collections of zero-cycles $(D_i)_{i\in I}$ where $D_i$ is an effective zero-cycle supported on $X_i$, of degree $n_i$, and such that the images of the $D_i$ in $X$ have disjoint supports. 

\bigskip

\begin{remark}[Case when $I$ is a semigroup]\label{Isemigroup} Assume for a moment that $I$ is of the form $I_0\setminus \{0\}$ where $I_0$ is a commutative monoid, that is, $I_0$ is endowed with some associative and commutative composition law with zero-element 0. Then there is a well-defined map 
$$\begin{array}{rccc}\lambda:& \N^{(I)}& \to & I_0\\
                          & \pi=(n_i)_{i\geq 1} & \mapsto & \sum_{i\in I}{in_i}
                           \end{array}$$
and we may also define, for any $n\in I_0$, $S^{n}\scr{X}$ to be the disjoint union of all the $S^{\pi} \scr{X}$ for $\pi\in \lambda^{-1}(n)$. 

\end{remark}
\begin{example}\label{example.allxiequal}In the particular case where $I=\N^*$ all the $X_i$ are equal to $X$, by definition, for any integer $n$, $S^n\scr{X}$ is the disjoint union of the locally closed subsets $S^{\pi}X$ described in section \ref{sympower}. In particular, we have the equality of classes $[S^n\scr{X}] = [S^nX]$ in $\svar_{R}$. 
Note that however, the natural scheme structure of the symmetric power $S^nX$ is not the same as the scheme structure of $S^n\scr{X}$. This won't have any importance for us because we will be mainly working in Grothendieck (semi)rings.
\end{example}

\begin{example}\label{generalexn} If $X=R$ then by example \ref{generalex}, $S^{n}\scr{X} = X_{n}$. 
\end{example}

The following definition comes as a natural generalisation of Kapranov's zeta function.
\begin{definition}\label{def.zetafunction} \index{motivic zeta function!of a family of varieties} Let $X$ be variety over $R$, $\scr{X} = (X_i)_{i\in I}$ a family of varieties over $X$. Consider also a family $\mathbf{t} = (t_i)_{i\in I}$ of indeterminates and denote by $\svar_{R}[[\mathbf{t}]]$ the semi-ring of power series in those indeterminates over $\svar_{R}$. The zeta function associated to $\scr{X}$  is the formal power series given by
$$Z_{\scr{X}}(\mathbf{t}) = \sum_{\pi\in\N^{(I)}}[S^{\pi}\scr{X}]\t^{\pi}\ \ \in\svar_{R}[[\mathbf{t}]],$$
where $\t^{\pi} := \prod_{i\geq 1}t_i^{n_i}$. 
\index{ZX@$Z_{\scr{X}}(\t)$}
In particular, if one assumes $I = \N^*$ and specialises the $t_i$ to $t_i = t^i$ for a single variable~$t$, one gets a power series 
 $$Z_{\scr{X}}(t) = \sum_{n\geq 0}[S^{n}\scr{X}]t^{n}\ \ \in\svar_{R}[[t]].$$
 More generally, if we assume $I =  \N^p\backslash\{0\}$ for some integer $p\geq 1$, we get a multi-variate variant of the above zeta-function: 
 $$Z_{\scr{X}}(t_1,\ldots,t_p) = \sum_{\n\in \N^p}[S^{\n}\scr{X}]t_1^{n_1}\ldots t_p^{n_p}\ \ \in\svar_{R}[[t_1,\ldots,t_p]].$$
where for every $\n = (n_1,\ldots,n_p)\in\N^p$, the variety $S^{\n}\scr{X}$ is the disjoint union of the $S^{\pi}\scr{X}$ for all $\pi = (n_i)_{i\in I}\in\N^{(I)}$ such that $\sum_{i\in I}in_i = \n$. 
 
 When we want to specify $R$, we are going to write $Z_{\scr{X}/R}$ instead. 
\end{definition}

\begin{example} Taking  $X_i = X$ for all $i\geq 1$, and using the fact that by example \ref{example.allxiequal} in this case $[S^n\scr{X}] = [S^nX]$, we recover Kapranov's zeta function 
$$Z_X(t) = \sum_{n\geq 0}[S^nX]t^n\ \ \in\svar_{R}[[t]].$$
\end{example}\index{Kapranov's zeta function}
 \subsection{Describing points of symmetric products}\label{symprodpoints}\index{symmetric product!points}
%Assume for simplicity that $R=k$ is a field. We are going to describe the geometric points of $S^{\pi}\scr{X}$. Fix~$\Omega$ an algebraically closed field containing~$k$. An element of $S^{\pi}X(\Omega)$ is of the form
 %\begin{equation}\label{0cycle}D = \sum_{i\geq 1} i (v_{i,1} + \ldots + v_{i,n_{i}})\end{equation}
%where the $v_{i,j}$ are distinct $\Omega$-points of $X$. Let us describe the geometric fibre $(S^{\pi}\scr{X})_D$ of $S^{\pi}\scr{X}$ above this point: for this, let $\Omega'$ be an algebraically closed field containing $\Omega$.  By analogy with the notation in (\ref{0cycle}), an element of $(S^{\pi}\scr{X})_D(\Omega')$  will be written in the form
%$$\sum_{i\geq 1}i(x_{i,1} + \ldots+x_{i,n_{i}})$$
%with $x_{i,j}$ a $\Omega'$-point of $X_{i}$ mapping to $v_{i,j}$ for all $i, j$.
Assume that $I$ is a commutative semigroup.  To describe the points of these symmetric products, it is convenient to use the term ``effective zero-cycle'' rather loosely, so that it applies to any finite formal sum of closed (or Galois orbits of geometric) points with coefficients in some semigroup.

Recall each $X_i$ comes with a morphism $\phi_i:X_i\to X$. Each point of $S^{\pi}\scr{X}$ has an image in $S^{\pi}X$ which, by construction, can be written as an effective zero-cycle on $X$ of the form
   \begin{equation}\label{0cycle}\sum_{i\in I} i(v_{i,1}+\ldots +v_{i,n_i})\in S^{\pi}X ,\end{equation}
   the $v_{i,j}$ being distinct (geometric) points of $X$. Moreover, for every $i\in I$, the degree $i$ part $$i(v_{i,1} +\ldots + v_{i,n_i})$$ comes from $(v_{i,1},\ldots,v_{i,n_i})\in X^{n_i}$, which in turn by definition comes from a point of~$X_i^{n_i}$. 
   
   Thus, by analogy with the notation used in (\ref{0cycle}), we will write elements of $S^{\pi}\scr{X}$ as effective zero-cycles on the disjoint union of (a finite number of) the $X_i$, of the form
   $$D' =\sum_{i\in I}i(x_{i,1} + \ldots + x_{i,n_i}) $$
   such that for all $i\in I$ and for all $j\in\{1,\ldots,n_i\},\ x_{i,j}$ is a geometric point of $X_i$, and
    $$\sum_{i\in I}i(\phi_i(x_{i,1}) + \ldots + \phi_i(x_{i,n_i})) \in S^{\pi}X,$$
    that is, the $\phi_i(x_{i,j})$ are distinct geometric points of $X$. One may also view an element of $S^{\pi}\scr{X}$ simply as a collection of effective zero-cycles $(D_i)_{i\in I}$ where for all $i\in I$, $D_i\in S^{n_i}X_i$, the support of the image of $D_i$ in $X$ is composed of $n_i$ distinct geometric points, and the supports of the images of $D_i$ and $D_j$ for $i\neq j$ are disjoint. 

If $D\in S^{\pi}X(\Omega)$ is a geometric point, for some algebraically closed field $\Omega$, then  for all $i\in I$ and $1\leq j\leq n_i$, $v_{i,j}\in X(\Omega).$ Let $\Omega'\supset \Omega$ be an algebraically closed field. Then the $\Omega'$-points of the fibre of $S^{\pi}X$ above $D$ are exactly those where for all $i,j$, $x_{i,j}\in X_i(\Omega')$.
  
  For $n\in I$, a geometric point of $S^{n}X$ is of the form $D = \sum_{v\in X} n_v v$, where the $n_v$ are non-negative integers, almost all zero and such that $\sum_{v}n_v = n$, and the points $v$ are distinct geometric points of $X$. The zero-cycle $D$ is an element of $S^{\pi}X$ if and only if the partition of the integer $n$ defined by the integers $(n_v)_v$ is exactly $\pi$. A geometric point of $S^{\pi}\scr{X}$ lying above $D$ will be written in the form $\sum_{v\in X} n_v x_v$, where for every $v$,  $x_v$ is a geometric point of $X_{n_v,v}$. The fibre of $S^{\pi}\scr{X}$ above a geometric point $D\in S^{\pi}X$ is
   \begin{equation}\label{formula.fibre}(S^{\pi}\scr{X})_D = \prod_{v\in D}X_{n_v,v}.\end{equation}

\subsection{Symmetric product of a family of symmetric products}\label{symproductquestion}
Assume we are given a family of varieties $\scr{X} = (X_i)_{i\in I}$ over some variety $X$, which itself is defined over $R$, and assume $R$ is itself a variety over some $k$-variety $R'$. For every $\pi\in \N^{(I)}$, this gives rise to a variety $S^{\pi}(\scr{X}/R)$ over $R$. %Put on $\N^{(I)}$ the order induced by the lexicographical order on $\N^{I}$. 
We can now consider the family 
$$S^{\bullet}(\scr{X}/R) = (S^{\pi}(\scr{X}/R))_{\pi\in \N^{(I)}},$$
of varieties over $R$ and, replacing $I$ by $\N^{(I)}$ in the previous paragraph, do the same construction again. We get a family of varieties indexed by the set
$$\N^{(\N^{(I)})} = \{(m_{\pi})_{\pi}\in \N^{\N^{(I)}}, \ \ m_{\pi} = 0\ \text{for all but finitely many $\pi$}\}.$$
For $\varpi\in \N^{(\N^{(I)})}$, we have
$$S^{\varpi}(S^{\bullet}(\scr{X}/R)/R') = \left(\prod_{\pi\in \N^{(I)}}(S^{\pi}(\scr{X}/R))^{m_{\pi}}\ \   _{/R'}\right)_{\!\!\!*,R}/\prod_{\pi\in \N^{(I)}}\Sym_{m_{\pi}},$$
where $/R'$ means the product is over $R'$. 

A natural question arises now: What is the link between this family $$(S^{\varpi}(S^{\bullet}(\scr{X}/R)/R'))_{\varpi\in\N^{\left(\N^{I}\right)}},$$ and the family $(S^{\pi}(\scr{X}/R'))_{\pi\in\N^{(I)}}$ obtained by doing the symmetric product construction for the family $\scr{X}$ but seeing $X$ directly as an $R'$-variety?

\subsection{Main result}

\begin{definition} \label{mu}
 The map $\mu:\N^{(\N^{(I)}\backslash\{0\})}\arr \N^{(I)}$ is defined to be the map that sends an element $(m_{\pi})_{\pi\in \N^{(I)}\backslash\{0\}}$ to 
$$\sum_{\pi\in \N^{(I)}\backslash\{0\}}m_{\pi}\pi\in \N^{(I)}.$$ In terms of the other notation, $\mu$ sends an element
$$[[a_{1,1},\ldots,a_{1,m_1}],\ldots,[a_{r,1},\ldots,a_{r,m_r}]]\in \N^{(\N^{(I)}\backslash\{0\})}$$
to 
$$[a_{1,1},\ldots,a_{1,m_1},a_{2,1},\ldots,a_{r,1},\ldots,a_{r,m_r}]\in \N^{(I)}.$$
\end{definition}

Before stating the main proposition, let us give a motivating example.
\begin{example}\label{ex.112} Let $I = \N^{*}$ and $\pi = [1,1,2]\in\N^{(I)}$, so that
$$\mu^{-1}(\pi) = \{\ [[1],[1],[2]],\ [[1,1],[2]],\ [[1],[1,2]],\ [[1,1,2]]\ \}.$$  We keep the notation from section \ref{symproductquestion}. The points of the variety \begin{equation}\label{sym.112}S^{\pi}(\scr{X}/R') = \left(X_1\times_{R'}X_1\times_{R'}X_2\right)_{*,X}/\Sym_2\times \Sym_1\end{equation}
are zero-cycles of the form $x + y + 2z$, with $x,y,z$ having distinct images in $X$, but all mapping to the same $r\in R'$.  We therefore may classify them depending on the relative positions of their images in $R$, which may be encoded by an element of $\mu^{-1}(\pi)$, by adding square brackets to gather integers corresponding to points having the same image in $R$. There are several cases to consider:
\begin{itemize}\item The points $x,y,z$ all have distinct images in $R$: this may be encoded by $\varpi_1 = [[1],[1],[2]]$.
\item The points $x$ and $y$ have the same image in $R$, but not $z$: this corresponds to $\varpi_2 = [[1,1],[2]]$.
\item The point $z$ has the same image as one of the points $x$ or $y$, but not the other: this is represented by $\varpi_3  = [[1],[1,2]]$.
\item They all have the same image: this gives $\varpi_4 = [[1,1,2]]$. 
\end{itemize} 
Thus, we have a decomposition of $S^{\pi}(\scr{X}/R')$ into four locally closed subsets $S^{\pi}_{\varpi_i}(\scr{X}/R')$, $i=1,2,3,4$ corresponding to these four cases. 
Proposition \ref{iterate} gives a direct way of constructing varieties isomorphic to these locally closed subsets, in the flavour of what has been done in section \ref{sympower}, when we gave direct constructions for the locally closed subsets $S^{\pi}X$ of $S^{n}X$. For $\varpi_2$ for example, we may remark that giving an element of $S^{\pi}_{\varpi_2}(\scr{X}/R')$ is equivalent to giving a zero-cycle in $S^{[1,1]}(\scr{X}/R)$ and a zero-cycle in $S^{[2]}(\scr{X}/R)$, and making sure they have distinct images in $R$ but the same image in $R'$, which results in:

\begin{equation}\label{var.112}\left(S^{[1,1]}(\scr{X}/R)\times_{R'} S^{[2]}(\scr{X}/R)\right)_{*,R} = S^{\varpi_2}(S^{\bullet}(\scr{X}/R)/R')\end{equation}
The right-hand side means that this amounts exactly to applying the symmetric product construction for the element $\varpi_2\in\N^{(N^{(I)}\setminus\{0\})}$, and the family $S^{\bullet}(\scr{X}/R)$ above the $R'$-variety $R$.
\end{example}

\begin{prop}\label{iterate} Let $R'$ be a variety over $k$, $R$ a variety over $R'$,  $X$ a variety over~$R$, and let $\scr{X} = (X_i)_{i\in I}$ be a family of varieties over $X$, indexed by a set $I$. Then for every $\pi\in\N^{(I)}$ and for every $\varpi\in\mu^{-1}(\pi)$, there is a piecewise isomorphism of the variety $S^{\varpi}(S^{\bullet}(\scr{X}/R)/R')$ onto a locally closed subset $S^{\pi}_{\varpi}(\scr{X}/R')$ of $S^{\pi}(\scr{X}/R')$, so that moreover $S^{\pi}(\scr{X}/R')$ is equal to the disjoint union of the sets $S^{\pi}_{\varpi}(\scr{X}/R')$. In particular, we have the equality
$$\sum_{\varpi\in \mu^{-1}(\pi)}[S^{\varpi}(S^{\bullet}(\scr{X}/R)/R')] = [S^{\pi}(\scr{X}/R')]$$
in $\svar_{R'}$. 
\end{prop}

\subsection{Proof of proposition \ref{iterate}}

To prove  proposition \ref{iterate}, by a spreading-out argument, we may assume $R'=k$ is a field. We write $J = \N^{(I)}\backslash\{0\}$, and for any $j\in J$, $\pi_j = (n_i^j)_{i\in I}\in \N^{(I)}$, so that every element $\varpi$ of $\N^{(J)}$ may be given as a family of multiplicities $(m_j)_{j\in J}$.

Put $\pi = (n_i)_{i\in I}$, and fix $\varpi\in\mu^{-1}(\pi)$. The condition that $\pi = \mu(\varpi)$ is equivalent to
$$n_i = \sum_{j\in J}m_j n_i^{j}$$
for all $i\in I$. Our proof decomposes in several steps. Both $S^{\varpi}(S^{\bullet}(\scr{X}/R))$ and $S^{\pi}(\scr{X})$ are constructed as some quotient of some product of the $X_i$. We are going to refrain from taking quotients first, and construct an immersion from the product giving the former to the product giving the latter.

\paragraph{The immersion before quotients}
For any two varieties $V$ and $W$ over $R$, there is an immersion $V \times_R W\hookrightarrow V\times W$. Thus, for every $j\in J$ there is an immersion 
$$\prod_{i\in I}X_i^{n_i^j}\ _{/R}\hookrightarrow \prod_{i\in I}X_i^{n_i^j}.$$
%Note that the variety on the right-hand side is equal to $$\prod_{i\in I}X_i^{\sum_{j\in J}n_i^jm_j} = \prod_{i\in I}X_i^{n_i}.$$
Restricting to the complement of the diagonal in $\prod_{i\in I}X^{n_i^{j}}$, we get an immersion
$$\left(\prod_{i\in I}X_i^{n_i^j}\ _{/R}\right)_{*,X}\hookrightarrow \left(\prod_{i\in I}X_i^{n_i^j}\right)_{\!\!\!*,X}.$$
Taking the product over all $j\in J$ of the $m_j$-th powers of those varieties gives:

$$\prod_{j\in J}\left(\left(\prod_{i\in I}X_i^{n_i^j}\ _{/R}\right)_{\!\!\!*,X}\right)^{m_j}
\hookrightarrow \prod_{j\in J}\left(\left(\prod_{i\in I}X_i^{n_i^j}\right)_{\!\!\!*,X}\right)^{m_j}.$$
For any two varieties $V$ and $W$ over $X$, we have the commutative diagram 
$$\xymatrix{(V\times W)_{*,R} \ar@{^{(}->}[r] \ar[d]& (V\times W)_{*,X} \ar@{^{(}->}[r] \ar[d]& V\times W \ar[d]\\
(X\times X)_{*,R}\ar@{^{(}->}[r]\ar[d] & (X\times X)_{*,X}  \ar@{^{(}->}[r]& X\times X\ar[d] \\
(R\times R)_{*,R} \ar@{^{(}->}[rr]& & R\times R}$$
where the horizontal arrows are all open immersions: indeed, recall that we assumed~$R$ and~$X$ to be quasi-projective over $k$, and therefore separated, so that complements of diagonals are open (and therefore so are their inverse images by the structural morphisms). In particular, we have an open immersion $(V\times W)_{*,R}\to (V\times W)_{*,X}.$ Thus, we can restrict to the complement of the diagonal of $\prod_{j\in J}R^{m_j}$ on the left, and to the complement of the diagonal of $\prod_{j\in J}\left(\prod_{i\in I}X^{n_i^j}\right)^{m_j}$ on the right, to get
\begin{equation}\label{VtoW1}\left(\prod_{j\in J}\left(\left(\prod_{i\in I}X_i^{n_i^j}\ _{/R}\right)_{\!\!\!*,X}\right)^{m_j}\right)_{\!\!\!*,R}
\hookrightarrow \left(\prod_{j\in J}\left(\prod_{i\in I}X_i^{n_i^j}\right)^{m_j}\right)_{\!\!\!*,X}.\end{equation}
Note that using the assumption $\mu(\varpi) = \pi$, we may write
\begin{equation}\label{rewrite}\left(\prod_{j\in J}\left(\prod_{i\in I} X_{i}^{n_i^j}\right)^{m_j}\right)_{\!\!\!*,X} = \left(\prod_{i\in I} X_i^{\sum_{j\in J}n_i^jm_j}\right)_{\!\!\!*,X} =  \left(\prod_{i\in I}X_i^{n_i}\right)_{\!\!\!*,X}. \end{equation}
Composing (\ref{VtoW1}) with this identification, we get an immersion
 \begin{equation}\label{VtoW}\left(\prod_{j\in J}\left(\left(\prod_{i\in I}X_i^{n_i^j}\ _{/R}\right)_{\!\!\!*,X}\right)^{m_j}\right)_{\!\!\!*,R}
\hookrightarrow \left(\prod_{i\in I}X_i^{n_i}\right)_{\!\!\!*,X}.\end{equation}

\begin{example} In example \ref{ex.112}, the immersion $(\ref{VtoW})$ corresponding to $\varpi_2$ is written in the form
$$\left((X_1\times_R X_1)_{*,X}\times X_2 \right)_{*,R} \hookrightarrow  (X_1\times X_1\times X_2)_{*,X},$$
where the variety on the left-hand side (resp. right-hand side) is exactly the one in $(\ref{var.112})$ (resp. $(\ref{sym.112})$), just without the permutation action quotients. (Recall we took $R' = k$.)%  In the same manner, the immersion corresponding to $\varpi_{3}$ will be
\end{example}
\paragraph{Description of the permutation actions}

 Let $V$ be the variety on the left-hand side, as well as its image through this morphism, and $W$ the variety on the right-hand side. 
 There is a natural action of $G = \prod_{i\in I}\Sym_{n_i}$ on $W$. As for $V$, we can distinguish two groups acting on it. The first one is  $$\prod_{j\in J}\left(\prod_{i\in I}\Sym_{n_{i}^j}\right)^{m_j}$$
 which comes from the natural permutation action of $\Sym_{n_i^j}$ on each product $\prod_{i\in I}X_i^{n_i^j}\ _{/R}$ for all $i\in I$ and all $j\in J$. Composing the morphisms $X_i\to X$ with the morphism $X\to R$, we get a map
 \begin{equation}\label{maptoR}\phi: \left(\prod_{j\in J}\left(\left(\prod_{i\in I}X_i^{n_i^j}\ _{/R}\right)_{\!\!\!*,X}\right)^{m_j}\right)_{\!\!\!*,R}\arr\left( \prod_{j\in J}R^{m_j}\right)_{*}\end{equation}
 the fibres of which are stable with respect to that action. On the other hand, there is also a permutation action of $\prod_{j\in J}\Sym_{m_j}$ on $\left( \prod_{j\in J}R^{m_j}\right)_{*}$, which pulls back to an action on the variety on the left-hand side in the following manner: 
for $x\in V$, denoting for every $j\in J$ and every $\ell\in\{1,\ldots,m_j\}$ by $x_{j,\ell}$ the projection of $x$ on the $\ell$-th copy of $\left(\prod_{i\in I}X_i^{n_i^j}\ _{/R}\right)_{*,X}$ occurring in $V$, the element
 $\sigma = (\sigma_j)_{j\in J}\in \prod_{j\in J}\Sym_{m_j}$ acts on $x = (x_{j,1},\ldots,x_{j,m_j})_{j\in I}$ via
 $$\sigma\cdot \left((x_{j,1},\ldots,x_{j,m_j})_{j\in I}\right) = \left(\left(x_{j,\sigma_j^{-1}(1)},\ldots,x_{j,\sigma_j^{-1}(m_j)}\right)_{j\in J}\right).$$

 Through immersion (\ref{VtoW}), these two actions give us two subgroups~$H_1$ and $H_2$ of $G = \prod_{i\in I}\Sym_{n_i}$.
 
 \begin{example}\label{example.subgroup}\begin{enumerate}\item In example \ref{ex.112}, we have $W = (X_1^2 \times X_2)_{*,X}$. Let us examine the subgroups of $G := \Sym_2\times \Sym_1$ corresponding to the different $\varpi_i$ occurring in that example. 
 \begin{itemize}\item For $\varpi_1 = [[1],[1],[2]]$, we have $V =  (X_1^2\times X_2)_{*,R}$, $H_1 = \{1\}$ and $H_2 = G$. 
 \item For $\varpi_{2} = [[1,1],[2]]$, we have $V = ((X_1\times_R X_1)_{*,X}\times X_2)_{*,R}$, $H_1 = G$ and $H_2 = \{1\}$. 
 \item For $\varpi_{3} = [[1,2],[1]]$, we have $V = ((X_1\times_R X_2)_{*,X}\times X_1)_{*,R}$, and $H_1 = H_2 = \{1\}$. 
 \item For $\varpi_{4} = [[1,1,2]]$, we have $V = ((X_1\times_RX_1\times_R X_2)_{*,X}$, $H_1 = G$ and $H_2 = \{1\}$. 
 \end{itemize}
 \item Let us examine another example: $\pi = [1,1,1,1,1,1]$, $\varpi = [[1,1],[1,1],[1],[1]]$. Then $W = (X_1^6)_{*,X}$, $G = \Sym_6$, and $$V = \left((X_1\times_R X_1)_{*,X}\times (X_1\times_R X_1)_{*,X} \times X_1\times X_1\right)_{*,R},$$
 so that $H_1$ is the subgroup of $G$ generated by the permutations $(12)$ and $(34)$, whereas~$H_2$ is generated by the permutations $(13)(24)$ and $(56)$. 
 \end{enumerate}
 \end{example}
 \begin{lemma} The subgroup $H_1$ is normalised by $H_2$. The subgroup $H:=H_1H_2$ they generate is the largest subgroup of $\prod_{i\in I}\Sym_{n_i}$ under the action of which $V$ is invariant inside~$W$. 
 \end{lemma}
\begin{proof} For all $\sigma\in H_1$ and $\tau\in H_2$, the element $\tau\sigma\tau^{-1}\in G$ stabilises the fibres of the morphism~$\phi$ in (\ref{maptoR}), which means that it stabilises each factor $\prod_{i}X_i^{n_i^{j}}\ _{/R}$ of $V$. Thus, $\tau\sigma\tau^{-1}$ is an element of~$H_1$.

Now, let $\sigma \in\prod_{i\in I}\Sym_{n_i}$ be such that for all $x\in V$, $\sigma x\in V$. Let $x\in V$, and let $\tau\in H_2$ be the element such that $\tau(\phi(x)) = \phi(\sigma(x))$. Then $\sigma\tau^{-1}$ stabilises the fibres of the map~$\phi$ in (\ref{maptoR}). This means that its action stabilises each factor $\prod_{i}X_i^{n_i^{j}}\ _{/R}$ of $V$, so $\sigma\tau^{-1}$ is an element of~$H_1$. 
\end{proof}

Our aim now is to describe a locally closed subset $W(\varpi)$ of $W$ containing $V$ and stable under the natural $\prod_{i\in I}\Sym_{n_i}$-action on $W$, and show that we can apply Proposition \ref{subquotient} to $V$ and $W(\varpi)$ to get an isomorphism
$$V/H \simeq W(\varpi)/\prod_{i\in I}\Sym_{n_i}$$
where the variety on the right-hand side will be called $S^{\pi}_{\varpi}(\scr{X})$.

 \paragraph{Equivalence relations on coordinates of points of $W$}
 Recall that $W$ is the variety $$\left(\prod_{i\in I} X_i^{n_i}\right)_{*,X}.$$ A point of this variety is of the form
$x = (x_{i,p})_{\substack{i\in I\\ 1\leq p\leq n_i}}$ where for all $i\in I$ and for all $1\leq p\leq n_i$, $x_{i,p}\in X_i$ and all coordinates $x_{i,p}$ have distinct images in $X$. Consider an equivalence relation $\rho$ on the set of indices $\{(i,p)\}_{\substack{i\in I\\ 1\leq p\leq n_i}}$.  Each equivalence class~$E$ is a subset of the latter, which we write in the form $E = \bigcup_{i\in I}E_i$ (disjoint union) where 
$$E_i  = E\cap X_i= \{(i,\alpha_{i,1}),\ldots,(i,\alpha_{i,\ell_i})\}$$ for some integers $(\ell_i)_{i\in I}$ (with $\ell_i\leq n_i$ for all $i$), and
$$1\leq \alpha_{i,1}<\ldots<\alpha_{i,\ell_i}\leq n_i\ \ \text{for all}\ \ i\in I.$$ Note that the equivalence classes $E$ of $\rho$ form a partition of the set of indices of the coordinates of the point $x$, and that therefore for every $i\in I$, the sets $E_i = E\cap X_i$ form a partition of the set $\{(i,1),\ldots,(i,n_i)\}$. Thus, the sum of the $\ell_i$ over all equivalence classes~$E$ is equal to~$n_i$. 

To each such non-empty $E$ we can associate the non-zero element $\pi(E) = (\ell_i)_{i\in I}\in\N^{(I)}$. The collection of all $\pi(E)$ for all equivalence classes $E$ of $\rho$, counted with multiplicities, then gives an element $\varpi(\rho)\in \N^{(\N^{(I)}\backslash\{0\})}$ such that $\mu(\varpi(\rho)) = \pi$, since the sum of the $\ell_i$ over all equivalence classes $E$ is $n_i$.

 \paragraph{Definition of $W(\varpi)$}
 
 For every $x\in W$, we define an equivalence relation $\rho_x$ on the set $$\{(i,p)\}_{\substack{i\in I\\ 1\leq p\leq n_i}}$$ by: $(i,p)\sim (i',p')$ if and only if the coordinates $x_{i,p}$ and $x_{i',p'}$ have the same image in $R$.

\begin{definition}
For every equivalence relation $\rho$ on $\{(i,p)\}_{\substack{i\in I\\ 1\leq p\leq n_i}}$ occurring in this way, define the locally closed subsets
$$W_{\rho} = \{x\in W,\ \rho_{x} = \rho\}\subset W$$
and $$W(\varpi) = \bigcup_{\varpi(\rho) = \varpi}W_{\rho}\subset W,$$
(this is a finite and disjoint union).
\end{definition}
\begin{example} Let $I = \N^{*}$ and $\pi = [1,1,2]$, so that $W = (X_1^2\times X_2)_{*,X}$. An element of~$W$ will be written $(x_{1,1},x_{1,2},x_{2,1}).$ Let $\rho$ be the equivalence relation on the set $\{(1,1), (1,2),(2,1)\}$ with equivalence classes $\{(1,1),(2,1)\}$ and $\{(1,2)\}$, giving rise respectively to the partitions $[1,2]$ and $[1]$, so that $\varpi:=\varpi(\rho) = [[1,2],[1]]$. The only other equivalence relation giving this element of $\mu^{-1}(\pi)$ is the one with equivalence classes $\{(1,2),(2,1)\}$ and $\{(1,1)\}$, denoted by $\rho'$.
Thus $W_{\rho}$ is the locally closed subset of triples $(x_{1,1},x_{1,2},x_{2,1})$ such that in~$R$, $x_{2,1}$ becomes equal to $x_{1,1}$ but not to $x_{1,2}$. In the same way, $W_{\rho'}$ corresponds to triples such that $x_{2,1}$ becomes equal to $x_{1,2}$ but not to $x_{1,1}$. Finally, $W(\varpi)$ is the union of these two sets, namely the set of triples such that in $R$, $x_{2,1}$ becomes equal
 either to $x_{1,1}$ or to~$x_{1,2}$, where the ``or'' is exclusive.

\end{example}

\paragraph{Taking quotients}

\begin{lemma} \begin{enumerate}[(a)]\item $W(\varpi)$ is stable under the action of $\prod_{i\in I}\Sym_{n_i}$ on $W$.
\item The group $\prod_{i\in I}\Sym_{n_i}$ acts transitively on the set of the $W_{\rho}$ with $\varpi(\rho) = \varpi$.
\end{enumerate}
\end{lemma}

\begin{proof}\begin{enumerate}[a)]\item 
Let $\sigma = (\sigma_{i})_{i\in I}\in \prod_{i\in I}\Sym_{n_i}$, $\rho$ some equivalence relation and $x\in W_{\rho}$. For every equivalence class $E$, the equivalence class $\sigma E = \bigcup_{i\in I}\sigma_i(E_i)$ gives the same numbers $\ell_i$: therefore $\varpi(\sigma\rho) = \varpi$.

\item  Let $\rho$ and $\rho'$ be two different equivalence relations such that $\varpi(\rho) = \varpi(\rho')$. Since they give rise to the same $\varpi$, they have the same number of equivalence classes, and moreover, to each equivalence class $E$ of $\rho$ we may associate an equivalence class~$E'$ of~$\rho'$ such that for every $i\in I$, we have
$$\# E_i = \# E'_i.$$
Denoting by $\ell_i$ this common value, write 
$$E_i = \{(i,\alpha_{i,1}),\ldots,(i,\alpha_{i,\ell_i})\}\ \ \ \text{and}\ \ \ E_i' = \{(i,\be_{i,1}),\ldots,(i,\be_{i,\ell_i})\}$$
for all $i\in I$. Define the restriction of the element $\sigma = \prod_{i\in I}\sigma_i\in \prod_{i\in I}\Sym_{n_i}$ to $$\prod_{i\in I}\{\al_{i,1},\ldots,\al_{i,\ell_i}\}$$ by
$\sigma_i(\al_{i,p}) = \be_{i,p}.$
Since the equivalence classes of $\rho$ form a partition of $\{(i,p)\}_{\substack{i\in I\\ 1\leq p\leq n_i}}$, doing this for all equivalence classes completely defines an element $\sigma\in \prod_{i\in I}\Sym_{n_i}$ such that for every $x\in W_{\rho}$, $\sigma x\in W_{\rho'}$. 
\end{enumerate}
\end{proof}

\begin{definition} For every $\varpi\in\mu^{-1}(\pi)$ we define $S^{\pi}_{\varpi}(\scr{X})$ to be the locally closed subset of $S^{\pi}\scr{X}$ given by taking the quotient of $W(\varpi)\subset W$ by $\prod_{i\in I}\Sym_{n_i}$.
\end{definition}

\begin{lemma} There is an equivalence relation $\rho$ such that the image $V$ of the immersion in (\ref{VtoW}) is equal to $W_{\rho}$. 
\end{lemma}

\begin{proof} Fix $x=(x_{i,p})_{\substack{i\in I\\ 1\leq p\leq n_i}}\in W$ in the image of the morphism in (\ref{VtoW}), and for all $i\in I,\ j\in J$ and $1\leq q \leq m_j$, denote by $E_{i,j,q}$ the set of indices of the coordinates of the projection of $x$ to the $q$-th factor~$X_{i}^{n_i^j}$. By definition, all coordinates $x_{i,p}$ of $x$ with index   $(i,p)\in E_{j,q}:= \bigcup_{i\in I} E_{i,j,q}$ have the same image $r_{j,q}(x)\in R$, and the elements~$r_{j,q}(x)$ for all $j\in J$ and $1\leq q \leq m_j$ are distinct. Therefore, the sets $E_{j,q}$ in fact don't depend on $x$, so that the elements of $V$ are exactly those subject to the equivalence relation $\rho$ with classes $(E_{j,q})_{\substack{j\in J\\ 1\leq q\leq m_j}}.$
\end{proof}

Putting everything together and applying Proposition \ref{subquotient} to $V\subset W$ with actions of the groups $H\subset \prod_{i\in I}\Sym_{n_i}$, we get the result, since $S^{\varpi}(S^{\bullet}(\scr{X}/R))$ is by definition equal to $V/H$. 

\section{Cutting into pieces}\label{cuttingintopieces}
\subsection{Introduction}
The aim of this section is to state and prove an analogue in our setting of the following classical result about symmetric powers (see for example \cite{CNS}, chapter 6, proposition 1.1.7):

\begin{prop}\label{cut} Let $X$ be a quasi-projective variety over a field $k$, and $Y$ a closed subvariety of $X$ with open complement $U$. For any integers $n\geq 1$ and $r\in\{0,\ldots,n\}$, the variety $$S^rU\times S^{n-r} Y$$ can be identified with the locally closed subset of $S^nX$ corresponding to effective zero-cycles of degree~$n$ the restriction of which to $U$ has degree $r$. Moreover, $S^nX$ is the disjoint union of these locally closed subsets. In particular, in terms of classes in $\svar_k$, we have
$$[S^nX] = \sum_{r=0}^n[S^rU][S^{n-r}Y]\in\svar_k.$$\end{prop}

Thus, starting with a zero-cycle $D$ of degree $n$ in $X$, we get through restriction a pair $(D_U,D_Y)$ of zero-cycles in $S^rU\times S^{n-r}Y$ for some $r$. Conversely, starting with a point of the latter variety, the sum of the two components yields an element in $S^nX$. More precisely, by the same argument we even have the following:

\begin{prop}\label{cut2} (Refinement of Proposition \ref{cut}) Let $k, X, U, Y$ be as in Proposition \ref{cut}, and let $\pi\in\N^{(\N^*)}$ be a partition. Then $S^{\pi}X$ is the disjoint union of locally closed subsets isomorphic to $S^{\pi'}U\times S^{\pi-\pi'} Y$ where $\pi'$ runs through all partitions such that $\pi'\leq \pi$. In particular, in terms of classes in $\svar_k$, we have
$$[S^{\pi}X] = \sum_{\pi'\leq \pi}[S^{\pi'}U][S^{\pi-\pi'} Y]\in\svar_k.$$\end{prop}

  To motivate the construction in the following paragraph, let us examine what exactly we need to get a result of this flavour for symmetric products. Fix a set $I$, and let $X$ be a variety over~$k$, and $\scr{X} = (X_i)_{i\in I}$ a family of varieties over~$X$. For all~$i\in I$, let~$Y_i$ be a closed subvariety of $X_i$, and~$U_i$ its complement, so that we get families of~$X$-varieties $\scr{Y}  = (Y_i)_{i\in I}$ and $\scr{U} = (U_i)_{i\in I}$. 

Let $\pi = (n_i)_{i\in I}\in\N^{(I)}$. Any point $ D\in S^{\pi}\scr{X}$ is a zero-cycle contained in the disjoint union of (a finite number of) the $X_i$, and we can consider its restriction $D_{\scr{U}}$ to $\cup_{i\in I}U_i$. As in the discussion above, this clearly gives us a point in $S^{\pi'}\scr{U}$ for some $\pi'\leq \pi$, and the restriction $D_{\scr{Y}}$ to the elements of the family $\scr{Y}$ will then be a point in $S^{\pi-\pi'}\scr{Y}$. In other words, there is a well-defined immersion
$$\alpha: S^{\pi}\scr{X}\arr \bigcup_{\pi'\leq \pi}S^{\pi'}\scr{U}\times S^{\pi-\pi'}\scr{Y}.$$
On the other hand, this morphism $\alpha$ will in general not be an isomorphism. Indeed, the inverse mapping $(D_1,D_2)\mapsto D_1 + D_2$ that worked in the above cases is well-defined only when $D_1$ and $D_2$ have disjoint supports, since, by definition, points of $S^{\pi}\scr{X}$ are zero-cycles with supports mapping injectively to $S^{\pi}X$. Thus the image of $\alpha$ is the subset of $\bigcup_{\pi'\leq \pi}S^{\pi'}\scr{U}\times S^{\pi-\pi'}\scr{Y}$ mapping to pairs in $\bigcup_{\pi'\leq \pi}S^{\pi'}X\times S^{\pi-\pi'}X$ with disjoint supports. Thus, to generalise proposition \ref{cut2}, we are going to define more general symmetric products which are \textit{mixed}, in the sense that we will combine different families of varieties before restricting to the complement of the diagonal in the base variety. In the case studied above, this construction will give us varieties $S^{\pi',\pi-\pi'}(\scr{U},\scr{Y})$, the union of which for all $\pi'\leq \pi$ corresponds exactly to the image of $\alpha$, so that the analogue of proposition \ref{cut2} for symmetric products will take the form
$$[S^{\pi}\scr{X}] = \sum_{\pi'\leq \pi}[S^{\pi',\pi-\pi'}(\scr{U},\scr{Y})]$$
in $\svar_k$.  %In short, we would like to be able to consider objects that combine different families of varieties before restricting to the complement of the diagonal in the base variety. This is what is done in the next paragraph.
 
\subsection{Mixed symmetric products}\label{mixedsymproducts}

 Let $X$ be a variety over $R$, $p\geq 1$ an integer, and $\scr{X}_{1},\ldots,\scr{X}_{p}$ families of varieties over~$X$, all indexed by the same set $I$. For all $j\in\{1,\ldots,p\}$ we write $\scr{X}_j = (X_{i,j})_{i\in I}$. We also fix for every $j\in\{1,\ldots,p\}$ an almost zero family of non-negative integers $\pi_j = (r_{i,j})_{i\in I}$ (that is, all $r_{i,j}$ but a finite number are zero). 
The product (over $R$)
$$\prod_{i\in I}X_{i,1}^{r_{i,1}}\times\ldots \times  X_{i,p}^{r_{i,p}}$$
is a variety over
$$\prod_{i\in I} X^{r_{i,1}}\times \ldots \times X^{r_{i,p}}.$$
As before, we restrict it to the open subset $\left(\prod_{i\in I}X_{i,1}^{r_{i,1}}\times\ldots \times  X_{i,p}^{r_{i,p}}\right)_*$ lying above the complement of the diagonal of $\prod_{i\in I} X^{r_{i,1}}\times \ldots \times X^{r_{i,p}} = X^{\sum_{i,j}r_{i,j}}$, that is, we remove points mapping to points having at least two equal coordinates. Then we take the quotient by the natural permutation action of $\prod_{i\in I}\Sym_{r_{i,1}}\times \ldots \times \Sym_{r_{i,j}}$. The resulting variety will be denoted by $$S^{\pi_1,\ldots,\pi_p}(\scr{X}_1,\ldots,\scr{X}_p).$$ 
\index{symmetric product!mixed}
If for every $j\in\{1,\ldots,p\}$ all varieties $X_{i,j}$ are equal to some $X_j$, we may write this simply $S^{\pi_1,\ldots,\pi_p}(X_1,\ldots,X_p)$.%Its points can be seen as collections of zero-cycles $(D_i)_{i\in I}$ such that 
%\begin{itemize} \item for every $i$, $D_i$ is contained in the disjoint union $\cup_{j=1}^pX_{i,j}$
%\item the restriction of $D_i$ to $X_{i,j}$ has degree $r_{i,j}$;
%\item the disjoint union of the supports of the $D_i$ map injectively into $X$.
%\end{itemize}  
%In particular, such zero-cycles map to zero-cycles inducing partition $\pi_1 + \ldots + \pi_p$ in $X$. 

As it was the case in the construction of simple symmetric products, we could also first take the quotient, and then remove the closed set lying above the image of the diagonal by the quotient map.

The following properties are obvious consequences of the definition:

\begin{fact}\begin{enumerate}\item In the case $p=1$, we recover exactly the symmetric product $S^{\pi}\scr{X}$ from section \ref{anyset}. 
\item For all $\sigma\in \Sym_p$, there is an isomorphism
$$S^{\pi_1,\ldots,\pi_p}(\scr{X}_1,\ldots,\scr{X}_p) \simeq S^{\pi_{\sigma(1)},\ldots,\pi_{\sigma(p)}}(\scr{X}_{\sigma(1)},\ldots,\scr{X}_{\sigma(p)}).$$
\item If $\pi_p= 0$, then 
$$S^{\pi_1,\ldots,\pi_p}(\scr{X}_1,\ldots,\scr{X}_p) = S^{\pi_1,\ldots,\pi_{p-1}}(\scr{X}_1,\ldots,\scr{X}_{p-1}).$$
\item If there exists  $i\in I$ such that $r_{i,p}>0$ and $X_{i,p} = \emptyset$, then 
$$S^{\pi_1,\ldots,\pi_p}(\scr{X}_1,\ldots,\scr{X}_p) = \emptyset.$$
\end{enumerate}
\end{fact}

\begin{remark}\label{separate} The open set $\left(\prod_{j=1}^pX^{\sum_{i\in I}r_{i,j}}\right)_*$  is a subset of
$\prod_{j=1}^p \left(X^{\sum_{i\in I}r_{i,j}}\right)_*$
consisting of points with projection to $X^{\sum_{i\in I}r_{i,j}}$ having distinct coordinates for all $j$. Thus, there is an open immersion
\begin{equation}\label{prodimmersion}S^{\pi_1,\ldots,\pi_p}(\scr{X}_1,\ldots,\scr{X}_p)\arr S^{\pi_1}\scr{X}_1\times \ldots \times S^{\pi_p}\scr{X}_p.\end{equation}
It is an isomorphism if there is a partition of $X$ into locally closed subsets $V_1,\ldots,V_p$ such that for all $i\in I$ and all $j\in\{1,\ldots,p\}$, the image of $X_{i,j} $ in $X$ is contained in $V_i$.  Indeed, then all elements in any $p$-tuple $(D_1,\ldots,D_p)\in S^{\pi_1}\scr{X}_1\times \ldots \times S^{\pi_p}\scr{X}_p$ have disjoint supports, and the inverse mapping is given by $(D_1,\ldots,D_p)\mapsto D_1 + \ldots + D_p$. 
\end{remark}

We can now state our generalisation of Propositions \ref{cut} and $\ref{cut2}$. 
\begin{prop}\label{pieces} Let $X, \scr{X}_1,\ldots,\scr{X}_p$ be as above, and consider moreover another family $\scr{X} = (X_i)_{i\in I}$ of quasi-projective varieties over $X$, together with families $\scr{Y} = (Y_i)_{i\in I}$ and $\scr{U} = (U_i)_{i\in I}$ such that for every $i\in I$, $Y_i$ is a closed subvariety of $X_i$ and $U_i$ its complement. Let $\pi_1,\ldots,\pi_p,\pi$ be elements of $\N^{(I)}$. Then for every $\pi'\leq \pi$ the variety
$$S^{\pi_1,\ldots,\pi_{p},\pi',\pi -\pi'}(\scr{X}_1,\ldots,\scr{X}_{p},\scr{U},\scr{Y})$$ is isomorphic to the locally closed subset of $S^{\pi_1,\ldots,\pi_p,\pi}(\scr{X}_1,\ldots,\scr{X}_p,\scr{X})$ corresponding to zero-cycles with $\scr{X}$-component inducing partition $\pi'$ on $\scr{U}$. Moreover, the variety $$S^{\pi_1,\ldots,\pi_p,\pi}(\scr{X}_1,\ldots,\scr{X}_p,\scr{X})$$ is the disjoint union of these locally closed subsets, so that in terms of classes in $\svar_{R}$, we have
$$\left[S^{\pi_1,\ldots,\pi_p,\pi}(\scr{X}_1,\ldots,\scr{X}_p,\scr{X})\right] = \sum_{\pi'\leq \pi}\left[S^{\pi_1,\ldots,\pi_{p},\pi',\pi-\pi'}(\scr{X}_1,\ldots,\scr{X}_{p},\scr{U},\scr{Y})\right].$$

\end{prop}

\begin{proof} We write $\pi = (n_i)_{i\in I}$, $\pi' = (m_i)_{i\in I}$ and for every $j\in\{1,\ldots,p\}$, $\pi_j = (r_{i,j})_{i\in I}$ According to proposition \ref{cut},
the variety $\prod_{i\in I}S^{r_{i,1}}X_{i,1}\times\ldots\times S^{r_{i,p}}X_{i,p}\times S^{n_i}X_i$ is the union of locally closed subsets isomorphic to
$$\prod_{i\in I}S^{r_{i,1}}X_{i,1}\times\ldots\times S^{r_{i,p}}X_{i,p}\times S^{m_{i}}U_i\times S^{n_i-m_i}Y_i $$
 for all $\pi\leq \pi_p$. Restricting to the images of points lying above the complements of the diagonal, we get the result.
\end{proof}

\subsection{Applications}\label{sect.cutapplications}
As an immediate consequence of Proposition \ref{pieces} (for $p=0$) and Remark \ref{separate}, we get
 \begin{cor}\label{multzeta}  Let $X$ be a variety over $R$ and $\scr{X} = (X_i)_{i\in I}$ a family of varieties over~$X$. Let $Y$ be a closed subvariety of $X$, $U$ its open complement, and for all $i\in I$, define varieties  $Y_i = X_i\times_X Y$, $U_i = X_i\times_X U$, and families $\scr{Y} = (Y_i)_{i\in I}$ and $\scr{U} = (U_i)_{i\in I}$  of varieties over $Y$ and~$U$, respectively. Then for all $\pi\in\N^{(I)}$ and for all $\pi'\leq \pi$, the variety $S^{\pi'}\scr{U}\times S^{\pi-\pi'}\scr{Y}$ is isomorphic to the locally closed subset of $S^{\pi}\scr{X}$ corresponding to families of zero-cycles inducing partition $\pi'$ on $\scr{U}$. Moreover, $S^{\pi}\scr{X}$ is the disjoint union of these locally closed subsets. In particular, in terms of classes in $\svar_R$, we have
 $$[S^{\pi}\scr{X}] = \sum_{\pi'\leq \pi}[S^{\pi'}\scr{U}][S^{\pi-\pi'}\scr{Y}].$$
 In terms of zeta-functions, this means that 
$$Z_{\scr{X}}(\mathbf{t}) = Z_{\scr{U}}(\mathbf{t})Z_{\scr{Y}}(\mathbf{t})$$
in $\svar_R[[\t]].$
\end{cor}

\begin{prop}\label{lociso} Let $X$ be a variety over $R$ and $\scr{X} = (X_i)_{i\in I}$ and $\scr{Y} = (Y_i)_{i\in I}$ families of varieties over $X$. Fix an element $\pi = (n_i)_{i\in I}\in\N^{(I)}$, and assume that for all $i\in I$ such that $n_i>0$, $X_i$ and  $Y_i$ are piecewise isomorphic. Then $S^{\pi}\scr{X}$ and $S^{\pi}\scr{Y}$ are piecewise isomorphic: in other words, we have the equality
$$Z_{\scr{X}}(\t) = Z_{\scr{Y}}(\t)$$
in $\svar_R[[\t]].$  

\end{prop}

\begin{proof} By assumption, there is an integer $p$ with the property that, for all $i\in I$ satisfying $n_i>0$, we can partition $X_i$ (resp. $Y_i$) into locally closed sets $X_{i,1},\ldots,X_{i,p}$ (resp. $Y_{i,1},\ldots,Y_{i,p}$) such that for all $j\in\{1,\ldots,p\}$,  $X_{i,j}$ is isomorphic to $Y_{i,j}$. For every $j\in\{1,\ldots,p\}$, we denote by $\scr{X}_j$ (resp. $\scr{Y}_j$) the family $(X_{i,j})_{i\in I}$ (resp. $(Y_{i,j})_{i\in I}$). Proposition \ref{pieces} together with an induction shows that $S^{\pi}\scr{X}$ is the disjoint union of locally closed subsets isomorphic to $S^{\pi_1,\ldots,\pi_p}(\scr{X}_1,\ldots,\scr{X}_p)$ for partitions $\pi_1,\ldots,\pi_p$ such that $\pi_1 + \ldots + \pi_p = \pi$. Since the same is true for $\scr{Y},\scr{Y}_1,\ldots,\scr{Y}_p$, and since the isomorphisms between the $X_{i,j}$ and $Y_{i,j}$ give isomorphisms 
$$S^{\pi_1,\ldots,\pi_p}(\scr{X}_1,\ldots,\scr{X}_p)\simeq S^{\pi_1,\ldots,\pi_p}(\scr{Y}_1,\ldots,\scr{Y}_p) $$
for all $\pi_1,\ldots,\pi_p$ such that $\pi_1+\ldots + \pi_p  = \pi$, the result follows.
\end{proof}

\section{Symmetric products and affine spaces}\label{affinespaces}
The following proposition is a direct generalisation of a result by Totaro (see \cite{CNS}, Proposition 5.1.5). In fact, it may be obtained directly by applying Totaro's statement to each factor in $\prod_{i\in I}S^{n_i}X_i$ and then restricting to the open subset $S^{\pi}X\subset \prod_{i\in I}S^nX$, but we prefer to give a direct proof in the flavour of Totaro's argument: it is more straightforward here since one does not need to go through the step of cutting up the symmetric power with respect to different partitions. 
\begin{prop}\label{affine} Let $X$ be a variety over $R$ and $\scr{X} = (X_i)_{i\in I}$ a family of varieties over $X$. Moreover, let $\mathbf{m} = (m_i)_{i\in I}$ be a family of non-negative integers, $\scr{L}^{\bf{m}}$ the family of affine spaces $(\A_{R}^{m_i})_{i\in I}$ and $\scr{X}\times \scr{L}^{\bf{m}}$ the family $(X_i\times \A_{R}^{m_i})_{i\in I}$, each $X_i\times \A_{R}^{m_i}$ being seen as an $X$-variety through the first projection. Then for all $\pi = (n_i)_{i\in I}\in \N^{(I)}$, the variety $S^{\pi}(\scr{X}\times\scr{L}^{\m})$ is endowed with the structure of a vector bundle of rank $\sum_{i\in I} m_in_i$ over $S^{\pi}(\scr{X})$, so that in particular
$$[S^{\pi}(\scr{X}\times \scr{L}^{\bf{m}})] = [S^{\pi}(\scr{X})]\LL^{\sum_{i\in I}m_in_i}$$ in $\svar_{R}$.
 
\end{prop}

\begin{proof} The first projections induce natural maps
$$\prod_{i\in I}(X_i\times \A_{R}^{m_i})^{n_i}\arr \prod_{i\in I}X_i^{n_i}$$
over $\prod_{i\in I}X^{n_i}$.  
Restricting to points mapping to the diagonal in $\prod_{i\in I}X^{n_i}$, we get a map
$$\left(\prod_{i\in I}(X_i\times \A_{R}^{m_i})^{n_i}\right)_*\arr \left(\prod_{i\in I}X_i^{n_i}\right)_*$$
that gives us in turn a map $$p:S^{\pi}(\scr{X}\times \scr{L}^{\bf{m}})\arr S^{\pi}(\scr{X})$$ after taking the quotient by the natural action of the group $\prod_{i\in I}\Sym_{n_i}$. 

On the other hand, the variety $\left(\prod_{i\in I}(X_i\times \A_{R}^{m_i})^{n_i}\right)_*$ is by definition isomorphic to 

$$\prod_{i\in I}(X_i\times \A_{R}^{m_i})^{n_i}\times_{\prod_{i\in I}X^{n_i}}\left(\prod_{i\in I}X^{n_i}\right)_*\simeq \left(\prod_{i\in I}X_i^{n_i}\right)_*\times \A_{R}^{\sum_{i\in I}n_im_i}.$$
Thus, we have a cartesian diagram 
$$\xymatrix{\left(\prod_{i\in I}X_i^{n_i}\right)_*\times \A_{R}^{\sum_{i\in I}n_im_i} \ar[r]\ar[d]^{q'} & \left(\prod_{i\in I}X_i^{n_i}\right)_*\ar[d]^q\\
S^{\pi}(\scr{X}\times \scr{L}^{\bf{m}})\ar[r]^{\ p}& S^{\pi}(\scr{X})
}$$
with $q'$ being the quotient by the permutation action of $\prod_{i\in I}\Sym_{n_i}$. Through our identification, this permutation action becomes linear, endowing $S^{\pi}(\scr{X}\times \scr{L}^{\bf{m}})$ étale-locally with a structure of vector bundle of rank $\sum_{i\in I}n_im_i$ over $S^{\pi}(\scr{X})$. By Hilbert's theorem 90,  $S^{\pi}(\scr{X}\times \scr{L}^{\bf{m}})$ is a vector bundle of the same rank over $S^{\pi}(\scr{X})$, whence the result. 
\end{proof}
This implies the following formula for zeta-functions:
\begin{cor}\label{zetaaffine} Let $X$, $\scr{X}$, $\scr{L}^{\bf{m}}$ be as above. Then writing $\LL^{\mathbf{m}}\t = (\LL^{m_i}t_i)_{i\in I}$ we have
$$Z_{\scr{X}}(\LL^{\m}\t) = Z_{\scr{X}\times \scr{L}^{\mathbf{m}}}(\t)$$
in $\svar_{R}[[\t]]$. 
\end{cor}

\section{Symmetric products of non-effective classes}\label{symprodclasses}
For the moment, we have only defined the symmetric products $S^{\pi}\scr{X}$ in the case where the family~$\scr{X}$ is a family of quasi-projective varieties over $X$. The aim of this section is to generalise this definition in the case when $\scr{X}$ is merely a family of classes in $\kvar_{X}$. For this, recall that in remark \ref{diagonal}, we noted that one could define the symmetric product $S^{\pi}\scr{X}$ of a family of varieties $\scr{X} = (X_i)_{i\in I}$ over a variety $X$ for $\pi = (n_i)_{i\in I}$ alternatively by first considering the product $\prod_{i\in I}S^{n_i}X_i$ and then restricting to the appropriate open subset, namely the one lying above the open subset $S^{\pi}X$ of $\prod_{i\in I}S^{n_i}X.$ In this section, we are first going to define, for a class $Y\in \kvar_X$, its symmetric power $S^nY$ as an element of $\kvar_{S^nX}$. The symmetric product $S^{\pi}\scr{X}$ of a family of classes $(X_i)_{i\in I}$ in $\kvar_X$ will then be defined, by analogy with remark \ref{diagonal}, as the pullback to $\kvar_{S^{\pi}X}$ of the element $$\prod_{i\in I}S^{n_i}X_i\in \kvar_{\prod_{i\in I}S^{n_i}X}.$$
\subsection{Relative symmetric powers}\label{sect.symprodgrouplaw}
Let $X$ be a quasi-projective variety over $R$. All (fibred and exterior) products in this section are over $R$, though we will not write it to simplify notation. The product $\prod_{n\geq 1}\kvar_{S^nX}$ is an additive group, but it may also be endowed with a commutative multiplicative group structure, in the following manner:

For $a = (a_n)_{n\geq 1}$ and $b = (b_n)_{n\geq 1}\in \prod_{n\geq 1}\kvar_{S^nX}$, we put
$$ (ab)_n = \sum_{k=0}^na_k\boxtimes b_{n-k}$$
where by convention $a_0 = b_0 = 1$, and $a_k\boxtimes b_{n-k}$ is the image of $(a_k,b_{n-k})$ through the composition
$$\kvar_{S^kX}\times \kvar_{S^{n-k}X}\xrightarrow{\boxtimes}\kvar_{S^kX\times S^{n-k}X}\to \kvar_{S^nX}$$
where the latter morphism is obtained from the natural map $X^n\to S^nX$ by passing to the quotient with respect to the natural action of the group $\Sym_k\times \Sym_{n-k}$ on $X^n$. Observe that the neutral element for this law is the family $(e_n)_{n\geq 1}$ with $e_i = 0$ for all $i\geq 1$. Associativity is obtained by using, for all $n\geq 1$ and all $i,j,k$ such that  $i + j + k = n$ the commutativity of the diagram
$$\xymatrix{(S^i X \times S^{j} X) \times S^kX \ar[r]\ar[d]^{\simeq} & S^{i+j}X\times S^k X \ar[r] & S^{n} X\\
S^iX\times (S^{j} X \times S^kX) \ar[r] & S^iX\times S^{j+k}X \ar[ur]
}$$

On the other hand, starting from $a = (a_n)_{n\geq 1}\in \prod_{n\geq 1}\kvar_{S^nX}$, we may define its inverse $b = (b_n)_{n\geq 1}\in \prod_{n\geq 1}\kvar_{S^nX}$ by induction: put $b_0 =1$, $b_1 = -a_1$, and assume $b_0,\ldots,b_{n-1}$ to be constructed. Then $b_n$ is the element of $\kvar_{S^nX}$ defined by 
$$b_n = -\sum_{k=1}^n a_k\boxtimes b_{n-k}.$$

\begin{notation} We denote by $\kvar_{S^{\bullet}X}$ the group $\prod_{n\geq 1}\kvar_{S^nX}$ with this law. 
\end{notation}

\begin{lemma}\label{varmorphism} There is a unique group morphism
$$S:\kvar_X \to \kvar_{S^{\bullet}X},$$
such that the image of the class of a quasi-projective variety $Y$ over $X$ is the family $(S^nY)_{n\geq 1}$. \index{S@$S$}
\end{lemma} 
\begin{proof} 
The condition in the statement defines a morphism  $S'$ on the free abelian group with generators the isomorphism classes of quasi-projective varieties over $X$. Consider a quasi-projective variety $Y$ over $X$, and a closed subscheme $Z$ of $Y$, with open complement $U$. We know that $S^nY$ is the disjoint union of locally closed subsets isomorphic to $S^kZ\times S^{n-k} U$ for $k=0,\ldots,n$, and these isomorphisms are over $S^nX$, so $S'(Y) = S'(Z)S'(U)$, and $S'$  descends to a well-defined group morphism as in the statement of the lemma, using the fact that $\kvar_X$ is generated by classes of quasi-projective varieties. 
\end{proof}
\begin{cor} For every $n\geq 1$ the class in $\kvar_{S^nX}$ of the symmetric power $S^{n}Y$ of a variety $Y$ over $X$ depends only on the class of $Y$ in $\kvar_{X}$.
\end{cor}

\begin{remark} In the same manner, one may put a (commutative) monoid structure on $\svar_{S^{\bullet}X} := \prod_{i\geq 1} \svar_{S^iX}$, and there is a unique morphism of monoids $$S^{+}:\svar_X \to \svar_{S^{\bullet}X},$$
such that the image of the class of a quasi-projective variety $Y$ over $X$ is the family $(S^nY)_{n\geq 1}$. Moreover, $S^{+}$ induces $S$ on $\kvar_{X}$. 
\end{remark}

\begin{definition}\label{sympowerdef} Let $\a \in\kvar_X$ and $n\geq 1$. The $n$-th symmetric power of $\a$, denoted by $S^n\a$, \index{symmetric power! of a class} is the image of $S(\a)$ through the projection 
$$\prod_{i\geq 1}\kvar_{S^iX}\to \kvar_{S^nX}$$
onto the $n$-th component.
\end{definition}

\begin{remark}\label{generalcut} The lemma implies that for any $\a,\b\in\kvar_X$ and for any $n\geq 1$, $$S^{n}(\a + \b) = \sum_{k=0}^nS^{k}\a\boxtimes S^{n-k}\b$$ in $\kvar_{S^nX}$, where every term on the right-hand side is seen as an element of $S^{n}X$ via the natural map $S^{k}X\times S^{n-k}X\to S^nX$. In particular, we have the equality of classes
$$[S^{n}(\a + \b)]= \sum_{k=0}^n[S^{k}\a][S^{n-k}\b]$$ in $\kvar_R$. 
\end{remark}
\subsection{Definition of symmetric products of classes in $\kvar_X$}

Let $X$ be a quasi-projective variety over $R$, $\scr{A} = (\a_i)_{i\in I}$ a family of elements of $\kvar_X$, $n\geq 1$ an integer and $\pi = (n_i)_{i\in I}$ a partition of $n$. Consider the natural projection morphism $$p: \prod_{i \in I}X^{n_i}\to \prod_{i\in  I}S^{n_i}X,$$
which is finite and surjective.  Let $U := \left(\prod_{i\in I}X^{n_i}\right)_*\subset \prod_{i\in I}X^{n_i}$ be the complement of the diagonal, that is, the open subset of points with all coordinates being distinct. Then by construction, $S^{\pi}X$ is the open subset $p(U)$ of the variety $\prod_{i\in I}S^{n_i}X.$

Using definition \ref{sympowerdef}, we may consider the element $\prod_{i\in I}S^{n_i}\a_i \in \kvar_{\prod_{i\in I}S^{n_i}X}$. 

\begin{definition}\label{symproddef}\begin{enumerate}\item  The element $S^{\pi}\scr{A}\in \kvar_{S^{\pi}X}$ is defined to be the image of $\prod_{i\in I}S^{n_i}\a_i $ through the restriction morphism
$$\kvar_{\prod_{i\in I}S^{n_i}X}\to \kvar_{S^{\pi}X}.$$

\item More generally, let $p\geq 1$ be an integer, $X_1,\ldots,X_p$ quasi-projective varieties over~$R$, and for every $j\in\{1,\ldots,p\}$, let $ \scr{A}_j= (\a_{ij})_{i\in I}$ be a family of elements of $\kvar_{X_i}$. For all partitions $\pi_1,\ldots,\pi_p$ with $\pi_i = (r_{i,j})_{i\in I}$ we define the \textit{mixed symmetric product} $S^{\pi_1,\ldots,\pi_p}(\scr{A}_1,\ldots,\scr{A}_p)$ as the image of $\prod_{i\in I}S^{r_{i,1}}\a_{i,1}\times\ldots \times S^{r_{i,p}}\a_{i,p}$ through the restriction morphism
$$\kvar_{\prod_{i\in I}S^{r_{i,1}}X_1\times\ldots\times S^{r_{i,p}}X_p}\to \kvar_{S^{\pi_1,\ldots,\pi_p}(X_1,\ldots,X_p)}.$$
\end{enumerate}
\end{definition}

%\subsubsection{Cutting into pieces}
\begin{remark}\label{separate2} The above restriction morphism factors through  $\kvar_{S^{\pi_1}X_1\times \ldots\times S^{\pi_p}X_p}$ by remark~\ref{separate}. In the case where immersion (\ref{prodimmersion}) is an isomorphism, we have the equality $$S^{\pi_1}\scr{A}_1\boxtimes \ldots\boxtimes S^{\pi_p}\scr{A}_p = S^{\pi_1,\ldots,\pi_p}(\scr{A}_1,\ldots,\scr{A}_p)$$
in $\kvar_{S^{\pi_1}X_1\times \ldots\times S^{\pi_p}X_p} = \kvar_{S^{\pi_1,\ldots,\pi_p}(X_1,\ldots,X_p)}.$
\end{remark}

\begin{remark}\label{generalzeta}Using the first part of the definition, we may extend definition \ref{def.zetafunction} and get a notion of zeta-function $Z_{\scr{A}}(\t)\in\kvar_{R}[[\t]]$ of a family of elements of $\kvar_X$. 
\end{remark}
\begin{prop}\label{generalcut2mixed}
Let $X$ be a quasi-projective variety over $R$. Let $\scr{A} = (\a_i)_{i\in I}$, $\scr{B} = (\b_i)_{i\in I}$, $\scr{C} = (\ce_i)_{i\in I}$ be families of elements of $\kvar_X$ such that for every $i\in I$, $\a_i = \b_i + \ce_i$. For every partition  $\pi= (n_i)_{i\in I }$ we have the equality
$$S^{\pi}\scr{A} = \sum_{\pi'\leq \pi}S^{\pi',\pi-\pi'}(\scr{B},\scr{C})$$
in $\kvar_{S^{\pi}X}$, where each term on the right-hand side is considered as an element of $\kvar_{S^{\pi}X}$ via the natural morphism $S^{\pi'}X\times S^{\pi-\pi'}X \to S^{\pi}X$. 
\end{prop}
\begin{proof} According to remark \ref{generalcut}, we may write
$$\prod_{i\in I}S^{n_i}\a_i = \prod_{i\in I}\left(\sum_{0\leq m_i\leq n_i}S^{m_i}\b_i\boxtimes S^{n_i-m_i}\ce_i\right) = \sum_{\pi' = (m_i)_{i\in I}\leq\pi}\prod_{i\in I}S^{m_i}\b_i\boxtimes S^{n_i-m_i}\ce_i$$
in $\kvar_{\prod_{i\in I}S^{n_i}X}$, where each  $S^{m_i}\b_i\times S^{n_i-m_i}\ce_i$ is seen as an element of $\kvar_{S^{n_i}X}$ through the natural morphism $S^{m_i}X\times S^{n_i-m_i}X\to S^{n_i}X$. This gives the result by restriction to~$S^{\pi}X$. 
\end{proof}
\begin{cor}\label{generalcut2}
Let $X$ be a quasi-projective variety over $R$, $Y$ a closed subscheme of $X$ and $U$ its open complement. Let $\scr{A} = (\a_i)_{i\in I}$ be a family of elements of $\kvar_X$, and define $\scr{B} = (\b_i)_{i\in I}$ and $\scr{C} = (\ce_i)_{i\in I}$ the families of elements of $\kvar_U$ (resp. $\kvar_Y$) obtained by restriction from~$\scr{A}$. For every partition  $\pi= (n_i)_{i\in I }$ we have the equality
$$S^{\pi}\scr{A} = \sum_{\pi'\leq \pi}S^{\pi'}\scr{B}\boxtimes S^{\pi-\pi'}\scr{C}$$
in $\kvar_{S^{\pi}X}$, where each term on the right-hand side is considered as an element of $\kvar_{S^{\pi}X}$ via the natural immersion $S^{\pi'}U\times S^{\pi-\pi'}Y \to S^{\pi}X$. In particular, we have the equality $Z_{\scr{A}}(\t) = Z_{\scr{B}}(\t)Z_{\scr{C}}(\t)$  in $\kvar_{R}[[\t]]$.
\end{cor}
%\begin{proof} According to remark \ref{generalcut}, we may write
%$$\prod_{i\in I}S^{n_i}\a_i = \prod_{i\in I}\left(\sum_{0\leq m_i\leq n_i}S^{m_i}\b_i\boxtimes S^{n_i-m_i}\ce_i\right) = \sum_{\pi' = (m_i)_{i\in I}\leq\pi}\prod_{i\in I}S^{m_i}\b_i\boxtimes S^{n_i-m_i}\ce_i$$
%in $\kvar_{\prod_{i\in I}S^{n_i}X}$, where each  $S^{m_i}\b_i\times S^{n_i-m_i}\ce_i$ is seen as an element of $\kvar_{S^{n_i}X}$ through the natural immersion $S^{m_i}U\times S^{n_i-m_i}Y\to S^{n_i}X$. This gives the result by restriction to~$S^{\pi}X$. 
%\end{proof}

\section{Symmetric products of varieties with exponentials}\label{sect.symprodexp}
 
\subsection{The symmetric product of a family of varieties with exponentials}\label{sect.symprodexpvarietiesdef}
Fix a variety $X$ over $R$, as well as $(\scr{X},f) = (X_i,f_i)_{i\in I}$ a family of varieties over $X$ with exponentials. For any $\pi\in\N^{(I)}$, recall that the symmetric product $S^{\pi}\scr{X}$ is constructed as a quotient of an open subset of the product
$$\prod_{i\in I}X_i^{n_i}.$$
The latter is endowed with the map 
$$\begin{array}{lccc}f^{\pi}:&\prod_{i\in I}X_i^{n_i}&\arr& \A^1\\
                        &(x_{i,1},\ldots,x_{i,n_i})_{i\in I} & \mapsto & \sum_{i\in I}(f_{i}(x_{i,1}) + \ldots + f_{i}(x_{i,n_i})).
\end{array}$$
This map restricts to points lying above the complement of the diagonal, and is invariant modulo the action of $\prod_{i\in I}\Sym_{n_i}$. Therefore it descends to a map
$$f^{(\pi)}:S^{\pi}\scr{X}\arr \A^1,$$\index{fpi@$f^{(\pi)}$}
and we may define the symmetric product $S^{\pi}(\scr{X},f)$ to be the variety $S^{\pi}\scr{X}$ endowed with the map $f^{(\pi)}$. \index{symmetric product!of varieties with exponentials}

\begin{remark} Let $X$ be a variety over $R$, $f:X\to \A^1$ a morphism and $n\geq 1$ an integer. Then there is a morphism $$\begin{array}{lccc}f^{(n)}:&S^nX&\to &\A^1\\
&x_1 + \ldots + x_n & \mapsto & f(x_1) + \ldots + f(x_n)
\end{array}$$  induced by $f^{\conv n}:X^n\to \A^1$. Its restriction to the locally closed subset $S^{\pi}X$ is given by 
$$\sum_{i\geq 1}i(x_{i,1} + \ldots + x_{i,n_i}) \mapsto \sum_{i\geq 1} i(f(x_{i,1}) + \ldots + f(x_{i,n_i})).$$
Thus, the restriction $f^{(n)}_{|S^{\pi}X}$ to $S^{\pi}X$ is exactly the map $g^{(\pi)}$ where $g$ is the family $(if)_{i\geq 1}$ of morphisms $X\to \A^1$. 
\end{remark}\index{fn@$f^{(n)}$}

\subsection{Iterated symmetric products}
Here we generalise proposition \ref{iterate} to varieties with exponentials.  We refer to the corresponding section for the definition of the map $\mu$ and other pieces of notation.
Let $X$ be a variety over a variety $R$, which itself lies above some $k$-variety~$R'$, and let $(\scr{X},f) = (X_i,f_i)_{i\in I}$ be a family of varieties with exponentials over $X$. Let $\pi\in\N^{(I)}$ and let $\varpi = (m_j)_{j\in J}\in\mu^{-1}(\pi)$. Recall from (\ref{VtoW}) and the discussion that follows that we have an immersion
\begin{equation}\label{VtoWbis}\left(\prod_{j\in J}\left(\left(\prod_{i\in I}X_i^{n_i^j}\ _{/R}\right)_{\!\!\!*,X}\right)^{m_j}_{/R'}\right)_{\!\!\!*,R}
\hookrightarrow \left(\prod_{i\in I}X_i^{n_i}\ _{/R'}\right)_{\!\!\!*,X}.\end{equation}
which after passing to the quotient by the appropriate group actions, induces the isomorphism
$$u_{\varpi}: S^{\varpi}(S^{\bullet}(\scr{X}/R)/R')\to S^{\pi}_{\varpi}(\scr{X}/R').$$

By the construction in section\ref{sect.symprodexpvarietiesdef} there is a morphism $f^{(\varpi)}:S^{\varpi}(S^{\bullet}(\scr{X}/R)/R')\to \A^1$ induced by the morphism $$\prod_{j\in J}\left(f^{(\pi_j)}\right)^{\conv m_j}:\ \ \prod_{j\in J}\left(S^{\pi_j}(\scr{X}/R)\right)^{m_j}\ _{/R'}\to \A^1$$
where each $f^{(\pi_j)}$ is itself induced by
$$f^{\pi} = \prod_{i\in I} f_i^{\conv n_i^j}:\ \prod_{i\in I}X_i^{n_i^{j}}\ _{/R}\to \A^1.$$
Thus, $f^{(\varpi)}$ is induced, via $(\ref{VtoWbis})$ and passing to the quotient, by the morphism $f^{\pi}= \prod_{i\in I}f_i^{\conv^{n_i}} = \prod_{j\in J}\left(\prod_{i\in I}f_i^{\conv n_i^j}\right)^{\conv m_j}$ defined on $\prod_{i\in I}X_i^{n_i}\ _{/R}$. We may conclude that for all $\varpi$, we have $f^{(\pi)}\circ u_{\varpi} = f^{(\varpi)}$, which gives the following proposition:  
\begin{prop}\label{expiterate} Let $X$ be a variety over a variety $R$, which itself lies above some $k$-variety~$R'$, and let $(\scr{X},f) = (X_i,f_i)_{i\in I}$ be a family of varieties with exponentials over~$X$. Then for every $\pi\in\N^{(I)}$ and for every $\varpi\in\mu^{-1}(\pi)$, there is an isomorphism $u_{\varpi}$ of the constructible set $S^{\varpi}(S^{\bullet}(\scr{X}/R)/R')$ onto a locally closed subset $S^{\pi}_{\varpi}(\scr{X}/R')$ of $S^{\pi}(\scr{X}/R')$, so that moreover $S^{\pi}(\scr{X}/R')$ is equal to the disjoint union of the sets $S^{\pi}_{\varpi}(\scr{X}/R')$ and $f^{(\pi)}\circ u_{\varpi} = f^{(\varpi)}$. In particular, we have the equality
$$\sum_{\varpi\in \mu^{-1}(\pi)}[S^{\varpi}(S^{\bullet}(\scr{X}/R)/R'),f^{(\varpi)}] = [S^{\pi}(\scr{X}/R'),f^{(\pi)}]$$
in $\sevar_{R'}$. 
\end{prop}

\subsection{Cutting into pieces}
More generally, starting from families of varieties with exponentials $(\scr{X}_1,f_1),\ldots,(\scr{X}_p,f_p)$ over $X$, for any $\pi_1,\ldots,\pi_p\in \N^{(I)}$ the mixed symmetric product $$S^{\pi_1,\ldots,\pi_p}(\scr{X}_1,\ldots,\scr{X}_p)$$ from \ref{mixedsymproducts} may be endowed with a map $(f_1,\ldots,f_p)^{(\pi_1,\ldots,\pi_p)}$ to $\A^1$. Moreover, proposition \ref{pieces} may be extended in the following way:

\begin{prop}\label{exppieces} Let $X$ be a variety over $R$, and $(\scr{X}_1,f_1),\ldots,(\scr{X}_p,f_p)$ families of varieties with exponentials over~$X$, and consider moreover another family $(\scr{X},g) = (X_i,g_i)_{i\in I}$ of varieties over~$X$, together with families 
$$(\scr{U},g_{|\scr{U}}) = (U_i, g_{i|U_i})_{i\in I}\ \ \text{and}\ \ (\scr{Y},g_{|\scr{Y}}) = (Y_i, g_{i|Y_i})_{i\in I}$$ such that for every $i\in I$, $Y_i$ is a closed subvariety of $X_i$ and $U_i$ its complement. Let $\pi_1,\ldots,\pi_p,\pi$ be elements of~$\N^{(I)}$. Then for every $\pi'\in\N^{(I)}$ such that $\pi'\leq \pi$ there is an isomorphism $u_{\pi'}$ from the variety
$$S^{\pi_1,\ldots,\pi_{p},\pi',\pi -\pi'}(\scr{X}_1,\ldots,\scr{X}_{p},\scr{U},\scr{Y})$$ to the locally closed subset of $S^{\pi_1,\ldots,\pi_p,\pi}(\scr{X}_1,\ldots,\scr{X}_p,\scr{X})$ corresponding to points with $\scr{X}$-component inducing partition $\pi'$ on $\scr{U}$, such that
$$(f_1,\ldots,f_p,g)^{(\pi_1,\ldots,\pi_p,\pi)}\circ u_{\pi'} = (f_1,\ldots,f_p,g_{|\scr{U}},g_{|\scr{Y}})^{(\pi_1,\ldots,\pi_p,\pi',\pi-\pi')}.$$
 Moreover, $S^{\pi_1,\ldots,\pi_p,\pi}(\scr{X}_1,\ldots,\scr{X}_p,\scr{X})$ is the disjoint union of these locally closed subsets, so that in terms of classes in $\sevar_{R}$, we have
$$\left[S^{\pi_1,\ldots,\pi_p,\pi}(\scr{X}_1,\ldots,\scr{X}_p,\scr{X}),(f_1,\ldots,f_p,g)^{(\pi_1,\ldots,\pi_p)}\right] $$
$$= \sum_{\pi'\leq \pi}\left[S^{\pi_1,\ldots,\pi_{p},\pi',\pi-\pi'}(\scr{X}_1,\ldots,\scr{X}_{p},\scr{U},\scr{Y}),(f_1,\ldots,f_p,g_{|\scr{U}},g_{|\scr{Y}})^{(\pi_1,\ldots,\pi_p,\pi',\pi-\pi')}\right].$$

\end{prop}
\subsection{Compatibility with affine spaces}
\begin{lemma}\label{expcompaffspaces} Let $Y$ be a quasi-projective variety over a field $k$ and $f:Y\to \A^1$ a morphism. Then for every $\lambda\in k$, and every integer $n\geq 0$  we have the equality:
 $$[S^{n}(Y\times\A^1, f\conv \lambda\id)] = \left\{\begin{array}{cc}[S^nY]\LL^{n}& \text{if}\ \lambda = 0\\
 0 & \text{otherwise}\end{array}\right.
 $$ 
in $\evar_{S^{n}Y}$.
\end{lemma}
\begin{proof} %Put $r = \sum_{i\geq 1}n_i$. % Denote by $\scr{A}$ the constant family of varieties with exponentials $$(Y\times \A^1,f\oplus \lambda\id)_{i\geq 1}$$ seen as $X$-varieties via the first projection. 
Consider the commutative diagram
$$\xymatrix{\left(Y\times \A^1\right)^n\ar[d]^{q'} \ar[r]^-{p'} & Y^n\ar[d]^q\\
S^n(Y\times \A^1)\ar[r]^-{p} & S^nY}$$
where the vertical arrows are the quotient maps.
%Consider the map $p:S^n(Y\times \A^1)\to S^nY$ induced by the projection $(Y\times \A^1)^n\to Y^n$. 
By the proof of Totaro's lemma (see lemma \ref{affinespaces}, or \cite{CNS} chapter 6, Proposition 3.1.5), for every partition $\pi$ of $n$, the lower arrow endows $p^{-1}(S^{\pi}Y)$ with the structure of a vector bundle of rank $n$ over $S^{\pi}Y$ which is locally trivial for the Zariski topology. On the other hand, by definition, the symmetric product $S^{n}(Y\times \A^1)$ is endowed with the morphism $(f\conv \lambda\id)^{(n)}$
%$$(f\conv \lambda\id)^{(\pi)}: \sum_{i\geq 1} i((y_{i,1},t_{i_1}) + \ldots + (y_{i,n_i},t_{i,n_i}))\mapsto \sum_{i\geq 1}(f(y_{i,1}) + \ldots + f(y_{i,n_i})) + \sum_{i\geq 1}\lambda(t_{i,1} + \ldots + t_{i,n_i})$$
induced by the morphism $(f\conv \lambda\id)^{\conv n}$ given by 
$$\begin{array}{ccc} (Y\times \A^1)^{n} = Y^n\times \A^n& \to & \A^1\\
 (y_1,\ldots,y_n,t_1,\ldots,t_n)& \mapsto & f(y_1) + \ldots + f(y_n) +\lambda(t_1+\ldots + t_n)
\end{array}$$
Consider a point $y\in S^{\pi}Y$. We know that the fibre of $p^{-1}(S^{\pi}Y)$ above this point is an affine space $\A^n_{\kappa(y)}$. The linear form $(t_1,\ldots , t_n) \mapsto \lambda(t_1 + \ldots + t_n)$ on the general fibre of the trivial vector bundle in the top row of the diagram descends (via the permutation action, which is linear) to some linear form $\ell$ on $\A^n_{\kappa(y)}$, which will be zero if and only if $\lambda$ is zero. Thus, since by construction the morphism $(f \conv\lambda\id)^{(n)}$ coincides with $f^{(n)}$ on the zero-section of the vector bundle $p^{-1}(S^{\pi}Y)\to S^{\pi}Y$, we have, for any $(t_1,\ldots,t_n)\in\A^n_{\kappa(y)}$
$$(f\conv \lambda\id)^{(n)}(y,t_1,\ldots,t_n) = f^{(n)}(y) + \ell(t_1,\ldots,t_n),$$
with  $\ell$ a linear form, which is zero if and only if $\lambda = 0$. Note that because of the last axiom defining $\evar_k$, we have, using a linear change of basis, that 
$$ [\A^n,\ell] = \left\{\begin{array}{rc}\LL^n & \text{if}\ \ell= 0\\
                                         0 &  \text{otherwise.}\end{array}\right.$$
                                    Taking $y$ to be a generic point of $S^{\pi}Y$ and spreading out to some trivialising open set $U$ for the vector bundle $p^{-1}(S^{\pi}Y)\to S^{\pi}Y$, we have $$[p^{-1}(S^{\pi}Y)_{|U}, (f\conv\lambda\id)^{(n)}_{|p^{-1}(S^{\pi}(Y))_{|U}}] = [U,f^{(n)}_{|U}][\A^n,\ell]
                                     = \left\{\begin{array}{rc}[U,f^{(n)}_{|U}]\LL^n & \text{if}\ \lambda= 0\\
                                         0 &  \text{otherwise.}\end{array}\right.$$
                                         Repeat the process for a generic point of $S^{\pi}Y\setminus U$. By Noetherian induction, the result follows by taking the sum over all partitions of $n$. 
                                         %Let $(y_1,\ldots,y_r)\in q^{-1}(y)$. Let $G$ be the subgroup of $\Sym_r$ defined as the image of the obvious injective group morphism
%$$\prod_{i\geq 1}\Sym_{n_i}\to \Sym_r.$$
%The preimage of the fibre $\{y\}\times \A^r_{\kappa(y)}$ through the morphism $q'$ is the disjoint union
%$$p^{-1}(q^{-1}(y)) = \bigcup_{\sigma\in G}\{(y_{\sigma(1)},\ldots,y_{\sigma(r)})\}\times \A^r_{\kappa(y)}$$
%on which $(f\conv \id)^{\pi}$ is given by 
%$$(y_{\sigma(1)},\ldots,y_{\sigma_r},t_1,\ldots,t_r)\mapsto f(y_1) + \ldots + f(y_r) + t_1 + \ldots + t_r.$$
%Therefore, the restriction of $(f\conv\id)^{(\pi)}$ to the fibre $\A_{\kappa(y)}^r$ is given by
%$$(t_1,\ldots,t_r) \mapsto f(y_1) + \ldots + f(y_r) + t_1 + \ldots + t_r.$$
\end{proof}
\subsection{Symmetric products of classes in Grothendieck rings with exponentials}
Let $X$ be a quasi-projective variety. The same procedure as in section \ref{sect.symprodgrouplaw} endows the product $\prod_{i\geq 1}\evar_{S^iX}$ with a group law, and this group will be denoted $\evar_{S^{\bullet}X}$.
\begin{prop} Let $X$ be a quasi-projective variety. There is a unique group morphism
$$S^{\mathrm{exp}}:\evar_X \to \evar_{S^{\bullet}X}$$\index{Se@$S^{\exp}$}
sending the class $[Y,f]$ of a quasi-projective variety $Y$ with a morphism $f:Y\to \A^1$, to the family of classes $([S^{i}Y,f^{(i)}])_{i\geq 1}.$ Moreover, there is a commutative diagram
$$\xymatrix{\kvar_X \ar[r]^S \ar[d]&\kvar_{S^{\bullet}X}\ar[d]\\
\evar_X \ar[r]^{S^{\exp}}& \evar_{S^{\bullet}X}}$$
where the vertical arrows are given by the injections $\kvar\to\evar$. \end{prop}
\begin{proof} Define a morphism $S'$ from the free abelian group on pairs $(Y,f)$, where $Y$ is a quasi-projective variety over $X$ and $f:Y\to \A^1$ a morphism, to $\evar_{S^{\bullet}X}$ by $S'((Y,f)) = ([S^{i}Y,f^{(i)}])_{i\geq 1}$. It suffices now to check that $S'$ passes to the quotient through the three relations defining $\evar_X$. It is clear that $S'$ is constant on isomorphic pairs. If $Z$ is a closed subscheme of $Y$ with open complement $U$, then for any $n\geq 1$, $[S^nY,f^{(n)}] = \sum_{i=0}^n[S^iU,f_{|U}^{(i)}][S^{n-i}Z,f_{|Z}^{(n-i)}]$, so that $S'(Y) = S'(U)S'(Z)$. Finally, it follows from lemma \ref{expcompaffspaces} that for any quasi-projective variety $Y$ over $X$, the class $[Y\times_{\Z} \A^1, \pr_2]$ goes to zero. The commutativity of the diagram may be checked for classes of varieties $[Y\to X]$, where it is immediate.
\end{proof}
The notion of symmetric product of a class in $\kvar_X$ may therefore be extended to $\evar_X$ in the following manner: 
\begin{definition} Let $\a\in\evar_X$ and $n\geq 1$. The $n$-th symmetric product of $\a$, denoted by $S^n\a$ is the image of $S^{\exp}(\a)$ through the projection 
$$\prod_{i\geq 1}\evar_{S^iX} \to \evar_{S^nX}$$
onto the $n$-th component.
\end{definition}

\begin{remark}\label{symprodexpdefandprop} If $X$ be a quasi-projective variety over $R$ and $\scr{A} = (A_i)_{i\geq 1}$ a family of varieties with exponentials, definition \ref{symproddef} may be extended in a compatible way to define, for any partition $\pi$, the symmetric product $S^{\pi}\scr{A}$ as an element of $\evar_{S^{\pi}X}$. We can also define mixed symmetric products, and proposition \ref{generalcut2mixed} and corollary \ref{generalcut2} are true more generally with $\kvar$ replaced by $\evar$. 
\end{remark}

\section{Symmetric products in localised Grothendieck rings}\label{sect.locsymproducts}
In this section, we are going to generalise the notion of symmetric product further, to include classes in the localised Grothendieck rings $\M_X$ and $\expp_X$. Since the Grothendieck ring $\kvar$ injects into $\evar$, it does not restrict generality to work with $\evar$ directly, which we will do in this section. 

\begin{lemma}\label{gencompaffinespaces} For every $\a\in \evar_X$, for any $n\geq 1$ and for any $m\geq 1$, one has
$$S^{n}(\a\LL^{m}) = S^{n}(\a)\LL^{nm}$$
in $\evar_{S^nX}$. %The same statement holds in $\kvar_{S^nX}$ if $\a$ is an element of $\kvar_X$. 
\end{lemma}
\begin{proof} Lemma \ref{expcompaffspaces} shows that this holds for effective elements of $\evar_X$. It suffices to prove that it holds for a difference $Y-Z$ of two effective elements. Using the fact that $S$ (resp. $S^{\exp}$) is a group morphism, one may write, by induction on $n$,
\begin{eqnarray*}S^{n}((Y-Z)\LL^{m}) &=& S^{n}(Y\LL^{m}) - \sum_{k=0}^{n-1}S^{n-1-k}((Y-Z)\LL^{m})S^k(Z\LL^m)\\
& = & S^n(Y)\LL^{mn} - \LL^{mn}\sum_{k=0}^{n-1}S^{n-1-k}(Y-Z)S^{k}(Z) \\
& = & S^n(Y-Z)\LL^{mn}.
\end{eqnarray*}
\end{proof}
Let $X$ be a quasi-projective variety. The same procedure as in section \ref{sect.symprodgrouplaw} endows the product $\prod_{i\geq 1}\expp_{S^iX}$ with a group law, and this group will be denoted $\expp_{S^{\bullet}X}$.
\begin{lemma}\label{sloc} There is a unique group morphism
$$S^{\loc}:\expp_{X}\to \expp_{S^{\bullet}X}$$ given by $\a\LL^{-m}\mapsto ((S^i\a)\LL^{-mi})_{i\geq 1}.$
\end{lemma}\index{Sloc@$S^{\loc}$}
\begin{proof} We already have a group morphism
\begin{equation}\label{slocequation}\evar_X\to \expp_{S^{\bullet}X}\end{equation}
obtained from $S$ by composing with the localisation morphism $\evar_{S^{\bullet} X}\to \expp_{S^{\bullet}X},$ and given by $\a \mapsto (S^{i}\a)_{i\geq 1}.$ Lemma \ref{gencompaffinespaces} shows that an element in the kernel of the localisation morphism $\alpha_X:\evar_X\to \expp_X$ will go to 0, so (\ref{slocequation}) factors through~$\alpha_X$, inducing a morphism $$\mathrm{im}(\alpha_X)\to \expp_{S^{\bullet}X}.$$
Let $\a \in \expp_{X}$. Then there exists an integer $m\geq 1$ such that $\a\LL^{m}$ belong to $\mathrm{im}(\alpha_X)$. Put 
$$S^{\loc}(\a) = (S^{i}(\a\LL^m)\LL^{-mi})_{i\geq 1}$$  
By lemma \ref{gencompaffinespaces} this does not depend on the choice of $m$, so $S^{\loc}$ is well-defined.
\end{proof}

\begin{remark}\label{finalsymproddef} We may now define, for any $\a\in \expp_X$, its symmetric powers $S^{i}(\a)$ to be the components of $S^{\loc}(\a)$. More generally, for any partition $\pi = (n_i)_{i\in I}$ we may define the symmetric products $S^{\pi}\scr{A}$ of any family $\scr{A}$ of elements of $\expp_X$. Furthermore, we may define mixed symmetric products of a finite number of families of elements of $\expp_X$, and as in remark \ref{symprodexpdefandprop}, proposition \ref{generalcut2mixed} remains true with $\kvar$ replaced with $\M, \evar$ or $\expp$. 
\end{remark}
We also have, for $A\in \{\kvar,\M,\evar,\expp\}$:
\begin{prop}\label{genaffinespaces} Let $X$ be a variety over $R$, and $(X_i)_{i\in I}$ a family of elements in $A_X$. Let $\m = (m_i)_{i\in I}$ be a family of non-negative integers, and denote by $\scr{X}\LL^{\m}$ the family $(X_i\LL^{m_i})_{i\in I}$. Then for all $\pi = (n_i)_{i\in I}\in\N^{(I)}$, we have
$$S^{\pi}(\scr{X}\LL^{\m}) = (S^{\pi}\scr{X})\LL^{\sum_{i\in I}m_in_i}$$
in $A_{S^{\pi}X}.$
\end{prop}
\begin{proof} By lemma \ref{gencompaffinespaces}, we have, for all $i\in I$,
$$S^{n_i}(X_i\LL^{m_i}) = S^{n_i}(X_i)\LL^{m_in_i}$$
in $A_{S^{n_i}X}$. Taking exterior products over $i\in I$ and restricting to $S^{\pi}X$ we get the result.
\end{proof}
\section{Euler products}\label{eulerprod}
\subsection{Definition and first properties}

Let $A\in\{\kvar,\M, \evar,\expp\}$. Let $X$ be a variety over $\base$ and $\scr{X} = (X_i)_{i\in I}$ a family of elements of $A_X$. We define the  Euler product notation to be
$$\prod_{u\in X/\base}\left( 1 + \sum_{i\in I} X_{i,u}t_i\right) :=  \sum_{ \pi\in\N^{(I)}} [S^{\pi}(\scr{X}/R)]\t^{\pi} = Z_{\scr{X}/\base}(\t)\ \ \ \in A_{\base}[[\t]],$$\index{Euler product}
 where the $t_i,\ i\in I$ are variables, and $\t^{\pi}$ is defined to be the finite product $\prod_{i\in I}t_i^{n_i}$, where $\pi = (n_i)_{i\in\I}$. When $\base = \spec k$ for a field $k$, we will leave out the mention of $\base$ in the product, writing simply $\prod_{u\in X}$.

 We are going to start by checking that our ``product'' actually behaves like a product, thereby justifying our notation. %A summary of the properties proved below is given in theorem \ref{theoremeprodeul} of the introduction.

\paragraph{Properties:} Let $X$ be a variety over $\base$ and $\scr{X} = (X_i)_{i\in I}$ a family of classes in $A_X$. \begin{enumerate}
\item (Product with one factor) When $X = \base$, the last part of Example $\ref{firstex}$ gives \begin{equation}\label{onepoint}\prod_{u\in \base/\base}\left(1 + \sum_{i\in I}X_{i,u}t_i\right) = 1 + \sum_{i\in I}X_i t_i.\end{equation}% where $[X_i]$ is the class of $X_i$ in $\svar_{\base}$. 
 \item (Associativity) Let $X  = U\cup Y$ be a partition of $X$ into a closed subscheme $Y$ and its complement $U$, and $\scr{U} = (U_i)_{i\in I}$ (resp. $\scr{Y}= (Y_i)_{i\in I}$) the restriction of $\scr{X}$ to $U$ (resp. to~$Y$). Then  corollary  \ref{generalcut2} can be reformulated as
 \begin{equation}\label{associativity} \prod_{u\in X/\base}\left(1 + \sum_{i\in I}X_{i,u}t_i\right) =  \prod_{u\in U/\base}\left(1 + \sum_{i\in 1} U_{i,u}t_i\right) \prod_{u\in Y/\base}\left(1 + \sum_{i\in I}Y_{i,u}t_i\right).\end{equation}
 Here we use remarks \ref{symprodexpdefandprop} and \ref{finalsymproddef}, which state that corollary \ref{generalcut2} is true more generally with $\kvar$ replaced with $\evar$, $\M$ or $\expp$. 
 \item (Finite products) Combining the previous two properties, we see that in the case where $X$ is a disjoint union of $m$ varieties $Y_1,\ldots,Y_m$ isomorphic to $\base$, 
 \begin{equation}\label{finiteprod}\prod_{v\in X/\base}\left(1 + \sum_{i\in I}X_{i,v}t_i\right)=\prod_{j=1}^m\left(1 + \sum_{i\in I}X_{i,j} t_i\right)\in A_{\base}[[\t]]\end{equation}
 where for all $i\in I$ $X_{i,j}$ is the restriction of $X_i$ to $Y_j\simeq \base$, and the product on the right-hand side is a finite product of power series (in the classical sense).
 \item (Change of variables of the form $\t\mapsto \LL^{\mathbf{m}}\t$) In terms of Euler products, Proposition \ref{genaffinespaces} means that for any $(m_i)_{i\in I}\in\N^{I}$, 
 $$\prod_{u\in X/R}\left(1 + \sum_{i\in I}X_{i,u}(\LL^{m_i}t_i)\right) = \prod_{u\in X/R}\left(1 + \sum_{i\in I}(X_{i,u}\LL^{m_i})t_i\right),$$
where the right-hand-side is the Euler product associated to the family $(X_i\LL^{m_i})_{i\in I}$, that is, the Euler product notation is compatible with respect to changes of variables of the form $(t_i\mapsto \LL^{m_i}t_i)_{i\in I},$ and factors of the form $\LL^{m_i}$ may pass from the variables to the coefficients. One must pay attention that this is specific to affine spaces because of their good behaviour with respect to Euler products.

 \end{enumerate}
 
 \begin{remark}
Let us now try to get a grip on what each factor of such an infinite product actually represents. For simplicity, we will consider $\base = \spec k$.  If $X = \spec k$, then according to (\ref{onepoint}),
 $$\prod_{u\in \spec k}\left(1 + \sum_{i \in I}X_{i,u}t_i\right) = 1 + \sum_{i\in I}X_it^i \in\kvar_{k}[[t]].$$
In this case the left-hand side has ``only one factor'', and identifying coefficients on both sides suggests that we can think of those objects $X_{i,u}$ as being the classes in $\kvar_{k}$ of the fibres of the~$X_i$ above the single closed point $u\in X$.
 
 Now let~$X$ be the disjoint union of a copy $Y$ of $\spec k$ (with a single closed point $y$) and of its open complement $U$, and let $\scr{Y}$ and $\scr{U}$ be the respective restrictions of $\scr{X}$ to $Y$ and $U$. Then using what we just remarked together with (\ref{associativity}),
 \begin{eqnarray*}\prod_{x\in X}\left( 1 + \sum_{i \in I}X_{i,x}t_i\right) &= &\prod_{y\in Y}\left(1 + \sum_{i \in I}Y_{i,y}t_i\right)\prod_{u\in U}\left(1 + \sum_{i \in I}U_{i,u}t_i\right)\\ &=& \left(1 + \sum_{i\in I}Y_{i} t_i\right)\prod_{u\in U}\left(1 + \sum_{i \in I}U_{i,u}t_i\right) \end{eqnarray*}
 This means that adding a closed point $\spec k$ to the variety consists exactly in adding a factor with coefficient of $t_i$ representing the class in $\kvar_{k(y)}$ of the fibre $Y_i = X_{i,y}$ above the added point~$y$. 
 \end{remark}
 \begin{example}\label{eulerprodexample} \begin{enumerate}\item \label{Kapranovexample} Kapranov's zeta function is obtained when taking $I = \N^{*}$, $X_i = X$ and specialising $t_i$ to $t^i$ for some single indeterminate $t$, for all $i$. Thus, for all closed points $v\in X$, the class of $X_{i,v}$ in $\svar_{k(v)}$ is 1 and the Euler product decomposition of Kapranov's zeta function can be written as
 $$\prod_{v\in X}\left(\sum_{i\geq 0}t^i\right) = 1 + \sum_{n\geq 1}[S^nX]t^n,$$\index{Kapranov's zeta function!Euler product}
 or even
 $$\prod_{v\in X}(1-t)^{-1} = 1 + \sum_{n\geq 1}[S^nX]t^n,$$
 \item For $I = \N^{*}$, put $X_1 = X$ and $X_i = \varnothing$ for all $i \geq 2$. Then the class of $X_{i,v}$ in $\svar_{k(v)}$ is $1$ if $i=1$ and 0 otherwise. Thus, writing $t_1=t$, we get the following Euler product decomposition for the generating function of zero-cycles without repetitions:
 $$\prod_{v\in X}(1 + t) = 1 + \sum_{n\geq 1}[S^n_*X]t^n.$$
 \end{enumerate}
 
 \end{example}
 
 \section{Double products}
\subsection{Double products with effective coefficients}
In this section we translate proposition \ref{iterate} into the language of Euler products, which will show the good behaviour of double products with \textit{effective} coefficients. We will then use this property in the next sections to deduce that this good behaviour extends in some sense to double products with not necessarily effective coefficients, when some natural assumptions on the set $I$ and the family of indeterminates $(t_i)_{i\in I}$ are made.  

%\label{doubleprodprop} 
We  place ourselves in the situation of proposition~\ref{iterate}: we consider $R'$ a variety over~$k$, $R$ a variety over $R'$,  $X$ a variety over~$R$, and $\scr{X} = (X_i)_{i\in I}$ a family of varieties over $X$, indexed by a set $I$.  Using our notation for the family of varieties $S^{\bullet}(\scr{X}/\base)$ over $\base$, we have
 $$\prod_{v\in\base/\base'}\left(1 + \sum_{\pi\in\N^{(I)}\backslash\{0\}}(S^{\pi}(\scr{X}/\base))_v\t^{\pi}\right) = \sum_{\varpi\in\N^{\left(\N^{(I)}\backslash\{0\}\right)}}[S^{\varpi}(S^{\bullet}(\scr{X}/\base)/\base'))]\t^{\varpi},$$
where, writing $\varpi = (m_{\pi})_{\pi\in \N^{(I)}\backslash\{0\}}$, and denoting by $\pi_i$ the number of times $i$ occurs in~$\pi$, $\t^{\varpi} $ is by definition  
$$\prod_{\pi\in\N^{(I)}\backslash\{0\}}(\t^{\pi})^{m_{\pi}} = \prod_{\pi\in\N^{(I)}\backslash\{0\}}\left(\prod_{i\in I}t_i^{\pi_i}\right)^{m_{\pi}} = \prod_{i\in I}t_i^{\sum_{\pi\in\N^{(I)}\backslash\{0\}}m_{\pi}\pi_i} = \t^{\mu(\varpi)}.$$
 In particular, $\t^{\varpi}$ and $\t^{\varpi'}$ are equal if and only if, in the notations of definition \ref{mu}, $\mu(\varpi) = \mu(\varpi')$. Thus, the above product can be rewritten in the form
 $$\prod_{v\in\base/\base'}\left(1 + \sum_{\pi\in\N^{(I)}\backslash\{0\}}(S^{\pi}(\scr{X}/\base))_v\t^{\pi}\right) $$
 $$= 1 + \sum_{\pi\in\N^{(I)}\backslash\{0\}}\left(\sum_{\varpi\in\mu^{-1}(\pi)}[S^{\varpi}(S^{\bullet}(\scr{X}/\base)/\base')]\right)\t^{\pi}.$$
We may therefore conclude, using proposition \ref{iterate} that the double product $$\prod_{v\in\base/\base'}\left(\prod_{u\in X/\base}\left(1 + \sum_{i\in I}X_{i,u}t_i\right)\right)_v$$ makes sense and that the following result is true:
\begin{prop}\label{prop.effectivedoubleprod} Let $R'$ be a variety over $k$, $R$ a variety over $R'$,  $X$ a variety over~$R$, and $\scr{X} = (X_i)_{i\in I}$ a family of varieties over $X$, indexed by a set $I$. Then we have the equality
 \begin{equation}\label{doubleprod}\prod_{v\in\base/\base'}\left(\prod_{u\in X/\base}\left(1 + \sum_{i\in I}X_{i,u}t_i\right)\right)_v = \prod_{u\in X/\base'}\left(1 + \sum_{i\in I}X_{i,u}t_i\right)\end{equation}
 in $\kvar_{R'}[[(t_i)_{i\in I}]]$. 
 \end{prop}
  By proposition \ref{expiterate} this remains true with $\kvar$ replaced by $\evar$ if $\scr{X}$ is a family of varieties with exponentials. We may also pass to the respective localisations $\M$ and $\expp$.

  Assume now that $X$ is the disjoint union of two copies of $\base$, written respectively $Y$ and~$Z$. For all $i \in I$, we will write $Y_i$ (resp. $Z_i$) for the restriction of $X_i$ to $Y$ (resp. to $Z$).  Using $(\ref{finiteprod})$, we get that
 $$\prod_{u\in X/\base}\left(1 + \sum_{i \in I}X_{i,u}t_i\right) = \left(1 + \sum_{i \in I}Y_it_i\right)\left(1 + \sum_{i \in I}Z_it_i\right).$$
  Taking the product over $\base$ relatively to $\base'$, we get
 $$\prod_{v\in\base/\base'}\left(1 + \sum_{i \in I}Y_{i,v}t_i\right)\left(1 + \sum_{i \in I}Z_{i,v}t_i\right)$$
 On the other hand, by (\ref{associativity}), we have
 $$\prod_{v\in X/\base'}\left(1 + \sum_{i \in I}X_{i,v}t_i\right) = \prod_{v\in Y/\base'}\left(1 + \sum_{i \in I}Y_{i,v}t_i \right)\prod_{v\in Z/\base'}\left(1 + \sum_{i \in I}Z_{i,v}t_i \right).$$
 Using $(\ref{doubleprod})$, we finally get:
 \begin{prop} Let $R'$ be a variety over $k$, $R$ a variety over~$R'$, and $(Y_i)_{i\in I}$, $(Z_i)_{i\in I}$ two families of varieties over~$R$. Then we have the equality 
$$\prod_{v\in\base/\base'}\left(1 + \sum_{i \in I}Y_{i,v}t_i\right)\left(1 + \sum_{i \in I}Z_{i,v}t_i\right) = \prod_{v\in \base/\base'}\left(1 + \sum_{i \in I}Y_{i,v}t_i \right)\prod_{v\in \base/\base'}\left(1 + \sum_{i \in I}Z_{i,v} t_i\right)
$$
 in $A_{R'}[[(t_i)_{i\in I}]]$, where $A\in\{\kvar, \evar, \M, \expm\}.$
 \end{prop}
 \begin{example} Let $X$ be a variety over $\base$ and $n\geq 1$ an integer. As an application of double products, let us compute Kapranov's zeta function of a $\Proj^n$-bundle $Y$ over $X$ in terms of $Z_X$. By definition,
$$Z_{Y} = \prod_{y\in Y}\frac{1}{1-t}.$$
Using (\ref{doubleprod}), we have
\begin{eqnarray*}Z_{Y}& =& \prod_{x\in X}\left(\prod_{y\in Y/X}\frac{1}{1-t}\right)_x\\
                                    & = &\prod_{x\in X}\prod_{y\in \Proj^n}\frac{1}{1-t}\\
                                   &= &\prod_{x\in X}\frac{1}{1-t}\frac{1}{1-\LL t}\ldots\frac{1}{1-\LL^{n}t}
\end{eqnarray*}
where the last line comes from the exact expression of $Z_{\Proj^{n}}(t)$. Using compatibility with finite products, we finally get
$$Z_{Y}=  \prod_{i=0}^n\prod_{x \in X}\frac{1}{1-\LL^i t}= Z_X(t)Z_X(\LL t)\ldots Z_X(\LL^{n} t).$$
  Note that this result can be obtained without the Euler product, by cutting the projective bundle into affine bundles and applying propositions \ref{affine} and  \ref{multzeta} (see \cite{LL}, Corollary~3.6). 
                                \end{example}
 \subsection{Double products with arbitrary coefficients}
 To get a multiplicativity result of this type for non-effective coefficients, we will make assumptions on the set $I$ and the indeterminates $(t_i)_{i\in I}$. 
 \paragraph{Hypothesis on $I$.} We assume that $I$ is of the form $I_0\setminus \{0\}$ where $I_0$ is a commutative monoid such that
 \begin{enumerate} 
 \item $I_0$ is endowed with a total order $<$ compatible with the monoid law, in the sense that if we have $p+q = i$ for some $q\neq 0$, then $p<i$.
 \item for all $i\in I$, the set
$\{p\in I,\ p<i\}$ is finite.
 \end{enumerate}
 A consequence of this is that $I$ has a minimal element.
 \paragraph{Hypothesis on $(t_i)_{i\in I}$.} We are going to consider collections of indeterminates $(t_i)_{i\in I}$ satisfying the condition
$t_pt_q = t_{p+q}$
for all $p,q\in I$. We will sometimes write $t_0 =1$. 
\begin{notation} Whenever we write sums or products indexed by the condition $p+q = i$, we mean that they go over all pairs $(p,q)\in I_0^2$ such that $p+q = i$. 
\end{notation} 

\begin{prop} \label{prop.findeffective} Assuming the above conditions on $I$ and $(t_i)_{i\in I}$, for every $k$-variety $X$ and for every family $(X_i)_{i\in I}$ of elements of $\kvar_X$, there exists a family of effective elements $(Y_i)_{i\in I}$ of $\kvar_X$ such that the coefficients of the expansion of the series
$$\left(1 + \sum_{i\in I}X_it_i\right)\left(1 + \sum_{i\in I} Y_it_i\right)\in \kvar_X[[\t]]$$
are effective.
\end{prop}
\begin{proof} We construct $(Y_i)_{i\in I}$ by induction, using the order $<$. We put $Y_0 = X$. Assume the $Y_p$ are constructed for all $p$ satisfying $p<i$ for some $i\in I$. We are looking for an effective $Y_i$ such that
$$\sum_{\substack{p+q = i\\ q\neq 0}}X_pY_q + Y_i$$
is effective. By the first condition on $I$ and by the induction hypothesis, all the $Y_i$ appearing in the sum are already defined, so such a $Y_i$ exists. 
\end{proof}

\begin{example} The main example of this situation that will be of interest to us is when $I_0 = \Z_{\geq 0}^{r}$ for some integer $r\geq 1$ (with the lexicographical order), and for $i = (i_1,\ldots,i_r)\in I$, the indeterminate $t_i$ is given by $s_1^{i_1}\ldots s_r^{i_r}$ for $r$ independent indeterminates $s_1,\ldots,s_r$. 
\end{example}

 The aim of the following sections is to prove: 
 \begin{prop}\label{compfinprodnoneff} Let $R$ be a variety, $X$ a variety over~$R$, $\scr{A} = (A_i)_{i\in I}$ a family of effective elements of $\kvar_X$, and $\scr{B} = (B_i)_{i\in I}$ a family of elements of $\kvar_X$. Then, under the above hypotheses on~$I$ and $(t_i)_{i\in I}$, 
$$\prod_{x\in X}\left(1 + \sum_{i\in I} A_{i,x}t_i\right)\left(1 + \sum_{i\in I} B_{i,x}t_i\right) = \prod_{x\in X}\left(1 + \sum_{i\in I} A_{i,x}t_i\right)\prod_{x\in X} \left(1 + \sum_{i\in I} B_{i,x}t_i\right)$$
in $\kvar_R[[(t_i)_{i\in I}]]$. 
\end{prop} 
Before proving the proposition, let us state the following corollary:

   \begin{cor} Let $R'$ be a variety, $R$ a variety over $R'$,  $X$ a variety over~$R$, and $\scr{X} = (X_i)_{i\in I}$ a family of classes in $\kvar_X$,  indexed by a set $I$. Then, under the above hypotheses on~$I$ and $(t_i)_{i\in I}$, we have the equality
            \begin{equation}\label{doubleprodnoneff}\prod_{v\in\base/\base'}\left(\prod_{u\in X/\base}\left(1 + \sum_{i\in I}X_{i,u}t_i\right)\right)_v = \prod_{u\in X/\base'}\left(1 + \sum_{i\in I}X_{i,u}t_i\right).\end{equation}
             in $\kvar_{R'}[[(t_i)_{i\in I}]]$. 
           \end{cor}
          
        \begin{proof} By proposition \ref{prop.effectivedoubleprod}, this is already known when the classes $X_i$ are effective. Using proposition \ref{prop.findeffective}, choose a family of effective classes $\scr{Y} = (Y_i)_{i\in I}$ such that 
        $$\left(1 + \sum_{i\in I}X_it_i\right)\left(1 + \sum_{i\in I}Y_it_i\right) = 1 + \sum_{i\in I}Z_it_i\in \kvar_{X}[[(t_i)_{i\in I}]]$$
        has effective coefficients. Then we have
        $$\prod_{v\in\base/\base'}\left(\prod_{u\in X/\base}\left(1 + \sum_{i\in I}Z_{i,u}t_i\right)\right)_v = \prod_{u\in X/\base'}\left(1 + \sum_{i\in I}Z_{i,u}t_i\right)$$
        for the effective family $(Z_i)_{i\in I}$. Thus, we have 
$$\prod_{v\in\base/\base'}\left(\prod_{u\in X/\base}\left(1 + \sum_{i\in I}X_{i,u}t_i\right)\left(1 + \sum_{i\in I}Y_{i,u}t_i\right)\right)_v = \prod_{u\in X/\base'}\left(1 + \sum_{i\in I}X_{i,u}t_i\right)\left(1 + \sum_{i\in I}Y_{i,u}t_i\right).$$
Using proposition \ref{compfinprodnoneff}, we get
$$\prod_{v\in\base/\base'}\left(\prod_{u\in X/\base}\left(1 + \sum_{i\in I}X_{i,u}t_i\right)\right)_v \prod_{v\in\base/\base'}\left(\prod_{u\in X/\base}\left(1 + \sum_{i\in I}Y_{i,u}t_i\right)\right)_v=$$
$$ \prod_{u\in X/\base'}\left(1 + \sum_{i\in I}X_{i,u}t_i\right)\prod_{u\in X/\base'}\left(1 + \sum_{i\in I}Y_{i,u}t_i\right).$$
By proposition \ref{prop.effectivedoubleprod}, the factors corresponding to $\scr{Y}$ cancel out and we have the result.
        \end{proof}
        
        \begin{remark} Though in this section we have stated everything for $\kvar_X$, the results remain true for Grothendieck rings of exponentials and localised Grothendieck rings. 
        \end{remark}
\subsection{Some combinatorics of partitions}\label{subsect.combinatorics}

 We use some terminology and notation on partitions due to Vakil-Wood (see \cite{VW}, section 2) and Howe (see \cite{Howe}, section~6.1). As before, a (generalised) partition $\kappa$ in an abelian semigroup $I$ may equivalently be seen as a family of non-negative integers $(k_i)_{i\in I}$, almost all of them being zero, or as a finite multiset $[i^{k_i}]_{i\in I}$ of elements of $I$, the number of occurrences of $i$ being given by $k_i$ (so that exponents in the notation $[i^{k_i}]_{i\in I}$ denote multiplicity). We use the notations
 \begin{itemize} \item $\sum\kappa:= \sum_{i\in I} k_ii$ for the sum of the partition $\kappa$, 
 \item $|\kappa| = \sum_{i\in I}k_i$ for the number of elements (or \textit{parts}) of $\kappa$,
 \item $||\kappa|| = |\{i\in I,\ k_i\neq 0\}|$ for the number of \textit{distinct} parts of $\kappa$. 
 \end{itemize}
 A \textit{formalisation} of the partition $\kappa$ is a partition $f(\kappa)$ in the abelian semigroup $\Z^+[a_i]_{i\in I}$ (where the $a_i$ are formal indeterminates) given by $f(\kappa) = [a_i^{k_i}]_{i\in I}$. In other words, every element of $\kappa$ is replaced by a formal indeterminate, keeping the same multiplicities. Note that $|\kappa|$ and $||\kappa||$ are preserved when passing to a formalisation. 
 
 We denote by $\calP$ the set of partitions of integers (that is, partitions in the case where $I = \Z_{>0}$.) Among those, we denote by  $\mathcal{Q}$ the set of partitions $\mu = (m_i)_{i\geq 1}$ such that for some $k\geq 1$, we have $m_i=0$ for $i>k$ and $m_i\neq 0$ for $1\leq i\leq k$. Such a partition $\mu$ can be also seen as an \textit{ordered} partition $(m_1,\ldots,m_k)$ of the integer $|\mu|$. From this point of view, the elements of $\calP$ may be reinterpreted as ordered partitions allowing zeroes. For any integer~$m$, we denote by $\calP_m\subset \mathcal{P}$ (resp. $\mathcal{Q}_m\subset \mathcal{Q}$) the subsets corresponding to partitions~$\mu$ with $|\mu| = m$. For non-negative integers $m_1,\ldots,m_r$, we will use the notations $(m_1,\ldots,m_r)$ or $1^{m_1}\ldots r^{m_r}$ to denote the partition $(m_1,\ldots,m_r,0,0,\ldots)$. 

There is a natural retraction map $c:\calP\to \mathcal{Q}$ given by removing all intermediate zeroes. For example, $c(1,0,3) = (1,3)$. 

\begin{remark}\label{remark:piles} A helpful visualisation is to consider an element $\mu =(\mu_i)_{i\geq 1}\in \calP$ as a pile of blocks arranged into columns numbered from left to right, with the $i$-th column containing $m_i$ blocks. The elements of $\calQ$ are piles with no empty columns, $|\mu|$ is the total number of blocks, and $||\mu||$ is the total number of non-empty columns. The contraction $c(\mu)$ is given by sliding all of the non-empty columns as far to the left as possible.
\end{remark}
\paragraph{A combinatorial lemma}
For every integer $n\geq 1$ will consider the \textit{decomposition sets}
$$\mathcal{Q}^{n,*} = \{(\mu_1,\ldots,\mu_n)\in \calP^{n}|\ \mu_1 + \ldots + \mu_n \in \mathcal{Q}\}.$$
There is a contraction map $c_n:\mathcal{Q}^{n,*}\to \mathcal{Q}^n$ given by $(\mu_1,\ldots,\mu_n)\mapsto (c(\mu_1),\ldots,c(\mu_n)).$
\begin{example}\label{contractionn1} When $n=1$, we have $\calQ^{1,\star} = \calQ$ and $c_1 = \id$. 
\end{example}
We have the following combinatorial lemma, which, together with its proof, has been communicated to us by Sean Howe:
\begin{lemma}\label{howelemma} For all $(\nu_1,\ldots,\nu_n)\in \mathcal{Q}^n$, 
$$\sum_{(\mu_1,\ldots,\mu_n)\in c_n^{-1}(\nu_1,\ldots,\nu_n)}(-1)^{||\mu_1 + \ldots + \mu_n||}  = (-1)^{||\nu_1|| + \ldots + ||\nu_n||}.$$
\end{lemma}
\begin{proof}The case $n=1$ is clear by Example \ref{contractionn1} We will show the case $n=2$ directly, and then proceed by induction for larger $n$.

We now consider the case $n=2$. To give an element $(\mu_1, \mu_2)$ in $c_2^{-1}(\nu_1, \nu_2)$ is the same as to give:
\begin{enumerate}
\item A non-negative integer $j$ such that $j \leq \min(||\nu_1||, ||\nu_2||)$.
\item A subdivision of $\{1, \ldots, ||\nu_1|| + ||\nu_2|| - j\}$ into a subset of size $j$, a subset of size $||\nu_1|| - j$, and a subset of size $||\nu_2|| - j$.
\end{enumerate}
Here $j$ corresponds to the number of integers $k$ such that $\mu_{1}$ and $\mu_2$ are both non-zero in the $k$-th spot; given this information, to determine $\mu_1$ and $\mu_2$ from $\nu_1$ and $\nu_2$, it suffices to pick $j$ spots for both to be non-zero, $||\nu_1|| - j$ for only $\mu_1$ to be non-zero, and $||\nu_2|| - j$ for only $\mu_2$ to be non-zero. Furthermore, for the resulting $(\mu_1, \mu_2)$, we have 
$$ (-1)^{||\mu_1 + \mu_2||} = (-1)^{||\nu_1|| + ||\nu_2||-j} = (-1)^{||\nu_1|| + ||\nu_2||} (-1)^j .$$
Thus, the sum depends only on $||\nu_1||=:a, ||\nu_2||=:b.$ Therefore, we may forget about multiplicities and we may assume that $\nu_1=1^1 2^1 \ldots a^1$ and $\nu_2=1^1 2^1 \ldots b^1$. Then, using the visualisation in terms of piles of blocks as in remark~\ref{remark:piles}, we find that we are summing $-1$ to the length of the pile over all piles of blocks created by inserting $a$ red blocks into a row of $b$ blue blocks with each red block going either between columns or on top of a blue block (with at most one red block on top of any blue block). We must show this sum is equal to $(-1)^{a+b}$. 

To prove this is the case, it suffices to consider the variant where the red blocks are numbered $1$ through $a$, and to show that the corresponding sum is equal to $(-1)^{a+b} a!$. In this variant, we may first choose where to insert the first block, then the second block, etc.; we will keep track of the sign for the corresponding term in the sum by considering whether inserting the block keeps the length of the pile the same or changes it by one.

For every $i\in\{0,\ldots,a\}$, denote by $S_i$ the corresponding sum where only $i$ red blocks have been inserted, so that $S_0 = (-1)^b$ and $S_a$ is the sum we are looking for. Denoting by $\Pi_i$ the set of possible piles of $i$ (numbered) red blocks and $b$ blue blocks, and for every $p\in \Pi_i$ by $||p||$ the length (that is, the number of columns) of the pile $p$, we have, for all $i\in\{1,\ldots,a\}:$
$$S_i = \sum_{p\in \Pi_i}(-1)^{||p||} = \sum_{p\in \Pi_{i-1}}(-1)^{||p||}(r_i(p) - r'_i(p))$$
where $r_i(p)$ (resp. $r'_i(p)$) is the number of ways of adding an $i$-th block to $p\in \Pi_{i-1}$ which don't change (resp. change) the length of $p$. We claim that the difference $r_i(p) - r'_i(p)$ only depends on $i$, and not on the pile $p$. Indeed, note that each red block may either be inserted on top of a blue block, which does not change the length of the pile, or between two columns, which changes the length of the pile by one. If a red block is placed on top of a blue block, then the number of positions for the next red block to be placed without changing the length goes down by one, while the number of positions it can be placed that change the length by one remains the same. If a red block is placed between two columns, then the number of positions for the next red block to be placed which don't change the length remains the same, but the number of positions for the next red block to be placed that change the length goes up by one (since the spot where this block was placed now is replaced with two possible spots, one to the left and one to the right). In either case (and independently of the pile we are considering), the \emph{difference} between the number of ways to place a block without changing the length and the number of ways to place it that change the length goes down by one. For the first block this difference is $b - (b+1) = -1$, so when inserting the $i$-th block it is $-i$. In other words, for all $i\in \{1,\ldots,a\}$ and all $p\in \Pi_i$, we have
$$r_i(p) - r'_i(p) = -i.$$
In particular, for all $i\in \{1,\ldots,a\}$, we have $S_i = (-i)S_{i-1}$, and therefore
$$S_a = (-a)S_{a-1} = (-a)(-a+1)\ldots (-1)S_{0} = (-1)^{a+b}a!,$$ as desired. This completes the argument when $n=2$.  

We now assume $n \geq 3$, and that the identity is known for all positive integers $m < n$. Giving $(\mu_1,\ldots,\mu_n)\in c_n^{-1}(\nu_1,\ldots,\nu_n)$ is equivalent to giving the following data:
\begin{enumerate}\item $(\mu'_1,\ldots,\mu'_{n-1})\in c_{n-1}^{-1}(\nu_1,\ldots,\nu_{n-1})$
\item $(\lambda,\mu_n)\in c_2^{-1}(\mu'_1+\ldots+\mu'_{n-1}, \nu_n)$.
\end{enumerate}
Indeed, to construct these data from $(\mu_1,\ldots,\mu_n)$, first of all put $\lambda = \mu_1 + \ldots + \mu_{n-1}$. It is an element of $\calP$, but not in general an element of $\calQ$, that is, in terms of the above description with piles of blocks, it may have some empty columns, which are exactly the empty columns common to all the $\mu_i$, $i\in\{1,\ldots,n-1\}$. To construct $\mu'_i$ from $\mu_i$, simply do the same operation of sliding columns to the left that we did when defining the retraction $c$, but making only these common empty columns disappear. Note that we have $||\mu_1 + \ldots + \mu_{n-1}|| = ||\mu'_1 + \ldots + \mu'_{n-1}||$, and $c(\mu'_i) = \nu_i$ for all $i\in\{1,\ldots,n-1\}$. Conversely, to reconstruct $(\mu_1,\ldots,\mu_n)$ from $(\lambda,\mu_n)$ and $(\mu'_1,\ldots,\mu'_{n-1})$, we simply look at what columns are empty in~$\lambda$, and create empty columns at exactly these spots in $\mu'_1,\ldots,\mu'_{n-1}$ by sliding appropriate columns to the right, starting from the leftmost empty spot and working our way up to the right. 

Therefore, the sum we are interested in may be rewritten as
$$ \sum_{(\mu_1, \ldots, \mu_{n-1})\ \in\ c_{n-1}^{-1} (\nu_1, \ldots, \nu_{n-1})} \ \ \ \ \sum_{ (\lambda, \mu_n)\ \in\ c_2^{-1} (\mu_1 + \ldots + \mu_{n-1}, \nu_n) } (-1)^{||\lambda + \mu_n||}. $$
Applying the case $n=2$ to the interior sum we find that this is equal to
\begin{multline*} \sum_{(\mu_1, \ldots, \mu_{n-1})\ \in\ c_{n-1}^{-1} (\nu_1, \ldots, \nu_{n-1})} (-1)^{||\mu_1 + \ldots + \mu_{n-1}|| + ||\nu_n||} = \\
(-1)^{||\nu_n||} \sum_{(\mu_1, \ldots, \mu_{n-1})\ \in\ c_{n-1}^{-1} (\nu_1, \ldots, \nu_{n-1})} (-1)^{||\mu_1 + \ldots + \mu_{n-1}||}, 
\end{multline*}
and we conclude by applying the induction hypothesis to the remaining sum. 
\end{proof}

\subsection{Symmetric products and the inverse of Kapranov's zeta function} For any variety~$Y$ over a field $k$ and for any $\mu = (m_i)_{i\geq 1}\in\calP$ we define
$$\mathrm{Sym}^{\mu}Y = \prod_{i\geq 1}S^{m_i}Y.$$
There is a natural morphism $\mathrm{Sym}^{\mu} Y\to S^{|\mu|}Y$ induced by starting with the identity $\prod_{i\geq 1}Y^{m_i}\to Y^{|\mu|}$ and passing to the appropriate quotients on both sides. Note that the symmetric product $S^{\mu}Y$ is by definition the restriction of $\mathrm{Sym}^{\mu}Y$ to the complement of the diagonal in $S^{|\mu|}Y$. 
\begin{remark} In what follows, we will often drop the square brackets when writing classes of varieties in a Grothendieck ring. 
\end{remark}
The following result will be very important to us:
\begin{lemma}\label{Smminuslemma} For any variety $Y$, and any integer $m\geq 1$, we have the equality 
\begin{equation}\label{Smminusformula} S^{m}(-Y) = \sum_{\mu\in \calQ_m}(-1)^{||\mu||}[\mathrm{Sym}^{\mu}Y]\end{equation}
in $\kvar_{S^mY}$. 
\end{lemma}
\begin{proof} Note that by definition (see section \ref{sect.symprodgrouplaw}), the elements of the family $(S^r(-Y))_{r\geq 1}$ satisfy
$$\sum_{i=0}^m S^{r}(Y)S^{m-r}(-Y) = 0$$
in $\kvar_{S^m(Y)}$ for every $m\geq 1$. In particular, $S^{1}(-Y) = -Y$, so the property is true for $m=1$. We now proceed by induction: write
$$S^m(Y) = -\sum_{r=0}^{m-1} S^{r}(Y)S^{m-r}(-Y) = -\sum_{r=0}^{m-1} S^{r}(Y)\sum_{\mu\in \mathcal{Q}_{m-r}}(-1)^{||\mu||}\mathrm{Sym}^{\mu}Y$$
 by using the formula in the statement for all $r\in\{1,\ldots,m-1\}$. For any such $r$ and for any partition $\mu = (m_i)_{i\geq 1}\in \calQ_{m-r}$, we have
$$S^{r}(Y)\cdot\mathrm{Sym}^{\mu}(Y) = \mathrm{Sym}^{\mu'}(Y),$$
where $\mu'$ is the partition in $\mathcal{Q}_m$ obtained from $\mu$ by adding a new part with multiplicity~$r$. Conversely, any $\mu' = (m_i)_{i\geq 1}\in \calQ_m$ is obtained uniquely in this way by choosing $r$ to be the value of $m_i$ for the largest $i$ such that $m_i>0$ and by defining $\mu\in \mathcal{Q}_{m-r}$ by setting this $m_i$ to zero. Observing that in this case $||\mu'|| = ||\mu|| +1$, we get the result. 
\end{proof}

\begin{remark} Lemma \ref{Smminuslemma} gives, using a formula by Vakil and Wood, a proof of a special case of proposition \ref{compfinprodnoneff}, namely of the formula
$$\prod_{x\in X}(1 + t + t^2 + \ldots )\prod_{x\in X}(1-t) = 1,$$
which says that the series $\prod_{x\in X}(1-t)$ is the inverse of Kapranov's zeta function $Z_X(t)$ (see example \ref{eulerprodexample}, \ref{Kapranovexample}). 
 
 First of all, let us point out that the right-hand side of (\ref{Smminusformula}) for $Y=X$ appears naturally as the coefficient of degree $m$ of the series $Z_{X}^{-1}$, as the computation
\begin{eqnarray*} Z_X(t)^{-1} & = & \left(1 + \sum_{i\geq 1}[S^iX]t^i\right)^{-1}\\
                           & = & \sum_{m\geq 0} \left(-\sum_{i\geq 1}[S^iX]t^i\right)^{m}\\
                           & = & \sum_{\mu\in \mathcal{Q}}(-1)^{||\mu||}[\mathrm{Sym}^{\mu}X]t^{|\mu|}\end{eqnarray*}
                           from \cite{VW} (just before paragraph 3.6) shows. On the other hand, the coefficients in the expansion of the series 
                           $\prod_{x\in X}(1-t) $ are the symmetric products of the family $(X_i)_{i\geq 1}$ given by $X_1 = -X$ and $X_i = 0$ for all $i\geq 2$, which are computed using definition \ref{symproddef}. In particular, the only partition occurring in the coefficient of degree $m$ is $1^m$, and the corresponding symmetric product is the class of the restriction $S^m(-X)_{|S^m_*X}$ of~$S^m(-X)$ to~$S^m_*X$:
       $$   \prod_{x\in X}(1-t) = 1 + \sum_{m\geq 1}S^m(-X)_{|S^m_*X}t^m. $$ 
                      
   Now, proposition 3.7 in \cite{VW} gives the equality 
   $$\sum_{\mu\in \mathcal{Q}}(-1)^{||\mu||}[\mathrm{Sym}^{\mu}X]t^{|\mu|} =\sum_{\mu\in \mathcal{Q}}(-1)^{||\mu||}[S^{\mu}X]t^{|\mu|},$$
 which implies precisely that $S^m(-X)$ is in fact supported above the complement of the diagonal in $S^mX$, that is, above $S^m_*X$, so
   $$   \prod_{x\in X}(1-t) = 1 + \sum_{m\geq 1}S^m(-X)t^m = Z_X(t)^{-1}. $$
\end{remark}

\subsection{Overlaps between partitions}\label{subsect.overlaps}
Let $k,\ell\in I$. For every partition $\kappa = (k_i)_{i\in I}$ of $k$ and $\lambda = (\ell_i)_{i\in I}$ of $\ell$, we denote by 
                $$f(\kappa) = [a_i^{k_i}]_{i\in I},\ \ f(\lambda) = [b_i^{\ell_i}]_{i\in I}$$
                their formalisations. We may define their concatenation $f(\kappa) \cdot f(\lambda)$ as a finite multiset with elements in the union $\{a_i\}_{i\in I}\cup\{b_i\}_{i\in I}$, $a_i$ having multiplicity $k_i$ and $b_j$ having multiplicity~$\ell_j$. The sum  of this partition is the element
                $$\sum f(\kappa) \cdot f(\lambda) = \sum_{i\in I}k_ia_i+ \sum_{j\in I}\ell_jb_j$$ in the free abelian monoid $M_{a,b}$ on $a_i, b_j, $ for all $i,j\in I$.
                  
                  On the other hand, we may consider  partitions with values in the abelian semigroup $M_{a,b}\setminus\{0\}$. Let $S_{\kappa,\lambda}$ be the set of such partitions $\gamma$ containing only elements of the form $a_i$, $b_j$ and $a_i+b_j$, and such that 
                  $$\sum \gamma = \sum f(\kappa)\cdot f(\lambda).$$  
                  Explicitly, if $\gamma = [(a_i+b_j)^{n_{i,j}}]_{i,j\in I_0},$ where by convention $a_0 = b_0 = 0$, then $$\sum \gamma = \sum_{i,j\in I_0}n_{i,j}(a_i+b_j)= \sum_{i\in I}\left({\sum_{j\in I_0}n_{i,j}}\right)a_{i}+ \sum_{j\in I}\left({\sum_{i\in I_0}n_{i,j}}\right)b_j,$$
                   so that we must have, for all $i,j\in I:$
                  $$k_i = \sum_{j\in I_0}n_{i,j}\ \ \ \text{and}\ \ \ \ell_j = \sum_{i\in I_0}n_{i,j}.$$
                In particular, giving a collection of integers $(n_{i,j})_{(i,j)\in I_0^2\setminus\{(0,0)\}}$ determines completely two partitions $\kappa$ and $\lambda$, and an ``overlap partition'' $\gamma$. Moreover, there is a natural way of ``deformalising''~$\gamma$, by putting
                $d(\gamma) = [m^{\gamma_m}]$ to be the partition with values in $I$ given by $$\gamma_m = \sum_{i + j = m}n_{i,j}.$$ Concretely, $d(\gamma)$ is obtained from $\gamma$ by replacing $a_i+ b_j$ by $i+j$ for all $(i,j)\in I_0^2\setminus\{(0,0)\}$. 
               
%                Such a generalised partition $\gamma$ encodes the possible overlaps of an element of $S^{\alpha}X$ and of an element of $S^{\beta}X$, and therefore index a decomposition
%                $$S^{\alpha}X\times S^{\beta}X = \bigsqcup_{\gamma\in S_{\al,\be}} S^{\alpha}X\times_{\gamma} S^{\beta}X.$$
%                There is a natural summation mapping
%                $$S^{\alpha}X\times_{\gamma} S^{\beta}X \to S^{d(\gamma)}X.$$
\subsection{Overlaps between symmetric products}\label{subsect.overlapssymproducts}
 Here we are going to  give a notation that is going to be useful in the proof of proposition \ref{compfinprodnoneff}. Assume we are given two partitions $\kappa = (k_i)_{i\in I}$ and $\lambda= (\ell_j)_{j\in J}$, as well as a family $\gamma = (n_{i,j})_{i,j\in I_0}$ such that for all $i,j\in I$, we have
 $$\sum_{i\in I_0} n_{i,j} \leq \ell_j\ \ \ \text{and}\ \ \ \sum_{j\in I_0} n_{i,j} \leq k_i.$$
 Then for any variety $X$, we define $\left(\prod_{i\in I} S^{k_i}X\right)_*\times_{\gamma} \left(\prod_{j\in I}S^{\ell_j}X\right)_*$ to be the locally closed subset of $\left(\prod_{i\in I} S^{k_i}X\right)_*\times \left(\prod_{j\in I}S^{\ell_j}X\right)_*$ given by points with $S^{k_i}X$-component and $S^{\ell_j}X$-component overlapping exactly along an effective zero-cycle of degree $n_{i,j}$. More generally, for a variety $A$ over $\left(\prod_{i\in I} S^{k_i}X\right)_*$ and a variety $B$ over $\left(\prod_{j\in I}S^{\ell_j}X\right)_*$, we may define the variety $A\times_{\gamma} B$ to be the restriction of $A\times B$ to $\left(\prod_{i\in I} S^{k_i}X\right)_*\times_{\gamma} \left(\prod_{j\in I}S^{\ell_j}X\right)_*$. In terms of classes in Grothendieck rings, for any elements $A\in \kvar_{\left(\prod_{i\in I}S^{k_i}X\right)_*}$ and $B\in \kvar_{\left(\prod_{j\in I}S^{\ell_j}X\right)_*}$, we denote by $A\times_{\gamma}B$ the element of $\kvar_{\left(\prod_{i\in I} S^{k_i}X\right)_*\times_{\gamma} \left(\prod_{j\in I}S^{\ell_j}X\right)_*}$ obtained from $A\boxtimes B$ by the restriction morphism 
 $$\kvar_{\left(\prod_{i\in I} S^{k_i}X\right)_*\times \left(\prod_{j\in I}S^{\ell_j}X\right)_*}\to \kvar_{\left(\prod_{i\in I} S^{k_i}X\right)_*\times_{\gamma} \left(\prod_{j\in I}S^{\ell_j}X\right)_*}.$$
 In particular, this defines $S^{\kappa}\scr{A}\times_{\gamma}S^{\lambda}\scr{B}$, whenever $(A_i)_{i\in I}$ and $(B_i)_{i\in I}$ are families of elements of $\kvar_X$.

\begin{remark} Recall that we defined, in chapter \ref{grothrings}, (\ref{exteriorproduct}), a (relative) exterior product operation on Grothendieck rings which is compatible with fibre products for effective classes. In the next sections, we will sometimes use the notation $\times$ instead of $\boxtimes$ when we want to stress the fact that we are working with effective classes. 
\end{remark} 
 \subsection{Proof of proposition \ref{compfinprodnoneff}}
 \subsubsection{Expansion of the left-hand side}
 The left-hand side is given by
$$\prod_{x\in X} \left(1 + \sum_{i\in I}\left(\sum_{p + q = i} A_{p,x} B_{q,x}\right) t_i\right)$$
and therefore it is the Euler product associated to the family 
$$\left(\left(\sum_{p + q = i} A_{p}\boxtimes_X B_{q}\right)\right)_{i\in I},$$
which we will denote by $\scr{A}\ast_X\scr{B}$. Expanding this, we get
$$1 + \sum_{n\in I}\left(\sum_{\pi\ \text{partition of}\ n}S^{\pi}(\scr{A}\ast_X\scr{B})\right) t_n.$$

Let $n\in I$ and let $\pi = (n_i)_{i\in I}$ be a partition of $n$. The contribution $S^{\pi}(\scr{A}\ast_X\scr{B})$ to the coefficient of $t_n$ in the expansion of the series is the restriction of
$$\prod_{i\in I} S^{n_i}\left(\sum_{p + q = i} A_{p}\boxtimes_X B_{q}\right)\in \kvar_{\prod_{i\in I}S^{n_i}X}$$
to $S^{\pi}X.$ This can be written 
$$\prod_{i\in I}\ \ \ \sum_{\substack{(n_{p,q})_{p+q=i}\\
                                n_{i,0} + \ldots + n_{0,i} = n_i}}\ \prod_{p+q = i}S^{n_{p,q}}(A_p\boxtimes_X B_q)$$
                                \begin{equation}\label{LHSexpansion}= \sum_{\substack{(n_{p,q})_{(p,q)\in I_0^2\setminus\{0\}}\\
                                \forall i\in I,\ \sum_{p+q = i}n_{p,q} = n_i}}\prod_{i\in I}\prod_{p+q = i}S^{n_{p,q}}(A_p\boxtimes_X B_q).\end{equation}
     \subsubsection{Expansion of the right-hand side}   The right-hand side is given by
                                $$\left(1 + \sum_{n\in I} S^n(\scr{A})t_n\right)\left(1 + \sum_{n\in I} S^n(\scr{B})t_n\right) = 1 + \sum_{n\in I}\left(\sum_{k+\ell = n}S^k(\scr{A})\boxtimes S^{\ell}(\scr{B})\right)t_n $$
                                The proposition is implied by the following claim:
                                
                                \textbf{Claim:} For every partition $\pi$ of $n$, the restriction to $S^{\pi}X$ of the coefficient $$\sum_{k+\ell = n}S^k(\scr{A})\boxtimes S^{\ell}(\scr{B})\in \kvar_{S^nX}$$ of $t_n$ on the right-hand side is equal to $S^{\pi}(\scr{A}\ast_X\scr{B})$. 
 
We are in fact going to prove a more precise statement, which implies the above claim.     The sum $\sum_{k+\ell = n}S^k(\scr{A})\boxtimes S^{\ell}(\scr{B})$ is a sum of terms of the form $S^{\kappa}(\scr{A})\boxtimes S^{\lambda}(\scr{B})$ where $\kappa$ and~$\lambda$ are partitions such that $\sum \kappa + \sum \lambda = n$.  On the other hand, note that in the expansion (\ref{LHSexpansion}) of the left-hand side, there was a sum over collections of integers $(n_{p,q})_{(p,q)\in I_0^2\setminus\{0\}}$ such that 
              $$\sum_{p+q = i}n_{p,q} = n_i.$$    
              In particular, each term of this sum corresponds to some $\gamma$ as in section \ref{subsect.overlaps} (together with some partitions~$\kappa$ and~$\lambda$) such that $d(\gamma) = \pi$. Our aim is to prove, for every given $\gamma = [(a_p+b_q)^{n_{p,q}}]$ such that $d(\gamma) = \pi$, and for the associated $\kappa,\lambda$, the equality
              \begin{equation}\label{preciseclaimequation}\left(\prod_{i\in I}\prod_{p+q = i}S^{n_{p,q}}(A_p\boxtimes_XB_q)\right)_{*} =  S^{\kappa}(\scr{A})\times_{\gamma}S^{\lambda}(\scr{B})\end{equation}
          in $\kvar_{S^{\pi}X}$   where the notation $(\cdot)_*$ on the left-hand side, as always, stands for restriction to the complement of the diagonal. Then summing over all $\gamma$ such that $d(\gamma) = \pi$ proves our claim. 
          
          \subsubsection{Computation of $S^{\kappa}(\scr{A})\times_{\gamma}S^{\lambda}(\scr{B})$}
           We will now compute $S^{\kappa}(\scr{A})\times_{\gamma}S^{\lambda}(\scr{B})$ in more detail, using the notations from section \ref{subsect.overlaps}.  We have
          $$S^{\kappa}(\scr{A})\times_{\gamma}S^{\lambda}(\scr{B}) = \left(\prod_{i\in I} S^{k_i}(A_i)\right)_*\times_{\gamma}\left(\prod_{j\in I} S^{\ell_j}(B_j)\right)_*.$$
          Recall that the notation $\times_{\gamma}$ from section \ref{subsect.overlapssymproducts} indicates that we restrict to points such that for all $i,j\in I$, the $S^{k_i}X$-component and the $S^{\ell_j}X$-component overlap exactly in $n_{i,j}$ geometric points. According to (\ref{preciseclaimequation}), we want to prove that this is equal to 
          $$\left(\prod_{(i,j)\in I_0^2\setminus\{(0,0)\}} S^{n_{i,j}}(A_i\boxtimes_XB_j)\right)_*.$$
          Since $k_i = \sum_{j\in I}n_{i,j}$ and $\ell_{j} = \sum_{i\in I} n_{i,j}$, by induction it suffices to prove that for any effective $K\in \kvar_{\prod_{p\neq i}S^{k_p}X_p}$ and for any $L\in \kvar_{\prod_{q\neq j}S^{\ell_q}X_q}$, denoting by $\gamma^{i,j}$ the family of integers obtained from $\gamma$ by setting $n_{i,j}$ to zero:
$$\left(S^{k_i}(A_i)\times K\right)_*\times_{\gamma} \left(S^{\ell_j}(B_j)\boxtimes L\right)_* =$$
\begin{equation}\label{klgammaaim} \left(S^{n_{i,j}}(A_i\boxtimes_XB_j)\boxtimes \left(S^{k_i-n_{i,j}}(A_i)\times K\right)_*\times_{\gamma^{i,j}} \left(S^{\ell_j-n_{i,j}}(B_j)\boxtimes L\right)_*\right)_*.\end{equation}
\begin{remark} The left-hand side of (\ref{klgammaaim}) is an element of the Grothendieck ring above $\left(\prod_{i\in I} S^{k_i}X_i\right)_*\times_{\gamma}\left(\prod_{j\in I} S^{\ell_j}X_j\right)_*,$ whereas the right-hand side is an element of the Grothendieck ring above 
$$\left(S^{n_{i,j}}(X)\times \left(S^{k_i-n_{i,j}}X\times \prod_{p\neq i}S^{k_p}X_p\right)_*\times_{\gamma^{i,j}}\left(S^{\ell_j-n_{i,j}}X\times \prod_{q\neq j}S^{\ell_q}X_p\right)_*\right)_*.$$
These two varieties are easily seen to be isomorphic.

%, via the morphism, defined on the latter variety that sends, for every $x\in S^{n_{i,j}}(X)$, $x'\in S^{k_i-n_{i,j}}X$, $x''\in S^{\ell_j-n_{i,j}}X$, $y\in \prod_{p\neq i}S^{k_p}X_p$ and $z\in  \prod_{q\neq j}S^{\ell_q}X_p$, the point $(x,x',y,x'',z)$ to $(x+x',y,x
\end{remark}
Writing $L = L'-L''$ for some effective $L',L''$, we can moreover assume that $L$ is effective. 

For every $j\in I$, fix effective elements $C_j,C'_j\in \kvar_X$ such that $B_j = C_j- C'_j$. We get 
\begin{eqnarray*}S^{\ell_j}(B_j) &=& \sum_{0\leq m\leq \ell_j}S^{\ell_j - m}(C_j) \boxtimes S^{m}(-C'_j)\\
& = & \sum_{0\leq m\leq \ell_j}\sum_{\mu\in \calQ_m}(-1)^{||\mu||} S^{\ell_j - m}(C_j)\times \mathrm{Sym}^{\mu}(C'_j),\end{eqnarray*}
using lemma \ref{Smminuslemma}, so that we need to compute
           \begin{equation}\label{klgammamuequation}\sum_{0\leq m\leq \ell_j}\sum_{\mu\in \calQ_m} (-1)^{||\mu||}\left(S^{k_i}(A_i)\times K\right)_*\times_{\gamma} \left(S^{\ell_j - m}(C_j)\times \mathrm{Sym}^{\mu}(C'_j)\times L\right)_*.\end{equation}
           Now, for any fixed $m$ and $\mu = (\mu_p)_{p\geq 1}\in \calQ_m$, we are looking for points of $$\left(S^{k_i}(A_i)\times K\right)_*\times_{\gamma} \left(S^{\ell_j - m}(C_j)\times \mathrm{Sym}^{\mu}(C'_j)\times L\right)_*$$ that is, with image in $S^{k_i}X$ of the component corresponding to $S^{k_i}A_i$ overlapping in $n_{i,j}$ geometric points with the image in $S^{\ell_j}X$ of the component corresponding to the product 
           $$S^{\ell_j-m}(C_j)\times\prod_{p\geq 1}S^{\mu_p}(C_j').$$ 
           For any such point, there is a non-negative integer $\delta\leq \min\{m,n_{i,j}\}$ such that $n_{i,j}-\delta\leq \ell_j-m$ and such that the overlap with $S^{\ell_j-m}C_j$ is along $n_{i,j}-\delta$ points and the overlap with $\prod_{p\geq 1 }S^{\mu_p}(C_j')$ is along $\delta$ points. The overlap with $\prod_{p\geq 1}S^{\mu_p}(C'_j)$ must moreover be distributed among the different factors $S^{\mu_p}(C'_j)$, which leads, for every possible value of~$\delta$ to the introduction, for every $p\geq 1$, of an integer $0\leq \epsilon_p\leq \mu_p$, the overlap with the component $S^{\mu_p}(C'_j)$, which defines a partition $\epsilon  = (\epsilon_p)_{p\geq 1}$ such that $\sum_{p}\epsilon_p = \delta$, that is, $\epsilon \in \calP_{\delta}$. In other words, for any integer $\delta$ satisfying the condition
           $$\max(0,n_{i,j} -\ell_j + m)\leq\delta\leq \min(m,n_{i,j})$$
           and for any $\epsilon\in \calP_{\delta}$ such that $\epsilon \leq \mu$, the variety
           $$\left(S^{n_{i,j}-\delta}(A_i\times_XC_j)\times \mathrm{Sym}^{\epsilon}(A_i\times_XC'_j)\times P(\delta,\mu-\epsilon)\right)_*,$$
           where for any partition $\nu$, $P(\delta,\nu$ stands for 
     $$P(\delta,\nu) =  \left(S^{k_i-n_{i,j}}(A_i)\times K\right)_*\times_{\gamma^{i,j}}\left(S^{\ell_j-m-n_{i,j}+\delta}(C_j)\times\mathrm{Sym}^{\nu}(C'_j)\times L\right)_*,$$ 
     is isomorphic to the locally closed subset of 
        $$\left(S^{k_i}(A_i)\times K\right)_*\times_{\gamma} \left(S^{\ell_j - m}(C_j)\times \mathrm{Sym}^{\mu}(C'_j)\times L\right)_*
     $$
     of points with image in $S^{k_i}(X)$ of the $S^{k_i}(A_i)$-component overlapping along an effective zero-cycle of degree $n_{i,j}-\delta$ with the image in $S^{\ell_j-m}X$ of the $S^{\ell_j-m}(C_j)$-component, and for every $p\geq 1$, along an effective zero-cycle of degree $\epsilon_p$ with the image in $S^{\mu_p}X$ of the  $S^{\mu_p}(C'_j)$-component. 
Moreover,   $$\left(S^{k_i}(A_i)\times K\right)_*\times_{\gamma} \left(S^{\ell_j - m}(C_j)\times\mathrm{Sym}^{\mu}(C'_j)\times L\right)_*
     $$
     is the disjoint union of these locally closed subsets, so that in terms of classes in the Grothendieck ring above $\left(\prod_{i\in I} S^{k_i}X\right)_{*}\times_{\gamma}\left(\prod_{i\in I} S^{\ell_i}X\right)_{*}$, we have 
     $$\left(S^{k_i}(A_i)\times K\right)_*\times_{\gamma} \left(S^{\ell_j - m}(C_j)\times\mathrm{Sym}^{\mu}(C'_j)\times L\right)_* = 
     $$
     $$\sum_{\substack{\max(0,n_{i,j} -\ell_j + m)\leq\delta\leq \min(m,n_{i,j})\\ \epsilon = (\epsilon_p)_{p\geq 1}\in \calP_{\delta},\ \epsilon \leq \mu}}\left(S^{n_{i,j}-\delta}(A_i\times_XC_j)\times\mathrm{Sym}^{\epsilon}(A_i\times_XC'_j)\times P(\delta,\mu-\epsilon)\right)_*.$$
 
%     In what follows, to abbreviate calculations, we are also going to use the notations
%     $$AC(\delta):= \left(S^{n_{i,j}-\delta}(A_i)\times_{n_{i,j}-\delta}S^{n_{i,j}-\delta}(C_j)\right)$$
%     and
%     $$AC'(\eta):= \left(\prod_{p\geq 1}S^{\epsilon_p}(A_i)\times_{\epsilon_p}S^{\epsilon_p}(C'_j)\right).$$
    In what follows, though we do not specify it to make indices of sums lighter, sums are taken for $\delta$ satisfying the condition $\max(0,n_{i,j} -\ell_j + m)\leq\delta\leq \min(m,n_{i,j}).$  
    
     Note that the factor  
      $\mathrm{Sym}^{\epsilon}(A_i\times_XC'_j)$ occurring in the term corresponding to $\epsilon$ depends only on $c(\epsilon)$ (recall the definition of $c$ in section \ref{subsect.combinatorics}), and not on $\epsilon$ itself. Therefore, we may write
      $$\left(S^{k_i}(A_i)\times K\right)_*\times_{\gamma} \left(S^{\ell_j - m}(C_j)\times \mathrm{Sym}^{\mu}(C'_j)\times L\right)_* = $$
     $$\sum_{\substack{\delta\\ \eta\in \calQ_{\delta}}}\left(S^{n_{i,j}-\delta}(A_i\times_XC_j)\times\mathrm{Sym}^{\eta}(A_i\times_XC'_j)\times  \sum_{\substack{\epsilon\leq \mu\\
          c(\epsilon) = \eta}}P(\delta,\mu-\epsilon)\right)_*
     $$ 
       Taking the sum over $\mu\in \calQ_m$ for fixed $m$ as in (\ref{klgammamuequation}), we get
   $$\sum_{\mu\in \calQ_m}(-1)^{||\mu||}\left(S^{k_i}(A_i)\times K\right)_*\times_{\gamma} \left(S^{\ell_j - m}(C_j)\times\mathrm{Sym}^{\mu}(C'_j)\times L\right)_* =$$
   \begin{equation}\label{equation_eta}\sum_{\substack{\delta\\ \eta\in \calQ_{\delta}}}\left(S^{n_{i,j}-\delta}(A_i\times_XC_j)\times\mathrm{Sym}^{\eta}(A_i\times_XC'_j) \boxtimes \sum_{\mu\in \calQ_m}(-1)^{||\mu||}\sum_{\substack{\epsilon\leq \mu\\
          c(\epsilon) = \eta}}P(\delta,\mu-\epsilon)\right)_*\end{equation}     
          For fixed $\eta$, note that $P(\delta,\mu-\epsilon)$ depends only on $\delta$ and $c(\mu-\epsilon)$, so that 
      \begin{eqnarray*} \sum_{\mu\in \calQ_m} (-1)^{||\mu||} \sum_{\substack{\epsilon\leq \mu\\
          c(\epsilon) = \eta}}P(\delta,\mu-\epsilon)&=& \sum_{\nu\in \calQ_{m-\delta}}P(\delta,\nu)\sum_{(\epsilon,\xi)\in c_2^{-1}(\eta,\nu)}(-1)^{||\epsilon +\xi||}\\
          & = &(-1)^{||\eta||}\sum_{\nu\in \calQ_{m-\delta}}(-1)^{||\nu||}P(\delta,\nu)\\ \end{eqnarray*}
          $$=(-1)^{||\eta||}\left(S^{k_i-n_{i,j}}(A_i)\times K\right)_*\times_{\gamma^{i,j}}\left(S^{\ell_j-m-n_{i,j}+\delta}(C_j)\boxtimes S^{m-\delta}(-C'_j)\boxtimes L\right)_*.$$
         where the second equality comes from lemma \ref{howelemma} and the third one from lemma \ref{Smminuslemma} and the definition of $P(\delta,\nu)$, by linearity. In particular, denoting, for every integer $r$ such that $0\leq r \leq \ell_j - n_{i,j}$, 
         $$P'(r):= \left(S^{k_i-n_{i,j}}(A_i)\times K\right)_*\times_{\gamma^{i,j}}\left(S^{\ell_j -n_{i,j}-r}(C_j)\boxtimes S^{r}(-C'_j)\boxtimes L\right)_*$$
         and plugging this into equation (\ref{equation_eta}), we get
$$\sum_{\mu\in \calQ_m}(-1)^{||\mu||}\left(S^{k_i}(A_i)\times K\right)_*\times_{\gamma} \left(S^{\ell_j - m}(C_j)\times\mathrm{Sym}^{\mu}(C'_j)\times L\right)_* $$
\begin{eqnarray*}&=&\sum_{\substack{\delta\\ \eta \in \calQ_{\delta}}}\left(S^{n_{i,j}-\delta}(A_i\times_XC_j)\times\mathrm{Sym}^{\eta}(A_i\times_XC'_j)\boxtimes (-1)^{||\eta||}P'(m-\delta)\right)_*\\
&=&\sum_{\delta}\left(S^{n_{i,j}-\delta}(A_i\times_XC_j)\boxtimes P'(m-\delta)\boxtimes\sum_{\eta\in \calQ_{\delta}}(-1)^{||\eta||}\mathrm{Sym}^{\eta}(A_i\times_X C'_j)\right)_*\\
& =&\sum_{\delta}\left(S^{n_{i,j}-\delta}(A_i\times_XC_j)\boxtimes S^{\delta}(-A_i\times_XC'_j)\boxtimes P'(m-\delta)\right)_*,\end{eqnarray*}
where the last equality comes again from lemma \ref{Smminuslemma}.         
   Summing this over $m$, we get that~(\ref{klgammamuequation}) is equal to 
 
$$\sum_{0\leq m\leq \ell_j}\ \ \   \sum_{\max(0,n_{i,j} -\ell_j + m)\leq\delta\leq \min(m,n_{i,j})}\left(S^{n_{i,j}-\delta}(A_i\times_XC_j)\boxtimes S^{\delta}(-A_i\times_XC'_j)\boxtimes P'(m-\delta)\right)_*$$
  $$ = \sum_{\delta = 0}^{n_{i,j}}\left(S^{n_{i,j}-\delta}(A_i\times_XC_j)\boxtimes S^{\delta}(-A_i\times_XC'_j)\boxtimes \sum_{\delta\leq m \leq \ell_j-n_{i,j} + \delta} P'(m-\delta)\right)_*.$$

On the other hand, observe that 
$$\sum_{\delta\leq m \leq \ell_j-n_{i,j} + \delta} P'(m-\delta) = $$
\begin{eqnarray*}&=& \left(S^{k_i-n_{i,j}}(A_i)\times K\right)_*\times_{\gamma^{i,j}}\left(\sum_{0\leq r\leq \ell_j-n_{i,j}}S^{\ell_j-n_{i,j}-r}(C_j)\boxtimes S^{r}(-C'_j)\boxtimes  L\right)_*\\
 &=&  \left(S^{k_i-n_{i,j}}(A_i)\times K\right)_*\times_{\gamma^{i,j}}\left(S^{\ell_j-n_{i,j}}(B_j)\boxtimes L\right)_*,\end{eqnarray*}
 because $B_j = C_j - C'_j$. We end up with 
   $$\sum_{0\leq m\leq \ell_j}\sum_{\mu\in \calQ_m} (-1)^{||\mu||}\left(S^{k_i}(A_i)\times K\right)_*\times_{\gamma} \left(S^{\ell_j - m}(C_j)\times\mathrm{Sym}^{\mu}(C'_j)\times L\right)_* $$
   \begin{eqnarray*}& = &  \left(\left(S^{k_i-n_{i,j}}(A_i)\times K\right)_*\times_{\gamma^{i,j}}\left(S^{\ell_j-n_{i,j}}(B_j)\boxtimes L\right)_*\boxtimes \sum_{\delta=0}^{n_{i,j}}  S^{n_{i,j}-\delta}(A_i\times_XC_j)\boxtimes S^{\delta}(-A_i\times_X C_j')\right)_*\\
   & = & \left(\left(S^{k_i-n_{i,j}}(A_i)\times K\right)_*\times_{\gamma^{i,j}}\left(S^{\ell_j-n_{i,j}}(B_j)\boxtimes L\right)_* \boxtimes S^{n_{i,j}}(A_i\boxtimes_X B_j)\right)_*
   \end{eqnarray*}    
      so that (\ref{klgammaaim}) is proved.

\section{Allowing other constant terms}\label{coefficients}\index{Euler product!other constant terms}
Until now, for simplicity we have only worked with series having constant terms equal to 1. Though for obvious reasons of convergence we cannot abandon this hypothesis completely, it is still possible to make sense of Euler products where a finite number of factors have arbitrary constant terms, by generalising our symmetric products a bit further. Let us motivate the construction first. In the simplest setting, we want to make sense of Euler products
$$\prod_{x\in X}(X_{0,x} + X_{1,x} t + X_{2,x}t^2 + \ldots )$$
where $X$ is a variety over an algebraically closed field $k$ and $(X_i)_{i\geq 0}$ is a family of varieties over $X$, such that $X_{0,x}=1$ for almost all closed points $x\in X$. In other words, our product looks like
$$\prod_{x\in U}(1 + X_{1,x} t + X_{2,x}t^2 + \ldots )\prod_{x\in F}(X_{0,x} + X_{1,x} t + X_{2,x}t^2 + \ldots )$$
for a finite set of closed points $F\subset X$ with open complement $U$.  When one expands this product naively, one sees that the contribution for a zero-cycle $D\in S^nX(k)$ depends on the intersection of the support $|D|$ of $D$ with the set $F$. Let us assume for simplicity that $F = \{x_0\}$ is a singleton. Then the expansion of this product may formally be written 
$$\sum_{n\geq 0}\left(\sum_{\substack{D = \sum n_x x\in S^nX(k)\\
                                      x_0\in |D|}} \prod_{x\in X}X_{n_x,x} + \sum_{D = \sum n_x x\in S^nU(k)                     }X_{0,x_0}\prod_{x\in U} X_{n_x,x}\right)t^n$$
                                   In other words, the fibre $X_{0,x_0}$ appears whenever the zero-cycle with respect to which we expand does not contain $x_0$ in its support, forcing us to choose the term of degree zero in the factor corresponding to $x_0$.   Thus, denoting by $\scr{X}$ the family $(X_i)_{i\geq 1}$, the coefficient of degree $n$ in this power series should be the class of the constructible set given by the product $X_{0,x}\times  S^{n}\scr{X}_{|S^nU}$ above the open subset $S^{n}U$, and by $S^{n}\scr{X}_{|S^nX\setminus S^nU}$ above its complement. 
                                   
                                   This construction generalises for general $F$: we cut $S^{n}X$ up into locally closed subsets $(S^nX)^{E}$ for all subsets $E\subset F$ such that $(S^nX)^{E}$ corresponds to zero-cycles $D$ such that $|D|\cap F = E$, and modify the symmetric product $S^{n}\scr{X}$ above each $(S^nX)^E$ accordingly, multiplying it by the product of the fibres $X_{0,x}$ for all $x$ in $F\setminus E$. 
\subsection{Another generalisation of symmetric products}\label{sect.addX0}
Throughout this section, let $A\in\{\kvar,\M,\evar,\expp\}.$
Let $I_0$ be an additive monoid, and $I = I_0\setminus\{0\}$. Let $X$ be a variety over a perfect field $k$ and $V$ an open set of $X$ such that its complement~$F$ is a finite set of closed points. Let $\scr{X} = (X_i)_{i\in I_0}$ be a family of varieties over $X$ (resp. of classes in the ring $A_X$), such that $X_0\times_XV\simeq V$ (resp. the image of $X_0$ in $A_V$ is the class  $1 = [V\xrightarrow{\id} V]$). Thus, as a motivic function, $X_0$ is constant equal to 1 on the set $V$. In this section we are going to define a slight modification of the notion of symmetric product which takes into account $X_0$, or more precisely, the finite number of fibres $X_{0,v}$ for $v\in F$. This will allow us to consider a finite number of constant terms other than 1 in our Euler products.

Recall that for any $\pi = (n_i)_{i\in I}\in\N^{(I)}$, $S^{\pi}X$ parametrises collections $D= (D_i)_{i\in I}$ of effective zero-cycles on $X$, with disjoint supports, such that $\deg D_i = n_i$. We denote by $|D| = \cup_{i\in I}|D_i|$ the union of the supports of the $D_i$. For any subset $E\subseteq F$ denote by $(S^{\pi}X)^E$ the constructible subset of $S^{\pi}X$ parametrising  families $D$ of effective zero-cycles  such that $|D|\cap F = E$. Also, denote by $\widetilde{\scr{X}}$ the family $(X_i)_{i\in I}$ (that is, $\scr{X}$ with $X_0$ removed), and by $\left(S^{\pi}\widetilde{\scr{X}}\right)^E$ the restriction of $S^{\pi}\widetilde{\scr{X}}$ to $(S^{\pi}X)^E$ (resp. the image of $S^{\pi}\widetilde{\scr{X}}$ in $A_{(S^{\pi}X)^E})$. 

\begin{definition}\label{addX0} \begin{enumerate}\item The symmetric product $S^{\pi}\scr{X}$ of the family of varieties $\scr{X} = (X_i)_{i\in I_0}$ is defined as the constructible set over $S^{\pi}X$ with restriction to $(S^{\pi}X)^E$ given by
$$S^{\pi}\scr{X}_{|(S^{\pi}X)^E} = \left(S^{\pi}\widetilde{\scr{X}}\right)^E\times \prod_{v\in F\backslash E}X_{0,v}.$$
\item Denote by $i_E$ the immersion $(S^{\pi}X)^E\to S^{\pi}X.$ The symmetric product $S^{\pi}\scr{X}$ of the family of classes $\scr{X} = (X_i)_{i\in I_0}$ in $A_X$ is defined as the element of $A_{S^{\pi}X}$ given by
$$S^{\pi}\scr{X} = \sum_{E\subseteq F} (i_E)_!\left(\left(S^{\pi}\widetilde{\scr{X}}\right)^E\right) \prod_{v\in F\backslash E}X_{0,v}.$$
\end{enumerate}
\end{definition} 

Note that since $X_{0,v}$ is trivial (that is, a point) for every $v\in V$, it makes sense to write 
$$\prod_{v\in F\backslash E}X_{0,v} = \prod_{v\not\in E}X_{0,v}.$$
\begin{notation} We denote by $(S^{\pi}\scr{X})^E$ the pullback of $S^{\pi}\scr{X}$ along $i_E$, given by $$\left(S^{\pi}\widetilde{\scr{X}}\right)^E \prod_{v\not\in E}X_{0,v}.$$
\end{notation}

By definition, the partition $\left\{(S^{\pi}X)^E,\ E\subseteq F\right\}$ of $S^{\pi}X$ into constructible subsets is such that for each $E\subseteq F$, the restriction $S^{\pi}\scr{X}\times_{S^{\pi}X}(S^{\pi}X)^E$ is a trivial fibration above $S^{\pi}\widetilde{\scr{X}}\times_{S^{\pi}X}(S^{\pi}X)^E$ with fibre $\prod_{v\not\in E}X_{0,v}.$ 
\begin{remark} It is clear from the definition of symmetric products of families of classes in $A_X$ that both parts of the definition are compatible, that is, the class in $A_{S^{\pi}X}$ of the constructible set $S^{\pi}\scr{X}$ from part 1 will be the element constructed in part 2 with the classes of the varieties $X_i$ in~$A_X$.
\end{remark}
\begin{remark} Definition \ref{addX0} does not depend on the choice of the set $F$. Choosing a bigger set~$F'$ only amounts to refining the partition $\left\{(S^{\pi}X)^E,\ E\subseteq F'\right\}$. Indeed, assume $F' = F\cup\{v_0\}$. Then for all $E\subset F$,
$$\left\{D,\ |D|\cap F = E\right\} = \left\{D,\ |D|\cap F' = E\right\}\cup \left\{D,\ |D|\cap F' = E\cup \{v_0\}\right\},$$
and $X_{0,v_0} = 1$, so that $\prod_{v\in F'\backslash E}X_{0,v} = \prod_{v\in F\backslash E}X_{0,v}.$
\end{remark}

\begin{remark} Assume that $F = \varnothing$, that is, $X_0=X$ (resp. $X_0 = 1\in A_X$). Then we get $S^{\pi}\scr{X} = S^{\pi}\widetilde{\scr{X}}$, so our definition is an extension of the previous definition of symmetric products. 
\end{remark}

\begin{remark} For $E = \varnothing$, $(S^{\pi}X)^E$ is an open subset of $S^{\pi}X$. Thus, the constructible set $S^{\pi}\scr{X}$ is birationally equivalent to $S^{\pi}\widetilde{\scr{X}}\times \prod_{v\in F}X_{0,v}$. 
\end{remark}

\begin{example} Take $\pi=0\in\N^{(I)}$. Then the only non-empty piece of $S^{\pi}\scr{X}$ is the one corresponding to $E = \varnothing$, and therefore 
$$S^{0}\scr{X} = \prod_{v\in X}X_{0,v}.$$ 
\end{example}

%\begin{example} Take $X = \spec k=F$. Then for $E=\emptyset$ we get  $S^{\pi}\scr{X}  = (S^{\pi}\widetilde{\scr{X}})^{F}
%\end{example}
The fibre of $S^{\pi}\widetilde{\scr{X}}$ above a family of effective zero-cycles $D = (D_i)_{i\in I}\in S^{\pi}X$ is given by $$\prod_{i\in I}\ \prod_{v\in |D_i|}X_{i,v}.$$ By definition, replacing $\widetilde{\scr{X}}$ by $\scr{X}$ consists in replacing this fibre by its product with $\prod_{v\in F\backslash |D|}X_{0,v} = \prod_{v\not\in |D|} X_{0,v}$. Thus, the fibre of $S^{\pi}\scr{X}$ above $D$ can be written $$\left(\prod_{v\not\in |D|}X_{0,v}\right)\left(\prod_{i\in I}\ \prod_{v\in |D_i|}X_{i,v}\right),$$ which is indeed a finite product since $\cup_{i}|D_i|$ is a finite set, and and only a finite number of fibres $X_{0,v}$ are non-trivial.

\subsection{Application to Euler products}

 %This will follow simply from the fact that an Euler product with a finite number of factors is given by taking the product of these factors as power series. 

\begin{definition}\label{constterm}  Let $X$ be a variety over a perfect field $k$ and $V$ an open subset of $X$ such that its complement~$F$ is a finite set of closed points. Let $\scr{X} = (X_i)_{i\in I_0}$ be a family of classes in $A_X$ such that $X_0$ maps to 1 in $A_V$. We define the zeta-function
$$Z_{\scr{X}}(\t):= \sum_{\pi\in\N^{(I)}}[S^{\pi}\scr{X}]\t^{\pi}\in A_{k}[[\t]],$$
where $S^{\pi}\scr{X}$ is understood to be the generalised symmetric product from the previous paragraph, as well as the Euler product notation for $Z_{\scr{X}}$:
$$\prod_{v\in X}\left(\sum_{i\in I_0}X_{i,v}t_i\right) :=Z_{\scr{X}}(\t).$$ 
\end{definition}
\index{Euler product!other constant terms}
\begin{lemma} Let $X,V,\scr{X}$ be like in Definition~\ref{constterm}. Let $Y$ be a closed subscheme of $X$ and $U$ its open complement. Define $\scr{U} = (U_i)_{i\in I_0}$ and $\scr{Y} = (Y_i)_{i\in I_0}$ to be the families of elements of $A_U$ (resp. $A_Y$) obtained by restriction from~$\scr{X}$. For every $\pi\in\N^{(I)}$ we have the equality
$$S^{\pi}\scr{X} = \sum_{\pi'\leq \pi}S^{\pi'}\scr{U}\boxtimes S^{\pi-\pi'}\scr{Y}$$
in $A_{S^{\pi}X}$, where each term on the right-hand side is considered as an element of $A_{S^{\pi}X}$ via the immersion $S^{\pi'}U\times S^{\pi-\pi'}Y \to S^{\pi}X$. In particular, we have the equality $$Z_{\scr{X}}(\t) = Z_{\scr{U}}(\t)Z_{\scr{Y}}(\t)$$  in $A_k[[\t]]$.
\end{lemma}
\begin{proof}  Let $E$ be a subset of $F$. Write $E_U$ (resp. $E_Y, F_U, F_Y$) for the intersection $E\cap U$ (resp. $E\cap Y, F\cap U, F\cap Y$). Define the families $\scr{U} = (U_i)_{i\in I_0}$, $\widetilde{\scr{U}} = (U_i)_{i\in I}$, $\scr{Y} = (Y_i)_{i\in I_0}$, $\widetilde{\scr{Y}} = (Y_i)_{i\in I}$. Applying corollary \ref{generalcut2} and pulling back along $i_E:(S^{\pi}X)^E\to S^{\pi}X$, we get 
 $$\left(S^{\pi}\widetilde{\scr{X}}\right)^{E} = \sum_{\pi'\leq \pi} \left(S^{\pi'}\widetilde{\scr{U}}\right)^{E_U}\boxtimes \left(S^{\pi-\pi'}\widetilde{\scr{Y}}\right)^{E_Y}$$
 in $A_{(S^{\pi}X)^E}$. Therefore, writing
$$\prod_{v\in F\backslash E} X_{0,v}= \prod_{v\in F_{U}\backslash E_{U}}U_{0,v}\prod_{v\in F_Y\backslash E_Y}Y_{0,v},$$
we get that 
\begin{eqnarray*}(S^{\pi}\scr{X})^{E}&=&\sum_{\pi'\leq \pi}\left((S^{\pi'}\widetilde{\scr{U}})^{E_U}\prod_{v\in F_{U}\backslash E_{U}}U_{0,v}\right)\boxtimes \left((S^{\pi-\pi'}\widetilde{\scr{Y}})^{E_Y}\prod_{v\in F_Y\backslash E_Y}Y_{0,v}\right) \\
&=& \sum_{\pi'\leq \pi}\left(S^{\pi'}\scr{U}\right)^{E_U}\times \left(S^{\pi-\pi'}\scr{Y}\right)^{E_Y}\end{eqnarray*}
in $A_{S^{\pi}X}$. On the other hand, denoting by $i_{E_U}$ (resp. $i_{E_Y}$) the immersion $(S^{\pi'}U)^{E_{U}}\to S^{\pi'}U$ (resp. $(S^{\pi-\pi'}Y)^{E_{Y}}\to S^{\pi-\pi'}Y$:
\begin{eqnarray*}S^{\pi'}\scr{U}\boxtimes S^{\pi-\pi'}\scr{Y} &=& \left(\sum_{E_U \subset F_U}(i_{E_U})_!\left((S^{\pi'}\scr{U})^{E_U}\right)\right)\boxtimes \left(\sum_{E_Y \subset F_Y}(i_{E_Y})_!(S^{\pi-\pi'}\scr{Y})^{E_Y}\right)\\ &=& \sum_{E\subset F}\left(i_{E_U}\boxtimes i_{E_Y}\right)_!\left(\left(S^{\pi'}\scr{U}\right)^{E_U}\boxtimes \left(S^{\pi-\pi'}\scr{Y}\right)^{E_Y}\right)\end{eqnarray*}
in $A_{S^{\pi'}U\times S^{\pi-\pi'}Y}$. We may conclude by commutativity of the diagram
$$\xymatrix{(S^{\pi'}U)^{E_U}\times (S^{\pi-\pi'}Y)^{E_Y}\ \ \ \ar[d] \ar[r]^-{i_{E_U}\boxtimes i_{E_Y}} &\ \ \  S^{\pi'}U\times S^{\pi-\pi'}Y\ar[d]\\
(S^{\pi}X)^E \ar[r]^{i_E}&S^{\pi}X 
}$$
where the vertical arrow on the right is the immersion from the statement of the lemma, and the vertical arrow on the left is the immersion it induces above $(S^{\pi}X)^{E}$. 
\end{proof}
 This lemma immediately implies that the associativity property from Section \ref{eulerprod} extends to this Euler product with constant terms.

\chapter{Mixed Hodge modules and convergence of Euler products}\label{hodgemodules}

The aim of this chapter is to introduce a topology on the Grothendieck ring of varieties which is suitable for getting the expected convergence results in chapter \ref{motheightzeta}.
This topology will be defined using the theory of mixed Hodge modules of Morihiko Saito. For any %separated
 complex variety $X$ there is a morphism
$$\chi_X^{\hdg}: \M_X^{\hat{\mu}} \to K_0(\mhm_{X}^{\mon}),$$
called the Hodge realisation, between the localised Grothendieck group of varieties over $X$ with $\hat{\mu}$-action, and the Grothendieck group associated to the category $\mhm_{X}^{\mon}$ of mixed Hodge modules on $X$ with monodromy action. This morphism becomes a ring morphism when $\M_X^{\hat{\mu}}$ is endowed with the convolution product $\ast$, and $K_0(\mhm_{X}^{\mon})$ with its Hodge-theoretic version. The notion of weight of a Hodge module defines an increasing sequence of subgroups $(W_{\leq n}K_0(\mhm_{X}^{\mon}))_{n\in\Z}$ of $K_0(\mhm_X^{\mon})$ which gives a filtration on the ring $K_0(\mhm_{X}^{\mon})$. We define the weight function $w_X:K_0(\mhm_X^{\mon})\to \Z$, by
$$w_X(\a) = \inf\{n\in\Z,\ \a\in W_{\leq n}K_0(\mhm_X^{\mon})\}.$$
We show it behaves well with respect to some natural operations in the derived category of mixed Hodge modules, like pushforwards, pullbacks, exterior products, but also with respect to a notion of \textit{symmetric product} of mixed Hodge modules due to Maxim, Saito and Schürmann: for any element~$M$ of the bounded derived category $D^b(\mhm^{\mon}_X)$, denoting by $S^nM$ its $n$-th symmetric product for any integer $n\geq 1$, which is naturally an element $D^{b}(\mhm_{S^nX}^{\mon})$, we show that
\begin{equation}\label{sympowerinequality}w_{S^nX}(S^nM)\leq n\, w_X(M).\end{equation}
Composing $w_X$ with $\chi_{X}^{\hdg}$ gives also a weight function on $\M_X^{\hat{\mu}}$. Finally, composing the latter with the total vanishing cycles measure
$$\Phi_X:\expp_X\to\M_X^{\hat{\mu}}$$
gives us a weight function on the ring  $\expp_X$. Using this, we may define a notion of convergence for power series with coefficients in this ring. Our aim is to formulate a result which, given a sufficiently convergent series $1 + \sum_{i\geq 1}X_i t^i \in \expp_{X}[[t]]$, predicts the convergence of its Euler product
$$\prod_{x\in X} \left(1 + X_{1,x} t + X_{2,x} t^2 + \ldots \right).$$
Thus we need to be able to bound the weights of the symmetric products of the family $\scr{X} = (X_i)_{i\geq 1}$ in terms of the weights of the elements of the family. In other words, we need to relate the weights of the elements $$\chi_{S^nX}^{\hdg}\circ \Phi_{S^nX} (S^n(\scr{X})) \in K_0(\mhm_{S^nX}^{\mon})$$ for all $i\geq 1$ to the weights of the elements
$$\chi_X^{\hdg}\circ \Phi_X(X_i)$$
for al $i\geq 1$. Our strategy consists in introducing a version of the total vanishing cycles measure for mixed Hodge modules, that is, a morphism
$$\Phi_X^{\hdg}: K_0(\mhm_{\A^1_X})\to K_0(\mhm_{X}^{\mon})$$ fitting into a commutative diagram
$$\xymatrix{\M_{\A^1_X}\ar[r]^{\Phi'_X}\ar[d]^{\chi^{\hdg}_{\A^1_X}} & \M^{\hat{\mu}}_{X}\ar[d]^{\chi_{X}^{\hdg}}\\
K_0(\mhm_{\A^1_X})\ar[r]^-{\Phi^{\hdg}_X} & K_0(\mhm_{X}^{\mon})
}$$
(here we replaced $\Phi_X$ with its composition $\Phi'_X$ with the quotient morphism $\M_{\A^1_X}\to \expp_X$). More precisely, we define a functor 
$$\vanhdg_X:\mhm_{\A^1_X}\to \mhm_{X}^{\mon}$$
and take $\Phi_X^{\hdg}$ to be the morphism it induces on Grothendieck groups. Again, it is important to show that this morphism is a ring morphism when the Grothendieck groups are endowed with appropriate products. Having done this, in the above problem, the composition $\chi\circ \Phi$ may be replaced with $\Phi^{\hdg}\circ \chi$, which is easier to deal with because most of the work can be done in categories of mixed Hodge modules, which are better behaved than categories of varieties (for example, they are abelian). 

Via the Hodge realisation, symmetric powers in the Grothendieck ring of varieties correspond to Maxim, Saito and Schürmann's symmetric powers of mixed Hodge modules, that is, for any complex variety $Z$ and any $\a\in\M_{Z}$, we have
$$S^n(\chi^{\hdg}_{Z}(\a)) = \chi^{\hdg}_{S^nZ}(S^n\a).$$
Using a Thom-Sebastiani theorem for Hodge modules due to Saito, we show that the functor $\vanhdg_X$ is, in a certain sense, compatible with these symmetric powers. Combining these two compatibilities with estimates of the form (\ref{sympowerinequality}) enables us to conclude.

We now describe the structure of this chapter. The first three sections are purely Hodge-theoretic. We will start by recalling some definitions and useful facts about mixed Hodge modules, with or without monodromy, in section \ref{sect.mhmbasics}. In section \ref{sect.totvanhdg} we define the \textit{total vanishing cycles functor} between categories of mixed Hodge modules and prove the Thom-Sebastiani property. We recall the definition of symmetric products of Hodge modules in section \ref{sect.symprodvancycles}, explain that it extends without trouble to Hodge modules with monodromy and describe the behaviour of the total vanishing cycles functor with respect to these symmetric products.

In section \ref{sect.motvancomp}, we go on to relate all this to the framework of chapter \ref{grothrings}, introducing the Hodge realisation and showing compatibilities of the motivic and Hodge-theoretic objects.
Then we finally arrive to the definition of the weight filtrations, first on Grothendieck rings of Hodge modules in section \ref{sect.weightmhm}, then on Grothendieck rings of varieties in section \ref{sect.weightfiltrationkvar}. In section \ref{sect.powerseriesconv}, we conclude the chapter by stating and proving a convergence lemma for motivic Euler products of power series with coefficients in a Grothendieck ring of varieties with exponentials which will be used in chapter \ref{motheightzeta}. We also include a result showing how, for a power series $Z(T) = \sum_{n\geq 0} [M_n] T^n \in \svar_{\C}[[T]]$ (that is, having effective coefficients) for which we know the exact order of the first pole and which has a meromorphic continuation beyond this pole, one may recover information about the growth of the dimensions and number of irreducible components of the coefficients $M_n$.  
\section{Mixed Hodge modules\index{mixed Hodge module}}\label{sect.mhmbasics}
We are going to use freely the language of mixed Hodge modules introduced by Saito and are going to fix some notations and recall some facts to this end in this section.  References are the original works \cite{Saito88} and \cite{Saito90}, the summary \cite{Saito89} by Saito himself, the axiomatic introduction by Peters and Steenbrink in \cite{PS} section 14.1.1, and Beilinson, Bernstein and Deligne's paper \cite{BBD} for definitions and properties of perverse sheaves. If $S$ is a variety over~$\C$, we denote by $\mhm_S$ the abelian category of mixed Hodge modules on~$S$, by $D(\mhm_S)$ its derived category and by $D^{b}(\mhm_S)$ its bounded derived category.  For any integer $a$, we also denote by $D^{\leq a}(\mhm_S)$ the full subcategory of $D(\mhm_S)$ of complexes only having cohomology in degree $\leq a$. We use square brackets to denote the shifting of complexes: for any complex $M$ of mixed Hodge modules,  for every $n\in \Z$ and any $i\in \Z$, $(M[n])^i =M^{i+n}$. 

\subsection{The $\mathrm{rat}$ functor}
For any variety $S$ over $\C$, there is an exact and faithful functor
$$\rat_S:\mhm_S\to \mathrm{Perv}_S$$
to the abelian category of perverse sheaves on $S$, extending to a functor
$$\rat_S:D^{b}(\mhm_S)\to D^{b}_{cs}(S)$$
where the category on the right is the full subcategory of cohomologically constructible complexes inside the bounded derived category of sheaves of $\Q$-vector spaces on $S$. Saito showed (Theorem 0.1 in \cite{Saito90} or Theorem 1.3 in \cite{Saito89}) that the usual operations $\boxtimes$, $\otimes$, as well as $f_*, f^*, f_!, f^!$ for any morphism $f$ of varieties over $\C$ lift to the corresponding derived categories of mixed Hodge modules in a way compatible with the functor $\rat$. In particular, there are adjunctions $(f^*,f_*)$ and $(f_!,f^!)$.  There is a morphism $f_! \to f_*$ which is an isomorphism when $f$ is proper.

%For any integer $n\in \Z$ there is an abelian subcategory $\mathrm{HM}^n_S\subset \MHM_S$ of pure polarisable Hodge modules of weight $n$. 

\subsection{Twists}

In the case where $S$ is a point, the category $\mhm_{\pt}$ is exactly the category of polarisable mixed Hodge structures (see \cite{Saito89}, Theorem 1.4), and the functor $\rat$ becomes the forgetful functor associating to a mixed Hodge structure its underlying $\Q$-vector space. For any integer $d\in \Z$, we denote by~$\Q_{\pt}^{\hdg}(d)\in \mhm_{\pt}$ the Hodge structure of type $(-d,-d)$ with underlying vector space~$\Q$. For $d=0$, it will be denoted simply by $\Q^{\hdg}_{\pt}$. For any complex variety $S$, tensoring with~$\Q_{\pt}^{\hdg}(1)$ defines a Tate twist on $D^{b}(\mhm_S)$. 

When $f$ is smooth of relative dimension $d$, then  
\begin{equation}\label{uppershriek} f^{!} \simeq f^*[2d](d).\end{equation}

\subsection{Weight filtration}
Each $M\in\mhm_S$ has a finite increasing weight filtration $W_{\bullet}M$, the graded parts of which will be denoted $\gr^{W}_{\bullet}$. For a bounded complex of mixed Hodge modules $M^{\bullet}$, we say~$M^{\bullet}$ has weight $\leq n$ if $\gr^W_i\mathcal{H}^j(M^{\bullet}) = 0$ for all integers $i$ and $j$ such that $i>j+n$. 

For varieties $X$ and $Y$ over $\C$ we say that a functor $F:D^{b}(\mhm_X)\to D^{b}(\mhm_Y)$ does not increase weights if for every $n\in\Z$ and every $M^{\bullet}\in D^{b}(\mhm_X)$ with weight $\leq n$, the complex $F(M^{\bullet})$ is also of weight $\leq n$. In particular, for any morphism of complex varieties~$f$, the functors $f_!$ and $f^{*}$ do not increase weights (see \cite{Saito90}, (4.5.2)). %Exterior products \text{add weights}

\subsection{Cohomological functors and cohomological amplitude} The usual truncation functor $\tau_{\leq}$ on $D^b(\mhm_S)$ corresponds to the perverse truncation $^p\tau_{\leq}$ on $D^{b}_{cs}(S)$, so that $\rat_S\circ \mathcal{H}^{\bullet} =\,^p\mathcal{H}^{\bullet}\circ \rat_S$, where $\mathcal{H}^{\bullet}$ is the usual cohomology on $D^{b}(\mhm_S)$ and $^p\mathcal{H}^{\bullet}$ is the perverse cohomology on $D^{b}_{cs}(S)$. 

\begin{definition} Let $T:D_1\to D_2$ be an exact functor between triangulated categories endowed with $t$-structures, and let $a$ be an integer. The functor $T$ is said to be of \textit{cohomological amplitude} $\leq a$ if $T(D_1^{\leq 0})\subset D_2^{\leq a}$. Denoting by $^t\mathcal{H}^{\bullet}$ the corresponding cohomological functors, this means that for all $i>a$ and all $X\in D_1^{\leq 0}$, $^t\mathcal{H}^i(T(X)) = 0$. 
\end{definition}

\begin{lemma}\label{cohamplitude} For a morphism $f:Y\to X$ of varieties over $\C$ with fibres of dimension $\leq d$, the functors
$$f_{!}:D^{b}(\mhm_Y)\to D^b(\mhm_X)\ \ \ \text{and}\ \ \ f^{*}:D^{b}(\mhm_X)\to D^{b}(\mhm_Y)$$
are of cohomological amplitude $\leq d$. 
\end{lemma}

\begin{proof} According to \cite{BBD} 4.2.4, this is true for the corresponding functors 
$$f_!:D^{b}_{cs}(Y)\to D^{b}_{cs}(X)\ \ \ \text{and}\ \ \ f^{*}:D^b_{cs}(X)\to D^{b}_{cs}(Y)$$ with the perverse $t$-structure. Let $M^{\bullet}\in D^{\leq 0}(\mhm_Y)$. Then by compatibility of $\rat$ with pushforwards and $t$-structures as explained above, we have, for every $i\in \Z$
$$\rat_X(\mathcal{H}^{i}f_{!}(M^{\bullet})) = \  ^p\mathcal{H}^i(f_!(\rat_Y(M^{\bullet})))$$
The right-hand side is zero for $i>d$ by \cite{BBD} 4.2.4. The functor $\rat$ being faithful, we therefore have $\mathcal{H}^{i}f_{!}(M^{\bullet}) = 0$ for $i>d$ by lemma \ref{faithfulfunctorzero} below. The same argument holds for~$f^{*}$. 
\end{proof}

\begin{lemma}\label{faithfulfunctorzero} Let $F:A\to B$ be a faithful functor between additive categories. If $F(X) = 0$ for some object $X$ of $A$, then $X=0$. 
\end{lemma}
\begin{proof} The assumption $F(X) = 0$ implies that $F(\id_X) = 0 = F(0_X)$ where $0_X$ is the constant zero map on $X$. By faithfulness we have $\id_X = 0_X$, which means that $X=0$. 
\end{proof}
\subsection{The trace morphism \index{trace morphism}for mixed Hodge modules} \label{section.trace}
Though the theory of trace morphisms for Hodge modules is probably classical, we include this paragraph for lack of an appropriate reference. We only treat the case of smooth morphisms because it is sufficient for our purposes.
\begin{notation}\label{qshdg}For any complex variety $S$, we denote by $a_S:S\to \spec \C$ its structural morphism and by $\Q_S^{\hdg}$ the complex of mixed Hodge modules $a_S^{*}\Q_{\pt}^{\hdg}$. 
\end{notation}
\begin{remark} In the case where $S$ is smooth and connected, the complex of mixed Hodge modules $\Q_S^{\hdg}$ is concentrated in degree $\dim S$, and $\mathcal{H}^{\dim S}\Q_S^{\hdg}$ is  pure of weight $\dim S$, given by the pure Hodge module associated to the constant (rank one) variation of Hodge structures of weight 0 on~$S$. When $S$ is not smooth, by lemma \ref{cohamplitude} the complex $\Q_S^{\hdg}$ is still an object of $D^{\leq \dim S}(\mhm_S)$, of weight $\leq 0$ because the functor $a_S^{*}$ does not increase weights, so that $\gr^W_i\mathcal{H}^{\dim S}(\Q_S^{\hdg}) = 0$ for $i>\dim S$. On the other hand, by \cite{Saito90}, (4.5.9), $\gr^{W}_{\dim S}\mathcal{H}^{\dim S}(\Q_S^{\hdg})$ is non-zero and simple, given by the intermediate extension of the constant weight 0 variation of Hodge structures on an open subset of $S$. 
\end{remark}

 \begin{remark}[Duality and the trace morphism for Hodge modules] \label{tracemap} Let $X$ be a smooth variety of dimension $d$ over $\C$. Then according to (\ref{uppershriek}), we have $a_X^{!}\Q_{\pt}^{\hdg} \simeq \Q_X^{\hdg}(d)[2d]$, and by \cite{Saito90} (4.4.2), there is a morphism of complexes of mixed Hodge structures $$(a_X)_!\Q_X^{\hdg}\to \Q_{\pt}^{\hdg}(-d)[-2d]$$ lifting the corresponding morphism in the derived category of sheaves of $\Q$-vector spaces on $X$. The cohomology of the complex $(a_X)_!\Q_X^{\hdg}$ is exactly the cohomology with compact supports of $X$,  and the morphism $(a_X)_!\Q_X^{\hdg}\to \Q_{\pt}^{\hdg}(-d)[-2d]$ induces the trace morphism $H^{2d}_c(X,\Q)\to \Q(-d)$ on the top cohomology. The lift described above is compatible with Deligne's Hodge theory (see e.g. lemma 14.8, corollary 14.9 and remark 14.10 in \cite{PS}), and turns the trace morphism
into a morphism of Hodge structures $H^{2d}_c(X,\Q)\to \Q^{\hdg}_{\pt}(-d)$ (which is an isomorphism when $X$ is irreducible).
\end{remark}
The following proposition generalises this over a base:
\begin{prop}\label{traceprop} Let $S$ be a variety over $\C$ of dimension $n$ and $p:X\to S$ a smooth morphism  with fibres of constant dimension $d\geq 0$. Then there exists a morphism of complexes of Hodge modules
$$f:p_!\Q_X^{\hdg} \to \Q_S^{\hdg}(-d)[-2d]$$
inducing a morphism of mixed Hodge modules
$$\mathcal{H}^{2d+n}(p_!\Q_X^{\hdg})\to \mathcal{H}^{2d+n}(\Q_S^{\hdg}(-d)[-2d])$$
which above every closed point $s\in S$ corresponds to the classical trace morphism
$$H_c^{2d}(X_s,\Q)\to \Q^{\hdg}_{\pt}(-d)$$
of mixed Hodge structures.
\end{prop} 
\begin{proof} The counit $p_!p^{!}\to \mathrm{id}$ associated to the adjunction $(p_!,p^{!})$ induces a morphism of complexes of Hodge modules 
$$p_!p^{!}\Q_S^{\hdg} \to \Q_S^{\hdg}.$$

Since $p$ is smooth, we have $p^{!} = p^*(d)[2d]$, and this morphism induces a morphism
$$f:p_!p^{*}\Q_{S}^{\hdg}\to \Q_S^{\hdg}(-d)[-2d].$$
Note that by lemma \ref{cohamplitude}, both these complexes are objects of $D^{\leq 2d + n}(\mhm_S)$. Moreover, by proper base change  (\cite{Saito90} 4.4.3), above every closed point $s\in S$, this induces a morphism of complexes of Hodge  structures $f_s: (p_s)_!(p_s)^{*}\Q_{\pt}^{\hdg}\to \Q_{\pt}^{\hdg}(-d)[-2d]$, where $p_s:X_s\to \C$ is the pullback of $p$ by the inclusion $i_s:s\to S$, that is, $p_s = a_{X_s}$ using notation \ref{qshdg}. This morphism induces the trace map on the top cohomology as explained in remark \ref{tracemap}.
%In the general case, there is a dense open subset $U$ of $X$ on which $p$ is smooth. More precisely, we can even choose $U$ such that its complement is of relative dimension $<d$. Denote by $j:U\to X$ the inclusion, and by $i:Z\to X$ the inclusion of the reduced closed complement into $X$. Then we have the distinguished triangle
%$$j_!j^{*}\Q^{\hdg}_X\to \Q^{\hdg}_X \to i_*i^{*}\Q_X^{\hdg}\to j_!j^{*}\Q_X[1]$$
%and applying $p_!$, we get (since $i$ is proper):
%$$(p\circ j)_!\Q_U^{\hdg}\to p_!\Q_X^{\hdg} \to (p\circ i)_!\Q_Z^{\hdg} \to (p\circ j)_!\Q_U^{\hdg}[1].$$
%Thus, to construct $f$ on $p_!\Q_X^{\hdg}$, it suffices to construct it for $(p\circ j)_!\Q_U^{\hdg}$, which has been done in the previous step, and for $(p\circ i)_!\Q_Z^{\hdg}$. For the latter, we just take the zero morphism. Above every point $s\in S$, our distinguished triangle then yields a long exact sequence involving the cohomologies of the fibres $U_s$, $X_s$ and $Z_s$. Since $\dim (Z_s) < d$, our construction gives the expected trace morphism on the top cohomology of $X_s$. 
\end{proof}

\begin{remark}\label{tracerem} Let $p:X\to S$ be as in the proposition, and assume moreover that $X$ is irreducible. Then the morphism of mixed Hodge modules 
$$\mathcal{H}^{2d+n}(p_!\Q_X^{\hdg})\to \mathcal{H}^{2d+n}(\Q_S^{\hdg}(-d)[-2d])$$
given by the proposition induces an isomorphism
$$\gr^{W}_{2d+n}\mathcal{H}^{2d+n}(p_!\Q_X^{\hdg})\simeq \gr^{W}_{2d+n}\mathcal{H}^{2d+n}(\Q_S^{\hdg}(-d)[-2d]).$$
Indeed, denoting by $K_1$ (resp. $K_2$) the kernel (resp. cokernel) of the above morphism, we have an exact sequence of mixed Hodge modules
$$0\to K_1\to \mathcal{H}^{2d+n}(p_!\Q_X^{\hdg})\to \mathcal{H}^{2d+n}(\Q_S^{\hdg}(-d)[-2d])\to K_2 \to 0,$$
which above $s\in S$ becomes
\begin{equation}\label{ksexact}0\to K_{1,s}\to H^{2d}_c(X_s,\Q)\to \Q_{\pt}^{\hdg}(-d)\to K_{2,s}\to 0.\end{equation}
The trace morphism being always surjective, we may conclude that $K_2=0$. Moreover, for every $s\in S$, 
since $X$ is irreducible, there is an open dense subset $U$ of $S$ such that for all $s\in S$, $X_s$ is irreducible, so that the trace morphism for $X_s$ is an isomorphism. In particular, $K_1$ is supported inside some closed subset of $S$ contained in the complement of $U$. Denote by $S_1$ the support of the pure Hodge module $\gr^{W}_{2d+n}K_1$. If $S_1$ is non-empty, there is an open dense subset of $S_1$ over which $\gr^{W}_{2d+n}K_1$ corresponds to a variation of pure Hodge structures, which must be of weight $2d$ because of the exact sequence (\ref{ksexact}). This would mean that the corresponding pure Hodge module is of weight $\dim S_1 + 2d < \dim S + 2d$, a contradiction.
\end{remark}

\subsection{Mixed Hodge modules with monodromy\index{mixed Hodge module!with monodromy}}\label{sect.mhmmonodromy}

A reference for this is Saito's unpublished paper about the Thom-Sebastiani theorem for mixed Hodge modules \cite{SaitoTS}. One can also consult the summary in section 2.9 of \cite{BBDJS}. We denote by $\mhm_X^{\mon}$ the category of mixed Hodge modules $M$ on a complex variety~$X$ endowed with commuting actions of a finite order operator $T_s:M\to M$ and a locally nilpotent operator $N:M\to M(-1)$. The category $\mhm_X$ can be identified with a full subcategory of $\mhm_X^{\mon}$ via the functor
$$\mhm_X\to \mhm_{X}^{\mon}$$\index{Ts@$T_s$, monodromy operator}\index{mhmmon@$\mhm_X^{\mon}$}
sending a mixed Hodge $M$ to itself with $T_s = \id$ and $N = 0$. The Tate twist and the cohomological pullback and pushforward operations still exist in this setting (see \cite{BBDJS}, second paragraph of section 2.9). However, we need a more appropriate, twisted version of the external tensor product. Saito gives two equivalent ways of defining it, one of them being the following. Let $M_i =(D_i,F,L_i,W; T_s,N)$, $i=1,2$ be two mixed Hodge modules with monodromy, with  underlying $\mathcal{D}_X-$modules $D_i$, underlying perverse complexes $L_i$ with isomorphisms $\mathrm{DR}(D_i)\simeq L_i\otimes \C$ given by the Riemann-Hilbert correspondence, Hodge filtrations~$F$, weight filtrations $W$ and monodromy actions $(T_s,N)$. 

For each rational number $\alpha\in(-1,0]$, let $D^{\alpha}_i = \ker(T_s - \exp(-2i\pi\alpha))\subset D_i$. We define
$$M_1\twtimes M_2 = (D,F,L,W;T_s,N)$$
by $D = D_1\boxtimes D_2$, $L = L_1\boxtimes L_2$, $T_s = T_s\boxtimes T_s$, $N = N\boxtimes\id + \id \boxtimes N$, the Hodge filtration being given by   \index{twisted exterior product} \index{boxT@$\twtimes$}
$$F_p(D_1^{\al}\boxtimes D_2^{\beta}) = \left\{\begin{array}{cl}\displaystyle{ \sum_{i+ j = p+1} F_i D_1^{\al}\boxtimes F_jD_2^{\be} }& \text{if}\ \al + \be \leq -1,\\
\displaystyle{\sum_{i+ j = p} F_iD_1^{\al}\boxtimes F_jD_2^{\be}} & \text{if}\ \al + \be > -1.\end{array}\right.$$
and the weight filtration by
\begin{equation}\label{twistedweight} W_k(D_1^{\al}\boxtimes D_2^{\beta}) = \left\{\begin{array}{cl}\displaystyle{ \sum_{i+ j = k} W_i D_1^{\al}\boxtimes W_jD_2^{\be} }& \text{if}\ \al\be =0,\\
\displaystyle{\sum_{i+ j = k-1} W_iD_1^{\al}\boxtimes W_jD_2^{\be}} & \text{if}\ \al\be \neq 0,\al + \be \neq -1,\\
\displaystyle{\sum_{i+ j = k-2} W_iD_1^{\al}\boxtimes W_jD_2^{\be}} & \text{if}\ \al + \be = -1\end{array}\right.\end{equation}
for the underlying $\mathcal{D}$-modules. The weight filtration on the complex $(L_1\otimes\C)\boxtimes(L_2\otimes\C)$ is defined in the same manner, and gives a weight filtration on  $L_1\boxtimes L_2$ via the action of the Galois group of $\Q$.  We refer to \cite{SaitoTS} for the definition of the isomorphism $\mathrm{DR}(D_1\boxtimes D_2)\simeq (L_1\otimes\C)\boxtimes(L_2\otimes\C)$. 
\begin{example}\label{twtimesexample} Consider the Hodge structure with monodromy $$H = (\Q_{\pt}^{\hdg},-\id,0)\in \mhm_{\C}^{\mon}.$$ Let us compute
$H\twtimes H$. The underlying $\Q$-vector space is of dimension one. Moreover, observe that $H = H^{-\frac{1}{2}}$, so that by definition, we have $W_{1}(H\twtimes H)=0$ and $W_{2}(H\twtimes H) = W_0H\boxtimes W_0 H$. Thus, $H\twtimes H$ is pure of weight 2, with trivial monodromy, that is, it is equal to the pure Hodge structure $\Q_{\pt}^{\hdg}(-1)$. 
\end{example}

There is also a twisted tensor product on the derived category $D^b(\mhm_X^{\mon})$, defined for $M_1,M_2\in D^b(\mhm_X^{\mon})$ by
$$M_1\twotimes M_2 := \delta^*(M_1\twtimes M_2)$$
where $\delta: X\to X\times X$ is the diagonal map.  It is clear from the definition that these twisted products $\twtimes, \twotimes$ coincide with the usual products $\boxtimes, \otimes$ for Hodge modules with trivial monodromy. \index{otime@$\twotimes$}\index{twisted tensor product}

More generally, for any complex variety $X$ and for varieties $Y,Z$ above $X$, we have a relative twisted exterior product \index{twisted exterior product!relative}\index{boxTX@$\twtimes_X$} 
$$\twtimes_X: D^b(\mhm^{\mon}_Y)\times D^b(\mhm^{\mon}_Z)\to D^b(\mhm^{\mon}_{Y\times_XZ})$$
given for $M_1\in D^b(\mhm^{\mon}_Y)$, $M_2\in D^b(\mhm^{\mon}_Z)$ by 
$$M_1\twtimes_XM_2 = i^*(M_1\twtimes M_2)$$
where $i$ is the closed immersion $Y\times_XZ\to Y\times_{\C} Z.$ We denote simply by $\boxtimes_X$ the relative exterior product it induced on mixed Hodge modules with trivial monodromy. \index{box@$\boxtimes_X$}

%In the same way as for the category $\mhm_X$, one may consider the Grothendieck group $K_0(\mhm_X^{\mon})$ of mixed Hodge modules with monodromy. It is endowed with a product coming from the tensor product $\twotimes$ on $D^b(\mhm_X^{\mon})$. 

\subsection{Grothendieck ring of mixed Hodge modules\index{Grothendieck ring!of mixed Hodge modules}} A reference for this is \cite{CNS}, chapter 1 section 3.1, especially 3.1.4, 3.1.9 and~3.1.10. The Grothendieck group $K_0(\mhm_S)$  associated to the abelian category $\mhm_S$ is defined as the quotient of the free abelian group on isomorphism classes of mixed Hodge modules by relations of the form $[X] - [Y] +  [Z]$ for all objects $X,Y,Z\in \mhm_S$ forming a short exact sequence
$$0\to X \to Y \to Z \to 0.$$ 
The full subcategory of $\mhm_S$ of objects of pure weight, denoted by $\mathrm{HM}_S$, is semi-simple (see \cite{Saito88}, Lemme 5). On the other hand, the natural group morphism
$$K_0(\mathrm{HM}_S)\to K_0(\mhm_S)$$
is an isomorphism, with inverse given by sending the class of a mixed Hodge module $M$ to the sum of the classes of its graded parts. Thus, $K_0(\mhm_S)$ may also be seen as the quotient of the free abelian group on isomorphism classes of Hodge modules of pure weight by relations of the form $[X] - [Y] +  [Z]$ for all objects $X,Y,Z\in \mathrm{HM}_S$ forming a \textit{split} short exact sequence
$$0\to X \to Y \to Z \to 0.$$

 On the other hand, the triangulated category $D^{b}(\mhm_S)$ has a Grothendieck group $$K_0^{\mathrm{tri}}(D^{b}(\mhm_S)),$$ which is defined as the quotient of the free abelian group on isomorphism classes of objects of $D^{b}(\mhm_S)$, by relations of the form $[X] - [Y] +  [Z]$, for any objects $X,Y,Z$ of the category $D^{b}(\mhm_S)$ fitting into a distinguished triangle
$$X \to Y \to Z \to X[1].$$
The $\otimes$ operation endows the ring $K_0^{\mathrm{tri}}(D^{b}(\mhm_S))$ with a ring structure.  There is a natural group morphism
$$K_0(\mhm_S)\to K_0^{\mathrm{tri}}(D^{b}(\mhm_S))$$
sending the class of a mixed Hodge module $M$ to the class of the complex defined by this Hodge module placed in degree zero. This morphism is an isomorphism, with inverse given by sending the class of any complex $M^{\bullet}$ of mixed Hodge modules to the alternating series of the classes of its cohomology groups $\sum_{i\in \Z} (-1)^{i}[\mathcal{H}^{i}(M^{\bullet})].$  In what follows, we will denote this group always by $K_0(\mhm_S)$, and consider it as a ring via the ring structure on $ K_0^{\mathrm{tri}}(D^{b}(\mhm_S))$.

As for Grothendieck rings of varieties, a morphism $f:T\to S$ of complex varieties induces a group morphism
$$f_!:K_0(\mhm_T)\to K_0(\mhm_S)$$
and a ring morphism
$$f^*:K_0(\mhm_S)\to K_0(\mhm_T).$$
In particular, for any complex variety $S$, $K_0(\mhm_S)$ is endowed with a $K_0(\mhm_{\pt})$-algebra structure.

One may also consider Grothendieck rings of mixed Hodge modules with monodromy $K_0(\mhm_S^{\mon})$,\index{Grothendieck ring!of mixed Hodge modules!with monodromy} defined in the same manner, the product being induced by $\twotimes$. We denote this product by $\ast$. 
\section{Vanishing cycles and mixed Hodge modules}\label{sect.totvanhdg}
\subsection{The classical theory of vanishing cycles}\label{sect.classicaltheoryvancycles}
Here we recall briefly how nearby and vanishing cycles are defined in the classical transcendental setting. The main reference for this is Deligne's article \cite{DeligneSGA} in SGA 7. For a summary with a view towards mixed Hodge modules, see section 8 of Schnell's notes \cite{Schnell}. Let $X$ be a complex manifold, $D\subset \C$ the open unit disc, and $f:X\to D$ a holomorphic function. We denote by $D^* = D\setminus\{0\}$ the punctured unit disc and by $\widetilde{D}^*$ its universal covering space, which may be viewed as the complex upper half plane via the covering map
$$\begin{array}{ccc}p: \widetilde{D}^* = \{z\in\C,\ \mathrm{Im}(z) >0\}&\to & D^*\\
                               z &\mapsto &\exp(2i\pi z)\end{array}.
                               $$
                               We denote by $X^*$ (resp $X_0$) the inverse image $f^{-1}(D^*)$ (resp. $f^{-1}(0)$), and by $\widetilde{X}^*$ the complex manifold making the rightmost square in the following diagram cartesian:
                               $$\xymatrix{X_0\ar[d]^f \ar@{^{(}->}[r]^{i} & X\ar[d]^{f}  & X^* \ar@{_{(}->}[l]_{j} \ar[d]^f & \widetilde{X}^* \ar[l]_p\ar[d] \\
                               \{0\} \ar@{^{(}->}[r] & D & D^* \ar@{_{(}->}[l] & \widetilde{D}^* \ar[l]_p}                               $$
                               The nearby cycle functor
                               $$\psi_f:D^b_c(X) \to D^b_c(X_0)$$
                               is defined in the following way:
                               for $\mathcal{F}^{\bullet}\in D^b_c(X)$ a bounded constructible complex of sheaves on $X$, we put 
                               $$\psi_f\mathcal{F}^{\bullet} = i^{-1} R(j\circ p)_*\, (j\circ p)^{-1} \mathcal{F}^{\bullet}. $$
                               The deck transformation $z\mapsto z+1$ on $\widetilde{D}^*$ induces an automorphism of $\psi_{f}\mathcal{F}^{\bullet}$ called the \textit{monodromy}.
                               The adjunction morphism
                               $$\mathcal{F}^{\bullet}\to R(j\circ p)_*\, (j\circ p)^{-1} \mathcal{F}^{\bullet}$$
                               gives, after applying the functor $i^{-1}$, a morphism
                               $$i^{-1}\mathcal{F}^{\bullet} \to \psi_f\mathcal{F}^{\bullet}.$$ The vanishing cycles complex $\phi_f\mathcal{F}^{\bullet}$ of $\mathcal{F}^{\bullet}$ at 0 is defined as the cone of this morphism, so that there is a distinguished triangle
                               $$ i^{-1}\mathcal{F}^{\bullet} \to \psi_f\mathcal{F}^{\bullet} \to \phi_{f}\mathcal{F}^{\bullet} \to i^{-1}\mathcal{F}^{\bullet}[1].$$
                            The functors $\psi_f$ and $\phi_f$ take distinguished triangles to distinguished triangles, and commute with shifting of complexes.   
                            
                          %  The terminology is motivated by the fact that the hypercohomology of $\psi_f(\Q)$ 

                            A theorem of Gabber (see \cite{Brylinski}, Théorème 1.2) says that if $\mathcal{F}^{\bullet}$ is a perverse complex, then both $^p\psi_f\mathcal{F}^{\bullet}:=\psi_{f}\mathcal{F}^{\bullet}[-1]$ and $^p\phi_f\mathcal{F}^{\bullet}:=\phi_{f}\mathcal{F}^{\bullet}[-1]$ are perverse.

\begin{lemma}\label{vancyclesfinite} Let $X$ be a complex algebraic variety, $f:X\to \A^1_{\C}$ a morphism and $\mathcal{F}^{\bullet}\in D^{b}_c(X)$. Then $\phi_{f-a}\mathcal{F}^{\bullet} = 0$ for all but a finite number of $a\in\C$.
\end{lemma}                              
\begin{proof}  The case when $\mathcal{F}^{\bullet}$ is a constructible sheaf in degree zero follows from theorem 2.13 in \cite{DeligneSGA45}. Indeed, though the latter is formulated in an $\ell$-adic setting, the proof is formal and goes over to the complex setting. Another argument may be given using local triviality results (see e.g. Corollaire 5.1 in \cite{Verdier}).%: we may choose a Zariski open subset $U$ of $\A^1_{\C}$ such that the restriction of $f$ to $f^{-1}(U)\to U$ is a locally trivial fibration. This means that around any $a\in\C$, we can find a small disk $D$ above which $f$ is a trivial fibration, so that the cohomology of $\mathcal{F}$ on the fibres is constant, and vanishing cycles are zero on $D$. 

 Since vanishing cycles commute with shifting of complexes, we may then proceed by induction on the amplitude of the complex $\mathcal{F}^{\bullet}$: there are complexes $\mathcal{G}^{\bullet}$ and $\mathcal{G}'^{\bullet}$ of strictly smaller amplitude fitting into a distinguished triangle
$$\mathcal{G}^{\bullet}\to \mathcal{F}^{\bullet}\to \mathcal{G}'^{\bullet} \to \mathcal{G}^{\bullet}[1].$$
Applying the functor $\phi_{f-a}$ and using the induction hypothesis, we get, for all $a$ but a finite number, $\phi_{f-a}\mathcal{F}^{\bullet}=0$
\end{proof}

\begin{remark} There are several conventions and notations concerning nearby and vanishing cycles, which may seem quite confusing. In this section we use the most common one, coming from SGA, which is also the one used by Saito. Another convention is the one by Kashiwara and Shapira from \cite{KS}. It has the same definition $\psi_f^{KS} = \psi_f$ for nearby cycles, but vanishing cycles are shifted by one:
$$\phi_f^{KS} := \phi_{f}[-1],$$
so that in Kashiwara and Shapira's theory, the above distinguished triangle is shifted and takes the form
\begin{equation}\label{KSdistinguished}\phi_f^{KS}\mathcal{F}^{\bullet}\to i^{-1}\mathcal{F}^{\bullet} \to \psi_{f}^{KS} \to\phi_f^{KS}\mathcal{F}^{\bullet}[1].\end{equation}
This latter convention is used, e.g., in David Massey's paper \cite{Massey} on the Thom-Sebastiani theorem in the derived category of constructible sheaves, because the shift chosen by Kashiwara and Shapira is the one that makes the theorem of Thom-Sebastiani work. For this reason, it is in fact also the convention used in the motivic setting of chapter \ref{grothrings}: the distinguished triangle (\ref{KSdistinguished}) induces, in the triangulated Grothendieck ring of the category $D^{b}_c(X_0)$, the identity
$$[\phi_f^{KS}\mathcal{F}^{\bullet}] = [i^{-1}\mathcal{F}^{\bullet}] - [\psi_f^{KS}\mathcal{F}^{\bullet}],$$
which corresponds to the way we defined motivic vanishing cycles in section \ref{sect.vanrecall} of chapter~\ref{grothrings}. This definition, which we took from Lunts and Schnürer's work \cite{LS} is the one which, simultaneously, enables us to define a group morphism using total motivic vanishing cycles and makes the motivic Thom-Sebastiani theorem work, so that this group morphism is actually a ring morphism, our motivic vanishing cycles measure. Denef and Loeser's motivic vanishing cycles $\scr{S}^{\phi}_f$ which are equal to $(-1)^{\dim X}$ times our motivic vanishing cycles, satisfy Thom Sebastiani (each side of the Thom-Sebastiani equality being multiplied by the same power of $-1$), but can't be combined into a group morphism because of obvious sign issues.\index{motivic vanishing cycles!sign}
\end{remark}

%\begin{example} 
%\end{example}
\subsection{Nearby and vanishing cycles for Hodge modules}\label{sect.vancycleshodgemodules}

For a morphism $f:X\to \A^1$ on a complex variety $X$, denoting $X_0(f) = f^{-1}(0)$, there are nearby and vanishing cycles functors 
$$\psi_f^{\hdg}, \phi_f^{\hdg}:\mhm_X\to \mhm_{X_0(f)}^{\mon}$$ \index{phifHdg@$\phi_f^{\hdg}$}\index{psifHdg@$\phi_f^{\hdg}$}
lifting the corresponding functors $^p\psi_f$ and $^p\phi_f$ on perverse sheaves. The monodromy operator is quasi-unipotent, so that its semisimple part is of finite order. The action of $T_s$ is given by this semisimple part, whereas the action of $N$ is given by the logarithm of the unipotent part of the monodromy.
 For any mixed Hodge module $M\in\mhm_{X}$, there are decompositions
$$\psi_{f}^{\hdg}(M) = \psi_{f,1}^{\hdg}(M) \oplus \psi_{f,\neq 1}^{\hdg}(M)$$
and
$$\phi_{f}^{\hdg}(M) = \phi_{f,1}^{\hdg}(M) \oplus \phi_{f,\neq 1}^{\hdg}(M)$$
in the category $\mhm_{X_0(f)}$
where $\psi_{f,1}^{\hdg}(M) = \ker (T_s-\id)$ and $\psi_{f,\neq 1}^{\hdg}(M) = \ker (T_s^{d-1} + \ldots + T_s + \id)$ where $d$ is the order of $T_s$ (same definition for $\phi_f$). 

The following compatibility result with proper morphisms is classical in the theory of nearby and vanishing cycles, and, in the case of Hodge modules, follows from \cite{Saito90} Theorem~2.14.
\begin{lemma}\label{vancyclhdgproper} Let $p:X\to Y$  and $f:Y\to \A^1$ be morphisms of complex varieties. Assume that $p$ is proper.
$$\xymatrix{X \ar[r]^p\ar[dr]_{f\circ p} & Y \ar[d]^f\\
                       &\A^1}$$
Denoting by $\tilde{p}$ the restriction of $p$ to $(f\circ p)^{-1}(0)$, for any $M\in\mhm_{X}$ there are isomorphisms
$$\tilde{p}_*(\psi_{f\circ p}(M))\simeq \psi_{f}(p_{*}(M))
$$
and
$$\tilde{p}_*(\phi_{f\circ p}(M))\simeq \phi_{f}(p_{*}(M))
$$
\end{lemma}

% There is a Thom-Sebastiani theorem in this setting:
%\begin{theorem}[\cite{SaitoTS}, Theorem 5.4] Let $X_1,X_2$ be complex varieties with morphisms $f_1:X_1\to \A^1_{\C}$, $f_2:X_2\to\A^1_{\C}$.  Let $i:X_0:=f_1^{-1}(0)\times f^{-1}_2(0)\to (f_1\conv f_2)^{-1}(0)$ be the inclusion. Then for all $M_i\in D^b(\mhm_{X_i})$ there is a canonical isomorphism 
%$$i^*\phi_f^{H}(M_1\boxtimes M_2) \simeq \phi_{f_1}^H(M_1)\twtimes \phi_{f_2}^H(M_2)$$
%in $D^b(\mhm_{X_0}^{\mon})$. 
%\end{theorem}

\subsection{Total vanishing cycles}\label{sect.totalvancycleshodgemodules}
Let $X$ be a complex variety, and denote by $\pr$ the projection $\A^1_X\to \A^1_{\C}$. By lemma \ref{vancyclesfinite} and faithfulness of the functor $\mathrm{rat}$, there is a well-defined \textit{total vanishing cycles functor}: \index{total vanishing cycles functor}
$$\begin{array}{rccc}\vanhdg_X:&\mhm_{\A^1_{X}} &\to& \mhm_{X}^{\mon}\\
    & M & \mapsto & \bigoplus_{a\in\C}\phi_{\pr-a}^{\hdg} M
\end{array}$$
 It satisfies a Thom-Sebastiani property: \index{Thom-Sebastiani!for $\phi_X^{\tot}$}
\begin{prop}\label{HodgeTS} For all $M_1,M_2\in D^b(\mhm_{\A^1_{X}})$, denoting by $\add$ the addition morphism $\A^1_{X}\times \A^1_{X}\to \A^1_{X^2}$, one has the isomorphism
$$\vanhdg_{X^2}\left((\add_{!})(M_1\boxtimes M_2)\right) \simeq \vanhdg_X(M_1)\twtimes \vanhdg_X(M_2)$$
in $D^b(\mhm_{X^2}^{\mon})$. 
\end{prop}
\begin{proof} Saito's Thom-Sebastiani theorem for Hodge modules (see \cite{SaitoTS}) gives, for any $a_1,a_2\in \C$, an isomorphism
$$i_{a_1,a_2}^*\phi^{\hdg}_{\pr\conv \pr - a_1 -a_2}(M_1\boxtimes M_2) \simeq \phi^{\hdg}_{\pr-a_1}M_1 \twtimes \phi_{\pr-a_2}^{\hdg}M_2,$$
in $D^{b}(\mhm^{\mon}_{\pr^{-1}(a_1)\times \pr^{-1}(a_2)})$, where 
$$i_{a_1,a_2}:\pr^{-1}(a_1)\times \pr^{-1}(a_2) \to (\pr\conv \pr)^{-1}(a_1 + a_2)$$
is the natural inclusion. For any $a\in\C$, the products $\pr^{-1}(a_1)\times \pr^{-1}(a_2)$ for all $a_1,a_2$ such that $a_1 + a_2 = a$ form a partition of  $(\pr\conv \pr)^{-1}(a)$. Therefore, taking the direct sum over all $a_1,a_2$, we get an isomorphism
$$\bigoplus_{a\in\C}\phi^{\hdg}_{\pr\conv \pr -a}(M_1\boxtimes M_2) \simeq \vanhdg_{X}(M_1)\twtimes \vanhdg_{X}(M_2)$$
in $D^{b}(\mhm^{\mon}_{X^2}).$
As in the motivic setting, a compactification argument allows us to write the left-hand side as
$$ \bigoplus_{a\in\C}\phi^{\hdg}_{\pr\conv \pr -a}(M_1\boxtimes M_2)  \simeq \bigoplus_{a\in\C} \phi^{\hdg}_{\pr -a} ((\add)_!(M_1\boxtimes M_2)),$$
whence the result.
\end{proof}

\section{Symmetric products and vanishing cycles}\label{sect.symprodvancycles}
\subsection{Symmetric products of mixed Hodge modules}\label{mhmsymproducts}
From now on, we are going to take symmetric products, so we assume the varieties to be quasi-projective over $\C$. A notion of \textit{symmetric power} of a complex of mixed Hodge modules was defined by Maxim, Saito and Schürmann in \cite{MSS}. Here we are going to explain how it extends to the setting of mixed Hodge modules with monodromy. 

Theorem 1.9 in \cite{MSS} gives, for an integer $n\geq 1$, bounded complexes of Hodge modules $M_1,\ldots,M_n$ on a quasi-projective complex variety $X$ and for all $\sigma \in \Sym_n$, an isomorphism
$$\sigma^{\sharp}: M_1\boxtimes \ldots \boxtimes M_n \xrightarrow{\sim} \sigma_{*}\left(M_{\sigma(1)}\boxtimes\ldots\boxtimes M_{\sigma(n)}  \right)$$
in $D^b(\mhm_{X^n})$ compatible with the analogous isomorphism on the underlying complexes of constructible sheaves. These isomorphisms in fact induce a right action of $\Sym_n$ on the exterior product  $M_1\boxtimes \ldots \boxtimes M_n$. When $X$ is smooth, these isomorphisms are constructed from analogous isomorphisms on the underlying complexes of $\mathcal{D}$-modules and constructible sheaves, which are checked to be compatible via the De Rham functor (proposition 1.5 in \cite{MSS}). The general non-smooth case is deduced from this by an embedding argument. 

Note that the above generalises to complexes of mixed Hodge modules with monodromy, if we replace the ordinary exterior product $\boxtimes$ by its twisted counterpart $\twtimes$. Indeed, the underlying complexes of $\mathcal{D}$-modules and constructible sheaves and the De Rham functor relating them remain the same, and the monodromy is compatible with the above action. 

Let now $M$ be a bounded complex of mixed Hodge modules with monodromy over a quasi-projective complex variety $X$. Let $n\geq 1$ be an integer, and denote by $\pi: X^n\to S^nX$ the canonical projection. The above construction leads to the definition of a right $\Sym_n$-action on the complex $\pi_*(\bigtwtimes{n} M)$. Finally, as in~\cite{MSS}, the idempotent $e = \frac{1}{n!}\sum_{\sigma\in \Sym_n}\sigma\in\Q[\Sym_n]$ defines an idempotent of $\pi_{*}(\bigtwtimes{n}M)$ in the category $D^b(\mhm_{S^nX}^{\mon})$, which splits by corollary~2.10 in~\cite{BS}, meaning that we may write
$$\pi_*(\bigtwtimes{n}  M) = \ker(e) \oplus \mathrm{Im}(e).$$
For any bounded complex of mixed Hodge modules~$M$ with monodromy on a complex variety $X$, the symmetric power \index{symmetric power!of mixed Hodge module} $S^n(M)$ is  then defined to be
$$S^n(M) = \left(\pi_*\left(\bigtwtimes{n}M\right)\right)^{\Sym_n} := \mathrm{Im}(e)\in D^b(\mhm^{\mon}_{S^nX}),$$
where $\pi: X^n\to S^nX$ is the canonical projection.
\begin{remark} In the case where $M$ has trivial monodromy, we recover the definition from  \cite{MSS}. In particular, in this case $S^nM$ is an element of $D^{b}(\mhm_{S^nX})$. 
\end{remark}
As in chapter \ref{eulerproducts}, section \ref{sect.symprodgrouplaw}, we may define a multiplicative group structure on the product $\prod_{i\geq 1}K_0(\mhm^{\mon}_{S^iX})$, by using the twisted exterior product $\twtimes$: for all $a =(a_i)_{i\geq 1}$ and $b= (b_i)_{i\geq 1}$ in the product $\prod_{i\geq 1}K_0(\mhm^{\mon}_{S^iX})$, put, for every $n\geq 1$, 
$$ (ab)_n = \sum_{k=0}^na_i\boxtimes b_{n-i}$$
where by convention $a_0 = b_0 = 1$, and $a_i\twtimes b_{n-i}$ is the image of $(a_i,b_{n-i})$ through the composition
$$K_0(\mhm^{\mon}_{S^iX})\times K_0(\mhm^{\mon}_{S^{n-i}X})\xrightarrow{\twtimes}K_0(\mhm^{\mon}_{S^iX\times S^{n-i}X})\to K_0(\mhm^{\mon}_{S^nX})$$
where the latter morphism is obtained from the quotient map $X^n\to S^nX$ by passing to the quotient with respect to the natural permutation action of the group $\Sym_i\times \Sym_{n-i}$ on~$X^n$. We denote this group by $K_0(\mhm^{\mon}_{S^{\bullet}X}).$
\begin{lemma} \label{mhmmorphism} There is a unique group morphism
$$S_X^{\hdg}: K_0(\mhm^{\mon}_X) \to K_0(\mhm^{\mon}_{S^{\bullet}X}),$$
such that the image of the class of a complex of mixed Hodge modules $M$ over $X$ is the family of classes~$([S^nM])_{n\geq 1}$. \index{SXHdg@$S_X^{\hdg}$}
\end{lemma}
\begin{proof} We define a map $S_X^{\hdg}$ from the free abelian group on Hodge modules to the group $K_0(\mhm^{\mon}_{S^{\bullet}X})$ by putting $S_X^{\hdg}(M) = ([S^nM])_{n\geq 1}$. By the discussion about the definition of $K_0(\mhm^{\mon}_X)$, to show that~$S_X^{\hdg}$ passes to the quotient, it suffices to show that whenever $M,M_0, M_1$ are Hodge modules such that $M = M_0\oplus M_1$, we have $S_X^{\hdg}(M) = S_X^{\hdg}(M_0)S_X^{\hdg}(M_1)$. In fact, denoting for every $k\in \{0,\ldots,n\}$ by $i_k$ the natural morphism $S^{n-k}X\times S^{k}X\to S^nX$, we will show that for every $n\geq 1$, we have, in $K_0(\mhm^{\mon}_{S^nX})$, the equality
$$S^nM = \sum_{k=0}^n i_{k,*}\left(S^{n-k}M_0\twtimes S^{k}M_1\right).$$
For this, note that in $D^b(\mhm^{\mon}_{X^n})$ we have the direct sum decomposition
$$\bigtwtimes{n}M = \bigoplus_{(\epsilon_1,\ldots,\epsilon_n)\in\{0,1\}^n}M_{\epsilon_1}\twtimes \ldots \twtimes M_{\epsilon_n}=  \bigoplus_{k=0}^n \bigoplus_{\substack{(\epsilon_1,\ldots,\epsilon_n)\in\{0,1\}^n\\ \epsilon_1 + \ldots + \epsilon_n = k}}M_{\epsilon_1}\twtimes\ldots\twtimes M_{\epsilon_n}.$$
Let $M_{(k)} :=\oplus_{\substack{(\epsilon_1,\ldots,\epsilon_n)\in\{0,1\}^n\\ \epsilon_1 + \ldots + \epsilon_n = k}}M_{\epsilon_1}\twtimes\ldots\twtimes M_{\epsilon_n}$ and let $\pi: X^n\to S^nX$ be the quotient morphism. After applying the exact functor $\pi_*$ to the above decomposition, we get
$$\pi_*\left(\bigtwtimes{n}M\right) = \bigoplus_{k=0}^n \pi_*(M_{(k)})$$
in $D^b(\mhm^{\mon}_{S^nX})$. Observe that each of the factors $\pi_{*}(M_{(k)})$ is stable under the action of the symmetric group $\Sym_n$, so that we have
$$\pi_*\left(\bigtwtimes{n} M\right)^{\Sym_n} = \bigoplus_{k=0}^n(\pi_*M_{(k)})^{\Sym_n}.$$
It suffices to prove that  for every $k$, $i_{k,*}\left(S^{n-k}M_0\twtimes S^{k}M_1\right)$ is isomorphic to $(\pi_*M_{(k)})^{\Sym_n}$. We denote again by $e$ the restriction of the idempotent $e$ to $\pi_*M_{(k)}$.

Fix $k\in\{0,\ldots,n\}$, and denote by $\pi_k:X^k\to S^kX$ and $\pi_{n-k}: X^{n-k}\to S^{n-k}X$ the quotient maps.  The corresponding idempotent on $(\pi_{n-k})_*\bigtwtimes{n-k}M_0$  (resp. $(\pi_{k})_*\bigtwtimes{k}M_1$) will be denoted by~$e^0$ (resp.~$e^1$). By the commutativity of the diagram
$$\begin{array}{ccc}X^{n-k}\times X^{k}\ \ \ \  &\xrightarrow{\mathrm{id}}&  X^n \\
 \downarrow\ \scriptstyle{\pi_{n-k}\times \pi_k} &   & \downarrow\ \scriptstyle{\pi}\\
                      S^{n-k}X\times S^kX\ \ \ \  &\xrightarrow{i_k} & S^nX\end{array}$$
                      and exactness of $\pi_*$, the inclusion of the direct factor $(\bigtwtimes{n-k} M_0)\twtimes(\bigtwtimes{k}M_1)\to M_{(k)}$ induces a monomorphism
                      $$i_{k,*}\left((\pi_{n-k})_*\left(\bigtwtimes{n-k}M_0\right) \twtimes (\pi_k)_*\left(\bigtwtimes{k}M_1\right)\right)\xrightarrow{f} \pi_*M_{(k)}.$$
                      The $\Sym_{n-k}\times \Sym_k$-action on the left-hand side is compatible with the $\Sym_n$-action on the right-hand side, when $\Sym_{n-k}\times \Sym_k$ is seen in a natural way as a subgroup of $\Sym_n$. Therefore, taking invariants, we have a monomorphism
                      $$i_{k,*}\left(S^{n-k}M_0\twtimes S^{k}M_1\right)\xrightarrow{\tilde{f}} (\pi_*M_{(k)})^{\Sym_n},$$
                      fitting into the commutative diagram
                      $$\xymatrix{i_{k,*}\left((\pi_{n-k})_*\left(\bigtwtimes{n-k}M_0\right) \twtimes (\pi_k)_*\left(\bigtwtimes{k}M_1\right)\right)\ar[r]^-{f}\ar[d]^{i_{k,*}(e^0\,\twtimes\, e^1)} &\pi_*M_{(k)}\ar[d]^{e}\\
                i_{k,*}\left(S^{n-k}M_0\twtimes S^{k}M_1\right)\ar[r]^-{\tilde{f}}& (\pi_*M_{(k)})^{\Sym_n}     }$$
              Since $e$ is surjective, $\tilde{f}$ is an isomorphism.                     
\end{proof}

\begin{definition} For any $\a\in K_0(\mhm_{X}^{\mon})$ and any $n\geq 1$, we define $S^n\a$ to be the element of $K_0(\mhm_{S^nX}^{\mon})$ given by the $n$-th component of $S^ {\hdg}(\a)$. 
\end{definition}
\begin{remark} If $\a$ is an element of $K_0(\mhm_{X})$ (i.e. has trivial monodromy), then $S^n\a$ is an element of $K_0(\mhm_{S^nX})$. 
\end{remark}
\subsection{Compatibility between symmetric products and total vanishing cycles}\label{sect.compsymprodvanhdg}
 Denote by $\add_n$ the addition map $\A^n\to \A^1$ on the group scheme $\A^n$, and for any quasi-projective variety $Y$, by $\pi_Y$ the quotient map $Y^n\to S^nY$.  Since $\add_n$ is invariant via the permutation of the coordinates of $\A^1$, for any quasi-projective complex variety $X$,  it induces a morphism $\overline{\add}_n$ fitting into a commutative diagram \index{add@$\add_n$, $\overline{\add}_n$}
$$\xymatrix{(\A^1_{X})^n\ar[r]^-{\add_n}\ar[d]_{\pi_{\A^1_X}} & \A^1_{X^n}\ar[d]^{\pi'_X} \\
  S^{n}(\A^1_{X}) \ar[r]^-{\overline{\add}_n} & \A^1_{S^nX}}$$
 where we denote by $\pi'_X$ the morphism $\A^1_{X^n}\to \A^1_{S^nX}$ induced by $\pi_X:X^n\to S^nX$. 
 
Let $M \in D^b(\mhm_{\A^1_{X}})$. From \ref{mhmsymproducts} we know that there is a $\Sym_n$-action on the mixed Hodge module $\bboxtimes^nM$ on $(\A^1_X)^n$,  compatible with the permutation action of $\Sym_{n}$ on~$(\A^1_X)^n$.  Since $\pi'_X\circ \add_n$ is equivariant if one equips $\A^1_{S^nX}$ with the trivial $\Sym_n$-action, this $\Sym_n$-action induces a $\Sym_n$-action on $(\pi'_X\circ \add_n)_! \left(\bboxtimes^nM\right)$.  The relation $\pi'_X\circ \add_n = \bar{\add}_n\circ \pi_{\A^1_X}$ shows that the functor $(\bar{\add}_n)_!$ sending $(\pi_{\A^1_X})_*(\bboxtimes^nM)$ to $(\pi'_X\circ\add_n)_{!}(\bboxtimes^nM)$ is compatible with $\Sym_n$-actions, which, taking invariants, gives us the relation:
\begin{lemma} Let $M\in D^b(\mhm_{\A^1_X})$. One has \label{addsymprod} \begin{equation}     
\label{addsymprodeq} (\bar{\add}_n)_!S^nM = \left((\pi'_X\circ \add_n)_!\left(\bboxtimes^nM\right)\right)^{\Sym_n}
\end{equation}
in $D^b(\mhm_{\A^1_{S^nX}})$. 
\end{lemma}

\begin{prop}\label{vancyclefunctorsymproducts} For any $M\in D^b(\mhm_{\A^1_{X}})$, we have
$$\vanhdg_{S^nX}\left(\left(\overline{\add}_n\right)_!S^nM\right) = S^n\left(\vanhdg_X(M)\right)$$
in $D^{b}(\mhm_{S^nX}^{\mon}).$
\end{prop}
\begin{proof} For all $M_1,\ldots M_n\in D^b(\mhm_{\A^1_{X}})$, the Thom-Sebastiani theorem for $\vanhdg$ says
$$\vanhdg_{X^n}\left((\add_n)_!\left(M_1\boxtimes\ldots\boxtimes M_n\right)\right) \simeq \vanhdg_X(M_1)\twtimes \ldots \twtimes \vanhdg_X(M_n)$$
in $D^{b}(\mhm_{X^n}^{\mon}).$
Composing with the functor $(\pi_X)_*$ where $\pi_X:X^n\to S^nX$ is the projection, and using lemma \ref{vancyclhdgproper}, we have
$$\vanhdg_{S^nX}((\pi'_X\circ \add_n)_!(M_1\boxtimes \ldots \boxtimes M_n)) \simeq (\pi_X)_*\left(\vanhdg_X(M_1)\twtimes \ldots \twtimes \vanhdg_X(M_n)\right),$$
where $\pi'_X$ is the morphism $\A^1_{X^n}\to \A^1_{S^nX}$ induced by $\pi_X$. 
In the same manner,  for every $\sigma\in\Sym_n$, we also have
$$\vanhdg_{S^nX}((\pi'_X\circ \add_n)_!(M_{\sigma^{-1}(1)}\boxtimes \ldots \boxtimes M_{\sigma^{-1}(n)~})) $$
$$\simeq (\pi_X)_*\left(\vanhdg_X(M_{\sigma^{-1}(1) })\twtimes \ldots \twtimes \vanhdg_X(M_{\sigma^{-1}(n)})\right),$$
Thus, for any $M\in D^b(\mhm_{\A^1_{X}})$, the idempotent $\frac{1}{n!}\sum_{\sigma\in\Sym_n}\sigma\in \Z[\Sym_n]$ induces idempotents~$e$ and~$e'$ of $$(\pi'_X\circ\add_n)_!\left(\bboxtimes^n M\right)\in D^b(\mhm_{\A^1_{X^n}})$$ and $$(\pi_X)_*\bigtwtimes{n}(\vanhdg_X(M))\in D^{b}(\mhm_{X^n}^{\mon})$$ such that $\vanhdg_{S^nX}(e) =e'$. Therefore, splittings being preserved by any additive functor, we have
$$\vanhdg_{S^nX}\left(\left((\pi'_X\circ\add_n)_!\bboxtimes^n M\right)^{\Sym_n}\right) \simeq \left(\bigtwtimes{n}(\vanhdg_X(M))\right)^{\Sym_n},$$
which, using the isomorphism (\ref{addsymprodeq}) above, gives the result. 
\end{proof}
\section{Compatibility with motivic vanishing cycles}\label{sect.motvancomp}

\subsection{The Hodge realisation}\label{sect.hodgereal}
 Let $S$ be a complex variety, and let $X\xrightarrow{p}S$ be an $S$-variety, endowed with a $\mu_n$-action~$\sigma$ for some $n\geq 1$. Then $\sigma(e^{\frac{2i\pi}{n}})$ induces an automorphism $T_s(\sigma)$ of finite order on each cohomology group of the complex of mixed Hodge modules $p_!\Q_X^{\hdg}$ (see notation \ref{qshdg}). Thus, as explained in \cite{GLM} section~3.16, this defines a group  morphism 
$$\begin{array}{rccc}\chi^{\hdg}_S:&\M_S^{\hat{\mu}}&\to &K_0(\mhm_S^{\mon})\\
                                     & [X\xrightarrow{p} S,\sigma] & \mapsto& \sum_{i\in\Z}(-1)^i[\mathcal{H}^i(f_!\Q_X^{\hdg}),T_s(\sigma), 0] \end{array}$$
called the \textit{Hodge realisation morphism}\index{Hodge realisation}. Here are some properties of this Hodge realisation (for a proof, see the ideas in \cite{GLM}, section 6):

\begin{prop}\label{chihdgprop} Let $S,T$ be complex varieties.
\begin{enumerate}
\item The morphism $\chi^{\hdg}_S$ commutes with twisted exterior products, that is, for any $\a\in\M_{S}^{\hat{\mu}}, \b\in\M_{T}^{\hat{\mu}}$, we have
$$\chi_{S\times T}^{\hdg}(\Psi(\a\boxtimes\b)) = \chi_{S}^{\hdg}(\a)\twtimes \chi_{T}^{\hdg}(\b)$$
where $\Psi:\M_{S\times T}^{\hat{\mu}\times \hat{\mu}} \to \M_{S\times T}^{\hat{\mu}}$ is the convolution morphism from chapter \ref{grothrings}, section \ref{sect.convolution}. 
\item For any morphism $f:T\to S$ between complex varieties, we have
$$\chi^{\hdg}_S\circ f_! = f_!\circ \chi_T^{\hdg}\ \ \ \ \text{and}\ \ \ \chi^{\hdg}_T\circ f^* = f^*\circ \chi_{S}^{\hdg}.$$ 
\item \label{chihgdpropring} The group morphism $\chi_S^{\hdg}$ is a ring morphism
$$(\M_S^{\hat{\mu}},\ast)\to (K_0(\mhm_S^{\mon}),\ast).$$
\end{enumerate}
\end{prop}

%\begin{proof} A COMPLETER
%\end{proof}
%\begin{remark} In fact, proposition 3.3.7 in \cite{CNS} gives $\chi^{\hdg}_S$ only on $\kvar_S$, but observing that $\chi^{\hdg}_{\pt}(\LL) = [\Q(-1)]$, which is an invertible class in $K_0(\mhm_{\pt})$, we may extend it by $\M_{\C}$-linearity to $\M_S$, putting, for any $\a\in \kvar_S$ and any integer $m$,
%$\chi^{\hdg}_S(\LL^{-m}\a) = [\Q^{\hdg}_{\pt}(m)]\chi_S^{\hdg}(\a).$
%\end{remark} 

\begin{example} For $S=\spec \C$, we have, for any separated complex variety $X$ with $\mu_n$-action $\sigma$,
$$\chi_{\pt}^{\hdg}([X,\sigma]) = \sum_{i=0}^{\dim X}(-1)^i[H^{i}_c(X(\C),\Q),T_s(\sigma)],$$
where $[H^{i}_c(X(\C),\Q),\sigma]$ is the class of the mixed Hodge structure defined by Deligne \cite{Deligne} on the singular cohomology group $H^{i}_c(X(\C),\Q)$, together with the automorphism of finite order induced by~$\sigma(e^{\frac{2i\pi}{n}})$.
\end{example}

\begin{example} %Take $X_1 = X_2 = \A^1_{\C}$ and $f_1 = f_2 = (x\mapsto x^2)$, so that $f_1\conv f_2:\A^2_{\C}\to \A^1_{\C}$ is given by $(x,y)\mapsto x^2 + y^2$. 

In section \ref{sect.TSexample} of chapter \ref{grothrings}, we showed that the motivic vanishing cycles $\phi_{x^2}$ of the function $\A^1\to \A^1$, $x\mapsto x^2$ are equal to the class $1-[\tilde{E},\mu_2]\in\M_{\C}^{\hat{\mu}}$, where $[\tilde{E},\mu_2]$ is the class of the union of two points with permutation action by $\mu_2$. The Hodge realisation of $[\tilde{E},\mu_2]$ is the Hodge structure with monodromy
$$\left((\Q_{\pt}^{\hdg})^2,\left(\begin{array}{cc} 0 & 1 \\
                                              1 & 0 \end{array}\right)\right),$$
                                              which may be decomposed as a direct sum
                                              $$(\Q_{\pt}^{\hdg},\id) \oplus (\Q_{\pt}^{\hdg}, -\id).$$
                                              Thus, the class $1- [\tilde{E},\mu_2]$ maps to the class of the Hodge structure with monodromy $H = (\Q_{\pt}^{\hdg},-\id,0)$. 
                                              
                                              We may conclude that  the equality 
                                              $$(1-[\tilde{E},\mu_2])\ast (1-[\tilde{E},\mu_2]) = \LL$$
                                              from section \ref{sect.TSexample} of chapter \ref{grothrings} becomes the equality
                                              $$H\twtimes H  = \Q_{\pt}^{\hdg}(-1)$$
          in $K_0(\mhm^{\mon}_{\pt})$. It is consistent with example                 \ref{twtimesexample}, where we actually proved the two sides were isomorphic. 
                                             
\end{example}

\begin{remark} The above Hodge realisation induces the classical Hodge realisation $\chi_S^{\hdg}:\M_S\to K_0(\mhm_S)$ (see e.g. \cite{CNS}, Chapter 1, paragraph 3.3) on the corresponding rings without monodromy. In what follows, we are also going to use this realisation.
\end{remark}
\subsection{Compatibility with symmetric products}
\begin{lemma}\label{chihdgsymprod} Let $p:Y\to X$ be a quasi-projective variety over $X$. Then $$\chi^{\hdg}_{S^nX}(S^nY) = S^n(\chi^{\hdg}_XY)$$ in $D^b(\mhm_{S^nX})$. 
\end{lemma}
\begin{proof} The relation we want is $$(S^np)_!\Q_{S^nY}^{\hdg} = S^{n}(p_!\Q_Y^{\hdg}),$$
which is exactly equation (1.11) in \cite{CMSSY}.
\end{proof}

\begin{lemma} The diagram
$$\xymatrix{(\M_X,+) \ar[r]^-S \ar[d] &\left(\prod_{n\geq 1}\M_{S^nX}, \cdot \right)\ar[d]\\
(K_0(\mhm_X),+) \ar[r]^-{S_X^{\hdg}} & \left(\prod_{n\geq 1}K_0(\mhm_{S^nX}),\cdot\right)
}$$
where the vertical arrows are given by the Hodge realisation morphisms, commutes. \end{lemma}

\begin{proof} We checked it is true for classes of quasi-projective varieties $Y$ over~$X$ in lemma \ref{chihdgsymprod}, and such classes generate $\kvar_X$. Moreover, note that for any $M\in D^b(\mhm_X)$, any $n\geq 1$ and any $k\in\Z$, we have $S^n(M(k)) = (S^nM)(nk)$ (recall that $M(k) = M\otimes \Q^{\hdg}_{\pt}(k)$ ). Thus, for any $\a\in\kvar_X$, any $k\in\Z$ and any $n\geq 1$, we have
\begin{eqnarray*} \chi_{S^nX}^{\hdg}(S^n(\LL^{-k}\a)) &= &\chi_{S^nX}^{\hdg}(\LL^{-kn}S^n\a)\\
                                               & = &(\chi_{S^nX}^{\hdg}(S^n\a))(kn)\\
                                               & = &S^n(\chi^{\hdg}_{X}(\a)(k))\\
                                               & = & S^n(\chi_{X}^{\hdg}(\LL^{-k}\a))
                                               \end{eqnarray*} 
                                               which shows that the diagram commutes also on the localisation. 
\end{proof}
\subsection{Grothendieck rings of Hodge modules over the affine line}
%Let $S$ be a complex variety. There are two natural $K_0(\mhm_{S}^{\mon})-$algebra structures on the group $K_0(\mhm_{\A^1_S})$. The first one is given by the pullback morphism
%$$\epsilon_S^*:(K_0(\mhm_{S}^{\mon}),\ast) \to (K_0(\mhm_{\A^1_S}^{\mon}),\ast)$$
%where $\epsilon_S:\A^1_S\to S$ is the structural morphism. 
%
%Denote by $\convv$ the product induced by the addition morphism $\add:\A^1_S\times_S\A^1_S\to \A^1_S$:
%$$\convv: K_0(\mhm_{\A^1_S}^{\mon})\times K_0(\mhm_{\A^1_S}^{\mon})\xrightarrow{\twtimes_S} K_0(\mhm_{\A^1_S\times_S\A^1_S}^{\mon})\xrightarrow{\add_!} K_0(\mhm_{\A^1_S}^{\mon})$$
There are two natural $K_0(\mhm_{X})-$algebra structures on the group $K_0(\mhm_{\A^1_{X}})$. The first one is given by the pullback morphism
$$(\epsilon_X)^*:K_0(\mhm_{X}) \to K_0(\mhm_{\A^1_{X}})$$
where $\epsilon_X:\A^1_{X}\to X$ is the structural morphism. 

Denote by $\convv$ the product induced by the addition morphism $\add:\A^2_X\to \A^1_{X}$:
$$\convv: K_0(\mhm_{\A^1_{X}})\times K_0(\mhm_{\A^1_{X}})\xrightarrow{\boxtimes_X} K_0(\mhm_{\A^2_X})\xrightarrow{\add_!} K_0(\mhm_{\A^1_{X}})$$
(see the end of section \ref{sect.mhmmonodromy} for the definition of $\boxtimes_X$ in the context of mixed Hodge modules). Moreover, we define $i_X:X\to \A^1_X$ and $i^2_X: X\to \A^2_X$ to be the morphisms induced by the inclusions $\{0\}\to \A^1_{\C}$ and $\{(0,0)\}\to \A^2_{\C}$, respectively.
\begin{lemma}
The functor $(i_X)_!$ induces a ring morphism
$$(i_X)_!: K_0(\mhm_{X}) \to (K_0(\mhm_{\A^1_{X}}),\convv),$$
endowing the ring $(K_0(\mhm_{\A^1_X}),\convv)$ with a $K_0(\mhm_X)$-algebra structure. 
\end{lemma}
\begin{proof} There is a commutative diagram
$$\xymatrix{X \ar[r]^{i^2_X} \ar[rd]^{i_X}& \A^2_X \ar[d]^{\add}\\
                              &\A^1_X}$$
                              which gives us for $M,M'\in D^b(\mhm_X),$
                              \begin{eqnarray*}(i_X)_!(M)\convv (i_X)_! (M') & = & (\add)_!((i_X)_!(M)\boxtimes_X (i_X)_!(M')) \\
        & = &                       (\add)_!(i^2_X)_! (M\otimes_XM') \\
                    &=&          (i_X)_!(M\otimes_X M')                                 
                              \end{eqnarray*}
 \end{proof}

\begin{lemma}\label{chihdgconvv} The Hodge realisation $\chi^{\hdg}_{\A^1_{X}}$ is a morphism 
$$\chi^{\hdg}_{\A^1_{X}} : (\M_{\A^1_{X}},\convv)\to (K_0(\mhm_{\A^1_{X}}),\convv)$$
of $\M_{X}$-algebras. 
\end{lemma}
\begin{proof} Let $\a,\b\in\M_{\A^1_{X}}$. Using the fact that the Hodge realisation commutes with pushforwards, pullbacks and exterior products (proposition \ref{chihdgprop}), we have  
\begin{eqnarray*}
 \chi_{\A^1_X}^{\hdg}(\a\convv \b) & = & \chi^{\hdg}_{\A^1_X}(\add_!(\a\boxtimes_X\b))\\
 & = & \add_{!} \left(\chi_{\A^2_X}^{\hdg}(\a\boxtimes_X\b)\right)\\
 & = & \add_!\left(\chi_{\A^1_X}^{\hdg}(\a)\boxtimes_X\chi_{\A^1_X}^{\hdg}(\b)\right)\\
 & = & \chi_{\A^1_X}^{\hdg}(\a)\convv\chi_{\A^1_X}^{\hdg}(\b).
\end{eqnarray*}
\end{proof}
\subsection{Compatibility with motivic vanishing cycles}
In section \ref{sect.vancycleshodgemodules} we defined a total vanishing cycle functor:
$$\vanhdg_X:\mhm_{\A^1_{X}}\to \mhm_{X}^{\tot}.$$
It induces a group morphism 
$$\Phi_X^{\hdg}: K_0(\mhm_{\A^1_{X}})\to K_0(\mhm_{X}^{\mon})$$
between the corresponding Grothendieck rings. 
\begin{prop}\label{HodgeTSgrothring} The morphism $\Phi^{\hdg}_X$ is a morphism of $K_0(\mhm_{X})$-algebras
$$\Phi_X^{\hdg}:(K_0(\mhm_{\A^1_{X}}),\convv)\to (K_0(\mhm_{X}^{\mon}),\ast).$$\index{Phihdg@$\Phi_X^\hdg$}
%such that for every mixed Hodge module $M$, 
%$$\Phi^{\hdg}_{\eps}(M) = \sum_{a\in X}\phi_{\id-a}^{\hdg}(M).$$
\end{prop}
\begin{proof} This is a direct consequence of the Thom-Sebastiani property for total vanishing cycles, proposition \ref{HodgeTS} .
\end{proof}

On the other hand, in chapter \ref{grothrings} we defined, for every variety $X$ over a field $k$ of characteristic zero, the motivic vanishing cycles measure
$$\Phi_X:\expp_X\to (\M_X^{\hat{\mu}},\ast).$$
Here, to be able to compare it with $\Phi^{\hdg}_X$, we are going to consider rather its composition 
$$\Phi'_{X}: 
 (\M_{\A^1_X},\convv) \to (\M_X^{\hat{\mu}},\ast)$$\index{phip@$\Phi'_X$}
 with the quotient morphism
$$(\M_{\A^1_X},\convv)\to \expp_{X}$$
(which is given, by definition, by sending to zero the elements $[\A^1_Y\to \A^1_X]$ for all morphisms $Y\to X$, the morphism $\A^1_Y\to \A^1_X$ being the identity on the $\A^1$--components, see the definition of Grothendieck rings with exponentials in chapter \ref{grothrings}, section \ref{subsect.grothrings}). Recall from property \ref{chihgdpropring} of lemma \ref{chihdgprop}  and lemma \ref{chihdgconvv} the multiplicative properties of the Hodge realisation morphisms.
\begin{prop}\label{commdiagramvancycles} The diagram
$$\xymatrix{(\M_{\A^1_{X}},\convv) \ar[r]^{\Phi'_{X}}\ar[d]^{\chi^{\hdg}} & (\M_{X}^{\hat{\mu}},\ast)\ar[d]^{\chi^{\hdg}}\\
(K_0(\mhm_{\A^1_{X}}),\convv)\ar[r]^{\Phi_X^{\hdg}} &(K_0(\mhm_{X}^{\mon}),\ast)
            }$$
            commutes. 
\end{prop}
\begin{proof} Proposition 3.17 in \cite{GLM} shows compatibility between the motivic nearby fibre morphism and the nearby fibre functor on mixed Hodge modules. As for the motivic vanishing cycle morphism, as noted just after notation 3.9 in \cite{DL01}, the motivic vanishing cycles $\scr{S}^{\phi}_f$  for $f:X\to \A^1$ as they were defined by Denef and Loeser should be seen as the motivic incarnation of $\phi_f[d-1]$ where $d$ is the dimension of $X$. With our notation, $\phi_f = (-1)^{d}\scr{S}^{\phi}_f$, so that our vanishing cycles should be the motivic incarnation of $\phi_f[-1]$, which is exactly the perverse sheaf underlying $\phi_f^{\hdg}$. 
\end{proof}

Recall that in section \ref{sect.compsymprodvanhdg} we defined a morphism
$$\overline{\add}_n:S^{n}(\A^1_{X})\to \A^1_{S^nX}$$
for every integer $n\geq 1$. 
\begin{cor}\begin{enumerate}[(a)]\item 
 For any $\a\in\M_{\A^1_X}$ and any integer $n\geq 1$, we have 
$$\chi^{\hdg}_{S^nX}\circ \Phi'_{S^nX}((\overline{\add}_n)_!(S^n\a)) = S^n(\chi^{\hdg}_X\circ \Phi'_X(\a)).$$
 \item  For any $\a\in\expp_X$ and any integer $n\geq 1$, we have 
$$\chi^{\hdg}_{S^nX}\circ \Phi_{S^nX}(S^n\a) = S^n(\chi^{\hdg}_X\circ \Phi_X(\a)).$$
\end{enumerate}
\end{cor}
\begin{proof} By proposition \ref{commdiagramvancycles}, to prove $(a)$, it suffices to prove
$$\Phi^{\hdg}_{S^nX}\circ \chi^{\hdg}_{\A^1_{S^nX}}((\overline{\add}_n)_!(S^n\a)) = S^n(\Phi_X^{\hdg}\circ \chi^{\hdg}_{\A^1_X}(\a)).$$
We have
\begin{eqnarray*}\Phi^{\hdg}_{S^nX}\circ \chi^{\hdg}_{\A^1_{S^nX}}((\overline{\add}_n)_!(S^n\a)) & = & \Phi^{\hdg}_{S^nX}\circ (\overline{\add}_n)_!\circ \chi^{\hdg}_{S^n(\A^1_X)}(S^n\a)\ \ \text{by proposition \ref{chihdgprop} }\\
& = & \Phi^{\hdg}_{S^nX}\circ (\overline{\add}_n)_!(S^n\chi^{\hdg}_{\A^1_X}(\a))\ \ \ \ \ \ \text{by lemma \ref{chihdgsymprod}}\\
& = & S^n(\Phi_X^{\hdg}\circ \chi^{\hdg}_{\A^1_X}(\a))\ \ \ \ \ \ \ \text{by proposition \ref{vancyclefunctorsymproducts}.}
\end{eqnarray*}
To prove $(b)$, take $\a\in\expp_X$, and, denoting by $q:\M_{\A^1_X}\to  \expp_X$ the quotient map, pick $\a'\in\M_{\A^1_X}$ such that $\a = q(\a')$. Applying $(a)$ to $\a'$, we have (recall $\Phi' = \Phi\circ q$)
$$\chi^{\hdg}_{S^nX}\circ \Phi_{S^nX}\circ q\circ (\overline{\add}_n)_!(S^n\a') = S^n(\chi^{\hdg}_X\circ \Phi_X(\a)).$$
It therefore remains to prove that $q\circ (\overline{\add}_n)_!(S^n\a') = S^n\a$. In other words, we want to show the commutativity of the diagram
\begin{equation}\label{qcircadddiag}\xymatrix{\M_{\A^1_X} \ar[r]^{S} \ar[d]^{q} & \M_{S^{\bullet}(\A^1_X)}\ar[d]^{q\circ (\overline{\add})_!} \\
\expp_{X}\ar[r]^S & \expp_{S^{\bullet}X}
}\end{equation}
(recall the group morphisms $S$ have been defined in chapter \ref{eulerproducts}, sections \ref{symprodclasses} and \ref{sect.locsymproducts}), where $\overline{\add}_! = \prod_{n\geq 1}(\overline{\add}_n)_!$. Let us start by checking that $q\circ (\overline{\add})_!$ is a group morphism. Let $\a = (\a_i)_{i\geq 1}$ and $\b = (\b_i)_{i\geq 1}$ be elements of $\M_{S^{\bullet}(\A^1_X)}$. We have 
$$\a\b = \left(\sum_{i=0}^n\gamma_!(\a_i\boxtimes \b_{n-i})\right)_{n\geq 1}$$
where %$\a_i\boxtimes \b_{n-i}$ is seen as an element of $\M_{S^{n}(\A^1_X)}$ via
$\gamma$ is the morphism $S^i(\A^1_X)\times S^{n-i}(\A^1_X)\to S^n(\A^1_X)$ induced by the identity $(\A^1_X)^i \times (\A^1_X)^{n-i} \to  (\A^1_X)^n$. To prove that
$$q\circ (\overline{\add})_!(\a\b) = q\circ (\overline{\add})_!(\a)q\circ (\overline{\add})_!(\b)$$
in $\expp_{S^{\bullet}X}$, it suffices to prove that for all $n\geq 1$ and all $i\in\{0,\ldots,n\}$, we have
\begin{equation}\label{qcircadd}q \circ (\overline{\add}_n)_!\gamma_!(\a_i\boxtimes \b_{n-i}) = \beta_!(q\circ (\overline{\add})_i(\a_i)\boxtimes q\circ (\overline{\add}_{n-i})_!(\b_{n-i}))\end{equation}
in $\expp_{S^{n}X},$  where $\beta:S^{i}X\times S^{n-i}X\to S^nX$ is the morphism induced by the identity $X^i\times X^{n-i}\to X^n$.  For this, consider the following diagram:
\begin{equation}\label{qcircadddiagram}\xymatrix{\M_{S^{i}(\A^1_X)}\times \M_{S^{n-i}(\A^1_X)}\ar[r]^-{\boxtimes}\ar[d]_{(\overline{\add}_i)_!\times (\overline{\add}_{n-i})_!} & \M_{S^{i}(\A^1_X)\times S^{n-i}(\A^1_X)}\ar[r]^-{\gamma_!}\ar[d]_-{(\overline{\add}_i\times\overline{\add}_{n-i})_!}&\M_{S^n(\A^1_X)}\ar[d]^{(\overline{\add}_n)_!}\\
\M_{\A^1_{S^iX}}\times \M_{\A^1_{S^{n-i}X}}\ar[r]^-{\boxtimes}\ar[d]_{q\times q}&\M_{\A^1_{S^iX}\times \A^1_{S^{n-i}X}}\ar[r]^-{(\beta\circ \add)_!}\ar[d]_{q\circ \add_!} & \M_{\A^1_{S^{n}X}}\ar[d]^q\\
\expp_{S^iX}\times \expp_{S^{n-i}X}\ar[r]^-{\boxtimes}& \expp_{S^iX\times S^{n-i}X} \ar[r]^-{\beta_!}&\expp_{S^nX}
}\end{equation}
Here $\beta$ denotes the morphism $S^{i}X\times S^{n-i}X\to S^{n}X$, as well as the morphism $\A^1_{S^{i}X\times S^{n-i}X}\to \A^1_{S^nX}$ it induces, and $\add$ refers to the morphism 
$$\A^1_{S^iX}\times \A^1_{S^{n-i}X}\to \A^1_{S^{i}X\times S^{n-i}X}$$
induced by the addition morphism on the $\A^1$-components. 

To prove (\ref{qcircadd}), it suffices to prove that this diagram is commutative. We do this square by square. The commutativity of the top left square comes from the fact that pushdowns commute with exterior products. The commutativity of the top right square comes from the commutativity of the square
$$\xymatrix{
(\A^1_X)^{i}\times (\A^1_X)^{n-i}\ar[r]^-{\id}\ar[d]^{\add_i\times \add_{n-i}} & (\A^1_X)^n\ar[d]^{\add_{n}}\\
\A^1_{X^i}\times \A^1_{X^{n-i}}\ar[r]^-{\add}&\A^1_{X^n}}$$
after taking quotients by the appropriate permutation actions. For the bottom left square, by bilinearity, it suffices to check commutativity for effective elements. For any morphisms $Y\xrightarrow{f}\A^1_{S^iX}$, and $Z\xrightarrow{g}\A^1_{S^{n-i}X}$, we have, by definition, 
\begin{eqnarray*}q\circ \add_!([Y\xrightarrow{f}\A^1_{S^iX}]\boxtimes[Z\xrightarrow{g}\A^1_{S^{n-i}X}])& = &q\circ \add_!([Y\times Z\xrightarrow{ f\times g}\A^1_{S^{i}X}\times \A^1_{S^{n-i}X}])\\
& = & [Y\times Z, \add\circ (f\times g)] \\
&=& [Y,f]\boxtimes [Z,g]\\
& = & q([Y\xrightarrow{f}\A^1_{S^iX}])\boxtimes q([Z\xrightarrow{g}\A^1_{S^{n-i}X}])\end{eqnarray*}
in $\expp_{S^{i}X\times S^{n-i}X}$.
The commutativity of the last square comes from the fact that $q$ commutes with~$\beta_!$. 

We come back to the proof of the main statement. Since all maps involved are group morphisms, it suffices to prove commutativity of diagram (\ref{qcircadddiag}) for effective elements. Let therefore $f:Y\to \A^1_X$ be a morphism, inducing $S^nf:S^nY\to S^n(\A^1_X)$, as well as $f^{(n)} = \overline{\add}_n\circ S^nf:S^nY\to \A^1_{S^nX}$. According to the definition of symmetric products of varieties with exponentials, we have, for all $n\geq 1$, 
\begin{eqnarray*} q\circ (\overline{\add}_n)_!(S^n[Y\xrightarrow{f}\A^1_X])& =& q\circ (\overline{\add}_n)_!([S^nY \xrightarrow{S^nf}S^{n}(\A^1_X)])\\
& = & q([S^nY \xrightarrow{\overline{\add}_n\circ S^nf} \A^1_{S^nX}])\\
& = & [S^nY,  f^{(n)}]\\
& = & S^n([Y,f])\\
& = & S^n(q([Y\xrightarrow{f}\A^1_X]))
\end{eqnarray*} 
in $\expp_{S^nX}$, whence the result.\end{proof}
\section{Weight filtration on Grothendieck rings of mixed Hodge modules}\label{sect.weightmhm}

\subsection{The weight filtration}\label{sect.weightfiltrationmhm}\index{weight filtration!on $K_0(\mhm_S^{\mon})$}
For any integer $n$, denote by $W_{\leq n}K_0(\mhm_S^{\mon})$ \index{WKM@$W_{\leq n}K_0(\mhm_S^{\mon})$} the subgroup of $K_0(\mhm_S^{\mon})$ generated by classes of pure Hodge modules $(M,\id,N)$ (i.e. with trivial semi-simple monodromy) of weight at most $n$ and by classes of pure Hodge modules with monodromy $(M,T_s,N)$ of weight at most $n-1$. 

\begin{remark}\label{eigenspacedecomp} The monodromy $T_s$ is an automorphism of finite order, so that a pure Hodge module $(M,T_s,N)$ over $S$ of weight $m$ with monodromy decomposes into $M = M^{0}\oplus M^{\neq 0}$ where $M^0 = \ker(T_s - \id)$ and $M^{\neq 0} = \ker(T_s^{k-1} + \ldots + T_s + \id)$, where $k$ is minimal such that $T_s^k = 1$. This Hodge module is an element of $W_{\leq m}K_0(\mhm_S^{\mon})$ if $M^{\neq 0} = 0$, and of $W_{\leq m+1}K_0(\mhm_S^{\mon})$ otherwise. 
\end{remark}

We have the following compatibility with respect to pushdowns, pullbacks and exterior products:
\begin{lemma}\label{weightpushpull}
\begin{enumerate}\item\label{pushpull}  Let $f:Y\to X$ be a morphism of complex varieties with fibres of dimension $\leq d$, then for all integers $n$, we have
$$f_{!}\left(W_{\leq n} K_0(\mhm_Y^{\mon})\right)\subset W_{\leq n + d}K_0(\mhm_X^{\mon})$$
and $$f^*\left(W_{\leq n} K_0(\mhm_X^{\mon})\right)\subset W_{\leq n + d}K_0(\mhm_Y^{\mon}).$$
\item \label{exterior}Let $X$ and $Y$ be complex varieties. Then for all integers $n,m$ we have
$$W_{\leq n}K_0(\mhm^{\mon}_X)\twtimes W_{\leq m}K_0(\mhm^{\mon}_Y)\subset W_{\leq n+m} K_0(\mhm^{\mon}_{X\times Y}).$$
\end{enumerate}
%In particular, for every $\a\in\kvar_X$, we have
%$$w_{\pt}((a_X)_!\a)\leq w_X(\a) + \dim X.$$    %A replacer au bon endroit
\end{lemma}
\begin{proof} Let $M$ be a pure Hodge module of weight at most $n$ (resp. $n-1$). Then, since the functor~$f_!$ does not increase weights, $f_{!}M$ is a complex of weight $\leq n$ (resp. $\leq n-1$) which belongs to $D^{\leq d}(\mhm_X)$ by lemma~\ref{cohamplitude}, so
$$[f_!M] = \sum_{i\leq d}(-1)^i[\mathcal{H}^if_!M]$$
is a sum of Hodge modules of weight $\leq n + d$ (resp. $\leq n-1 + d$). If the monodromy on~$M$ is trivial, then it is also trivial on all mixed Hodge modules $\mathcal{H}^if_!M$, so  in any case, we have $[f_!M]\in W_{\leq n+d}K_0(\mhm_S^{\mon})$. The proof is the same for $f^{*}$.

Let $(M_1,T_{s,1},N_1)\in W_{\leq n}K_0(\mhm_X)$ and $(M_2,T_{s,2},N_2)\in W_{\leq m}K_0(\mhm_X)$ be two pure Hodge modules with monodromy. By remark \ref{eigenspacedecomp}, it suffices to treat the following cases :
\begin{itemize} \item $M_1 = M_1^0$ is of weight $n$, $M_2 = M_2^0$ is of weight $m$;
\item $M_1 = M_1^0$ is of weight $n$, $M_2 = M_2^{\neq 0}$ is of weight $m-1$;
\item $M_1 = M_1^{\neq 0}$ is of weight $n-1$ and $M_2 = M_2^{\neq 0}$ is of weight $m-1$.
\end{itemize}
By the definition of the weight filtration on $M_1\twtimes M_2$, in the first case $M_1\twtimes M_2$ is pure of weight $m+n$, with trivial monodromy, and in the second case, pure of weight $m+n-1$. Thus, in both cases, $M_1\twtimes M_2$ is an element of $W_{\leq m+n}K_0(\mhm_S^{\mon})$. 
This leaves us with the third case. For any $\alpha,\beta\in(-1,0]$ such that $\exp(-2i\pi\al)$ (resp. $\exp(-2i\pi\beta)$) is an eigenvalue of $T_{s,1}$ (resp. $T_{s,2}$), the complex number $\exp(-2i\pi(\al + \be))$ is an eigenvalue of monodromy on $M_1\twtimes M_2$, and  $(M_1\twtimes M_2)^1$ is non-zero if and only if there exist such $\al,\beta$ with $\al + \beta = -1$. The weight filtration on $M_1\twtimes M_2$ is such that $(M_1\twtimes M_2)^{\neq 1}$ is of weight $m+n-2$, and $(M_1\twtimes M_2)^{1}$ is of weight $m+n$, so that the Hodge module $M_1\twtimes M_2$ is an element of $W_{\leq m + n}K_0(\mhm_S^{\mon})$.
\end{proof}

%\begin{lemma} Let $i:Z\to X$ be a closed immersion. Then for all integers $n\in \Z$,
%$$i^*(W^{\leq n}K_0(\mhm_X))\subset W^{\leq n + \dim Z - \dim X}K_0(\mhm_Z).$$
%\end{lemma}
%\begin{proof} Let $M$ be a pure Hodge module of weight $k$ on $X$. By remark 4.9 Assume first that $M$ is supported inside some proper closed subset $Y$ of $X$. If $Y\cap Z = \emptyset$, then $i^*M = 0$ and we are done. Otherwise, we replace $X$ by $Y$ and $Z$ by $Z\cap Y$. 
%
%Assume now that the support of $M$ is all of $X$. Then there is a smooth open subset $U$ of $X$ above which $M$ corresponds to a variation of pure Hodge structures of weight $k-\dim X$ Assume  
%\end{proof}
For any element $\a\in K_0(\mhm_S^{\mon})$, we put
$$w_S(\a):= \inf \{n,\ \a\in W_{\leq n}K_0(\mhm_S^{\mon})\},$$
which defines a function $w_S:K_0(\mhm_S^{\mon})\to \Z\cup \{-\infty\}$. \index{weight function!on $K_0(\mhm_S^{\mon})$}\index{wS@$w_S$ (weight function)} Lemma \ref{weightpushpull} gives us the following:
\begin{lemma}\label{weight.properties} For any complex varieties $S$ and $T$, any $\a,\a'\in K_0(\mhm_S^{\mon})$ and $\mathfrak{b}\in K_0(\mhm_T^{\mon})$ the weight function satisfies the following properties:
\begin{enumerate}[(a)]\item\label{weightofzero} $w_S(0) = -\infty$
\item \label{weight.sum}$w_S(\a + \a')\leq \max\{w_S(\a),w_S(\a')\}$, with equality if $w_S(\a)\neq w_S(\a')$. 
\item \label{weight.extproducts} $w_{S\times T}(\a\,\twtimes\,\mathfrak{b})\leq w_S(\a) + w_T(\mathfrak{b}).$
%\item\label{weight.products} $w_{S}(\a\a')\leq w_S(\a) + w_S(\a') - \dim S.$
\item \label{weight.push} If $f:S\to T$ is a morphism with fibres of dimension $\leq d$, then
$$w_T(f_!(\a))\leq w_S(\a) + d.$$
\item  \label{weight.pull} If $f:S\to T$ is a morphism with fibres of dimension $\leq d$, then
$$w_S(f^*(\b))\leq w_T(\b) + d.$$
%\item \label{weight.products} $w_{S}(\a\otimes \a') \leq w_S(a) + w_S(\a')$
\item\label{elemcomparison} If $M^{\bullet}\in D^{\leq a}(\mhm_S)$ is a complex of mixed Hodge modules of weight $\leq n$, then
$w_S([M^{\bullet}])\leq a + n$. Equality is achieved if and only if $\gr^{W}_{a+n}\mathcal{H}^a(M^{\bullet}) \neq 0$. 
\end{enumerate}
\end{lemma}
\begin{remark}\label{absweightremark} As a special case of (\ref{weight.push}), denoting by $w$ the weight function $w_{\pt}$ on $K_0(\mhm^{\mon}_{\pt})$ and by $a_S:S\to \pt$ the structural morphism, we have
$$w((a_S)_!\a)\leq w_S(\a) + \dim S.$$
\end{remark}
\subsection{Weights of symmetric powers of mixed Hodge modules}
Recall that in section \ref{mhmsymproducts} we defined a group $K_0(\mhm_{S^{\bullet}X}^{\mon})$ and a morphism
$$S_X^{\hdg}: K_0(\mhm^{\mon}_X) \to K_0(\mhm^{\mon}_{S^{\bullet}X})$$
sending the class of a mixed Hodge module $M$ to $(S^nM)_{n\geq 1}$. Our goal here is to show that $S_X^{\hdg}$ behaves well with respect to the weight filtration from section \ref{sect.weightfiltrationmhm}. 
We define the following natural filtration on $K_0(\mhm^{\mon}_{S^{\bullet}X})$:
$$W_{\leq d}K_0(\mhm^{\mon}_{S^{\bullet}X}) := \prod_{n\geq 1} W_{\leq nd}K_0(\mhm^{\mon}_{S^{n}X})\subset \prod_{n\geq 1}K_0(\mhm^{\mon}_{S^nX}).$$
We have the following properties:

\begin{prop}\label{propweightsymprod} \begin{enumerate}\item For every $d$, $W_{\leq d}K_0(\mhm^{\mon}_{S^{\bullet}X}) $ is a subgroup of the group $K_0(\mhm^{\mon}_{S^{\bullet}X})$. 
\item  For any integer $d$, $S_X^{\hdg}(W_{\leq d}K_0(\mhm^{\mon}_X))\subset W_{\leq d}K_0(\mhm^{\mon}_{S^{\bullet}X}).$
\item \label{mainweightsymprodineq} For every $\a\in K_0(\mhm^{\mon}_X)$ and every integer $n\geq 0$, we have $w_{S^nX}(S^n\a)\leq nw_X(\a)$. \index{weight!of symmetric power}
\end{enumerate}
\end{prop}

\begin{proof} 1.  Let $a = (a_n)_{n\geq 1}$ and $b = (b_n)_{n\geq 1}$ be two elements of $W_{\leq d}K_0(\mhm^{\mon}_{S^{\bullet}X})$. Then for all $n\geq 1$ and all $i\in\{0,\ldots,n\}$, by property \ref{exterior} of lemma \ref{weightpushpull} we have $a_i\twtimes b_{n-i}\in K_0(\mhm^{\mon}_{S^iX\times S^{n-i}X})$ of weight $\leq id + (n-i)d = nd$. The map $S^iX\times S^{n-i}X\to S^nX$ has fibre dimension zero, so by property \ref{pushpull} of lemma \ref{weightpushpull} the same estimate is valid for $a_i\twtimes b_{n-i}$ seen as an element of $K_0(\mhm^{\mon}_{S^nX})$.

2.  Let $M\in \mhm^{\mon}_X$ be a pure Hodge module of weight $d$. Then $\bigtwtimes{n} M$ is a pure Hodge module of weight $nd$, and denoting by $p$ the quotient morphism $X^n\to S^nX$, the complex $p_!(\bigtwtimes{n} M)$ is of weight $\leq nd$ by property \ref{pushpull} of lemma \ref{weightpushpull}, since $p$ has fibres of dimension $0$. Finally, $S^nM$ is obtained as a subobject of $p_!(\bigtwtimes{n} M)$, so its weight is $\leq nd$ again. Statement 3 is a direct consequence of 2.
\end{proof}
%\begin{proof} Properties (\ref{weightofzero}) and (\ref{weight.sum}) are immediate. Property (\ref{weight.extproducts}) comes from lemma \ref{weightpushpull}, part \ref{exterior}. 
%\end{proof}

\section{Weight filtration on Grothendieck rings of varieties}\label{sect.weightfiltrationkvar}
In this section, we are going to use the previously defined weight filtration to define a notion of weight on Grothendieck rings of varieties with exponentials. For this, we are going to use the motivic vanishing cycles measure from chapter \ref{grothrings}, and the Hodge realisation from section \ref{sect.hodgereal}. 
\subsection{The weight filtration and completion}
Let $S$ be a complex variety. Recall that $\Phi_S:\expp_S\to (\M_S^{\hat{\mu}},\ast)$ is the motivic vanishing cycles measure from theorem \ref{motmeasuremain} of chapter \ref{grothrings}. 
\begin{definition}\label{weightfiltration}\begin{enumerate}\item The weight filtration on the ring $\M_S^{\hat{\mu}}$ is given by 
$$W_{\leq n}\M_S^{\hat{\mu}} := (\chi^{\hdg}_S)^{-1}(W_{\leq n}K_0(\mhm_{S}^{\mon}))$$
for every $n\in\Z$. \index{WnM@$W_{\leq n}\M_S^{\hat{\mu}}$}
The weight function on $\M_S^{\hat{\mu}}$, again denoted by $w_S$, is the composition
$$\M_S^{\hat{\mu}} \xrightarrow{\chi_S^{\hdg}} K_0(\mhm_{S}^{\mon}) \xrightarrow{w_S} \Z.$$
\item The weight filtration on the ring $\expp_S$ is given by 
$$W_{\leq n}\expp_S := (\chi^{\hdg}_S\circ\Phi_S)^{-1}(W_{\leq n}K_0(\mhm_{S}^{\mon}))$$\index{WnE@$W_{\leq n}\expp_S$}
for every $n\in\Z$. The weight function on $\expp_S$, again denoted by $w_S$, is the composition 
$$\expp_S \xrightarrow{\Phi_S} \M_S^{\hat{\mu}} \xrightarrow{\chi_S^{\hdg}} K_0(\mhm_{S}^{\mon}) \xrightarrow{w_S} \Z.$$
\end{enumerate}
\end{definition}
\index{weight function!on $\M_S^{\hat{\mu}}$}\index{weight function!on $\expp_S$}\index{wS@$w_S$ (weight function)}
\begin{remark}\label{weight.grothringprops} Properties $(\ref{weightofzero})-(\ref{weight.pull})$ of lemma \ref{weight.properties} remain true for $w_S$ on $\M_S^{\hat{\mu}}$ or $\expp_S$. Indeed, this is obvious for $(\ref{weightofzero})$, for $(\ref{weight.push})$ and $(\ref{weight.pull})$ it follows from the fact that the Hodge realisation commutes with pushforwards and pullbacks, and for $(\ref{weight.sum})$ it comes from the fact that $\chi^{\hdg}_S$ and~$\Phi_S$ are group morphisms. Property $(\ref{weight.extproducts})$ comes from the Thom-Sebastiani property for $\Phi_S$, and from the fact that $\chi^{\hdg}$ is compatible with twisted exterior products.
\end{remark}
\begin{remark} Both these definitions induce the same weight filtration $(W_{\leq n}\M_S)_{n\in\Z}$ and the same weight map $w_S:\M_S\to \Z$ on the localised Grothendieck ring $\M_S$, because the restriction of $\Phi_S$ to $\M_S$ is the inclusion $\M_S\to \M_S^{\hat{\mu}}$. 
\end{remark}
\begin{notation} For a variety $X$ over $S$, we use $w_S(X)$ as a shorthand for $w_S([X])$. We denote by $w$ the weight function for $S = \spec \C$. % (the class of~$X$ with the zero morphism to $\A^1$). This is the same as $w_S([X,\1])$ (the class of $X$ with trivial $\hat{\mu}$-action) because $\Phi_S([X,0]) = [X,\1]$.
\end{notation}

\begin{definition} We define the completion of the ring $\expp_S$ with respect to the weight topology as
$$\widehat{\expp_{S}} = \varprojlim_n \expp_{S}/  W_{\leq n} \expp_S. $$ 
\end{definition}\index{expMh@$\widehat{\expp_{S}}$}\index{completion!weight topology}\index{weight filtration!completion}
\subsection{Weights of symmetric products}
\begin{lemma}\label{weightssymproducts} Let $I$ be a set and let $\pi = (n_i)_{i\in I}\in\N^{(I)}$. Let $X$ be a complex variety, and  let $\scr{A}=(\a_i)_{i\in I}$ be a family of elements of $\expp_X$. Then 
$$w_{S^{\pi}X}(S^{\pi}\scr{A})\leq \sum_{i\in I}n_iw_{X}(\a_i).$$
\end{lemma}\index{weight!of symmetric product}
\begin{proof} Recall that by definition, $S^{\pi}\scr{A}$ is the element of $\expp_{S^{\pi}X}$ obtained by pulling back the product $\prod_{i\in I} S^{n_i}\a_i \in \expp_{\prod_{i\in I}S^{n_i}X}$ along the open immersion $j:S^{\pi}X\to \prod_{i\in I}S^{n_i}X$. By property (\ref{weight.extproducts}) of lemma \ref{weight.properties} and property \ref{mainweightsymprodineq} of proposition \ref{propweightsymprod} we have 
$$w_{\prod_{i\in I}S^{n_i}X}\left(\bboxtimes_{i\in I}S^{n_i}\a_i\right)\leq \sum_{i\in I}w_{S^{n_i}X}(S^{n_i}\a_i)\leq \sum_{i\in I}n_iw_X(\a_i).$$
Applying property (\ref{weight.push}) of lemma \ref{weight.properties} to $j$, which has fibre dimension 0, we get the result. 
\end{proof}
\subsection{Weight and dimension}
For effective classes in the Grothendieck ring of varieties, weight and dimension are closely linked, as shown in the following lemma:

\begin{lemma}\label{weightdimension} Let $S$ be a complex variety and $X$ a variety over $S$. One has the equality
$$w_S(X) = 2\dim_SX + \dim S.$$\index{weight!vs. dimension}
\end{lemma}
\begin{proof} We are going to denote by $f:X\to S$ the structural morphism of $X$, and by $d$ the relative dimension $\dim_SX$, that is, the supremum of the dimensions of the fibres of $f$. Since the functors $a_X^*$ and $f_!$ do not increase weights, the complex $f_!\Q_X^{\hdg} =f_!a_X^*\Q_{\pt}^{\hdg}$ is of weight $\leq 0$. Moreover, we see by lemma \ref{cohamplitude} that $f_!\Q_X^{\hdg}$ is an object of $D^{\leq 2d + \dim S}(\mhm_S)$.  By property (\ref{elemcomparison}), it suffices to prove that the top cohomology $\mathcal{H}^{2d + \dim S}(f_!\Q_X^{\hdg})$ has non-zero graded part of weight $2d + \dim S$. By \cite{Saito89}~1.20, we may write the Leray spectral sequence for $a_X = a_S\circ f$: for any $M^{\bullet}\in D^b(\mhm_X)$,
$$ \mathcal{H}^p (a_S)_! \left(\mathcal{H}^qf_{!}M^{\bullet}\right)\Longrightarrow \mathcal{H}^{p+q}((a_X)_! M^{\bullet}).$$
Applying this with $M^{\bullet} =  \Q_X^{\hdg}$, $p= \dim S$ and $q=2d + \dim S$, and recalling that the cohomology of the complex of mixed Hodge structures $(a_X)_!\Q_X^{\hdg}$ is exactly the cohomology with compact supports of $X$ with coefficients in $\Q$ with its standard Hodge structure, we have
$$\mathcal{H}^{\dim S}(a_S)_!\left( \mathcal{H}^{2d + \dim S} f_!\Q_X^{\hdg}\right) \Longrightarrow H^{2 \dim X}_c(X,\Q).$$
 If $\gr^W_{2d + \dim S}\mathcal{H}^{2d + \dim S} f_!\Q_X^{\hdg}=0$, then  the graded part of weight $2 \dim X$ of the left-hand side is zero. But the right-hand side is a sub-object of a quotient of the left-hand side, and therefore its graded part of weight $2\dim X$ should be zero as well, which it is not: it is classical that $H^{2\dim X}_c(X,\Q)$ is pure of weight $2\dim X$, isomorphic to $\Q_{\pt}^{\hdg}(-\dim X)^{r}$ where~$r$ is the number of irreducible components of~$X$. 
\end{proof}
As a consequence, we have
\begin{lemma} \label{weightdimensionnoneffective}Let $S$ be a complex variety and $\a$ an element of $\M_S^{\hat{\mu}}$. Then
$$w_S(\a)\leq 2\dim_S\a + \dim S.$$\index{weight!vs. dimension}
\end{lemma}
\begin{proof} We may assume $\a$ is of the form $\LL^{-m}([X]-[Y])$ for some $S$-varieties $X$ and $Y$ with $\hat{\mu}$-actions. Assume moreover that $\max\{\dim_SX,\dim_SY\}$ is minimal, so that $$\dim_S\a = \max\{\dim_SX,\dim_SY\}-m.$$ 
Then, using lemma \ref{weight.properties} and the fact that $\chi^{\hdg}_{\pt}(\LL^{-m}) = \Q^{\hdg}_{\pt}(m)$ is of weight $-2m$, we have
\begin{eqnarray*}w_S(\a)&\leq & w_{\pt}(\LL^{-m}) + w_S([X]-[Y]) \\
                        &\leq &-2m + \max\{w_S(X),w_S(Y) \}
                        \end{eqnarray*}
                        By lemma \ref{weightdimension}, we therefore get
                        $$w_S(\a)\leq -2m + 2\max\{\dim_S(X),\dim_S(Y)\} + \dim S$$
                        whence the result.
\end{proof}
We may therefore deduce the triangular inequality for the weight topology:
\begin{lemma}[Triangular inequality for weights]\label{triangularwt} Let $S$ be a variety over $\C$, $X$ a variety over $S$ and $f:X\to \A^1_{\C}$ a morphism. Then
$$w_S([X,f])\leq w_S(X).$$ \index{triangular inequality!for weights}\index{weight!triangular inequality}
\end{lemma}
\begin{proof}
By lemmas \ref{triangulardim} and \ref{weightdimension}  we have
$$w_S([X,f])\leq 2\dim_S(\Phi_S([X,f])) + \dim S \leq 2\dim_SX + \dim S = w_S(X).$$
\end{proof} 
The following property, which follows from our discussion of the trace morphism (lemma \ref{traceprop} and remark \ref{tracerem}) and states that there is a drop in weights for certain simple non-effective classes, will be very important to us:
\begin{lemma}[Cancellation of maximal weights]\label{weightcancellation} Let $S$ be a complex variety and $p:X\to S$, $q:Y\to S$ morphisms with fibres of constant dimension $d\geq 0$, with $X$ and $Y$ irreducible. Then
$$w_S([X\xrightarrow{p} S] - [Y\xrightarrow{q} S]) \leq 2 d + \dim S -1.$$\index{cancellation of maximal weights}
\end{lemma}
\begin{proof} The classes $[X]$ and $[Y]$ are of weights $\leq 2d + \dim S$, and according to remark \ref{tracerem} the graded parts of weight exactly $2d + \dim S$ of the corresponding  complexes of mixed Hodge modules cancel out. 
\end{proof}

\section{Convergence of power series}\label{sect.powerseriesconv}
\subsection{Radius of convergence}
Recall that for a classical series $\sum_{i\geq 0}a_iz^i$, the radius of convergence is given by $$\left(\limsup \left(|a_i|\right)^{\frac{1}{i}}\right)^{-1}.$$ Analogously, in our setting, we have:
\begin{definition} Let $F(T) = \sum_{i\geq 0} X_i T^i\in\expp_X[[T]]$. The radius of convergence of $F$ is defined by 
$$\sigma_F = \limsup_{i\geq 1} \frac{w_X(X_i)}{2i}.$$ 
We say that $F$ converges for $|T|<\LL^{-r}$ if $r\geq \sigma_F$. 
\end{definition}\index{radius of convergence}\index{sigmaF@$\sigma_F$, radius of convergence}

When $F$ converges for $|T|<\LL^{-r}$, it converges also for $|T|<\LL^{-r'}$ for any $r'>r$.  The subset of power series converging for $|T| < \LL^{-r}$ is a subring of $\expp_{X}[[T]]$. 
\begin{remark} If $r>\sigma_F$,  there is some $i_0$ such that for all $i\geq i_0$, 
$\frac{w_X(X_i)}{2i} < r,$ which means that the set of integers
$\{w_X(X_i)-2ri,\ i\geq 0\}$
is bounded from above.  Conversely, if this set is bounded from above for some $r$, then we may conclude that $r \geq \sigma_F$, that is, $F$ converges for $|T|<\LL^{-r}$. Thus, in general, we are going to prove that a series converges by finding a linear bound for $w_X(X_i)$. 

However, one does not in general have $\{w_X(X_i)-2\sigma_F i,\ i\geq 0\}$ bounded from above: see for example the series $\sum_{i\geq 0}\LL^{i + \lceil\sqrt{i}\rceil}T^i$. %This being said, in the cases we are going to encounter in this paper, the radius of convergence will indeed correspond to the minimum $r$ such that $w(X_i)-2ri$ is bounded from above, and will be determined by seeking linear bounds for $w(X_i)$. 
\end{remark}

If $F(T)$ converges for $|T|<\LL^{-r}$, then for any element $\a\in\expp_{\C}$ such that $w(\a) < -2r$, $F(\a)$ exists as an element of $\widehat{\expp}_X.$ In particular, $F(\LL^{-m})$ exists as an element of $\widehat{\expp}_{X}$ if $m>r$.

\begin{example} Let $X$ be a complex quasi-projective variety, and consider $Z_X(T) = \sum_{i\geq 0} [S^i X]T^i\in\expp_{\C}[[T]]$ its Kapranov zeta function. We have
$$w(S^iX) = 2i\dim X$$
for all $i\geq 0$, so that the radius of convergence of $Z_X(T)$ is $\dim X$. \index{radius of convergence!Kapranov's zeta function}
\end{example}
\subsection{A convergence criterion}

\index{Euler product!convergence}\index{convergence criterion}
\begin{prop}\label{convergence} Assume $F(T) = 1 + \sum_{i\geq 1} X_iT^i\in \expp_X[[T]]$ is such that there exists an integer $M\geq 0$ and real numbers $\epsilon >0$, $\al < 1$ and $\beta$ such that
\begin{itemize}\item for all $i\in\{1,\ldots,M\}$,\ $w_X(X_i) \leq (i-\frac12 -\epsilon)w(X)$
\item for all $i\geq M+1$, $w_X(X_i) \leq (\alpha i+ \beta - \frac12)  w(X).$
\end{itemize}
Then there exists $\delta >0$ such that the Euler product $\prod_{v\in X}F_v(T)\in\expp_{\C}[[T]]$
\begin{itemize}
\item  converges for $|T|<\LL^{-\frac{w(X)}{2}\left(1-\delta + \frac{\beta}{M+1}\right)}$ 
\item for any $0\leq \eta< \delta$, takes non-zero values for $|T|\leq \LL^{-\frac{w(X)}{2}\left(1-\eta + \frac{\beta}{M+1}\right)}$ (that is, for every $\a\in \expp_{\C}$ such that $w(\a)< -w(X)\left(1 -\eta + \frac{\beta}{M+1}\right)$).
\end{itemize}
\end{prop}
\begin{proof} Let $n\geq 1$ be an integer, and $\pi = (n_i)_{i\geq 1}$ a partition of $n$. Then we have
\begin{eqnarray*}w (S^{\pi}\scr{X}) &\leq & w_{S^{\pi}X}(S^{\pi}\scr{X}) + \dim(S^{\pi}X) \ \ \ \ \ \ \ \ \text{by remark \ref{absweightremark}}\\
									&\leq &\sum_{i\geq 1}n_iw_X(X_i) + \frac{1}{2}\sum_{i\geq 1}n_i w(X)\ \ \ \ \ \ \text{by lemmas \ref{weightssymproducts} and \ref{weightdimension}} \\
                                     & \leq  &\sum_{i=1}^Mn_iiw(X) - \epsilon\sum_{i=1}^Mn_iw(X)+ \sum_{i\geq M+1}\alpha in_i w(X) +\beta w(X)\sum_{i\geq M+1}n_i \\
                                     &\leq & \sum_{i=1}^Mn_iiw(X) - \frac{\epsilon}{M}\sum_{i=1}^M n_iiw(X)  + \alpha \sum_{i\geq M+1}in_i w(X) + \frac{\beta w(X)}{M+1}\sum_{i\geq M+1}in_i \\
                                     & = & \left(1-\frac{\epsilon}{M}\right)                                     \sum_{i=1}^Mn_iiw(X)+ \left(\alpha + \frac{\beta}{M+1}\right) \sum_{i\geq M+1}n_ii w(X)\\
                                     & \leq & \left(1-\delta + \frac{\beta}{M+1}\right) nw(X)
                                     \end{eqnarray*}
                                     where $1-\delta =  \max\{1-\frac{\epsilon}{M},\alpha\}<1$ (in the case when $M = 0$, we put $1-\delta = \alpha$). The desired convergence follows. Moreover, one sees that for $n\geq 1$, any $0\leq \eta < \delta$ and any $\a\in \expp_{\C}$ such that $w(\a)\leq -w(X)\left(1 - \eta + \frac{\beta}{M+1}\right)$, we have
                                     $$w(S^{n}\scr{X}\a^n) \leq -(\delta-\eta) n w(X) < 0,$$ so the value of the product at $\a$ is equal to 1 plus some terms of negative weight: it is therefore non-zero. 

\end{proof}

\begin{example} Let $Z_X(T) = \sum_{i\geq 0}[S^iX]T^i$ be Kapranov's zeta function for some quasi-projective variety $X$. Then $Z_X(T) = \prod_{v\in X} F_v(T)$ where
$$F(T) = \sum_{i\geq 1}T^i \in \M_{X}[[T]],$$
that is, every coefficient is equal to $1 = [X]\in\M_X$. Take $M=0$, $\al = 0$, $\beta = 1$ and $\eta = 1 -\frac{1}{2\dim X}< \delta = 1$. Then, since $w_X(X) = \dim X =  \frac12 w(X)$, the condition in the lemma is satisfied, and we get that $Z_X(T)$ converges for $|T| < \LL^{-\dim X}$ and takes non-zero values for $|T|\leq \LL^{-\dim X-\frac{1}{2}}$. %This was to be expected, as $Z_X(T)$ is the motivic analogue of the classical zeta function $\zeta_X(s)$ of a variety $X$, which converges for $\mathrm{Re}(s) >1$

Note that each factor $F(T) = \sum_{i\geq 0} T^i$ has radius of convergence $0$, so taking the Euler product has the effect of shifting the radius of convergence by exactly the dimension of the base variety.
\end{example}

\begin{example} Let $X$ be a quasi-projective variety over $\C$, and let $\a\in \M_X$ be an element such that $w_X(\a) \leq \dim X + 1$. As an example of such an element, by lemma \ref{weightcancellation} we may take $\a = Y-Z$ for two irreducible varieties $Y,Z$ over $X$ of relative dimension 1. Consider the polynomial $F(T) = 1 + \a T^2$, so that
$$\prod_{v\in  X} F_v(T)  = \prod_{v\in X} ( 1 + \a_v T^2) = \sum_{n\geq 0} S^n_{*,X}(\a) T^{2n}.$$ 
Taking $M = 2, \eps = 1 - \frac{1}{w(X)}, \alpha = 0, \beta = 0$, we get convergence for $|T| < \LL^{-\frac12\dim X - \frac14}$.

Let us check that we get the same convergence by estimating the radius of convergence directly: for this, note that
$$w(S^n_{*,X}(\a))\leq w_{S^n_*X}(S^n_{*,X}\a) + \dim S^n_*X \leq n(\dim X + 1) + n\dim X = 2n\dim X + n.$$
Thus, taking the $\limsup$ over all even $n$, the radius of convergence is smaller than
$$\limsup \frac{n\dim X + \frac12 n}{2n} = \frac12\dim X + \frac14.$$

%Again, here the radius of convergence shifts by $\dim X$ after taking the product. 
\end{example}
\subsection{Growth of coefficients}

We finish this section by a lemma that allows one to get information about growth of coefficients of a power series from the fact that it possesses a pole of some order at $T=\LL^{-1}$. It shows that we can predict the behaviour of a positive proportion of the coefficients of the Hodge-Deligne polynomials of the coefficient of degree $n$ for large $n$. For any constructible set $M$, denote by $\kappa(M)$ the number of irreducible components of maximal dimension of~$M$. 
\begin{prop}\label{coefgrowth}
Let $Z(T) = \sum_{n\geq 0}[M_n]T^n\in\kvar_{\C}^{+}[[T]]$ be a power series with effective coefficients such that there exist integers $a,r\geq 1$, a real number $\delta>0$ and a power series $F(T) = \sum_{i\geq 0}f_iT^i\in\M_{\C}[[T]]$ converging for $|T|<\LL^{-1 + \delta}$ and taking a non-zero effective value at $T = \LL^{-1}$, such that
$$Z(T) = \frac{F(T)}{(1-\LL^aT^a)^r}.$$
Then for every $p\in\{0,\ldots,a-1\}$, one of the following cases occur when $n$ tends to infinity in the congruence class of $p$ modulo $a$:
\begin{enumerate}[(i)] \item Either $\limsup \frac{\dim(M_n)}{n}< 1$.
\item Or $\dim(M_n) -n$ has finite limit $d_0\in \Z$ and $\frac{\log(\kappa(M_n))}{\log n}$ converges to some integer in the set $\{0,\ldots,r-1\}$. More generally, for every real number $\eta$ such that $0 < \eta < \delta$ and for sufficiently large $n$ in the congruence class of $p$ modulo $a$, the coefficients of the Hodge-Deligne polynomial $HD(M_n)$ of degrees contained in the interval $$[2(1- \eta)n + 2d_0, 2n + 2d_0]$$ are polynomials in $\frac{n-p}{a}$ of degree at most $r-1$. 
\end{enumerate}
Moreover, the second case happens for at least one value of $p$.
\end{prop}

\begin{proof} Note first that since $[M_n]$ is effective, we have $2\dim [M_n] = w(M_n)$, so it suffices to prove the statement with dimensions replaced by weights divided by two. 

First of all, replace $Z(T)$ by $Z(\LL^{-1} T)$ and $F(T)$ with $F(\LL^{-1} T)$, so that $F$ is now a power series converging for $|T| < \LL^{\delta}$ and taking a non-zero effective value at $T = 1$, and $Z(T) = F(T) (1-T^a) ^{r}$ with $Z(T) = \sum_{n\in \Z}[M_n]\LL^{-n} T^n$. We are going to do calculations in the case $a=1$, and explain later how one can reduce to this case. Note first that if~$F$ converges for $|T|<\LL^{\delta}$, then the same is true for all its derivatives. We may write its Taylor expansion at $T =1$:

$$F(T) = \sum_{i\geq 0} \frac{F^{(i)}(1)}{i!}(T-1)^{i} = \sum_{i\geq 0}\frac{F^{(i)}(1)(-1)^i}{i!}(1-T)^i.$$
Put $G(T) = \sum_{i\geq r} \frac{F^{(i)}(1)(-1)^i}{i!}(1-T)^{i-r}.$ Then
\begin{eqnarray*} G(T) & = & \sum_{i\geq r} \frac{F^{(i)}(1)(-1)^i}{i!} \sum_{j=0}^{i-r}{i-r \choose j}(-1)^jT^j. \\
 & = & \sum_{j\geq 0} \left(\sum_{i\geq r + j}\frac{F^{(i)}(1)(-1)^i}{i!}{i-r \choose j}\right)(-T)^j,\end{eqnarray*}
 so that the coefficient of degree $j$ of $G(T)$ is exactly $g_j = (-1)^j\left(\sum_{i\geq r + j}\frac{F^{(i)}(1)(-1)^i}{i!}{i-r \choose j}\right).$ Writing $F(T) =  \sum_{i\geq 0} f_j T^j$, we have $w(f_j)\to -\infty$ as $j\to +\infty$. More precisely, for any $\eta$ such that $0 < \eta < \delta$ and for sufficiently large $j$, we have
 \begin{equation}\label{coefestimate} w(f_j) < -2\eta j.\end{equation}
 Thus, since
 $$F^{(i)}(1) = \sum_{j\geq i} j(j-1)\ldots (j-i + 1) f_j,$$
we see that as $i$ grows, $w(F^{(i)}(1))\to -\infty$ linearly in $i$. In particular, $w(g_j)\to -\infty$ linearly in $j$, and, more precisely, the estimate
\begin{equation}\label{coefestimategj} w(g_j) < -2\eta j\end{equation}
coming from (\ref{coefestimate}) holds for all sufficiently large $j$. Write now
\begin{eqnarray*} Z(T) & = & \frac{F(T)}{(1-T)^{r}}\\
                      & = & G(T) + \sum_{i=0}^{r-1}  \frac{F^{(i)}(1)(-1)^i}{i!(1-T)^{r-i}}\\
                      & = & G(T) + \sum_{i=0}^{r-1}  \frac{F^{(i)}(1)(-1)^i}{i!}\sum_{n\geq 0} {n + r-i -1\choose r-i-1}T^n\\
                      & = & G(T) + \sum_{n\geq 0} \left(\sum_{i=0}^{r-1}  \frac{F^{(i)}(1)(-1)^i}{i!}{n + r-i -1\choose r-i-1}\right) T^n.
                      \end{eqnarray*}
                      Thus, identifying coefficients, we have 
                      \begin{equation}\label{Mnexpansion}[M_n] \LL^{-n} = g_n + \sum_{i=0}^{r-1}  \frac{F^{(i)}(1)(-1)^i}{i!}{n + r-i -1\choose r-i-1}.\end{equation}
Since by assumption $[M_{n}]$ is an element of $\kvar_k^{+}$, its Hodge-Deligne polynomial is of the form 
\begin{equation}\label{Mnform}\kappa(M_n) (uv)^{\dim(M_n)}+ \ \text{terms of lower degree}.\end{equation}
To get asymptotics for  $\dim(M_n)$ and for the coefficients of high degree of $HD(M_n)$ when $n$ goes to infinity, we therefore need to keep track of the dominant terms of the Hodge-Deligne series of $(\ref{Mnexpansion})$.

We denote by $\{\a \}_d$ the coefficient of $(uv)^d$ in the Hodge-Deligne series of $\a\in \widehat {\M}_{\C}$. Let~$d_0$ be the largest integer $d$ such that there exists $i\in\{0,\ldots,r-1\}$ with $\{F^{(i)}(1)\}_d\neq 0$. Such a $d_0$ does exist since by assumption, we have $F(1)$ effective and non-zero, and therefore there exists some integer $b$ such that $ \{F(1)\}_b\neq 0$. Then for all sufficiently large $n$ (namely, for $n$ such that $w(g_n) < 2d_0$), and for all $d\geq d_0$, , we have
\begin{equation}\label{Mncoef}\{[M_n]\LL^{-n}\}_{d} = \sum_{i=0}^{r-1}  \frac{(-1)^i}{i!}{n + r-i -1\choose r-i-1}\{F^{(i)}(1)\}_d.\end{equation}

   Then, for $d >d_0$, the right-hand side of (\ref{Mncoef}) is zero, forcing the left-hand side to be zero as well, so that $w(M_n\LL^{-n})\leq 2d_0.$ Put now $d=d_0$, and let $i_0$ be the smallest $i$ such that  $\{F^{(i)}(1)\}_d\neq 0$. Then we have

 $$\{[M_n]\LL^{-n}\}_{d_0} \sim_{n\to\infty}\{F^{(i_0)}(1)\}_{d_0}\frac{(-1)^{i_0}}{i_0!(r-i_0-1)!}n^{r-i_0-1},$$
 so that for sufficiently large $n$, $w([M_n]\LL^{-n}) = 2d_0$, and moreover
 $$ \frac{\log\kappa(M_n)}{\log n}  \arr r-i_0-1 \in\{0,\ldots,r-1\}.$$
 
 More generally, going back to equation  (\ref{Mnexpansion}), we see that for sufficiently large $n$, the effective element $M_n$ is the sum of the element $\LL^ng_n$ of $\M_k$ which is of weight strictly less than $2(1-\eta)n$ by estimate (\ref{coefestimategj}), and of the sum 
 $$\sum_{i=0}^{r-1}  \frac{F^{(i)}(1)(-1)^i}{i!}{n + r-i -1\choose r-i-1},$$
 which is a polynomial of degree at most $r-1$ in $n$ with coefficients in $\M_k$ and of weight~$2n$. The statement on the coefficients of the Hodge-Deligne polynomial follows.

It remains to show how to reduce to this when $a >1$. We may decompose $F$ in the following manner:

$$F(T) = \sum_{p=0}^{a-1}\sum_{j\geq 0}f_{aj+p}T^{aj+p} = \sum_{p=0}^{a-1}T^{p}F_p(T^{a}),$$
where $F_p(T) = \sum_{j\geq 0}f_{aj+p}T^{j},$ so that $$Z(T) = \sum_{p=0}^{a-1}T^{p}\frac{F_p(T^a)}{(1-T^a)^r}.$$

Using the expansion we did above and putting $G_p(T) = \sum_{i\geq r} \frac{F_p^{(i)}(1)(-1)^i}{i!}(1-T)^{i-r} = \sum_{m\geq 0}g_{p,m}T^m$, we then have
$$Z(T) = \sum_{p=0}^{a-1} T^p\left(G_p(T^a) + \sum_{m\geq 0}T^{am}\sum_{i=0}^{r-1}  \frac{F_p^{(i)}(1)(-1)^i}{i!}{m + r-i -1\choose r-i-1}\right) $$
 Thus, for every $p\in\{0,\ldots,a-1\}$ and every $m\geq 0$, we have 
 $$[M_{am+p}]\LL^{-(am+p)} = g_{p,m} + \sum_{i=0}^{r-1}  \frac{F_p^{(i)}(1)(-1)^i}{i!}{m + r-i -1\choose r-i-1}.$$
Fix $p\in\{0,\ldots,a-1\}$, and assume first that there is some $d\in\Z$ and some $i\in\{0,\ldots,r-1\}$ such that $\{F_p^{(i)}(1)\}_d\neq 0.$ Then we may conclude as above. If on the contrary such a $d$ does not exist, this means that $$w([M_{am+p}]\LL^{-(am+p)})\to -\infty$$ linearly in $m$ (because $w(g_{p,m})\to -\infty$ linearly in $m$), so that $\limsup \frac{\dim M_n}{n}<1$ when $n$ goes to infinity in the congruence class of $p$ modulo~$a$. 

It remains to show that this last case does not occur for all $p$. For this, recall that $F(1) = \sum_{p=0}^{a-1}F_p(1)$, and, $F(1)$ being effective and non-zero, there exists $d$ such that $\{F(1)\}_d\neq 0$. This means that $\{F_p(1)\}_d\neq 0$ for at least one $p$.   
\end{proof}

\chapter{The motivic Poisson formula}\label{poissonformula}

The aim of this chapter is to extend the scope of Hrushovski and Kazhdan's motivic Poisson formula. As explained in the introduction, it is an analogue of a weakened form of the classical Poisson formula 
for Schwartz-Bruhat functions $f:(\Ad_F)^n\to\C$ on (the $n$-th power, for some $n\geq 1$, of) the adeles of a global field $F$:
\begin{equation}\label{classicalpoisson}\sum_{x\in F^n} f(x) = c\sum_{y\in F^n} \four f(y)\end{equation}\index{Poisson formula!classical}
for some multiplicative constant $c\in \C$, the Fourier transform being calculated with respect to a Haar measure on the locally compact group $\Ad_F^{n}$. The restriction $f_v$ of a  Schwartz-Bruhat function $f:\Ad_F\to C$ to the completion $F_v$ of $F$ at some non-archimedean place~$v$ is locally constant and compactly supported: by compactness, this means that there are integers $M\leq N$ such that, denoting by $\OO_v, t$ respectively the ring of integers and a uniformiser of $F_v$, the function $f_v$ is supported inside $t^{M}\OO_v$ and invariant modulo $t^{N}\OO_v$. Thus, $f_v$ may be viewed as a function on the quotient
$t^M\OO_v/t^N\OO_v.$ Denoting by $\kappa(v)$ the residue field of $F_v$, this quotient group can naturally be identified with the $\kappa(v)$-points of the affine space of dimension $N-M$ over $\kappa(v)$, via
\begin{equation}\label{Akiso}\begin{array}{ccc} t^M\OO_v/t^N\OO_v &\to &\A_{\kappa(v)}^{(M,N)}:=\A_{\kappa(v)}^{N-M}\\
                      t^Mx_M + \ldots + t^{N-1}x_{N-1} + t^{N}\OO_v & \mapsto & (x_M,\ldots,x_{N-1})\end{array}\end{equation}
                      Since the residue field $\kappa(v)$ is finite, the function $f_v$ takes only a finite number of values, and its integral over $F_v$ is given by the formula
                      \begin{equation}\label{equation.integral}\int_{F_v}f_v = q_v^{-N}\sum_{x\in t^M\OO_v/t^N\OO_v} f_v(x),\end{equation}
                      where $q_v$ is the cardinality of $\kappa(v)$. 
                      
 All the above definitions and equalities make heavy use of the local compactness of the adeles and of the local fields $F_v$. When the field $F$ is the function field $k(C)$ of a smooth projective connected curve $C$ over an algebraically closed field $k$, these local compactness properties fail. Hrushovski and Kazhdan's formalism from \cite{HK} allows nevertheless to define analogous objects in this setting, via a form of motivic integration. For example, as suggested by the identification in (\ref{Akiso}), one defines local Schwarz-Bruhat functions as elements of a relative Grothendieck ring with exponentials $\expp_{\A_k^{(M,N)}}$ for some integers $M\leq N$. When $M,N$ vary, these rings fit canonically into an inductive system, via maps that are interpreted respectively as extension by zero and pullback of functions. On the other hand, following formula (\ref{equation.integral}), the \textit{definition} of the integral of such a function $f\in \expp_{\A_k^{(M,N)}}$ is 
 $$\int f = \LL^{-N}\sum_{x\in \A_k^{(M,N)}}f(x),$$
 where the sum in the right-hand side is a notation which stands for the image of $f$ in the absolute Grothendieck ring $\expp_{k}$ via the forgetful morphism $\expp_{\A_k^{(M,N)}}\to \expp_k$. More generally, one can define Fourier transforms of such functions by using the same kind of analogy. 
 
 In the same manner, an element $f$ of the relative Grothendieck ring 
 $$\expp_{\prod_{s\in S}\left(\A_k^{(M_s,N_s)}\right)^n},$$
 for integers $M_s\leq N_s$, $s\in S$ may be seen as a motivic analogue of a Schwartz-Bruhat function on a finite product of powers of local fields $\prod_{s\in S} F_s^n$, for some finite set $S$ of non-archimedean places of the global field $F$ and some integer $n\geq 1$. The inclusion of the Riemann-Roch space 
 $$L(D) = k(C)\cap \prod_{s\in S} t_s^{M_s}\OO_s = \{0\}\cup \{x\in k(C)^{\times},\ \div(x)\geq -D\}$$  for the divisor $D = -\sum_{s\in S}M_s[s]$ on $C$ into $\prod t_s^{M_s}\OO_s$ induces a morphism
 $$\theta:L(D)^n\to \prod_{s\in S}\left(\A_k^{(M_s,N_s)}\right)^n$$
 via which we can pull back $f$. The class of $\theta^*f$ in $\expp_k$ is then denoted $\sum_{x\in k(C)^n}f(x)$. Moreover, the same kind of construction can be done for the motivic Fourier transform $\four f$ of $f$.  Denoting by $g$ the genus of the curve $C$, the motivic Poisson formula
 $$\sum_{x\in k(C)^n}f(x) = \LL^{(1-g)n}\sum_{y\in k(C)^n} \four f(y)$$
 \index{Poisson formula!motivic} proved by Hrushovski and Kazhdan is the motivic analogue of the above classical Poisson formula (\ref{classicalpoisson}) for such functions. 
 
 This chapter focuses on building a framework in which this Poisson formula may be applied for families of such Schwartz-Bruhat functions, with varying set $S$, which will be crucial in chapter \ref{motheightzeta}. The notion of symmetric product of a family of varieties from chapter~\ref{eulerproducts}, in the special case where the set of indices is $\N^p$ for some integer $p\geq 1$, will be central here, and therefore, we start by a review of those in this case in section \ref{sect.reviewsymproducts}. In chapter \ref{motheightzeta}, the integer $p$ will be the cardinality of the set $\scr{A}$ of irreducible components of the divisor at infinity in the equivariant compactification we will consider. 
 Our families of functions will be defined as elements of the relative localised Grothendieck ring with exponentials over symmetric products
 \begin{equation}\label{equation.domdefinition}\scr{A}_{\m}(\al,\be,M,N):= S^{\m}((\A_C^{(\al - M_{\i}, \be + N_{\i})})_{\i\in \N^p}),\end{equation}
 where 
 \begin{itemize} \item $\m\in\N^{p}$ is a $p$-tuple, which in chapter \ref{motheightzeta} will contain the degrees of the sections we count with respect to each of the irreducible components of the divisor at infinity.
 \item $(M_{\i})_{\i\in \N^p}$ and $(N_{\i})_{\i\in \N^p}$ will be families of non-negative integers, with $M_0 = N_0 = 0$.
 \item $\al,\be:C\to \Z$, $\al\leq 0\leq \be$,  are functions on the curve which are zero over some dense open subset $U$ of $C$, which will enable us to take into account the irregular behaviour of the height function at a finite number of places, e.g. places of bad reduction.
 \item $\A_C^{(\al - M_{\i}, \be + N_{\i})}$ is the constructible set over the curve $C$ given by $U\times \A_k^{(-M_{\i},N_{\i})}$ above $U$, and with fibres above $v\in C\setminus U$ given by the affine spaces $\A_k^{(\al_v -M_{\i},\be_v + N_{\i})}$.
 \end{itemize}
 By construction, this symmetric product has a morphism to $S^{\m}C$. A point $D\in S^{\m}C(k)$ may be seen as an ``effective zero-cycle'' $\sum_{v\in C}\m_v v$ where $\m_v\in\N^p$ is such that $\sum_{v}\m_v = \m$. The fibre of the variety (\ref{equation.domdefinition}) above $D$ will be 
 \begin{equation}\label{equation.introfibre}\prod_{v\in C}\A_k^{(\al_v-M_{\m_v}, \be_v + N_{\m_v})},\end{equation}
 that is, a product of affine spaces on which Schwartz-Bruhat functions in the sense of Hrushovski and Kazhdan may be considered. In the framework of chapter \ref{motheightzeta}, this fibre will be the domain of definition of the characteristic function of the sections with poles of orders the coordinates $m_{v,1},\ldots,m_{v,p}$ of the vector $\m_v$ along the $p$ irreducible components of the divisor at infinity. 
 
 The identification (\ref{Akiso}) involves the choice of a uniformiser at the place $v$, which in Hrushovski and Kazhdan's theory may be made arbitrarily, because only a finite number of places  of the field $F$ are involved. To be able to make such identifications for all places of $F$ and keep all operations on families of Schwartz-Bruhat functions algebraic, we discuss in section \ref{sect.uniformisers} how we can choose uniformisers in a ``uniform'' way.
 
 Though we define families of Schwartz-Bruhat functions (of level $\m$) in all generality to be elements of the ring $\expp_{\scr{A}_{\m}(\al,\be,M,N)}$, in fact we will use this definition in two important special cases, namely
 \begin{itemize}\item If all the integers $N_{\i}$ are zero, the family is said to be \textit{uniformly smooth}. By looking at the fibre (\ref{equation.introfibre}), we see indeed that all the functions in such a family will be invariant modulo the same ``compact open'' subset $\prod_{v\in C}t_v^{\be_v}\OO_v$ of the adele ring of~$F$. This will be the case for the family of characteristic functions of sections with given poles mentioned above. The arithmetic analogue of this is the fact that height functions are invariant modulo some compact open subset of the adeles. 
 \item If all the integers $M_{\i}$ are zero, the family is said to be \textit{uniformly compactly supported}. Again, the terminology is clear from the fact that all functions in the family are zero outside $\prod_{v\in C}t_v^{\al_v}\OO_v$. 
 \end{itemize}
 In section \ref{sect.fouriertransform}, we go on to define Fourier transformation for such families of functions, so that it coincides with Hrushovski and Kazhdan's Fourier transform in each fibre above a rational point. This operation exchanges the above two types of families of functions. In section \ref{summation.ratpoints}, we extend Hrushovski and Kazhdan's summation over rational points to families of uniformly compactly supported functions. Finally, in section \ref{Poisson.families}, we formulate the Poisson formula for families of uniformly smooth functions.

%\begin{itemize}
%%\item recall Hrushovski and Kazhdan's formalism
%%\item parametrising domains of definition makes uniform choice of uniformisers important 
%%\item two types of families of Schwartz-Bruhat functions
%%\item Fourier kernel and Fourier transform
%\item Uniform summation over rational points
%\item Poisson formula and reversing order of summation
%\end{itemize}
\section{Symmetric products}\label{sect.reviewsymproducts}
\subsection{Multidimensional partitions}
Fix an integer $p\geq 1$, and denote $\I = \N^{p}-\{0\}$ and $\I_0 = \N^p$. Consider the free abelian monoid over~$\I$:
$$\N^{(\I)} = \{(m_{\i})_{\i\in \I}\in\N^{\I},\ m_{\i} = 0\ \text{for almost all}\ \i\}.$$
To an element $\pi = (m_{\i})_{\i\in \I}\in \N^{(\I)}$ we can associate canonically a $p$-tuple
$$\lambda(\pi) = \sum_{\i\in \I}m_{\i}\i\in \I_0.$$
Thus, we have a well-defined map 
$$\lambda: \N^{(\I)}\arr \I_0.$$
We say $\pi$ is a partition of $\m\in \I$ if $\lambda(\pi) = \m$.

\begin{notation} Recall from notation \ref{partitionnotation} that another notation for partitions is as follows: a partition of $\m$ can be written in the form $[\a_1,\ldots,\a_r]$ where $\a_1,\ldots,\a_r\in \I$ are not necessarily distinct and such that $\a_1 + \ldots + \a_r = \m$. The order of the $\a_i$ in this notation is not important: we consider $[\a_1,\ldots,\a_r]$ to be the same as $[\a_{\sigma(1)},\ldots,\a_{\sigma(r)}]$ for all $\sigma\in\Sym_r$. 
\end{notation}

\begin{example} For $p=1$, we recover partitions of integers: indeed, in this case an element $\pi$ of $\N^{(\I)}$ is a finite family $(m_i)_{i\geq 0}$ of non-negative integers, $\lambda(\pi) = \sum_{i\geq 1}m_ii$ is some integer $m$, and $\pi$ determines a partition $$\sum_{i\geq 1} \underbrace{(i + \ldots + i)}_{m_i\ \text{times}} = m$$of the integer $m$, the non-negative integer $m_i$ being the number of occurrences of $i$ in this partition.

  For $p=2$, consider for example 
$$\pi = \left[\left(\begin{array}{c} 2 \\ 1  \end{array}\right),\left(\begin{array}{c} 2 \\ 1  \end{array}\right),\left(\begin{array}{c} 0 \\ 3  \end{array}\right)\right]$$
It is a partition of $$\left(\begin{array}{c} 4 \\ 5  \end{array}\right) = 2 \left(\begin{array}{c} 2 \\ 1  \end{array}\right) + \left(\begin{array}{c} 0 \\ 3  \end{array}\right).$$
Note that this 2-dimensional partition gives in particular a one-dimensional partition for each coordinate: $[2,2]$ for the first coordinate, and $[1,1,3]$ for the second one. However, it carries more information than just the choice of these two partitions, since it also matches up their parts in some way. Thus, the partition
$$\left(\begin{array}{c} 4 \\ 5  \end{array}\right) = \left(\begin{array}{c} 0 \\ 1  \end{array}\right) + \left(\begin{array}{c} 2 \\ 1  \end{array}\right) + \left(\begin{array}{c} 2 \\ 3  \end{array}\right).$$
is different from $\pi$, but yields the same partitions of its coordinates.
\end{example}

\subsection{From partitions to symmetric products}\label{symproducts}

Let $k$ be a perfect field. Let $\pi = (n_{\i})_{\i\in\I}\in \N^{(\I)}$ and let $\scr{X} = (X_{\i})_{\i\in \I_0}$ be a family of constructible subsets of projective varieties over a quasi-projective $k$-variety $X$. Assume moreover that there is an open subset $U$ of $X$ such that $X_0\times_XU \simeq U$ and such that $X\setminus U$ is a finite union of closed points. In chapter~\ref{eulerproducts}, in particular in sections \ref{definition}, \ref{anyset} and \ref{sect.addX0}, we defined a notion of \emph{symmetric product}~$S^{\pi}\scr{X}$. It follows from the construction that $S^{\pi}\scr{X}$ comes with a natural morphism to $S^{\pi}X$. Define also for any $\m\in\I$, the constructible set $S^{\m}(\scr{X})$ to be the disjoint union of the $S^{\pi}(\scr{X})$ for all partitions $\pi$ of $\m$. %A point $D\in S^{\m}\scr{X}$ may be seen as a formal sum $\sum_{v\in X}\i_vx_v$ where $x_v$ is a geometric point of the fibre $X_{\i_v,v}$ above the geometric point $v\in X$ and $\sum_{v}\i_v = \m$. The fibre of $S^{\m}\scr{X}$ above $D = \sum_{v\in X}\i_v v\in S^{\m}X$ is given by~$\prod_{v\in X}X_{\i_v,v}.$ 

%Fix $\Omega$ an algebraically closed field containing $k$. An element of $S^{\pi}X(\Omega)$ is of the form
%$$D = \sum_{\i\in\I} \i (v_{\i,1} + \ldots + v_{\i,n_{\i}})$$
%where the $v_{\i,j}$ are distinct $\Omega$-points of $X$. Let us describe the geometric fibre $S^{\pi}(\scr{X})_D$ of $S^{\pi}(\scr{X})$ above this point: for this, let $\Omega'$ be an algebraically closed field containing $\Omega$.  An element of $S^{\pi}(\scr{X})_D(\Omega')$ is of the form
%$$\sum_{\i\in \I}\i(x_{\i,1} + \ldots+x_{\i,n_{\i}})$$
%with $x_{\i,j}$ a $\Omega'$-point of $X_{\i}$ mapping to $v_{\i,j}$ for all $\i, j.$

\begin{remark} Recall that when $p=1$, for any quasi-projective variety $X$, the variety $S^nX$ can also be obtained directly by taking the quotient of $X^n$ by the natural permutation action of the symmetric group $\Sym_n$. For $p\geq 2$ and $\n = (n_1,\ldots,n_p)\in\N^{p}$, note that giving an element $\sum_v\n_vv$ of $S^{\n}X$ is equivalent to giving its $p$ components
$$\left(\sum_{v}n_{i,v}v\right)_{1\leq i\leq p}\in S^{n_1}X\times \ldots\times S^{n_p}X.$$
Thus, we in fact have a piecewise isomorphism 
$$S^{\n}X\simeq S^{n_1}X\times \ldots\times S^{n_p}X.$$
\end{remark}

 %We recall in \ref{SBreview} how to define Schwartz-Bruhat functions and their Fourier transforms on the ring of adeles of $F$ depending on a fixed finite set of places $S$ of a curve $C$. In the following sections we will show how these definitions extend naturally to families of Schwartz-Bruhat functions depending on a finite set of places which is allowed to vary. These families will be parametrised by some symmetric power of the curve~$C$.

\section{Motivic Schwartz-Bruhat functions and  Poisson formula}\label{SBreview}
We start with a review of Hrushovski and Kazhdan's motivic Poisson formula from \cite{HK}, following the exposition in sections 1.2 and 1.3 of \cite{CL}. Let $k$ be a perfect field.  
\subsection{Local Schwartz-Bruhat functions}\label{sect.localSB}\index{Schwartz-Bruhat function!local} Let $F = k((t))$ be the completion of a function field of a curve at a closed point, with uniformiser $t$, ring of integers~$\OO$ and residue field $k$. In \cite{HK}, Hrushovski and Kazhdan considered  local motivic exponential Schwartz-Bruhat functions on $F$: such functions are analogues of classical Schwartz-Bruhat functions on non-archimedean local fields, that is, locally constant and compactly supported functions. For each such function $\phi$, there exist integers $M\leq N$ such that $\phi$ is zero outside $t^M\OO$, and invariant modulo $t^N\OO$, so that $\phi$ can be seen as a function on the quotient $t^M\OO/t^N\OO$. The latter can be endowed with the structure of a $k$-variety, and more precisely of an affine space over $k$, through the following identification:

\begin{equation}\label{affineidentification}\begin{array}{ccc}t^M\OO/t^N\OO&\arr& \A_{k}^{N-M}(k)\\
                 x_Mt^M + \ldots + x_{N-1}t^{N-1} + t^N\OO & \mapsto & (x_M,\ldots,x_{N-1})\end{array}\end{equation}
                 This affine space is denoted by $\aff$. \index{Ak@$\aff$} More generally, for any $n\geq 1$ we denote by $\A_k^{n(M,N)}$ \index{Ak@$\A_k^{n(M,N)}$} the affine space $(\aff)^n$, which is viewed as a motivic incarnation of $(t^M\OO/t^N\OO)^n$. Thus, a Schwartz-Bruhat function of level $(M,N)$ \index{Schwartz-Bruhat function!of level $(M,N)$} on $F^n$ will by definition be an element of  $\Sch(F^n;(M,N)) := \expp_{\A_k^{n(M,N)}}$.  \index{SFM@$\Sch(F^n;(M,N))$} An element $E$ of this ring can indeed be interpreted as a function 
                 $$\phi:\A_k^{n(M,N)}\arr\expp_{k(x)}$$
                  by sending a point $x\in\A_k^{n(M,N)}$ to the class of the fibre $E_x$, where $k(x)$ is the residue field of $x$. As $M$ and $N$ vary, the rings $\Sch(F^n;(M,N))$  fit into a directed system the direct limit of which is the total ring $\Sch(F^n)$ \index{SF@$\Sch(F^n)$} of Schwartz-Bruhat functions. More precisely, let us point out that the natural injection 
                  $t^{M}\OO/t^{N}\OO \to t^{M-1}\OO/t^{N}\OO$ gives rise to the closed immersion \begin{equation}\label{immersion}\begin{array}{rcl}i:\A_k^{(M,N)}&\arr& \A_k^{(M-1,N)} \\
                  (x_M,\ldots,x_{N-1})&\mapsto& (0,x_M,\ldots,x_{N-1})\end{array}\end{equation} whereas the natural projection $t^{M}\OO/t^{N+1}\OO\to t^{M}\OO/t^{N}\OO$ induces a morphism
                \begin{equation}\label{projection}\begin{array}{rcl}p:\A_k^{(M,N+1)}&\arr& \A_k^{(M,N)} \\
                  (x_M,\ldots,x_{N})&\mapsto& (x_M,\ldots,x_{N-1})\end{array}\end{equation}  which is a trivial fibration with fibre $\A^1$. They induce ring morphisms $i_!:\Sch(F^n;(M,N))\to \Sch(F^n;(M-1,N))$ (extension by zero) and $p^{*}:\Sch(F^n;(M,N))\to \Sch(F^n;(M,N+1))$.
                  
               \subsection{Integration}\label{HKintegration} For any Schwartz-Bruhat function $\phi\in\Sch(F^n)$, choosing a pair $(M,N)$ such that $\phi\in\Sch(F^n;(M,N))$ one may define, using the exponential sum notation,
$$\int_{F^n}\phi(x) \dx x     = \LL^{-nN}\sum_{x\in\A_k^{n(M,N)}}\phi(x)\in\expp_k.$$    
This does not depend on the choice of $(M,N)$, and defines an $\expp_k$-linear map 
$$\int_{F^n}:\Sch(F^n)\to \expp_k.$$

\subsection{Fourier kernel} Fix a $k$-linear function $r:F\arr k$ such that there is an integer $a$ with $r_{|t^a\OO} = 0$. The least such integer $a$ is called the \textit{conductor} of $r$ and denoted by $\nu$. Note that because of its linearity, $r$ is invariant modulo $r^{\nu}\OO$, so that for any pair of integers $(M,N)$ such that $M\leq N$ and $\nu\leq N$, it induces a well-defined morphism $r^{(M,N)}:\A_k^{(M,N)}\arr \A_k^1$. 

On the other hand, for any two pairs of integers $(M,N)$ and $(M',N')$ satisfying $M\leq N$ and $M'\leq N'$, the product map $F\times F\arr F$ induces a well-defined map on classes
$$t^{M}\OO/t^{N}\OO \times t^{M'}\OO/t^{N'}\OO\arr t^{M+M'}\OO/t^{N''}\OO$$
where $N''\geq  \mathrm{min}(M'+N,M+N')$, that is, a well-defined morphism
\begin{equation}\label{localprod}\A_k^{(M,N)}\times \A_k^{(M',N')}\arr \A_k^{(M+M',N'')}.\end{equation}
Whenever $N''\geq \nu$, this map may be composed with $r^{(M+M',N'')}$ which yields a morphism
$$\A_k^{(M,N)}\times \A_k^{(M',N')}\arr \A_k^1.$$
More generally, taking $n$-th powers and summing the corresponding maps, we get a morphism
$$\A_k^{n(M,N)}\times \A_k^{n(M',N')}\arr \A_k^1.$$
Note that when $M' = \nu-N$ and $N' = \nu-M$, the condition $N''\geq \nu$ is satisfied. The morphism
\begin{equation}\label{localkernel}r:\A_k^{n(M,N)}\times \A_k^{n(\nu-N,\nu-M)}\to \A^1_k\end{equation}
defined in this setting is called the \textit{Fourier kernel}.
\index{Fourier kernel!local}
\subsection{Local Fourier transform} \label{sect.localfouriertransform} The Fourier transform \index{Fourier transform!local} of a motivic Schwartz-Bruhat function $\phi\in\Sch(F^n;(M,N))$ is defined to be the element $\four \phi\in \Sch(F^n;(\nu-N,\nu-M))$ given by
$$\four \phi = \LL^{-Nn}\phi\cdot[\A_k^{n(M,N)}\times \A_k^{n(\nu-N,\nu-M)},r],$$
 where $r$ is the morphism (\ref{localkernel}), and the product is taken in $\expp_{\A_k^{n(M,N)}}$, and viewed in $\expp_{\A_k^{n(\nu-N,\nu-M)}}$. For every $y\in \A_k^{n(\nu-N,\nu-M)}$, using the notation from section~\ref{exponentialsumnotation} of chapter~\ref{grothrings},  as well as the definition of the integral in section \ref{HKintegration} we have
$$\four \phi(y) = \int_F\phi(x)\psi(r(xy))\dx x.$$

\subsection{Global Schwartz-Bruhat functions}\label{CLglobal}
                    
                  One can extend the above definitions to finite products of fields. Consider a finite family $(F_v)_{v\in S}$ of such fields $F_v = k_v((t_v))$, with local parameters $t_v$ and residue fields $k_v$ (which are assumed to be finite extensions of $k$) and an integer $n\geq 1$.
                  For any  family of pairs of integers $(M_v,N_v)_{v\in S}$, with $M_v\leq N_v$ the space of Schwartz-Bruhat functions on $\prod_{v\in S}F_v^n$ of levels $(M_v,N_v)_{v\in S}$ is defined to be 
                  $$\Sch\left(\prod_{v\in S}F_v^n; (M_v,N_v)_{v\in S}\right):= \expp_{\prod_{v\in S}\Res_{k_v/k}\A_{k_v}^{n(M_v,N_v)}},$$
                  where $\Res_{k_v/k}$ denotes the functor of Weil restriction of scalars. In the case where $k$ is algebraically closed, we have
                  $$\Sch\left(\prod_{v\in S}F_v^n; (M_v,N_v)_{v\in S}\right):= \expp_{\prod_{v\in S}\A_{k}^{n(M_v,N_v)}}.$$
                  The ring $\Sch\left(\prod_{v\in S}F_v^{n}\right)$ is defined as a direct limit of these rings, with the appropriate compatibilities.
                  The notions of integral, Fourier kernel and Fourier transform defined above extend easily to such functions (see \cite{CL}, 1.2.10). \index{Fourier kernel!global}\index{Fourier transform!global}

        We are going to use this in the following setting: let $k$ be a perfect field, $C$ a smooth projective curve over $k$, $F = k(C)$ its function field. Denote by $\mathbb{A}_F$ the ring of adeles of the field~$F$. The rings $\Sch(\prod_{v\in S}F_v^n)$, for finite sets $S$ of closed points of $C$, form a directed system, and their direct limit is the ring $\Sch(\mathbb{A}^n_F)$ of global motivic Schwartz-Bruhat functions on $\mathbb{A}^n_F$. \index{Schwartz-Bruhat function!global}
          
            \subsection{Summation over rational points} \label{CLsummation}\index{summation over rational points}
               
          For details on the contents of this paragraph, see \cite{CL}, 1.3.5.  Let $\phi$ be a global Schwartz-Bruhat function on $\mathbb{A}^n_F$, represented by a class in the ring $$\expp_{\prod_{v\in S}\Res_{k_v/k}\A_{k_v}^{n(M_v,N_v)}}$$ for some finite set $S$ of closed points of $C$ and some family $(M_v,N_v)_{v\in S}$  of pairs of integers such that $M_v\leq N_v$ for all $v\in S$. 
            
              Consider the divisor $D = -\sum_{v\in S}M_vv$ on $C$. For every $v\in C$, the natural embedding of the field $F = k(C)$ into its completion $F_v$ maps the Riemann-Roch space 
$$L(D) = \{0\}\cup \{f\in k(C)^{\times},\div(f)\geq \sum_{v}M_vv\} $$ 
into $t^{M_v}\OO_v$. This gives rise to a morphism of algebraic varieties
$$\theta: L(D)^n\arr\left(\prod_{v}\Res_{k_v/k}\A_{k_v}^{(M_v,N_v)}\right)^n.$$
The sum over rational points of $\phi\in \expp_{\left(\prod_{v}\Res_{k_v/k}\A_{k_v}^{(M_v,N_v)}\right)^n}$, denoted by 
$\sum_{x\in F^n}\phi(x),$ \index{sum@$\sum_{x\in k(C)^n}$}
 is then defined to be the image in $\expp_k$ of the pull-back $\theta^*\phi\in\expp_{L(D)^n}.$ It does not depend on choices.
 
\begin{remark} This definition is motivated by the fact that, when $k$ is a finite field, for a Schwartz-Bruhat function $\phi:(\A_F)^n\to\C$ which is supported inside $\left(\prod_vt^{M_v}\OO_v\right)^n$, we have $\phi(x)=0$ for all $x\not\in F^n\cap\left(\prod_vt^{M_v}\OO_v\right)^n = L(D)^n$, so that we have the equality
$$\sum_{x\in F^n}\phi(x) = \sum_{x\in L(D)^n}\phi(x).$$
\end{remark}
 
 \subsection{Motivic Poisson formula}\label{CLpoissonformula}  We fix a non-zero meromorphic differential form $\omega\in\Omega^{1}_{F/k}$. For every $v\in C$, we choose the linear map $r_v:F_v\to k$ defined by   $r_v:x\mapsto \res_v(x\omega)$ and we compute Fourier transforms with respect to those. Theorem 1.3.10 in \cite{CL} states that for $\phi\in\Sch(\mathbb{A}^n_F)$, we have $\four\phi\in\Sch(\mathbb{A}^n_F)$ and
 $$\sum_{x\in F^n}\phi(x) = \LL^{(1-g)n}\sum_{y\in F^n}\four\phi(y).$$ 
              \index{Poisson formula!motivic}

\section{Families of Schwartz-Bruhat functions}
                 
                 \subsection{Parametrising domains of definition} \label{sect.domainsofdef}
                 
                Let $k$ be an algebraically closed field of characteristic zero, and $C$ a smooth projective connected curve over~$k$.% The first step is to define families of varieties of the form $\prod_{v\in S}\A_{k}^{n(M_v,N_v)}$. 
                 
\begin{definition} Let $X$ be a variety over $k$. A function $\alpha:X\to \Z$ is said to be constructible if for every $n\in \Z$, $\alpha^{-1}(n)$ is a constructible subset of $X$. \index{constructible function}
\end{definition}

\begin{remark} When $X = C$, $\alpha:C\to \Z$ is constructible if and only if it is constant on some dense open subset of $C$. If it is zero on some dense open subset of $C$, we say it is \textit{almost zero}. \index{almost zero function}
\end{remark}

The value of a constructible function $\al$ at a point $v\in C$ will be denoted $\al_v$.

\begin{definition}\label{affinespacedef} Let $M,N:C\arr\Z$ be constructible functions such that $M\leq N$. Let  $U\subset C$ be a dense open set over which they are constant, equal respectively to $M_0\in \Z$ and $N_0\in \Z$. We will denote by $\A_C^{(M,N)}$ \index{Ac@$\A_C^{(M,N)}$} the variety over $C$ isomorphic to $U\times \A_k^{(M_0,N_0)}$ over $U$, and with fibre above $u\not\in U$ given by~$\A_{k}^{(M_u,N_u)}$. Furthermore, we will denote by $\left(\A_C^{(M,N)}\right)^{n}$,\ or $\A_C^{n(M,N)}$, \index{Acn@$\A_C^{n(M,N)}$} the variety over $C$ defined by
$$\A_C^{(M,N)}\times_C\ldots \times _C\A_C^{(M,N)},$$
where the product contains $n$ factors. 
\end{definition}

%\subsubsection{Parametrising families of domains of definition of Schwartz-Bruhat functions}
Recall that we denote by $\I_0$ the additive monoid $\Z_{\geq 0}^p$ and $\I = \I_0\setminus\{0\}$. Fix two almost zero functions $\alpha,\beta:C\arr\Z$ such that $\alpha\leq 0\leq \beta$. Denote by $U$ a dense open subset of $C$  over which $\alpha$ and $\beta$ are zero. Fix also two families of non-negative integers~$M = (M_{\i})_{\i\in \I_0}$ and $N = (N_{\i})_{\i\in \I_0}$, with $M_0 = 0$ and $N_0 = 0$. 

We have a family of varieties $\left(\A_C^{n(\al-M_{\i},\be + N_{\i})}\right)_{\i\in \I_0}$ over $C$, giving rise to symmetric products 
\begin{equation}\label{domdefinitionformula}\scr{A}_{\m}(\al,\be,M,N) := S^{\m}\left(\left(\A_C^{n(\al-M_{\i},\be + N_{\i})}\right)_{\i\in \I_0}\right)\end{equation}
 for all $\m\in \I_0$. \index{Am@$\scr{A}_{\m}(\al,\be,M,N)$!definition}

By the definition of symmetric products, all these objects are varieties endowed with natural morphisms to $S^{\m}C$, which we denote $\varpi_{\m}:\scr{A}_{\m}(\al,\be,M,N)\arr S^{\m}C$
for every $\m\in \I_0$. 
\begin{remark} For clarity, let us point out that, where in definition \ref{affinespacedef} objects denoted $M,N$ were \textit{constructible functions}, from now on, except when explicitly stated (that is, except in section \ref{localkernel}), they will denote integers. The possible variation above a finite number of places will be taken care of by the almost zero functions $\al$ and $\be$. 
\end{remark}
\begin{remark}\label{remark.m0expansion} Denote by $\Sigma = \{x_1,\ldots,x_s\}$ the complement $C\setminus U$, which is a finite union of closed points. By corollary \ref{multzeta} from section \ref{sect.cutapplications} of chapter \ref{eulerproducts}, $\scr{A}_{\m}(\al,\be,M,N)$ is the disjoint union of locally closed subsets isomorphic to products
\begin{equation}\label{m0expansion}S^{\m_0}\left(\left(\A_{U}^{n(-M_{\i},N_{\i})}\right)_{\i\in\I}\right)\times \prod_{j=1}^sS^{\m_j}\left(\left(\A_{\{x_j\}}^{n(\al_{x_j} - M_{\i},\be_{x_j} + N_{\i})}\right)_{\i\in\I_0}\right)\end{equation}
for all $\m_0,\ldots,\m_s\in\I_0$ such that $\m_0 + \ldots + \m_s = \m$. By example \ref{generalexn} of chapter \ref{eulerproducts}, the variety~(\ref{m0expansion}) is isomorphic to 
$$S^{\m_0}\left(\left(\A_{U}^{n(-M_{\i},N_{\i})}\right)_{\i\in\I}\right)\times \prod_{j=1}^s \A_{\{x_j\}}^{n(\al_{x_j} - M_{\m_j},\be_{x_j} + N_{\m_j})}.$$
%Let $\pi = (m_{\i})_{\i\in\I}$ be a partition of $\m_0$. Then by proposition \ref{affine} from section \ref{affinespaces} of chapter \ref{eulerproducts}, the subvariety 
%\begin{equation}\label{piexpansion}S^{\pi}\left(\left(\A_{U}^{(-M_{\i},N_{\i})}\right)_{\i\in\I}\right)\times \prod_{j=1}^sS^{\m_i}\left(\left(\A_{\{x_j\}}^{(\al_{x_j} - M_{\i},\be_{x_j} + N_{\i})}\right)_{\i\in\I_0}\right)\end{equation}
%of (\ref{m0expansion}) corresponding to $\pi$ is endowed with the structure of a vector bundle of rank $\sum_{\i\in\I_0}N_{\i} + M_{\i}$
\end{remark}

\begin{remark} Though these definitions depend on the choice of $U$, the ring $$\expp_{\scr{A}_{\m}(\al,\be,M,N)}$$ which we will consider later won't depend on it.  
\end{remark}

\subsection{The fibres of the domains of definition}\label{sect.fibres} \index{Am@$\scr{A}_{\m}(\al,\be,M,N)$!fibre} Let $D\in S^{\m}C$, with residue field $\kappa(D)$. We want to describe the fibre $\varpi_{\m}^{-1}(D)$ above the point $D$. We know that $S^{\m}C$ is the disjoint union of locally closed subsets isomorphic to 
$$S^{\m_0}U\times S^{\m-\m_0}\Sigma$$
for all $\m_0\in\I_0$ such that $\m_0\leq \m$.  Let $\m_0\leq \m$ be such that $D$ belongs to the subset corresponding to $\m_0$. Since the field $k$ is algebraically closed, $S^{\m-\m_0}\Sigma$ is a disjoint union of a finite number of closed points, and therefore the variety $S^{\m_0}U\times S^{\m-\m_0}\Sigma$ has a finite number of connected components each corresponding to a point of $S^{\m-\m_0}\Sigma$. Thus, the schematic point $D$ of $S^{\m}C$ is of the form $(D_U,D_{\Sigma})$, where $D_U\in S^{\m_0} U$ and $D_{\Sigma}\in S^{\m-\m_0}\Sigma$. Let $\pi = (m_{\i})_{\i\in\I}$ be the partition of $\m_0$ such that $D_U\in S^{\pi}U$. On the other hand, $D_{\Sigma}$ is an effective zero-cycle with coefficients in $\I_0$ and with support contained in $\Sigma$, so it may be written in the form
$$D_{\Sigma} = \m_1 x_1 + \ldots + \m_s x_s\in S^{\m-\m_0} \Sigma$$
for $\m_1,\dots,\m_s\in\I_0$ such that $\m_0 + \ldots + \m_s = \m$. Using remark \ref{remark.m0expansion} as well as proposition~\ref{affine} from section~\ref{affinespaces} of chapter~\ref{eulerproducts}, the fibre above $D$ is of the form
\begin{equation}\label{fibre}\prod_{\i\in\I} \A_{\kappa(D)}^{m_{\i}n(N_{\i} + M_{\i}) }\times_{\kappa(D)} \prod_{j=1}^s \A_{\kappa(D)}^{n(\al_{x_j} - M_{\m_j},\be_{x_j} + N_{\m_j})}.\end{equation}

More precisely, we have the diagram
$$\xymatrix{\left(\prod_{\i\in\I} U^{m_{\i}}\right)_* \ar[d] & \ar[l]\left(\prod_{\i\in \I}\A^{m_{\i}n(-M_{\i},N_{\i})}_U\right)_{*,U}\ar[d]\\
S^{\pi}U & \ar[l] S^{\pi}((\A^{m_{\i}n(-M_{\i},N_{\i})}_U)_{\i\in\I})}$$
where the vertical maps are the quotient morphisms, the upper horizontal line is a trivial vector bundle, and the lower line is a vector bundle. Let $D'_U$ be a point of $\left(\prod_{\i\in\I} U^{m_{\i}}\right)_*$ lifting $D_U\in S^{\pi}U$. Taking fibres above $D_U$ and $D'_U$ and denoting by $K$ the residue field of~$D'_U$ (so that $K$ is a finite extension of $\kappa(D)$), the diagram becomes

$$\xymatrix{ \spec K \ar[d] & \ar[l] \prod_{\i\in \I}\A_K^{m_{\i}n(-M_{\i},N_{\i})}\ar[d] \\
            \spec \kappa(D) & \ar[l]\prod_{\i\in\I} \A_{\kappa(D)}^{m_{\i}n(N_{\i} + M_{\i}) }}$$
            so that we have a linear $K$-isomorphism
            $$\prod_{\i\in \I}\A_K^{m_{\i}n(-M_{\i},N_{\i})} \simeq \prod_{\i\in \I}\A_{\kappa(D)}^{m_{\i}n(N_{\i} + M_{\i}) }\otimes_{\kappa(D)}K.$$
            In other words, the fibre above $D_U$ is a \textit{twisted form} of the variety
            $$\prod_{\i\in \I}\A_{\kappa(D)}^{m_{\i}n(-M_{\i},N_{\i})}$$
            which splits above the finite extension $K$ of $\kappa(D)$. %[Dépendance en $D'$ du morphisme à expliquer]
            
            We may conclude that the fibre $\scr{A}_{\m}(\al,\be,M,N)_{D}$ above $D$ may be seen as the domain of definition of a Schwartz-Bruhat function, up to extension of scalars to some finite extension of $\kappa(D)$. 
\begin{remark} We write $\A_{\kappa(D)}^{m_{\i}n(N_{\i} + M_{\i})}$ instead of $\A_{\kappa(D)}^{m_{\i}n(M_{\i},N_{\i})}$ to signify that through the quotient morphism, the chosen identification of the form (\ref{affineidentification}) is twisted. Therefore, when looking at functions on such a fibre $\omega_{\m}^{-1}(D)$ later, for example in section \ref{twistedsummation}, we will pull them back via the quotient morphism before performing operations on them which via (\ref{affineidentification}) can be understood as analogues of operations from classical Fourier theory . %As an example, consider the affine space $\A_K^{2(-1,0)}$ of dimension 2, which is the spectrum $\spec K[x,y]$. Its quotient via the permutation action of $\Sym_2$ is $\spec \kappa(D)[\sigma_1,\sigma_2]$, the quotient map coming from the $K$-algebra morphism $K[\sigma_1,\sigma_2]$ sending 
\end{remark}
%Let us rewrite this in another, more explicit form. The schematic point $D_U\in S^{\pi}U$ pulls back to some $K$-point of $\left(\prod_{\i\in \I}U^{m_{\i}}\right)_*$, where $K$ is some finite extension of $\kappa(D)$. We denote by $v_{\i,j}$ the projection of the latter on the $j$-th copy of $U$ in the factor $U^{m_{\i}}$. This gives us a collection
% $\{v_{\i,j}\}_{\substack{\i\in I\\ j\in\{1,\ldots,m_{\i}\}}}$ of $K$-points of the open subset $U$ of the curve $C$, all distinct because they come as projections from a point in the complement of the diagonal. Put now, for every $K$-point $v$ of $C$:
% $$\i_v = \left\{\begin{array}{cl} \i & \text{if $v = v_{\i,j}$ for some $\i\in I$ and $j\in\{1,\ldots,m_{\i}\}$}\\
%                                    \m_i & \text{if $v = x_j$ for some $j\in\{1,\ldots,s\}$}\\
%                                    0 & \text{otherwise}
%                                    \end{array}\right.$$
%  
%  Then the fibre (\ref{fibre}) may be written in the form 

Let us make the particular case where $D\in S^{\m}C(k)$ more explicit. Recall $k$ is algebraically closed. Thus, $D$ may be seen as an effective zero-cycle $\sum \i_v v$ for points $v\in C(k)$ and $\i_v\in \I_0$, and (\ref{fibre}) may be written in the form
$$\prod_{v\in C} \A_k^{n (\al_v - M_{\i_v},\be_v + N_{\i_v})}$$ 
because the field $k$ is algebraically closed.

%, where $v$ denote distinct geometric points of $C$. The discussion in \ref{symproducts} and \ref{symprodpoints} shows that the fibre $\varpi_{\m}^{-1}(D)$ is given by $\prod_{v}\A_{\kappa(v)}^{n( \al_v-M_{\m_v},\be_v + N_{\m_v})}.$ %In the same manner, the fibre $(\varpi^{\m})^{-1}(D)$ is given by $\prod_{v}\A_k^{n(\al_v,\be_v + N_{\m_v})}.$

%The respective fibres of $\scr{A}_{\m}(\al,\be)$ and  $\scr{A}^{\m}(\al,\be)$ above an effective multi-zero-cycle $D = \sum_{v}\i_vv\in S^{\m}C$ are
%$$\scr{A}_{\m}(\al,\be)_D = \varpi_{\m}^{-1}(D) = \prod_{v\in C} \A_{k}^{n(-N_{\i_v} + \al_v,\be_v)}\ \ \text{and}\ \ \scr{A}^{\m}(\al,\be) = \left(\varpi^{\m}\right)^{-1}(D) = \prod_{v\in C} \A_{k}^{n(\al_v,\be_v+ N_{\i_v})}.$$
%Moreover, $\scr{A}_0(\al,\be)$ and $\scr{A}^0(\al,\be)$ are both equal to $\prod_{v}\A_k^{n(\al_v,\be_v)}.$

\subsection{Uniform choice of uniformisers}\label{sect.uniformisers}\index{uniformisers}
In section \ref{sect.domainsofdef}, we have defined families of domains of definition of Schwartz-Bruhat functions: the product $\prod_{v}\A_k^{(M_v,N_v)}$ has to be understood as representing
$\prod_{v}t_v^{M_v}\OO_v/t_v^{N_v}\OO_v.$ However, this identification depends on the choice of the uniformisers $t_v$ at each place $v$, and therefore so will some of the operations we are going to perform in what follows. This choice has to be made as uniformly as possible, so that these operations remain algebraic. We explain in this section how this can be done.
\begin{lemma}
Fix a non-constant element $t\in k(C)$. Then there is a dense open set $U\subset C$ such that for all $v\in U$, the function $t_v = t-t(v)$ is a local parameter at $v$. 
\end{lemma}

\begin{proof} Denote by $U_0$ an open dense set of $C$ on which $t$ is regular. We therefore get a holomorphic differential $\mathrm{d}t$ on $U_0$. It is non-vanishing when restricted to some open dense subset $U$ of $U_0$. At any $v\in U$, the function $t_v$ is an element of the maximal ideal $\mathfrak{m}_v$, and its differential $\mathrm{d}(t_v) =  \mathrm{d}t$ is non-zero, so it is a local parameter.
\end{proof}

From now on, we fix such an element $t\in k(C)$. The uniformiser at any place $v$ in the open set  $U$ furnished by the lemma will be given by $t_v = t-t(v)$. For $v\in C\backslash U$, we fix some arbitrary uniformiser~$t_v$.
\begin{lemma}\label{coef} Let $f\in k(C)$. For any $v\in C$, write the $t_v-$adic expansion of $f$ as
$$\sum_{p\in \Z}a_p(f,v)t_v^p\in F_v.$$
There is an open dense subset $U'$ of $U$ such that for any integer $p$, the map 
$$\begin{array}{rcl}a_p(f,\cdot): U'&\arr& k\\
v&\mapsto& a_p(f,v)\end{array}$$ 
is a regular function.
\end{lemma}

\begin{proof} Denote by $U'$ an open dense subset of $U$ over which $f$ is regular, so that the $a_{i}$ with $ i<0$ are identically zero. For any $v\in U'$, we have by definition $f(v) = a_0(f,v)$, and therefore $a_0(f,\cdot) = f_{|U'}$ is a regular map. Now, the differential $\mathrm{d}f$ is a holomorphic differential on $U'$, so there is a regular function $f_1:U'\arr k$ such that $\mathrm{d}f = f_1\mathrm{d}t.$ On the other hand, differentiating the $t_v$-adic expansion of $f$, we get, since $\mathrm{d}t = \mathrm{d}t_v$, 
$$\mathrm{d}f = (a_1(f,v) + 2a_2(f,v) t_v + 3a_2(f,v)t_v^2+ \ldots)\mathrm{d}t$$
Thus, since the differential $\mathrm{d}t$ doesn't vanish on $U'$, we have $a_1(f,\cdot) = f_1$, which is regular. To prove regularity of $a_2(f,\cdot)$, replace $f$ by $f_1$ and proceed in the same way. By induction, we get regularity of all $a_p$'s.
\end{proof}

%In what follows, we denote by $a_p$ the constructible morphism 
%$$\begin{array}{rcl}a_p: C&\arr& k\\
%v&\mapsto& a_p(v)\end{array}$$

\subsection{Families of Schwartz-Bruhat functions}

\begin{definition}\label{def.constr.families} Let $\m\in \I_0$. Let $\al,be:C\to \Z$ be almost zero functions such that $\al\leq 0 \leq \be$, and let $M = (M_{\i})_{\i\in\I_0}$, $N= (N_{\i})_{\i\in\I_0}$  be two families of non-negative integers such that $M_0 = N_0 = 0$. The elements of $\expp_{\scr{A}_{\m}(\al,\be,M,N)}$\ are called constructible families of Schwartz-Bruhat functions of level $\m$. \index{Schwartz-Bruhat function!family}
\end{definition}

Let $\Phi\in\expp_{\scr{A}_{\m}(\al,\be,M,N)}$ be such a family of functions, and let $D$ be a schematic point of $S^{\m}C$. The fibre of $\scr{A_{\m}}(\al,\be,M,N)$ above $D$ has been computed in (\ref{fibre}). Restricting~$\Phi$ to it, we obtain, up to extension of scalars to a finite extension of the residue field of $D$, a Schwartz-Bruhat function $\Phi_D$ in the sense of Hrushovski and Kazhdan. Thus, $\Phi$ gives rise to a family of ``twisted'' Schwartz-Bruhat functions $(\Phi_D)_{D\in S^{\m}C}$ indexed by $D\in S^{\m}C$.

%$\Phi_D \in\Sch(\prod_{\i\in I}\prod_{j = 1}^m_{\i}\kappa(D)(C)_{v_{\i,j}}\prod_{v\in \Sigma} $
%$$\Phi_D\in \scr{S}\left(\prod_{\i\in \scr{I}}\prod_{j=1}^{m_{\i}} \kappa(D)(C)^n_{v_{\i,j}}\times \prod_{v\in\Sigma}\kappa(D)(C)_v^n, (\al_v -M_{\

In the particular case when $D\in S^{\m}C(k)$, denoting by $|D|$ the support of the effective zero-cycle~$D$, in the notation of section \ref{CLglobal}, we have
$$\Phi_D\in\scr{S}\left(\prod_{v\in |D|\cup \Sigma}k(C)_v^n,(\al_v-M_{\i_v},\be_v+N_{\i_v})_{v\in |D|\cup \Sigma}\right),$$
 and $k(C)_v$ is the completion of the function field $k(C)$ at the place $v$. 

%The same is true for an element $\Psi$ of $\expp_{\scr{A}^{\m}(\al,\be)}$, except that in this case the function $\Psi_D$ will be  an element of $\scr{S}(\prod_{v\in |D|\cup S}F_v^n,(\al_v,\be_v+N_{\i_v})_{v\in S}).$ %In both cases, we are thus able to parametrise families of Schwartz-Bruhat functions with domain of definition of fixed dimension 
%$$n\sum_{v\in |D|\cup S}(\be_v -\al_v + N_{\i_v}) = n(|\be|-|\alpha| + \sum_{v\in |D|}N_{\i_v}).$$

%Explicitly, if $\Phi = [V,h]\in \expp_{\scr{A}_{\m}(\al,\be,M,N)}$ for a variety $V$ over  $\scr{A}_{\m}(\al,\be,M,N)$, and for every $D\in S^{\m}C$, $\Phi_D$ is given by 
%\begin{equation}\label{explicitphiD}[U\times_{\scr{A}_{\m}(\al,\be,M,N)}{\scr{A}_{\m}(\al,\be,M,N)_D}, h\circ \pr_1]\in \expp_{\scr{A}_{\m}(\al,\be,M,N)_D}.\end{equation}

\subsection{Uniformly smooth or uniformly compactly supported families}\label{sect.uniformfamilies}
In a similar manner to (\ref{immersion}) and (\ref{projection}), we may define constructible morphisms
$$ p: \A_C^{(\al - M_{\i},\be + N_{\i})}\to \A_C^{(\al-M_{\i},\be)}$$ and
$$i: \A_C^{(\al,\be + N_{\i})}\to \A_C^{(\al - M_{\i},\be + N_{\i})}$$
for every $\i\in \I_0$. Above $v\in C$, the first one is a projection on the first $\beta_v-\alpha_v + M_{\i}$ coordinates, and the second one is 
$$(x_{\al_v},\ldots , x_{\be_v + N_{\i} -1})\mapsto (\!\!\!\underbrace{0,\ldots,0}_{M_{\i} \ \text{coordinates}}\!\!\!, x_{\al_v},\ldots , x_{\be_v + N_{\i} -1}).$$
Taking symmetric products, we get, for every $\m\in\I_0$, morphisms
$$p_{\m}: S^{\m} ((\A_C^{(\al - M_{\i},\be + N_{\i})})_{\i\in\I_0}) \to  S^{\m} ((\A_C^{(\al - M_{\i},\be)})_{\i\in\I_0}),$$
$$i_{\m}: S^{\m} ((\A_C^{(\al,\be + N_{\i})})_{\i\in\I_0}) \to  S^{\m} ((\A_C^{(\al - M_{\i},\be + N_{\i})})_{\i\in\I_0}),$$
that is, $$p_{\m}:\scr{A}_{\m}(\al,\be,M,N) \to \scr{A}_{\m}(\al,\be,M,0) $$ and $$i_{\m}:\scr{A}_{\m}(\al,\be,0,N) \to \scr{A}_{\m}(\al,\be,M,N) .$$ Those induce injective ring morphisms
$$p_{\m}^*:\expp_{\scr{A}_{\m}(\al,\be,M,0)} \to \expp_{\scr{A}_{\m}(\al,\be,M,N)}$$ and
$$i_{\m,!}:\expp_{\scr{A}_{\m}(\al,\be,0,N)} \to \expp_{\scr{A}_{\m}(\al,\be,M,N)}.$$

\begin{definition} Let $\Phi\in\expp_{\scr{A}_{\m}(\al,\be,M,N)}$ be a constructible family of Schwartz-Bruhat functions of level $\m$. It is said to be
\begin{itemize}\item \textit{uniformly smooth} if it belongs to the image of the morphism $p^*_{\m}$.
\item \textit{uniformly compactly supported} if it belongs to the image of the morphism $i_{\m,*}$.
\end{itemize}\index{Schwartz-Bruhat function!family!uniformly smooth}\index{Schwartz-Bruhat function!family!uniformly compactly supported}
\end{definition}
A word of explanation on this terminology, which is inherited from the classical $p$-adic setting: let  $\Phi$ be a constructible family of functions. If $\Phi$ is uniformly smooth, then every~$\Phi_D$ for $D = \sum_{v}\i_vv\in S^{\m}C(k)$ may be seen as an element of $\expp_{\prod_{v}\A_k^{n(\al_v-M_{\i_v},\be_v)}}$, that is as a function on $\prod_{v}t^{\al_v-M_{\i_v}}\OO_v$, invariant modulo $\prod_{v}t^{\be_v}\OO_v$, this invariance domain being independent of $D$. In the same manner, if~$\Phi$ is uniformly compactly supported, all~$\Phi_D$ are supported inside $\prod_{v}t^{\al_v}\OO_v$ independently of $D$.

\section{Fourier transformation in families}\label{sect.fouriertransform}

\subsection{Local construction of the Fourier kernel in families}\label{localkernelfamilies}\index{Fourier kernel!in families!local}
Let $\omega\in \Omega_{F/k}$ be a non-zero meromorphic differential form. For every closed point $v\in C$ we write $\nu_v=-\ord_v\omega$, so that $\div\, \omega = -\sum_{v\in C}\nu_vv$. Then for every $v$ we get a $k$-linear map $r_v:F_v\arr k$ given by
$$r_v(x) = \res_v(x\omega).$$
It is non-zero, and its conductor, that is, the least integer $a$ such that $r_{v|t^a\OO_v}$ is zero, is equal to $\nu_v$. 

Let $M\leq N$ be constructible functions as in definition \ref{affinespacedef}, with the additional assumption that for every~$v$, we have $\nu_v\leq N_v$. Using lemma \ref{coef}, we see that the map~$r$ gives rise to a piecewise morphism $r^{(M,N)}:\A_C^{(M,N)}\arr\A_k^1$, sending an element $(v,x)$ to $r_v(x)$. 

Fix two additional constructible functions $M',N':C\arr\Z$ such that $M'\leq N'$. For every $v$, the product map $F_v\times F_v\arr F_v$ defines a  morphism
\begin{equation}\label{product}\A_{\kappa(v)}^{(M_v,N_v)}\times_{\kappa(v)} \A_{\kappa(v)}^{(M'_v,N'_v)}\arr \A_{\kappa(v)}^{(M_v+M_v',N_v'')}\end{equation}
where $N'' = \min\{M+N',M'+N\}$ (see (\ref{localprod})). More precisely, there is a morphism of constructible sets over~$C$
$$\A_C^{(M,N)}\times_C \A_C^{(M',N')}\arr \A_C^{(M+M',N'')}$$
 such that for every $v\in C$ the induced morphism on the fibre above $v$ is (\ref{product}). When $N''\geq \nu$, for example when $M'=\nu-N$ and $N' =\nu-M$, this can be composed with $r^{(M+M',N'')}$ to get a constructible morphism
\begin{equation}\label{kernelfamily}\A_C^{(M,N)}\times_C \A_C^{(M',N')}\arr\A^1.\end{equation}
 The restriction to the fibre above $v$ is given by the Fourier kernel $(x,y)\mapsto r_v(xy)$ from~(\ref{localkernel}). Thus, the map (\ref{kernelfamily}) may be interpreted as a parametrisation of all local Fourier kernel maps induced by the differential form $\omega$. 
\subsection{Global construction of the Fourier kernel}\label{sect.fourkernel}
\index{Fourier kernel!in families!global}
Fix two almost zero functions $\al\leq 0\leq \be$, and two non-negative families of integers $M = (M_{\i})_{\i\in\I_0}$ and $N = (N_{\i})_{\i\in\I_0}$ such that $M_0 =N_0 = 0$. According to the discussion in the previous paragraph, for any $\i$ there exists a constructible Fourier kernel morphism
$$\A_C^{n(\al-M_{\i} ,\be + N_{\i})}\times_C\A_C^{n(\nu-\be-N_{\i},\nu- \al + M_{\i})}\to \A^1$$
Taking symmetric products, we get morphisms
\begin{equation}\label{eq.fourierkernel}r_{\m}:\scr{A}_{\m}(\al,\be,M,N)\times_{S^{\m}C}\scr{A}_{\m}(\nu-\be,\nu-\al,N,M)\to \A^1.\end{equation}
for any $\m\in\I_0,$ using the following straightforward lemma:
\begin{lemma}\label{symmetricfibreproduct} Let $\scr{X} = (X_{\i})_{\i}$ and $\scr{Y} = (Y_{\i})_{\i}$ be families of constructible sets over $X$, and assume  that for every~$\i$ we are given a constructible morphism $f_{\i}:X_{\i}\times_XY_{\i}\to \A^1$. Denote by $\scr{X}\times_X\scr{Y}$ the family $(X_{\i}\times_X Y_{\i})_{\i}$. Then  for every $\pi\in\N^{(\I)}$  there is a natural piecewise isomorphism
$$S^{\pi}(\scr{X}\times_X\scr{Y})\simeq S^{\pi}(\scr{X})\times_{S^{\pi}(X)}S^{\pi}(\scr{Y})
$$
given by 
$$\sum_{\i}\i((x_{\i,1},y_{\i,1}),\ldots,(x_{\i,n_{\i}},y_{\i,n_{\i}}))\mapsto \left(\sum_{\i}\i(x_{\i,1} + \ldots + x_{\i,n_{\i}}),\sum_{\i}\i(y_{\i,1} + \ldots + y_{\i,n_{\i}})\right), $$
through which the morphism $f^{(\pi)}:S^{\pi}(\scr{X}\times_X\scr{Y})\to \A^1$ becomes 
$$\left(\sum_{\i}\i(x_{\i,1} + \ldots + x_{\i,n_{\i}}),\sum_{\i}\i(y_{\i,1} + \ldots + y_{\i,n_{\i}})\right)\mapsto \sum_{\i}(f(x_{\i,1},y_{\i,1}) + \ldots f(x_{\i,n_{\i}},y_{\i,n_{\i}})).$$

\end{lemma}
Using the description of the fibre above $D$ from \ref{sect.fibres}, we see that for every $D\in S^{\m}C$, the morphism 
$$r_{D}: \scr{A}_{\m}(\al,\be,M,N)_D\times_{\kappa(D)}\scr{A}_{\m}(\nu-\be,\nu-\al,N,M)_D\to \A^1$$
induced by $r_{\m}$ on the fibre above $D$ is a twisted form of the Fourier kernel of Hrushovski and Kazhdan. The twisting being linear, $r_D$ is a $\kappa(D)$-bilinear map. 

%Using the notation from remark \ref{remark.fibre} and choosing $\Sigma$ containing all closed  points $v\in C$ such that $\nu_v\neq 0$, the morphism induced by $r_{\m}$ on the fibre above a schematic point $D \in S^{\m}C$ is exactly a product
%%$$\prod_{v}\A_{\Omega}^{(\al_v-M_{\i_v},\be_v+N_{\i_v})}\times \prod_{v}\A_{\Omega}^{(\nu_v-\be_v-N_{\i_v},\nu_v-\al_v + M_{\m_v})}\to \A^1$$
%of Fourier kernels (\ref{localkernel}), the part above $U$ being of the form
%$$\prod_{\i\in\I} \A_{\kappa(D)}^{m_{\i}n(-M_{\i},N_{\i}) }\times_{\kappa(D)} \prod_{\i\in\I} \A_{\kappa(D)}^{m_{\i}n(-N_{\i},M_{\i}) }\to \A^1$$
%and the part above $\Sigma$ being of the form
%$$\prod_{j=1}^s \A_{\kappa(D)}^{n(\al_{x_j} - M_{\m_j},\be_{x_j} + N_{\m_j})}\times_{\kappa(D)} \prod_{j=1}^s \A_{\kappa(D)}^{n(\nu_{x_j} - \be_{x_j} -N_{\m_j}, \nu_{x_j} -\al_{x_j} + M_{\m_j})}\to \A^1$$

\subsection{Fourier transform}\index{Fourier transform!in families}
To define the Fourier transform of a family of Schwartz-Bruhat functions $$\Phi\in\expp_{\scr{A}_{\m}(\al,\be,M,N)},$$ we start by defining the factor we need to normalise it by, so that it does not depend on the choice of $\beta$ and $N$. 

For this, we start with the family of $C$-varieties $(\A_C^{n(\beta + N_{\i})})_{\i\in\I_0}$. Taking symmetric products, we get a constructible morphism
\begin{equation}\label{normalisation}S^{\m}((\A_C^{n(\beta + N_{\i})})_{\i\in\I_0})\to S^{\m}C.\end{equation}
Using the notation in remark \ref{remark.m0expansion} and section \ref{sect.fibres}, $U$ an open dense subset of $C$ above which $\be$ is zero, $\Sigma = C\setminus U$, $\m_0\in\I_0$ such that $\m_0 \leq \m$, $D_{\Sigma} = \sum_{v\in\Sigma}\i_v [v]\in S^{\m-\m_0}\Sigma$ and $\pi = (m_{\i})_{\i\in \I}$ a partition of $\m_0$ we get, by proposition~\ref{affine} from section~\ref{affinespaces} of chapter~\ref{eulerproducts}, that the restriction of (\ref{normalisation}) above the locally closed subset 
$S^{\pi}U\times \{D_{\Sigma}\}$ of $S^{\m_0}C$ is a vector bundle of rank $\sum_{\i}nm_{\i}N_{\i} + \sum_{v\in\Sigma}n(\be_{v} + N_{\i_v})$. Thus, the class of (\ref{normalisation}) in $\expp_{S^{\m}C}$ is
$$\sum_{\m_0 \leq \m}\ \ \sum_{\substack{D_{\Sigma} \in S^{\m-\m_0}\Sigma\\
                                      D_{\Sigma}  =\sum_{v\in\Sigma}\i_v [v]}}\ \ \  \sum_{\substack{\pi = (m_{\i})_{\i\in \I}\\
                                                     \sum_{\i\in\I}m_{\i}\i = \m_0}}\LL^{\sum_{\i}nm_{\i}N_{\i} + \sum_{v\in\Sigma}n(\be_{v} + N_{\i_v})}[S^{\pi}U\times \{D_{\Sigma}\}\to S^{\m}C],$$
                                                     and it makes sense to consider
                                                     $$[S^{\m}((\A_C^{n(\beta + N_{\i})})_{\i\in\I_0})]^{-1}\in\expp_{S^{\m}C}$$
                                                     defined by the formula
$$\sum_{\m_0 \leq \m}\ \ \sum_{\substack{D_{\Sigma} \in S^{\m-\m_0}\Sigma\\
                                      D_{\Sigma}  =\sum_{v\in\Sigma}\i_v [v]}}\ \ \  \sum_{\substack{\pi = (m_{\i})_{\i\in \I}\\
                                                     \sum_{\i\in\I}m_{\i}\i = \m_0}}\LL^{-\sum_{\i}nm_{\i}N_{\i} - \sum_{v\in\Sigma}n(\be_{v} + N_{\i_v})}[S^{\pi}U\times \{D_{\Sigma}\}\to S^{\m}C],$$
                                                     that is, the same one as above, but with the powers of $\LL$ inverted.
\begin{remark} This element is indeed the inverse of $[S^{\m}((\A_C^{n(\beta + N_{\i})})_{\i\in\I_0})]$ in the ring $\expp_{S^{\m}C}$, so our notation is consistent. 
\end{remark}                                                     
 We denote by $R_{\m}$ the element of the Grothendieck ring
$$\expp_{\scr{A}_{\m}(\al,\be,M,N)\times_{S^{\m}C}\scr{A}_{\m}(\nu-\be,\nu-\al,N,M)}$$
given by
$$R_{\m}:= [\scr{A}_{\m}(\al,\be,M,N)\times_{S^{\m}C}\scr{A}_{\m}(\nu-\be,\nu-\al,N,M), r_{\m}].$$
Moreover, we denote by $\pr_1,\pr_2$ the projections
$$\pr_1:\scr{A}_{\m}(\al,\be,M,N)\times_{S^{\m}C}\scr{A}_{\m}(\nu-\be,\nu-\al,N,M)\to \scr{A}_{\m}(\al,\be,M,N)$$
and
$$\pr_2:\scr{A}_{\m}(\al,\be,M,N)\times_{S^{\m}C}\scr{A}_{\m}(\nu-\be,\nu-\al,N,M)\to \scr{A}_{\m}(\nu-\be,\nu-\al,N,M).$$
Let $\Phi\in\expp_{\scr{A}_{\m}(\al,\be,M,N)}$ be a constructible family of Schwartz-Bruhat functions. The family  $\four\Phi\in\expp_{\scr{A}_{\m}(\nu-\be,\nu - \al,N,M)}$ is defined by the formula
$$\four \Phi:= [(S^{\m}((\A_C^{n(\beta + N_{\i})})_{\i\in\I_0}))]^{-1}(\pr_2)_!((\pr_1)^*\Phi \cdot R_{\m})\in\expp_{\scr{A}_{\m}(\nu-\be,\nu-\al,N,M)}$$
where $\cdot$ is the product in the Grothendieck ring $\expp_{\scr{A}_{\m}(\al,\be,M,N)\times_{S^{\m}C}\scr{A}_{\m}(\nu-\be,\nu-\al,N,M)}$. 

Explicitly, if $\Phi = [V,f]\in\expp_{\scr{A}_{\m}(\al,\be,M,N)}$, then $\four \Phi$ is given by
$$ [V\times_{\scr{A}_{\m}(\al,\be,M,N)}\scr{A}_{\m}(\al,\be,M,N)\times_{S^{\m}C}\scr{A}_{\m}(\nu- \be,\nu-\al,N,M), f\circ\pr_1 + r_{\m}(\pr_2\cdot\pr_3)]$$
multiplied by the normalisation factor $[(S^{\m}((\A_C^{n(\beta + N_{\i})})_{\i\in\I_0}))]^{-1}$. 

\begin{remark}\label{duality} Taking $N=0$ (resp. $M=0$) one can see that the Fourier transform of a family of uniformly smooth (resp. uniformly compactly supported) functions is a family of uniformly compactly supported (resp. uniformly smooth) functions.  
\end{remark}
\subsection{Compatibility between symmetric products and Fourier transformation}\label{sect.symprodfourtransform}
This section deals with the special case where $\Phi\in \expp_{\scr{A}_{\m}(\al,\be,M,N)}$ is given by a symmetric product. More precisely, suppose we are given, for any $\i\in\I_o$, an element $\phi_{\i}\in \expp_{\A_C^{n(\al-M_{\i},\be + N_{\i})}}$, and that $\Phi$ is the symmetric product $S^{\m}((\phi_{\i})_{\i \in \I_0})$. There is a natural notion of Fourier transformation for elements of the ring  $\expp_{\A_C^{n(\al-M_{\i},\be + N_{\i})}}$, defined in the following manner: denote by $r_{\i}$ the Fourier kernel 
$$\A_C^{n(\al-M_{\i},\be + N_{\i})}\times_C \A_C^{n(\nu -\be+ N_{\i},\nu - \al + M_{\i})}\to \A^1$$
defined in (\ref{kernelfamily}) and by $R_{\i}$ the element of the Grothendieck ring 
$$\expp_{\A_C^{n(\al-M_{\i},\be + N_{\i})}\times_C \A_C^{n(\nu -\be+ N_{\i},\nu - \al + M_{\i})}}$$ given by 
$$R_{\i} = [\A_C^{n(\al-M_{\i},\be + N_{\i})}\times_C \A_C^{n(\nu -\be+ N_{\i},\nu - \al + M_{\i})},r_{\i}].$$
Moreover, denoting again by $U$ the open subset of $C$ on which $\al$ and $\be$ are zero and  by $\Sigma$ its complement, we define the element $\left(\A_C^{n(\be + N_i)}\right)^{-1}$ of $\expp_C$ by 
$$\LL^{-nN_{\i}}[U\to C]+ \sum_{v\in \Sigma}\LL^{-n(\be_v - N_{i})}[\{v\}\to C].$$
We denote by $\pr_1$, $\pr_2$ the projections
$$\pr_1: \A_C^{n(\al-M_{\i},\be + N_{\i})}\times_C \A_C^{n(\nu -\be+ N_{\i},\nu - \al + M_{\i})}\to \A_C^{n(\al-M_{\i},\be + N_{\i})}$$
and 
$$\pr_2: \A_C^{n(\al-M_{\i},\be + N_{\i})}\times_C \A_C^{n(\nu -\be+ N_{\i},\nu - \al + M_{\i})}\to \A_C^{n(\nu -\be+ N_{\i},\nu - \al + M_{\i})}.$$ We then define
$$\four \phi_i = \left(\A_C^{n(\be + N_i)}\right)^{-1}(\pr_2)_!((\pr_1)^*\phi_i\cdot R_{\i} )\in \expp_{\A_C^{n(\nu -\be+ N_{\i},\nu - \al + M_{\i})}}$$
where $\cdot$ is the product in the ring $\expp_{\A_C^{n(\al-M_{\i},\be + N_{\i})}\times_C \A_C^{n(\nu -\be+ N_{\i},\nu - \al + M_{\i})}}.$
\begin{remark}\label{familyrestrictiontov} For every $v\in C$, denote by $\phi_{\i,v}\in \expp_{\A_{\kappa(v)}^{n(\al_v-M_{\i},\be_v + N_{\i}}}$ the local Schwartz-Bruhat function obtained from $\phi_{\i}$ by restriction to the fibre $\A_{\kappa(v)}^{n(\al_v-M_{\i},\be_v + N_{\i})}$ of $\A_C^{n(\al-M_{\i},\be + N_{\i})}$ above $v$. By definition of the Fourier transform of a local Schwartz-Bruhat function (see section \ref{sect.localfouriertransform} )  as well as of the Fourier kernel $r_{\i}$ (see section \ref{localkernelfamilies}), we see that 
$$\four(\phi_{\i,v}) = (\four\phi_i)_v$$
in $\expp_{\A_{\kappa(v)}^{n(\nu_v -\be_v+ N_{\i},\nu_v - \al_v + M_{\i})}},$
where the right-hand side is the restriction of $\four\phi_i$ to the fibre $\A_{\kappa(v)}^{n(\nu_v -\be_v+ N_{\i},\nu_v - \al_v + M_{\i})}$. Thus, the operation we just defined performs Fourier transformation on families of local Schwartz-Bruhat functions parametrised by $C$, with level constant except at a finite number of closed points. 
\end{remark}
\begin{prop} \label{prop.symproductfourtransform} We have
$$\four S^{\m}((\phi_{\i})_{\i\in\I_0}) = S^{\m}((\four \phi_{\i})_{\i\in\I_0}) .$$
\end{prop}
\begin{proof} By definition of $r_{\m}$ and of $\Phi$, with the notations from the previous sections we have
$$(\pr_2)_!((\pr_1)^*S^{\m}((\phi_{\i})_{\i\in\I_0})\cdot R_{\m}) = S^{\m}\left( ((\pr_2)_!((\pr_1)^*\phi_i\cdot R_{\i} ))_{\i\in\I_0}\right),$$
since the projections in the previous section are obtained from those in the previous section by taking symmetric products. Therefore, by lemma \ref{symmetricfibreproduct} we get the result, comparing the normalisation factors. 
\end{proof}
\subsection{Inversion formula}

This section will not be used in what follows, but we include it for the sake of completeness. 

For every $\i\in \I_0$, define a constructible morphism $\A_C^{(\al-M_{\i},\be + N_{\i})}\to \A_C^{(\al-M_{\i},\be + N_{\i})} $ of $C$-varieties by
$$(v,x_{\al_{v}-M_{\i}},\ldots,x_{\be_v + N_{\i}-1})\to (v,-x_{\al_{v}-M_{\i}},\ldots,-x_{\be_v + N_{\i}-1}).$$
Taking symmetric products, it induces a constructible morphism of $S^{\m}C$-varieties
$$\mathrm{inv}_{m}:\scr{A}_{\m}(\al,\be,M,N)\to \scr{A}_{\m}(\al,\be,M,N).$$

This Fourier transform satisfies the following Fourier inversion formula:
\begin{prop} For every $\Phi$ in $\expp_{\scr{A}_{\m}(\al,\be,M,N)}$, we have
$$\four\four \Phi = \LL^{n(2g-2)}\Phi\circ \mathrm{inv}_{\m}$$
in $\expp_{\scr{A}_{\m}(\al,\be,M,N)}$. \index{inversion formula!families}
\end{prop}

\begin{proof} Note that looking at the fibres of the constructions in the previous paragraphs above every rational point $D\in S^{\m}C(k)$, we recover the theory from \cite{CL}, described in section \ref{SBreview}.  For a general schematic point $D\in S^{\m}C$, we recover a twisted version of this theory. Theorem 1.2.9 in \cite{CL} implies that we have  $$\four\four \Phi_D(x) = \LL^{n(2g-2)}\Phi_D(-x)$$ for all rational points $D\in S^{\m}C(k)$. The heart of the proof of theorem 1.2.9 is the fact that the domains of definition of our functions are affine spaces, that the Fourier kernel is bilinear and the use of the relation $[\A^1,\id] = 0$. Therefore, this proof generalises easily to the twisted setting, and the Fourier inversion formula is valid for $\Phi_D$ for any schematic point $D\in S^{\m}C$. We may conclude using lemma \ref{function.equality}.
\end{proof}

\section{Summation over $k(C)^n$} \label{summation.ratpoints} 

\subsection{Summation for twisted Schwartz-Bruhat functions}\label{twistedsummation}
We defined constructible sets $\scr{A}_{\m}(\al,\be,M,N)$ lying above the symmetric power $S^{\m}C$ in section \ref{sect.domainsofdef}, and gave an explicit description of the fibre above a point $D \in S^{\m}C$ in section \ref{sect.fibres}, which shows that it may be seen as a twisted version of the domain of definition of a Schwartz-Bruhat function in the sense of Hrushovski and Kazhdan's theory.
A function $\Phi$ on $\scr{A}_{\m}(\al,\be,M,N)$ may therefore be seen as a family $(\Phi_D)_{D\in S^{\m}C}$, each $\Phi_D$ being a function on $\scr{A}_{\m}(\al,\be,M,N)_D$. In this section, we explain how Hrushovski and Kazhdan's operation of summation over rational points from section \ref{CLsummation} may be extended to the functions $\Phi_D$.

In the notation of section \ref{sect.fibres}, we have $D = (D_U,D_{\Sigma})\in S^{\m_0}U\times S^{\m-\m_0}\Sigma$ where $U$ is a dense open subset over which $\al$ and $\be$ are zero, and $\Sigma = C\setminus U$, a finite union of closed points, is its complement. Moreover, for some partition $\pi = (m_{\i})_{\i\in\I}$ of $\m_0$, we have $D_U\in S^{\pi}U$. Via the quotient morphism
$$\left(\prod_{\i\in \I} U^{m_{\i}}\right)_*\to S^{\pi} U,$$
the schematic point $D_U\in S^{\pi}U$ pulls back to some $K$-point $D'_U$ of $\left(\prod_{\i\in \I}U^{m_{\i}}\right)_*$, where~$K$ is a finite extension of the residue field $\kappa(D)$ of $D$. We choose $D'_U$ so that $K$ is of minimal degree. We denote by $v_{\i,j}$ the projection of $D'_U$ on the $j$-th copy of $U$ in the factor~$U^{m_{\i}}$. This gives us a collection
 $\{v_{\i,j}\}_{\substack{\i\in \I\\ j\in\{1,\ldots,m_{\i}\}}}$ of~$K$-points of the open subset $U$ of the curve~$C$, all distinct because they come as projections from a point in the complement of the diagonal. %Every $v_{\i,j}$ is given by a morphism $\spec K \to U$: the schematic point of $U$ which is the image of this morphism will be denoted $w_{\i,j}$. 

\begin{remark}\label{rationalpointfibre} A different choice of $D'_U$ amounts to a permutation of the points $v_{\i,j}$ via the $G = \prod_{\i\in\I}\Sym_{m_{\i}}$-action on $\left(\prod_{\i\in \I}U^{m_{\i}}\right)_*$. More precisely, the quotient morphism $\left(\prod_{\i\in \I} U^{m_{\i}}\right)_*\to S^{\pi} U$ being étale, the fibre product $\left(\prod_{i\in \I}U^{m_{\i}}\right)_* \times_{S^{\pi}U}\spec \kappa(D)$ is the spectrum of an étale algebra $\mathcal{E}$ over $\kappa(D)$, endowed with a $G$-action such that
$\mathcal{E}^{G} = \kappa(D)$. The étale algebra $\mathcal{E}$ is isomorphic to some power of $K$, and the point~$D'_U$ is one of the irreducible components of $\spec \mathcal{E}$. If we denote by $H$ the subgroup of $G$ stabilising $D'_U$, then the invariant field $K^H$ is $\kappa(D)$ (see e.g. proposition \ref{subquotient} in chapter~\ref{eulerproducts}). Consequently, in the commutative diagram
$$\xymatrix{\spec K\ar[d]  & \ar[d]^{q_{D_U}}\ar[l]\prod_{\i\in \I}\A_K^{m_{\i}n(-M_{\i},N_{\i})}\\
            \spec \kappa(D)& \ar[l]\prod_{\i\in \I}\A_{\kappa(D)}^{m_{\i}n(M_{\i}+N_{\i})}}
$$
the vertical morphisms, induced by the quotient morphisms on the fibres above $D_U$ and~$D'_U$, are exactly the quotients by the action of the finite group $H$. The diagram gives an equivariant $K$-linear isomorphism
$$\prod_{\i\in \I}\A_K^{m_{\i}n(-M_{\i},N_{\i})}\simeq \prod_{\i\in \I}\A_{\kappa(D)}^{m_{\i}n(M_{\i}+N_{\i})}\times_{\kappa(D)} K.$$
This induces a $\kappa(D)$-linear isomorphism between the $\kappa(D)$-points of the left-hand side  (that is, the points invariant by the $H$-action) and the affine space $\prod_{\i\in \I}\A_{\kappa(D)}^{m_{\i}n(M_{\i}+N_{\i})}$. In other words, the morphism $q_{D_U}$ induces a $\kappa(D)$-linear isomorphism on $\kappa(D)$-points. 
\end{remark} 
 We define an effective zero-cycle $E_D$ on the curve $C_{\kappa(D)}$, that is, the curve $C$ seen as a curve over the field $\kappa(D)$, by
 $$E_D:= \sum_{\i\in \I}M_{\i}(v_{\i,1} + \ldots + v_{\i,m_{\i}}) - \sum_{v\in\Sigma}(\al_v - M_{\i_v})v$$
 where the $\i_v$ are given by $D_{\Sigma} = \sum_{v\in\Sigma}\i_v[v]$. The zero-cycle $E_D$ has the same field of definition as $D$, that is, $\kappa(D)$, and an associated Riemann-Roch space
 $$L_{\kappa(D)}(E_D ) = \{0\}\cup \{f\in \kappa(D)(C), \div f \geq -E_D\},$$
 which is a finite-dimensional vector space  over $\kappa(D)$.  
 
Then we may define a morphism
$$\theta_D: L_{\kappa(D)}(E_D)^n  \to \scr{A}_{\m}(\al,\be,M,N)_D = \prod_{\i\in \I}\A_{\kappa(D)}^{nm_{\i}(N_{\i} + M_{\i})}\times_{\kappa(D)} \prod_{v\in \Sigma}\A_{\kappa(D)}^{n(\al_v - M_{\i_v},\be_v + N_{\i_v})}$$
in the following manner. We start by defining an intermediate morphism $$\theta'_D: L_{\kappa(D)}(E_D)^n  \to \prod_{\i\in\I}\A_K^{nm_{\i}(-M_{\i},N_{\i})}\times_{K}\prod_{v\in \Sigma}\A_{K}^{n(\al_v - M_{\i_v},\be_v + N_{\i_v})}.$$
For simplicity, assume $n=1$.  Then, for any $f\in L_{\kappa(D)}(E_D)$:

\begin{enumerate}\item \textbf{Image of $f$ in the component $\prod_{v\in \Sigma}\A_{K}^{(\al_v - M_{\i_v},\be_v + N_{\i_v})}$:}
it is given, for each $v\in \Sigma$,  by the coefficients of the $v$-adic expansion of $f$ in the range $\al_v -M_{\i_v},\ldots ,\be_v + N_{\i_v} -1$. 
 
 \item \textbf{Image of $f$ in the component $\prod_{\i\in \I}\A_{K}^{m_{\i}(-M_{\i}, N_{\i})}$:}
 We may consider $f$ as an element of the function field $K(C)$ of the curve $C$ seen as a curve over $K$. We therefore send $f$ to the point of 
$$\prod_{\i\in \I}\prod_{j=1}^{m_{\i}}\A_K^{(-M_{\i},N_{\i})}$$
with $(\i,j)$-component given by the coefficients of orders $-M_{\i},\ldots, N_{\i}-1$ of the $v_{i,j}$-adic expansion of $f$. Note that $f$ will define a $\kappa(D)$-point of this above affine space. 
%The quotient morphism
%$$\left(\prod_{\i\in \I}\left(\A_U^{(-M_{\i},N_{\i})}\right)^{m_{\i}}\right)_{*,U}\to S^{\pi}((\A_U^{(-%M_{\i},N_{\i})})_{\i\in \I})$$ 
%to get a $\kappa(D)$-point of  $\prod_{\i\in \I}\A_{\kappa(D)}^{m_{\i}(N_{\i} + M_{\i})}$.
\end{enumerate}
Then we compose $\theta'_D$ with the quotient morphism
$$q_{D}: \prod_{\i\in\I}\A_K^{m_{\i}n(-M_{\i},N_{\i})}\times_{K}\prod_{v\in \Sigma}\A_{K}^{n(\al_v - M_{\i_v},\be_v + N_{\i_v})}\to \scr{A}_{\m}(\al,\be,M,N)_D$$ 
to get $\theta_D = q_D\circ \theta'_D$. 

\begin{remark} Because of the composition with the quotient morphism, a different choice of $D'_U$ gives the same $\theta_D$.
\end{remark}

\begin{definition} For $\Phi_D\in \expp_{\scr{A}_{\m}(\al,\be,M,N)_D}$, we define its \textit{summation over rational points}, denoted 
$\sum_{x\in\kappa(D)(C)} \Phi_D(x)$, to be the class of $\theta_D^*\Phi_D$ in $\expp_{\kappa(D)}.$ \index{summation over rational points!twisted}
\end{definition}

\begin{lemma}[Poisson formula] \label{twistedpoisson} For $\Phi_D\in \expp_{\scr{A}_{\m}(\al,\be,M,N)_D}$, we have 
$$\sum_{x\in \kappa(D)(C)^n}\Phi_D(x) = \LL^{(1-g)n} \sum_{x\in \kappa(D)(C)^n}\four \Phi_D(x).$$ \index{Poisson formula!twisted}
\end{lemma}
\begin{proof} By the same kind of reduction as in the proof of theorem 1.3.10 in \cite{CL}, it suffices to prove the formula in the case where $\Phi_D$ is given by a class $[\{a\}\to \scr{A}(\al,\be,M,N)_D]$ where $a$ is a rational point of $\scr{A}(\al,\be,M,N)_D$. We may also assume $n=1$. 

We are going to use the notations from section \ref{twistedsummation} throughout the proof, denoting with a tilde the objects pertaining to the Fourier side: $\tilde{q}_D$ for the quotient morphism
$$\prod_{\i\in\I}\A_K^{(-N_{\i},M_{\i})} \times_{K}\prod_{v\in\Sigma}\A_{K}^{(\nu_v -\be_v - N_{\i_v}, \nu_v -\al_v + M_{\i_v})}\to\scr{A}_{\m}(\nu-\be,\nu-\al, N,M)_D,$$
$\tilde{E}_D$ for the divisor
$$\tilde{E}_D = \sum_{\i\in\I}N_{\i}(v_{\i,1} + \ldots + v_{\i,m_{\i}}) - \sum_{v\in\Sigma}(\nu_v - \be_v -N_{\i_{v}})v,$$
$\tilde{\theta}_D, \tilde{\theta'}_D$ for the summation morphisms, etc. 

Define the zero-cycle 
$$\Lambda_D = \sum_{\i,j}N_{\i}v_{\i,j} + \sum_{v\in \Sigma}(\be_v + N_{\i_v})v = \tilde{E}_D - \div\, \omega$$
on $C_{\kappa(D)}$. It has the same field of definition as $D$, namely $\kappa(D)$. 

The proof is essentially the same as the proof of theorem 1.3.10 in \cite{CL}: it will boil down to the theorem of Riemann-Roch and Serre duality for the divisor $\Lambda_D$ on the curve $C_{\kappa(D)}$ over the field $\kappa(D)$. We refer to \cite{CL}, 1.3.7 for reminders on these results. 

Denote by $F_D$ the function field $\kappa(D)(C)$ of the curve $C_{\kappa(D)}$. For any divisor $E$ on $C_{\kappa(D)}$ define
$$\Omega(E) = \{\omega\in \Omega_{F_D/\kappa(D)}, \div\, \omega \geq E\},$$
$$\Ad_{F_D}(E) = \{a\in \Ad_{F_D}, \div\, a \geq -E\}.$$
Recall that Serre's duality theorem says that for any divisor $E$ on $C_{\kappa(D)}$, the morphism
$$\Omega_{F_D/\kappa(D)} \to \mathrm{Hom}(\Ad_{F_D},\kappa(D))$$
given by $$\omega\mapsto \left( (x_s)_s\mapsto \sum_{s}\mathrm{res}_s(x_s\omega)\right)$$
induces an isomorphism
$$\Omega(E) \to \mathrm{Hom}(\Ad_{F_D} / (\Ad_{F_D}(E) + F_D), \kappa(D))$$
identifying $\Omega(E)$ with the orthogonal subspace of $\Ad_{F_D}(E) + F_D$ in $\mathrm{Hom}(\Ad_{F_D},\kappa(D))$, which itself is isomorphic to the dual of the cohomology group $H^1(\scr{L}(E))$. 

By remark \ref{rationalpointfibre}, a rational point $a\in \scr{A}_{\m}(\al,\be,M,N)_D$ comes from a $\kappa(D)$-point $b$ of the fibre 
$$\prod_{\i\in\I}\A_K^{m_{\i}(-M_{\i},N_{\i})} \times_k \prod_{v\in\Sigma}\A_k^{(\al_v - M_{\i_v},\be_v + N_{\i_v})} $$
via the quotient morphism
$$q_D:\prod_{\i\in\I}\A_K^{m_{\i}(-M_{\i},N_{\i})} \times_k \prod_{v\in\Sigma}\A_k^{(\al_v - M_{\i_v},\be_v + N_{\i_v})}\to \scr{A}_{\m}(\al,\be,M,N)_D.$$
Thus, it corresponds to the characteristic function of a polydisc inside the adeles $\Ad_{K(C)}$ with $\kappa(D)$-rational centre $b$ and radius described by the divisor $\Lambda_D$. The Fourier transform $\four \Phi_D$ is defined on the $\kappa(D)$-variety $\scr{A}_{\m}(\nu-\be,\nu-\al,N,M)_D$, which is the image of the quotient map
$$\tilde{q}_D:\prod_{\i\in\I}\A_K^{m_{\i}(-N_{\i},M_{\i})} \times_k \prod_{v\in\Sigma}\A_k^{(\nu -\be_v - N_{\i_v},\nu -\al_v - M_{\i_v})}\to \scr{A}_{\m}(\nu-\be,\nu-\al,N,M)_D.$$

Let us compute the right-hand side of the Poisson formula. For this, recall that the Fourier kernel
$$r_D:\scr{A}_{\m}(\al,\be,M,N)\times_{\kappa(D)}\scr{A}_{\m}(\nu-\be,\nu-\al,N,M)\to \A^1$$
is a $\kappa(D)$-bilinear morphism satisfying
$$r_D(q_D(u),\tilde{q}_D(v)) = r(u,v)$$
for any $\kappa(D)$-rational points $u,v$ of the fibres described above, where $r$ is the Fourier kernel associated to the differential form $\omega$ on the adeles $\Ad_{K(C)}$. 

By definition, for any $y\in \scr{A}_{\m}(\nu-\be,\nu-\al, N,M)_D(\kappa(D))$, we have
\begin{eqnarray*}\four \Phi_D(y)& = &\LL^{-\deg \Lambda_D}[\{a\}\times_{\scr{A}_{\m}(\al,\be,M,N)_D}\scr{A}_{\m}(\al,\be,M,N)_D\times_{\kappa(D)}\{y\}, r_D(\pr_2,\pr_3)]\\
                                & = & \LL^{-\deg \Lambda_D}[\spec k,r_D(a,y)]\\
                                & = & \LL^{-\deg \Lambda_D}\psi(r_D(a,y))\end{eqnarray*}
                                
                            For any $f\in L_{\kappa(D)}(\div\,\omega + \Lambda_D)$, we have
                            $$r_D(a,\tilde{\theta}_D(f)) = r(b,\tilde{\theta'}_D(f)) = \sum_{s} \res_s(b_sf\omega)$$
                            where the sum goes over the points of the curve $C_K$. Note that the map $f\mapsto f\omega$ identifies $L_{\kappa(D)}(\div(\omega) + \Lambda_D)$ with $\Omega(-\Lambda_D)$. By invariance of the residue, $f\mapsto r_D(a,\theta_D(f)) $ is identically zero on $L_{\kappa(D)}(\div\,\omega + \Lambda_D)$ if and only if $b\in \Omega(-\Lambda_D)^{\perp}.$ 
                            By lemma 1.1.11 in \cite{CL}, we have 
                                \begin{eqnarray*}\sum_{x\in F_D}\four\Phi_D(x) &= &\LL^{-\deg \Lambda_D} \sum_{f\in L(\div(w) + \Lambda_D)} \psi(r_D(a,\tilde{\theta}_D(f)))\\
                                & = & \left\{\begin{array}{cc} \LL^{-\deg\Lambda_D + \dim L_{\kappa(D)}(\div\,\omega + \Lambda_D)}& \text{if\ $b\in \Omega(-\Lambda_D)^{\perp} $}\\
                                0 & \text{otherwise}.\end{array}\right. \end{eqnarray*} 
                                Thus, by the Riemann-Roch theorem applied to $\Lambda_D$ on the curve $C_{\kappa(D)}$  we have:
                                $$\LL^{1-g}\sum_{x\in F_D}\four \Phi_D(x) =\left\{\begin{array}{cc} \LL^{\dim L(-\Lambda_D)}& \text{if\ $b\in \Omega(-\Lambda_D)^{\perp} $}\\
                                0 & \text{otherwise}\end{array}\right.$$ 
We now compute the left-hand side of the Poisson formula. If $b = (b_s)_s \in \Omega(-\Lambda_D)^{\perp} = \Ad_{F_D}(-\Lambda_D) + F_D$, there exists $c\in F_D$ such that $\div(c-b)\geq \Lambda_D$. In other words, there exists an element $c$ of $F_D$ in the polydisc of centre $b$ with radii controlled by the divisor~$\Lambda_D$. Then the intersection of $F_D$ with this polydisc is exactly $c + L(-\Lambda_D)$, so that
$$\sum_{x\in F_D} \Phi_D(x) = \sum_{x\in F_D} \Phi_D(x-c) = \sum_{x\in L(-D)}1 = \LL^{\dim L(-D)}.$$
If on the other hand $b\not\in \Omega(-\Lambda_D)^{\perp}$, then the intersection of $F_D$ with the polydisc is empty, and the sum on the left-hand side of the Poisson formula is zero. This concludes the proof.
\end{proof}

\subsection{Uniform summation for uniformly compactly supported functions}

In what follows, we are going to show that summation over $\kappa(D)(C)$ may be done \textit{uniformly} in~$D$ if $\Phi$ is a family of uniformly compactly supported functions, that is, if all integers in the family $M$ are zero. Let us recall the key point of the construction in this particular case: the zero-cycle $E_D$ from section \ref{twistedsummation} is equal to $D_{\al} = -\sum_{v}\al_v v$, and we are interested in the space $L_{\kappa(D)}(D_{\al})$. Since $D_{\al}$ is defined over $k$, by flat base change (see e.g. \cite{Milne}, theorem 4.2 $(a)$) we have a $\kappa(D)$-linear canonical isomorphism
$$L_{\kappa(D)}(D_{\al})\simeq L(D_{\al})\otimes_k\kappa(D)$$
where $$L(D_{\al}) = \{0\}\cup \{f\in k(C), \div f \geq -D_{\al}\}.$$
 There is a morphism of algebraic varieties over $\kappa(D)$:
$$\theta_D: L(D_{\al})^n\otimes_k \kappa(D)\arr \scr{A}_{\m}(\al,\be,0,N)_D,$$
constructed in the previous section. %\left(\prod_{v}\A_{\kappa(D)}^{(\al_v,\be_v + N_{\m_v})}\right)^n.$$
%The sum over rational points of a function $\phi\in \expp_{\left(\prod_{v}\A_{\Omega}^{(\al_v,\be_v + N_{\m_v})}\right)^n}$
% is then defined to be the image in $\expp_{\kappa(D)}$ of the pull-back $\theta_D^*\phi\in\expp_{L(D_{\al})^n\otimes_k\kappa(D)}.$
 
 The aim of the following proposition is to show that, as $D$ varies in $S^{\m}C$, the maps $\theta_D$ can be combined into a morphism performing summation over rational points uniformly in $D$. Of course, our uniform choice of uniformisers will be crucial here.

\begin{prop} \label{prop.uniformsummationuc} There exists a constructible morphism
$$\theta_{\m}: L(D_{\al})^n\times S^{\m}C \arr \scr{A}_{\m}(\al,\be,0,N)$$
over $S^{\m}C$ such that for any $D\in S^{\m}C$, the induced morphism 
$$L(D_{\al})^n\otimes_k \kappa(D)\arr \scr{A}_{\m}(\al,\be,0,N)_D$$
on fibres above $D$ is exactly the morphism $\theta_D$ constructed above.
\end{prop}

\begin{proof}  Lemma \ref{coef} shows that for any $f\in L(D_{\al})$ and any $\i\in \I_0$, there is a constructible morphism
$$\phi_{f,\i}: C\arr \A_C^{(\al,\be + N_{\i})}.$$
over $C$, sending $v\in C$ to $(v, a_{\al_{v}}(f,v),\ldots,a_{\be_v + N_{\i}-1}(f,v))$ (where we use the notation from lemma \ref{coef}).
Taking products over $C$, for any $f = (f_1,\ldots,f_n)\in L(D_{\al})^n,$ and any $\i\in \I_0$, there is a constructible morphism
$$\phi_{f,\i} :C\arr \A_C^{n(\al,\be + N_{\i})}.$$
Taking symmetric products, for any $\m\in \I_0$ we get morphisms
$$S^{\m}\phi_f: S^{\m}C\arr \scr{A}_{\m}(\al,\be,0,N).$$
%Explicitly, this morphism sends $D = \sum_{v}\m_vv\in S^{\m}C(k)$ to the point of 
%$\prod_{v}\A_{k}^{n(\al_v,\be_v + N_{\m_v})}$
%giving, for every $v$, the coefficients of the $t_v-$adic expansion of $f_1,\ldots,f_n$ in a range between $\alpha_v$ and $\be_v + N_{\m_v}$.

Finally, $L(D_{\al})^n$ being a finite-dimensional $k$-vector space, combining these morphisms for a basis of the latter, we get constructible morphisms
$$\theta_{\m}: L(D_{\al})^n\times S^{\m}C \arr \scr{A}_{\m}(\al,\be,0,N).$$
By construction, the restriction to the fibre over a schematic point $D\in S^{\m}C$ corresponds indeed to $\theta_D$.
\end{proof}
We see $L(D_{\al})^n\times S^{\m}C$ as a variety over $S^{\m}C$ via the second projection, which yields a group morphism 
$$\expp_{L(D_{\al})^n\times S^{\m}C}\arr \expp_{S^{\m}C},$$
interpreted as ``summation over rational points in each fibre''. Using this morphism, we can define uniform summation over rational points:

\begin{definition} \label{cs.summable} Let $\Phi\in\expp_{\scr{A}^{\m}(\al,\be,0,N)}$ be a family of uniformly compactly supported functions. The uniform summation of $\Phi$ over rational points, denoted by $$\left(\sum_{x\in \kappa(D)(C)^n}\Phi_D(x)\right)_{D\in S^{\m}C},$$ is the image in $\expp_{S^{\m}C}$ of the pullback $\theta_{\m}^*\Phi$.
\end{definition}

\begin{remark}[Order of summation] \label{orderofsummation} Starting from a uniformly compactly supported family $\Phi\in \scr{A}^{\m}(\al,\be,0,N)$ and pulling it back to $L(D_{\al})^n\times S^{\m}C$ via $\theta_{\m}$, we obtain an object of $\expp_{L(D_{\al})^n\times S^{\m}C}$, from which we can get an object of $\expp_k$, by either projecting first to $S^{\m}C$ and then to $k$, or first to $L(D_{\alpha})$ and then to $k$. The fact that these two operations commute can be interpreted in terms of motivic sums as the possibility of interchanging the order of summation:
$$\sum_{D\in S^{\m}C}\ \sum_{x\in \kappa(D)(C)^n}\Phi_D(x) = \sum_{x\in k(C)^n}\sum_{D\in S^{\m}C} \Phi_D(x).$$
Here $\sum_{x\in k(C)^n}$ denotes summation over $L(D_{\alpha})$.
\end{remark}

\section{Poisson formula in families}\label{Poisson.families}
In this section we describe the way in which the  Poisson formula is going to be used in section \ref{applicPoisson} of chapter \ref{motheightzeta}. 
\subsection{Uniformly summable families}\label{sect.uniformlysummable}
\begin{definition} We say that a constructible family of functions $$\Phi\in\expp_{\scr{A}_{\m}(\al,\be,M,N)}$$ is uniformly summable \index{uniformly summable} if there is an element  $\Sigma \in \expp_{S^{\m}C}$ such that for any schematic point $D\in S^{\m}(C)$, the pullback of~$\Sigma$ along $D$ is $\sum_{x\in \kappa(D)(C)^n}\Phi_D(x)\in \expp_{\kappa(D)}$. 
\end{definition}
\begin{remark}\label{uniqueness.sum} By lemma \ref{function.equality}, such an element is unique. 
\end{remark}
\begin{notation}\label{summable.not} The element $\Sigma$ from the definition will be denoted $$\left(\sum_{x\in \kappa(D)(C)^n}\Phi_D(x)\right)_{D\in S^{\m}C}.$$
\end{notation}
\begin{example}\label{cs.summable.ex} We showed in the previous section that a uniformly compactly supported family of functions is uniformly summable. The notation in definition \ref{cs.summable} is consistent with the one in \ref{summable.not}.
\end{example}

\subsection{Motivic Poisson formula in families}\label{poissonfamiliesproof}
Let $\Phi\in\expp_{\scr{A}_{\m}(\al,\be,M,N)}$ be a constructible family of Schwartz-Bruhat functions, which we assume to be uniformly summable.  By the motivic Poisson formula as it is stated in lemma \ref{twistedpoisson}, for any schematic point $D\in S^{\m}C$, we have 
$$\sum_{x\in \kappa(D)(C)^n}\Phi_D(x) = \LL^{(1-g)n}\sum_{y\in \kappa(D)(C)^n}\four\Phi_D(y).$$
We may conclude from this and from lemma \ref{function.equality} that the family $$\four \Phi\in\expp_{\scr{A}_{\m}(\nu-\be,\nu-\al,N,M)}$$ is uniformly summable as well, and that
\begin{equation}\label{uniform.Poisson}\left(\sum_{x\in \kappa(D)(C)^n}\Phi_D(x)\right)_{D\in S^{\m}C} = \LL^{(1-g)n}\left(\sum_{y\in \kappa(D)(C)^n}\four\Phi_D(y)\right)_{D\in S^{\m}C}\end{equation}
as elements of $\expp_{S^{\m}C}.$ 

In other words, the Poisson formula from lemma \ref{twistedpoisson} shows that $\Phi$ is uniformly summable if and only if  $\four \Phi$ is uniformly summable, and that in this case the families of their sums are related by a Poisson formula \ref{uniform.Poisson}. \index{Poisson formula!in families}

\subsection{The case of a uniformly smooth family}\label{us.poisson}
In chapter \ref{motheightzeta}, section \ref{applicPoisson}, we are going to use this formula in the case where $\Phi$ is a uniformly smooth family, so that $\four\Phi$ is uniformly compactly supported. By example \ref{cs.summable.ex}, $\four \Phi$ is then uniformly summable, and section \ref{poissonfamiliesproof} states that so is $\Phi$, and that we have equality (\ref{uniform.Poisson}). Taking classes in $\expp_k$ (written as motivic sums), we then have:

\begin{equation}\label{sc.Poisson}\sum_{D\in S^{\m}C}\sum_{x\in \kappa(D)(C)^n}\Phi_D(x) = \LL^{(1-g)n} \sum_{D\in S^{\m}C}\sum_{y\in \kappa(D)(C)^n}\four\Phi_D(y).\end{equation}
\index{Poisson formula!uniformly smooth family}

\subsection{Reversing the order of summation}\label{reversesummation} By remark \ref{orderofsummation}, it makes sense to reverse the order of summation in the right-hand-side of (\ref{sc.Poisson}), to get:
%The element $\sum_{D\in S^{\m}C}\sum_{y\in k(C)^n}\four\Phi_D(y)$ in the right-hand side is the image of $\theta_{\m}^*\four\Phi$ in $\expp_k$ (see definition \ref{cs.summable}). The latter is equal to the element obtained by taking the image in $\expp_k$ of the image  of $\theta_{\m}^*\Phi$ through the group morphism
%$$\expp_{L(D_{\al})^n\times S^{\m}C}\arr \expp_{L(D_{\al})^n}$$
%induced by the first projection, which in terms of motivic sums, may be denoted 
%$$ \sum_{y\in k(C)^n}\sum_{D\in S^{\m}C}\four\Phi_D(y).$$
%In particular, we have 
\begin{equation}\label{equation.reverseorder}\sum_{D\in S^{\m}C}\sum_{x\in \kappa(D)(C)^n}\Phi_D(x) = \LL^{(1-g)n}\sum_{y\in k(C)^n}\sum_{D\in S^{\m}C}\four\Phi_D(y).\end{equation}

\begin{remark}\label{dropkappaD} For simplicity of notation, in chapter \ref{motheightzeta} we will drop the mention of $\kappa(D)$ in the summations, and write simply $\sum_{x\in k(C)}$. 
\end{remark}
%We may conclude the following:

%\begin{prop} Let $\Phi\in\expp_{\scr{A}_{\m}(\al,\be)}$ be a family of Schwartz-Bruhat functions. Then there is an element $M\in\expp_{S^{\m}C}$ such that for every point $D\in S^{\m}C$, 
%$$M_D = \sum_{x\in k(C)^n}\Phi_{D}(x).$$
%Moreover, 

%\end{prop}

\chapter{Motivic height zeta functions}\label{motheightzeta}
This final chapter leads up to the proof of theorem \ref{main} and corollary \ref{maincor} from the introduction. Section~\ref{sectGeometry} starts by defining equivariant compactifications of $\G^n_a$, giving some of their properties, fixing some notations and recalling the geometric framework from assumptions~\ref{hypothese.geom} and~\ref{hypothese.section}, as well as from \cite{CL}. 

In section \ref{sect.heightzeta} we define the moduli spaces we are interested in and introduce their generating series, namely motivic height zeta functions.
We then decompose the moduli spaces of sections with respect to their local behaviour at each place $v\in C$, and make use of this to write them as summations of suitable constructible families of Schwartz-Bruhat functions
in the terminology of chapter~\ref{poissonformula}. Applying the motivic Poisson formula as stated in section \ref{Poisson.families} of chapter~\ref{poissonformula} we obtain a decomposition 
 \begin{equation}\label{equation.motivicheightintro}Z(T) = \sum_{\xi\in G(F)}Z(T,\xi) = \sum_{\xi\in G(F)} \prod_{v\in C} Z_v(T,\xi)
 \end{equation}
of the  motivic height zeta function (in fact, even of its multivariate version) as a motivic sum of series with coefficients in $\expp_{k}$, each of which has an Euler product decomposition. We then investigate in section~\ref{sectLocal} the convergence of the factors $Z_v(T,\xi)$, first for the places $v\in C_0$, then for places $v\in S = C\setminus C_0$. Each factor takes the form of a motivic Fourier transform which can be rewritten, using ideas from~\cite{CL}, as a motivic integral over the arc space of the chosen model~$\scr{X}$ of our equivariant compactification~$X$. The first three subsections of section~\ref{sectLocal} recall some results from~\cite{CL} used for this, and contain some general computations of motivic integrals. These local computations, together with the convergence result (proposition \ref{convergence}) from chapter \ref{hodgemodules}, yield convergence statements for each $Z(T,\xi)$ for $\xi$ both trivial and non-trivial. More precisely:
\begin{itemize} \item We show that the product $\prod_{v\in C_0}Z_v(T,0)$ has a meromorphic continuation for $|T|<\LL^{-1 + \delta}$ for some $\delta>0$ and compute the exact order of its pole at $\LL^{-1}$, thanks to assumption~\ref{hypothese.section} on local existence of sections.  
\item By an analysis analogous to the one in \cite{CLT}, we show that for every non-zero~$\xi$, the product $\prod_{v\in C_0}Z_v(T,\xi)$ has a meromorphic continuation for $|T|<\LL^{-1 + \delta}$ for some~$\delta>0$, and get a bound on the order of the pole.
\item  To get the exact contribution of $\prod_{v\in S}Z_v(T,0)$ to the order of the pole at $\LL^{-1}$, we rewrite each factor using the language of Clemens complexes.
\item Using arguments from \cite{CLT,CLTi, CL}, we observe that the summation in equation~(\ref{equation.motivicheightintro}) is in fact over $\xi$ in some finite-dimensional $k$-vector space $V$, and we bound the order of the pole of $\prod_{v\in S} Z_v(T,\xi)$ uniformly on the strata of a finite constructible partition of~$V\setminus\{0\}$, in terms of dimensions of some explicit subcomplexes of the Clemens complexes appearing for $\xi = 0$. 
\end{itemize}
We finish by combining these results to establish theorem~\ref{main}, and deduce corollary \ref{maincor}, first over $\C$, then over a general algebraically closed field~$k$ of characteristic zero. 
\section{Geometric setting}\label{sectGeometry}
\subsection{Equivariant compactifications}\label{sect.equivariant}
Let $k$ be an algebraically closed field of characteristic zero. Let $C_0$\index{C0@$C_0$} be a smooth quasi-projective curve over~$k$, $C$ be its  smooth projective compactification, and  $S = C\setminus C_0$. We denote by $F = k(C)$ the function field of $C$, $g$ its genus, and $G$ the additive group scheme~$\G^{n}_{a}$. A smooth projective \emph{equivariant compactification} \index{equivariant compactification} of $G_F$ is a smooth projective $F$-scheme $X$ containing $G_F$ \index{GF@$G_F$} as a dense open subset, and such that  the group law $G_F\times G_F\arr G_F$ extends to a group action $G_F\times X\arr X$. 

The geometry of such varieties has been investigated in \cite{HT}. We summarise here the main facts that will be used in this chapter.

\begin{prop} Let $X$ be a smooth projective equivariant compactification of~$G_{F}$.
\begin{enumerate} \item The boundary $X\setminus G_F$ is a divisor.  
\item The Picard group of~$X$ is freely generated by the irreducible components $(D_{\alpha})_{\alpha\in\scr{A}}$ of $X\setminus G_F$. \index{A@$\scr{A}$}\index{Da@$D_{\al}$}
\item The closed cone of effective divisors $\Lambda_{\mathrm{eff}}(X) \subset \mathrm{Pic}(X)_{\R}$ is given by
$$\Lambda_{\mathrm{eff}}(X) = \bigoplus_{\al\in \scr{A}}\R_{\geq 0}D_{\al}.$$
\item Up to multiplication by a scalar, there is a unique $G_F$-invariant meromophic differential form~$\omega_X$ on $X$. Its restriction to $G_F$ is proportional to the form $\dx x_1\wedge \ldots \wedge \dx x_n$.
\item There exist integers $\rho_{\al}\geq 2$ such that the divisor \index{omega@$\omega_X$} $-\div(\omega_X)$ is given by
$$-\div(\omega_X) = \sum_{\alpha\in\scr{A}}\rho_{\alpha}D_{\alpha}.$$ \index{rho@$\rho_{\al}$}
\end{enumerate}
\end{prop}
We will moreover assume that the divisor $X\setminus G_F$ has strict normal crossings. This  means that the components $D_{\alpha}$ are geometrically irreducible, smooth and meet transversally. 

\subsection{Partial compactifications}
A \textit{partial compactification} \index{partial compactification} \index{equivariant compactification!partial} of $G_F$ is a smooth quasi-projective scheme~$U$, containing $G_F$ as an open subset, and endowed with an action of $G_F$ which extends the group law $G_F\times G_F\arr G_F$. If $U$ is an open subscheme of a projective smooth equivariant compactification~$X$ of $G_F$, globally invariant under the action of $G_F$, with complement a reduced divisor $D=X\setminus U$, we denote by $\scr{A}_D$ the subset of $\scr{A}$ such that
$$D = \sum_{\al\in \scr{A}_D}D_{\al},$$\index{AD@$\scr{A}_D$}
where $(D_{\al})_{\al\in\scr{A}}$ are the irreducible components of $X\setminus G_F$. The log-anticanonical class \index{log-anticanonical class} of the partial compactification $U$ is the class of the divisor 
$$-K_X(D) = -(K_X + D) = \sum_{\al\in\scr{A}_D}(\rho_{\al}-1)D_{\al} + \sum_{\al\not\in\scr{A}_D}\rho_{\al}D_{\al} = \sum_{\al\in\scr{A}}\rho'_{\al}D_{\al},$$
\index{KXD@$K_X(D)$} \index{KX@$K_X$} \index{rhop@$\rho'_{\al}$}
where $\rho'_{\al} = \rho_{\al}-1$ for $\al\in\scr{A}_D$, and $\rho_{\al}$ otherwise. Since $\rho_{\al} \geq 2$, it belongs to the interior of the effective cone of $X$, and therefore it is big.

%\subsection{Models}

%We are interested in sections $\sigma: C\to \scr{X}$ of fixed degree $n$ with respect to~$\scr{L}$ such that $\sigma(\eta_C)\in \G^{n}_{a}(F)$ (where $\eta_C$ is the generic point of $C$), and $\sigma(C_0)\subset \scr{U}$. We assume that there is no local obstruction for such sections to exist, that is, for all $v\in C_0$ we have $\G^n_{a}(F_v)\cap \scr{U}(\OO_v) \neq \varnothing$, where $F_v$ is the completion of $F$ at $v$, and $\OO_v$ its ring of integers.

\subsection{Choice of a good model}\label{sect.goodmodelchoice}
%A \emph{model} of $X$ over $X$ is a projective flat scheme $\pi:\scr{X}\arr C$ with generic fibre equal to $X$. We will say that $\scr{X}$ is a \emph{good model} if moreover
%\begin{itemize}\item it is a regular scheme;
%\item the sum of the non-smooth fibres of $\scr{X}$ and of the closures $\scr{D}_{\al}$ of the divisors $D_{\al}$ is a divisor with strict normal crossings on $\scr{X}$.  
%\end{itemize} 

%We choose a good model $\pi:\scr{X}\to C$ of $X$ over $C$ (its existence is guaranteed by embedded resolution of singularities in characteristic zero).
Let $k, C_0, C, F$ be as in section \ref{sect.equivariant}. We now assume we are in the situation described in the introduction, namely that we are given  a projective irreducible $k$-scheme $\scr{X}$ \index{X@$\scr{X}$} together with a non-constant morphism $\pi:\scr{X}\to C$, $\scr{U}$ a Zariski open subset of $\scr{X}$, and $\scr{L}$ a line bundle on $\scr{X}$. We make the following assumptions on the generic fibres $X = \scr{X}_F$ and $U = \scr{U}_{F}$:
\begin{itemize}\item $X$ is a smooth equivariant compactification of $G_F$, and $U$ is a partial compactification of $G_F.$ 
\item the boundary $D= X\setminus U$ is a strict normal crossings divisor. 
\item the restriction $L$ of the line bundle $\scr{L}$ to $X$ is the log-anticanonical line bundle $-K_X(D)$. 
\end{itemize} 
We also assume that for all $v\in C_0$ we have $\G(F_v)\cap \scr{U}(\OO_v) \neq \varnothing$, where $F_v$ is the completion of~$F$ at $v$, and $\OO_v$ its ring of integers. According to lemma 3.4.1 in \cite{CL}, up to modifying the  models without changing this hypothesis nor the counting problem we are going to deal with, we may assume additionally that in fact~$\scr{X}$ is a good model, \index{good model} that is, it is smooth over $k$ and the sum of the non-smooth fibres of $\scr{X}$ and of the closures~$\scr{D}_{\al}$ \index{dalpha@$\scr{D}_{\alpha}$} of the irreducible components $D_{\al}$ of $X\setminus U$ is a divisor with strict normal crossings on~$\scr{X}$. Moreover, we may assume that $\scr{U}$ is the complement in $\scr{X}$ to a divisor with strict normal crossings. We will make use of the notations concerning equivariant compactifications introduced in section \ref{sect.equivariant}.

\subsection{Vertical components} \index{vertical components} For every $v\in C(k)$, we write $\scr{B}_v$ for the set of irreducible components of $\pi^{-1}(v)$, and $\scr{B}$ for the disjoint union of all $\scr{B}_v$, $v\in C(k)$. \index{B@$\scr{B}$, $\scr{B}_v$} For $\beta\in \scr{B}_v$, we denote by $E_{\be}$ \index{Eb@$E_{\be}$} the corresponding component, and by $\mu_{\be}$ \index{mub@$\mu_{\be}$} its multiplicity in the special fibre of $\scr{X}$ at $v$.  

The line bundle $\scr{L}$ is of the form $\sum_{\al\in\scr{A}}\rho'_{\al}\scr{L}_{\al}$ where for every $\al$, the line bundle $\scr{L}_{\al}$ on~$\scr{X}$ extends $D_{\al}$. We may write:
$$\scr{L}_{\al} = \scr{D}_{\al} + \sum_{\be\in \scr{B}}e_{\al,\be}E_{\be}$$\index{Lalpha@$\scr{L}_{\alpha}$}
where the integers $e_{\al,\be}$ are zero for almost all $\be$. We also define integers $\rho_{\be}$ such that
$$-\div(\omega_X) = \sum_{\al}\rho_{\al}\scr{D}_{\al} + \sum_{\be\in\scr{B}}\rho_{\be}E_{\be}$$
where $\omega_X$ is seen as a meromorphic section of $K_{\scr{X}/C}$.   

\subsection{Weak Néron models}\label{Neron} \index{weak Néron model} Let $\scr{B}_1$ \index{B1@$\scr{B}_1$, $\scr{B}_{1,v}$} be the subset of $\scr{B}$ consisting of those $\beta$ for which  the multiplicity $\mu_{\be}$ is equal to~1. Let $\scr{B}_{1,v} = \scr{B}_1\cap \scr{B}_v$. Let $\scr{X}_1$ be the complement in $\scr{X}$ of the union of the components $E_{\be}$ for $\be\in \scr{B}\setminus\scr{B}_1$ as well as of the intersections of distinct vertical components. It is a smooth scheme over~$C$. By lemma 3.2.1 in \cite{CL}, the $C$-scheme $\scr{X}_1$ is a weak Néron model of $X$. This means that for every smooth $C$-scheme~$\scr{Z}$, the canonical map
$$\mathrm{Hom}_C(\scr{Z},\scr{X}_1)\to \mathrm{Hom}_F(\scr{Z}_F,X)$$
is a bijection.

Applying this to $\scr{Z} = C$, we see that in particular, any rational point $g:\spec F\to X$ extends to a section $\sigma_g:C\to \scr{X}_1$ with image inside $\scr{X}_1$. %In other words, there is a bijection between the set of sections $\sigma:C\to \scr{X}_1$ and the set of rational points $X(F)$

%\subsubsection{Heights}
%There is a bijection
%$$\begin{array}\left\{\text{sections}\ \sigma:C\to\scr{X}
\section{Height zeta functions} \label{sect.heightzeta}
\subsection{Definition} \label{sect.heightzetadef}
  Let $\lambda =(\lambda_{\al})_{\al\in\scr{A}}$ be a family of positive integers, and let $$\scr{L}_{\lambda} = \sum_{\al\in\scr{A}}\lambda_{\al}\scr{L}_{\al}$$ be the corresponding line bundle on the good model $\scr{X}$. % We assume that for all $v\in C_0$ we have $\G^n_{a}(F_v)\cap \scr{U}(\OO_v) \neq \varnothing$, where $F_v$ is the completion of $F$ at $v$, and $\OO_v$ its ring of integers. This ensures local existence of the sections we are going to count. 
  % Lemma 3.4.1 in \cite{CL} allows us to assume without loss of generality that $\scr{U}$ is an open subset of $\scr{X}$.
  For every integer $n\in\Z$ and any $\n = (n_{\al})_{\al\in \scr{A}}\in\Z^{\scr{A}}$ let $M_{U,n}$ be the moduli space of sections $\sigma:C\to\scr{X}$ such that 
  \begin{itemize} \item the section $\sigma$ maps the generic point $\eta_C$ of $C$ to a point of $G_F$. 
                  \item it represents an $S$-integral point of $U$, that is, $\sigma(C_0)\subset \scr{U}$. 
                  \item $\deg_{C}(\sigma^*\scr{L}_{\lambda}) = n$
                  \end{itemize}
and 
 $M_{U,\mathbf{n}}$ the moduli space of sections $\sigma:C\to\scr{X}$ such that \index{moduli space of sections} \index{MUn@$M_{U,n}$, $M_{U,\n}$}
 \begin{itemize} \item the section $\sigma$ maps the generic point $\eta_C$ of $C$ to a point of $G_F$.
 \item it represents an $S$-integral point of $U$, that is, $\sigma(C_0)\subset \scr{U}$. 
  \item for all $\alpha\in \scr{A}$, $\deg_{C}(\sigma^*\scr{L}_{\al}) = n_{\al}.$
  \end{itemize}
 According to Proposition 2.2.2 in \cite{CL}, these moduli spaces exist as constructible sets over $k$, and there exists an integer $m$ such that $M_{U,n} $ (resp. $M_{U,\mathbf{n}}$) is empty when $n<m$ (resp. when $n_{\al}<m$ for all $\al\in\scr{A}$). Moreover, $M_{U,n}$ can be seen as the disjoint union of all $M_{U,\mathbf{n}}$ such that $\sum_{\al\in\scr{A}}\lambda_{\al}n_{\al} = n$. 
 
 The multivariate motivic height zeta function is given by
 $$Z(\T) = \sum_{\mathbf{n}\in\Z^{\scr{A}}}[M_{U,\mathbf{n}}]\T^{\mathbf{n}}\in\M_k[[(T_{\al})_{\al\in\scr{A}}]]\left[\prod_{\al\in\scr{A}}T_{\al}^{-1}\right],$$ \index{ZT@$Z(\T)$, $Z(T)$, motivic height zeta function}\index{motivic height zeta function!multivariate}
 and the motivic height zeta function associated to the line bundle $\scr{L}$, defined as $$Z_{\lambda}(T) = \sum_{n\in \Z}[M_{U,n}]T^n\in\M_k[[T]][T^{-1}],$$ can be written in the form
 $$Z_{\lambda}(T) = Z((T^{\lambda_{\al}})) = \sum_{m\in \Z}\left(\sum_{\substack{\mathbf{n}\in\Z^{\scr{A}}\\\sum_{\al}\lambda_{\al}n_{\al} = m}}[M_{U,\mathbf{n}}]\right)T^m.$$\index{motivic height zeta function!associated to line bundle}\index{Zl@$Z_{\lambda}(T)$}
 As stated in the introduction, we will investigate the case  where $\scr{L} = \scr{L}_{\rho'}= \sum_{\al\in\scr{A}}\rho'_{\al}\scr{L}_{\al}$ is generically the log-anticanonical line bundle. Denote by $Z(T)$ the corresponding height zeta function. 
 
%\subsection{Local descriptions}
\subsection{Local intersection degrees}
Let $v\in C(k)$. To every point $g\in G(F_v)$, one can attach local intersection degrees $(g,\scr{D}_{\al})_v\in\N$ for all $\al\in\scr{A}$ and $(g,E_{\beta})_v$ for all $\beta\in\scr{B}_v$ such that \index{local intersection degree}
\begin{enumerate}\item For every $g\in G(F)$ and every $\al\in\scr{A}$, one has
$$\deg_{C}(\sigma_g^*(\scr{D}_{\al})) = \sum_{v\in C(k)}(g,\scr{D}_{\al})_v$$\index{gd@$(g,\scr{D}_{\al})_v$}
where $\sigma_g:C\to\scr{X}$ is the canonical section extending $g$.
\item There is exactly one $\beta\in\scr{B}_v$ such that $(g,E_{\beta})_v = 1$, and this $\beta$ has multiplicity one. For any $\beta'\in \scr{B}_v$ different from this $\beta$, one has $(g,E_{\beta'})_v = 0$ (see section \ref{Neron}).
\end{enumerate}
We refer to \cite{CL}, 3.3 for details. 

\subsection{Decomposition of $G(F_v)$}\label{decomposition}
We can decompose $G(F_v)$ into definable (in the Denef-Pas language, see section~2.1 of \cite{CluckLos}) bounded domains on which all the above intersection degrees are constant: for all $\m\in\N^{\scr{A}}$ and all $\beta\in\scr{B}_v$, we define
$$G(\m,\beta)_v =\{g\in G(F_v),\ (g,E_{\beta})_v = 1\ \text{and}\ (g,\scr{D}_{\al})_v = m_{\al}\ \text{for all}\ \al\in\scr{A}\}$$
and $G(\m)_v = \cup_{\beta\in\scr{B}_v}G(\m,\beta)_v$. \index{gmbeta@$G(\m,\beta)_v$, $G(\m)_v$}
Lemma 3.3.2 in \cite{CL} says that for any $\m\in\N^{\scr{A}}$ and any $\beta\in\scr{B}_v$, the set $G(\m,\beta)_v$ is a bounded definable subset of $G(F_v)$, and that $G(F_v)$ is the disjoint union of all the $G(\m,\beta)_v$ for $\m\in\N^{\scr{A}}$ and $\be\in\scr{B}_v$.

Moreover, Lemmas 3.3.3 and 3.3.4 from \cite{CL} can be summarised in the following proposition:

\begin{prop}\label{gm} There exists a dense open subset $C_1$ of $C_0$ such that for every $v\in C_1(k)$:
\begin{enumerate}[(i)]\item The set $\scr{B}_v$ contains only one element $\be_v$.
\item The set $G(\mathbf{0},\be)_v$ is equal to the subgroup $G(\OO_v)$ of $G(F_v)$.  
\end{enumerate}
Moreover, there is an almost zero function $s:C\to\Z$ such that for every $v\in C(k)$, $\m\in\N^{\scr{A}}$ and $\beta\in\scr{B}_v$, the set $G(\m,\beta)_v$ is invariant under the subgroup $G(\mathfrak{m}_v^{s_v})$ of $G(\OO_v)$, where $\mathfrak{m}_v$ is the maximal ideal of $\OO_v$, and one can take $s_v=0$ for all $v\in C_1(k)$. 
\end{prop}
As a consequence of this, as in Corollary 3.3.5 from \cite{CL}, the characteristic function of each set $G(\m,\beta)_v$ may be seen as motivic Schwartz-Bruhat functions on $G(F_v)$ in the sense of \ref{sect.localSB}, with $N=s_v$ and with $M$ such that $G(\m,\beta)_v\subset (t^M\OO_v)^n$ (which exists by boundedness).  
\begin{notation}\label{C1} In what follows, it will be convenient for us to consider an even smaller set~$C_1$. Namely, from now on $C_1$ denotes the open subset of places $v\in C_0$ satisfying 
\begin{itemize}
\item the conditions in proposition \ref{gm};
\item $e_{\al,\be_v} = 0$ for all $\al\in\scr{A}$; 
\item $\rho_{\be_v} = 0$
\item For all $\al\in \scr{A}\setminus \scr{A}_D$, $\scr{D}_{\al}\times_C C_1\to C_1$ is smooth.\end{itemize}\end{notation}\index{C1@$C_1$}

\begin{prop}\label{gmdeffamily} There exist almost zero functions $s$ and $s'$, an unbounded family $N = (N_{\m})_{\m \in \N^{\scr{A}}}$ with $N_{\mathbf{0}} = 0$, and for every $\m\in \N^{\scr{A}}$ and every $\beta\in \prod_{v}\scr{B}_v$, a constructible subset $G_{\m,\beta}$ of $\A_C^{n(s'-N_{\m},s)}$, the fibre of which at every $v\in C(k)$ is exactly $G(\m,\beta_v)_v$. \index{Gmbeta@$G_{\m,\be}$}
\end{prop}
\begin{proof} The definable boundedness of $G(\m,\be)_v$ means that there exists an almost zero function $s':C\to \Z$, and, for every $\m\in\N^{\scr{A}}$, a non-negative integer $N_{\m}$ such that 
$$G(\m,\be)_v\subset (t_v^{s'_v-N_{\m}}\OO_v)^{n},$$
for all $v\in C$.  Statement $(ii)$ in proposition \ref{gm} implies that we may take $N_{\mathbf{0}} = 0$. Since $G(F_v)$ is the union of all the $G(\m,\beta)_v$, the family $(N_{\m})_{\m}$ is necessarily unbounded. Taking the almost zero function $s$ from proposition \ref{gm}, we see that $G(\m,\be)_v$ defines naturally a constructible subset of $\A_k^{n(s'_v-N_{\m},s_v)}$: the conditions on the intersection degrees will indeed translate into Zariski open  and Zariski closed polynomial conditions on the coordinates of $\A_k^{n(s'_v-N_{\m},s_v)}$.  Let $\scr{Y}$ be an affine subset of $\scr{X}$ such that $\pi_{|\scr{Y}}:\scr{Y}\to C$ is non-constant. Let~$C'$ be an open dense subset of $C$ contained in the intersection of the image of $\pi_{|\scr{Y}}$ and of the open subset $C_1$ from proposition \ref{gm}. For every $\al\in\scr{D}_{\al}$, let $f_{\al}$ be a local equation for $\scr{D}_{\al}$ in $\scr{Y}$. %, and for every $\beta\in \scr{B}_v$ with $v\in \pi_{|\scr{Y}}(\scr{Y}}$, let $g_{\beta}$ be an equation of $E_{\beta}$ in $\scr{Y}$.
 Then the condition that $(g,\scr{D}_{\al})_v = m_{\al}$ may be written $\ord (f_{\al}(g)) = m_{\al}$ for all $v$ in the image of $\pi_{|\scr{Y}}$. We thus see that we may take the same integer $N_{\m}$ for all $v$, and that for all $v\in C'$, the equations defining $G(\m)_v$ inside $\A_k^{n(s'_v-N_{\m},s_v)}$ are uniform in $v$ (in the sense that they are defined by the same formula in  the Denef-Pas Language for every such $v$). This, together with the fact that the $G(\m,\beta)_v$, for $v\in C\setminus C'$ are constructible, guarantees the existence of a constructible subset of $\A_C^{n(s'-N_{\m},s)}$ as in the statement of the proposition.
\end{proof}

\subsection{Integral points}\label{sect.integral}

The  complement of the model $\scr{U}$ inside $\scr{X}$ is the union of the divisors $\scr{D}_{\al}$ for $\al\in\scr{A}_D$, and of the vertical components $E_{\be}$ for a finite subset $\scr{B}^0$ of~$\scr{B}$. We then set $\scr{B}_v^{0} = \scr{B}^0\cap \scr{B}_v$ for every $v\in C(k)$, and define 
$$\scr{B}_0 = \scr{B}_1\setminus \left(\cup_{v\in C_0}\scr{B}_v^0\right),$$
and $\scr{B}_{0,v} = \scr{B}_0\cap \scr{B}_v$. In other words, $\scr{B}_{0}$ corresponds to vertical components of multiplicity one which either lie above $S$ or are contained in $\scr{U}$. Thus in particular $\scr{B}_{0,v} = \scr{B}_{1,v}$ for $v\in S$. \index{B0@$\scr{B}_0$, $\scr{B}_{0,v}$}

Let $\m_v\in\N^{\scr{A}}$ and $\be_v\in\scr{B}_v$. We say that the pair $(\m_v,\be_v)$ is $v$-\textit{integral} if \index{vintegral@$v$-integral}
\begin{itemize}\item either $v\in S$
\item or $v\in C_0$, $\be_v\in\scr{B}_0$ and $m_{\al,v} = 0$ for every $\al\in \scr{A}_D$. 
\end{itemize}
In other words, the union of the sets $G(\m_v,\be_v)_v$ for all $v$-integral pairs $(\m_v,\be_v)$ is equal to $\scr{U}(\OO_v)$ if $v\in C_0$, and $G(F_v)$ otherwise.

For any $(\m_v,\be_v)\in \N^{\scr{A}}\times \scr{B}_v$, define 
$$H(\m_v,\be_v)_v = \left\{\begin{array}{cc} G(\m_v,\be_v)_v& \text{if}\  (\m_v,\be_v)\ \text{is~$v$-integral}\\\varnothing& \text{else}.\end{array}\right.$$
 Then the union of the sets $H(\m_v,\be_v)_v$ for all $(\m_v,\be_v)$ is equal to~$\scr{U}(\OO_v)$ if~$v\in C_0$, and~$G(F_v)$ otherwise. We also define $H(\m_v)_v  = \cup_{\be\in\scr{B}_v} H(\m_v,\be)_v$. 
\index{Hmbeta@$H(\m,\be)_v$, $H(\m)_v$}
Let $\m = (\m_v)_{v}$ and $\beta = (\beta_v)_v$ be families indexed by $v\in C(k)$, where $\m_v = (m_{\al,v})\in\N^{\scr{A}}$ and $\beta_v\in\scr{B}_v$ for all~$v$, such that $\m_v = 0$ for almost all $v$. The element $\m$ must be seen as an effective zero-cycle on $C$ with coefficients in~$\N^{\scr{A}}$.  We say $(\m,\be)$ is \emph{integral} if $(\m_v,\be_v)$ is $v$-integral for every~$v$. \index{integral pair $(\m,\be)$}
\begin{remark} For fixed $\n\in \N^{\scr{A}}$ and $\be\in \prod_{v}\scr{B}_{0,v}$, families $\m = (\m_v)_v$ such that the pair $(\m,\be)$ is integral and such that $\sum_{v\in C}\m_v = \n$ are parametrised by the symmetric product $S^{\n'}(C\setminus C_0)\times S^{\n''}(C)$ (where $\n' = (n_{\al})_{\al\in \scr{A}_D}$ and $\n'' = (n_{\al})_{\al\in\scr{A}\setminus \scr{A}_D}$)
which may naturally be seen as a constructible subset of $S^{\n}(C)$. 
\end{remark}
 For any pair $(\m,\be)$, the characteristic functions of the subsets
$$G(\m,\beta) = \prod_{v\in C(k)}G(\m_v,\be_v)_v\subset G(\mathbb{A}_F)$$ \index{gmbeta@$G(\m,\beta)$ (adelic set)}
and
$$H(\m,\beta) = \prod_{v\in C(k)}H(\m_v,\be_v)_v\subset G(\mathbb{A}_F)$$ \index{hmbeta@$H(\m,\beta)$ (adelic set)}
may be seen as global motivic Schwartz-Bruhat functions by proposition \ref{gm}. More precisely, using the notation of proposition \ref{gmdeffamily}, they may be seen as elements of $$\expp_{\prod_{v}\A_k^{n(s'_v-N_{\m_v},s_v)}}.$$ We have $H(\m,\be) = G(\m,\be)$ if $(\m,\be) $ is integral, and $H(\m,\be)=\varnothing$ else.

In the same manner as the $G(\m,\be)_v$, the $H(\m,\be)_v$ may be combined into a constructible set $H_{\m,\be}$:
\begin{prop}\label{hmdeffamily} For any $\m\in\N^{\scr{A}}$ and any $\be = (\be_v)_v\in\prod_{v}\scr{B}_v$, there is a constructible subset $H_{\m,\be}\subset G_{\m,\be}$ the fibre of which at any $v\in C(k)$ is exactly $H(\m,\be_v)_v$. \index{Hmbeta@$H_{\m,\be}$}
\end{prop}
\begin{proof} There are two cases to consider. Assume first that $\m$ is such that $\m_{\al} = 0$ for all $\al \in \scr{A}_D$. Then $H_{\m,\be}$ is obtained from $G_{\m,\be}$ by removing the fibres above the finite number of points $v\in C(k)$ such that $(\m,\be_v)$ is not $v$-integral, that is, such that $v\in C_0$ and $\be_v\not\in \scr{B}_0$. If on the contrary there exists $\al\in \scr{A}_D$ such that $\m_{\al}\neq 0$, then $H_{\m,\be}$ is the restriction of $G_{\m,\be}$ to the finite set of points $S$.  
\end{proof}
\subsection{Two constructible families of Schwartz-Bruhat functions}\label{twoconstructiblefamilies}
In propositions \ref{gmdeffamily} and \ref{hmdeffamily} we have combined, for any $\m\in\N^{\scr{A}}$ and any $\beta = (\be_v)_v\in \prod_{v}\scr{B}_v$, the sets $G(\m,\be_v)_v$ (resp. $H(\m,\be_v)_v$) into a family $G_{\m,\beta}\subset \A_C^{n(s'-N_{\m},s)}$ (resp. a family $H_{\m,\be}\subset G_{\m,\be}$) above~$C$. The symmetric product construction then allows us to consider, for any $\n\in\N^{\scr{A}}$ and any $\be\in\prod_{v}\scr{B}_v$, constructible subsets
$$S^{\n}((H_{\m,\be})_{\m\in\N^{\scr{A}}})\subset S^{\n}((G_{\m,\be})_{\m\in\N^{\scr{A}}}) \subset S^{\n}\left(\left(\A_C^{n(s'-N_{\m},s)}\right)_{\m\in\N^{\scr{A}}}\right) = \scr{A}_{\n}(s',s,N,0).$$
with the notation of section \ref{sect.domainsofdef}. Therefore, using the terminology of section \ref{sect.uniformfamilies}, this defines two uniformly smooth constructible families of Schwartz-Bruhat functions of level~$\n$.  They parametrise the characteristic functions of the adelic sets $H(\m,\be)$ and $G(\m,\be)$ for fixed $\beta\in\prod_{v}\scr{B}_v$ and with $\m$ varying inside $S^{\n}C$: we will therefore denote these families $(\1_{H(\m,\be)})_{\m\in S^{\n}C}$ and $(\1_{G(\m,\be)})_{\m\in S^{\n}C}$. Their Fourier transforms will then be uniformly compactly supported constructible families of  functions, defined on $\scr{A}_{\n}(\nu -s,\nu-s',0,N)$.

\subsection{Applying the Poisson summation formula}\label{applicPoisson}
%{\color{red}[A modifier en fonction de l'énoncé de la formule de Poisson]}

Let $\n\in \Z^{\scr{A}}$.  For any $\beta = (\be_v)_{v}\in \prod_{v}\scr{B}_{0,v}$, and  $\al\in\scr{A}$, put $n^{\beta}_{\al} := n_{\al} - \sum_{v}e_{\al,\be_v}.$ We define $M_{\n,\beta}$ to be the constructible subset of $M_{\n}$ of sections~$\sigma_g$ such that for all $v$, $(g,E_{\be_v})_v = 1$ (so that $(g,E_{\be'_v}) = 0$ for all $\be'_{v}\neq \be_v$). By definition, these sections satisfy $\deg (\sigma_g^*\scr{D}_{\al}) = n_{\al}^{\be}$ for all $\alpha \in  \scr{A}$. Thus, since the $\scr{D}_{\al}$ are effective, for any $\n$ such that $M_{\n,\be}\neq \varnothing$, the $n_{\al}^{\be},\ \al\in\scr{A}$, are non-negative integers. %Write $\n = (\n',\n'')\in \N^{\scr{A}_{D}}\times \N^{\scr{A}\setminus \scr{A}_D}$, and $\n^{\be} = (n_{\al}^{\be})_{\al} = (\n'^{\be},\n''^{\be}).$ 

 \begin{lemma} Let $\n\in\Z^{\scr{A}}$ and $\beta\in \prod_{v}\scr{B}_{0,v}$ be such that $M_{\n,\be}$ is non-empty. There is a morphism of constructible sets defined by
$$\begin{array}{ccc}M_{\n,\be}&\to& S^{\n^{\be}}C\\
                     \sigma_g & \mapsto & \displaystyle{\sum_{v\in C(k)}}\left((g,\scr{D}_{\al})_v\right)_{\al\in\scr{A}}[v]\end{array} $$ 
                     \end{lemma}
                     \begin{proof} We will start by making some reductions. To simplify notations, write $M_{\n,\be} = M$, and $\n^{\be} = \n$. Since $S^{\n}C = \prod_{\al\in\scr{A}}S^{n_{\al}}C$, it suffices to prove constructibility for the map 
                     $$\begin{array}{rcl} M&\to &S^{n_{\al}}C\\
                                     \sigma_g & \mapsto & \sum_{v}(g,\scr{D}_{\al})_{v}[v]
                                     \end{array}$$
                                     associated to  one $\scr{D}_{\al}$. To simplify notations further, write $n_{\al} = n$ and $\scr{D}_{\al} = \scr{D}$. By definition of the moduli space of sections, there is a morphism
                                     $$\begin{array}{rccc} \epsilon: &C\times M &\to& \scr{X}\\
                                                           &(v,\sigma) & \mapsto & \sigma(v)\end{array}$$
                                                           Denote by $s_{\scr{D}}$ the canonical section of the line bundle $\OO_{\scr{X}}(\scr{D})$ and put $\Delta = \div (\epsilon^*s_{\scr{D}})$. This is a closed subscheme of $C\times M$, finite over $M$. By generic flatness, we may stratify~$M$ into locally closed subsets $U_i$ such that for every $i$, $\Delta\times_MU_i\subset C\times U_i$ is flat over $U_i$. By definition of Hilbert schemes, it therefore defines a morphism $U_i\to \mathrm{Hilb}(C)$ to the Hilbert scheme of points of $C$. Moreover,                                                        for every $\sigma \in M$, the fibre $\Delta_{\sigma} = \div(\sigma^{*}s_{\scr{D}})$ is a zero-dimensional subscheme of $C$ of length $n$. Thus, the image of the above morphism is in fact contained in the Hilbert scheme of $n$ points of $C$, which we may identify with the symmetric product $S^{n}C$. The constructible morphism we want is then obtained by combining these morphisms $U_i\to S^nC$ for every $i$.
                     \end{proof}
% The $k$-variety $S^{\n'^{\be}}(C\setminus C_0)$ is finite over $k$. 
 Recall that a point in $M_{\n,\be}$ represents a section that intersects components in $\scr{A}_D$ only above places in $S = C\setminus C_0$. Thus, in fact, $M_{\n,\be}$ lies above the constructible subset of $S^{\n^{\be}}C$ consisting of zero-cycles $\m = (\m_v)_{v}$ with components with respect to $\al\in\scr{A}_D$ supported inside $S$: in other words, these are the zero-cycles $\m$ such that $(\m,\be)$ is integral. 

The exact correspondence between sections $\sigma:C\to \scr{X}$ and elements of $X(F)$ via $\sigma \mapsto \sigma(\eta_C)$ restricts to an exact correspondence between
sections $\sigma \in M_{\n,\be}$ lying above $\m\in S^{\n^{\be}}C$ and elements in $G(F)\cap H(\m,\be)$, where $H(\m,\be)$ is the adelic set defined in section \ref{sect.integral}. We denote $H_F(\m,\be):=G(F)\cap H(\m,\be)$, which is a constructible set over~$k$. Note that by definition of summation over rational points (see section \ref{twistedsummation} in chapter \ref{poissonformula}), we have, for every $\m\in S^{\n^{\be}}C$
$$\sum_{x \in \kappa(\m)(C)^n}\1_{H(\m,\be)}(x) = [H_F(\m,\be)]  = (M_{\n,\be})_{\m}\in \M_{\kappa(\m)}.$$
 With the notation of section \ref{twoconstructiblefamilies}, the uniformly smooth family $(\1_{H(\m,\be)})_{\m\in S^{\n^\be}C}$ is uniformly summable (see section \ref{sect.uniformlysummable} of chapter \ref{poissonformula}), and its sum is exactly the class of $M_{\n,\be}$ in $\M_{S^{\n^{\be}}C}$.  Taking classes in $\M_k$, we may therefore write the motivic summation
 $$ [M_{\n,\be}] = \sum_{\m\in S^{\n^{\be}}C}[H_F(\m,\be)] .$$
  \begin{remark} Here the existence of the moduli spaces $M_{\n,\be}$ shows that the family of functions $(\1_{H(\m,\be)})_{\m\in S^{\n^{\be}}C}$ is uniformly summable, independently of section \ref{us.poisson}. By uniqueness of the sum (remark \ref{uniqueness.sum}), this sum (the class of $M_{\n,\be}$ in $\expp_{S^{\n^\be}C}$) is equal to the one given by the Poisson formula as stated in \ref{us.poisson}. 
 \end{remark}
  Note that 
  $$\T^{\n} = \prod_{\al\in\scr{A}}T_{\al}^{n_{\al}} = \prod_{\al\in\scr{A}}T_{\al}^{\sum_{v}e_{\al,\be_v}}\prod_{\al\in\scr{A}}\T_{\al}^{n_{\al}^{\be}}.$$
  
  Write $\T^{||\be||}$ for $ \prod_{\al\in\scr{A}}T_{\al}^{\sum_{v}e_{\al,\be_v}}$. Then
\begin{eqnarray*}Z(\T) = \sum_{\n\in \Z^{\scr{A}}}[M_{\n}]\T^{\n}&= &\sum_{\n\in\Z^{\scr{A}}}\sum_{\be\in\prod_{v}\scr{B}_{0,v}}[M_{\n,\be}]\T^{\n}\\
 & = & \sum_{\n\in\Z^{\scr{A}}}\left(\sum_{\be\in\prod_{v}\scr{B}_{0,v}}\sum_{\m\in S^{\n^{\be}}C}[H_F(\m,\be)]\right)\T^{\n}\\
 & \substack{=\\ \n \leftarrow \n^{\be}}& \sum_{\be\in\prod_{v}\scr{B}_{0,v}}\T^{||\be||}\sum_{\n\in\N^{\scr{A}}}\sum_{\m\in S^{\n}C} [H_F(\m,\be)] \T^{\n}%\\
% & = & \sum_{\be\in\prod_{v}\scr{B}_v}\T^{||\be||}\sum_{\n\in\N^{\scr{A}}}\sum_{\m\in S^{\n}C}\sum_{x\in k(C)^n} \1_{G(\m,\be)}(x)\T^{\n}
 \end{eqnarray*}
 For clarity, let us remark that in this summation, we have three different types of sums:  the sum over $\prod_{v}\scr{B}_{0,v}$, which is a finite sum, the one over $\Z^{\scr{A}}$ or over $\N^{\scr{A}}$ which is the sum of the formal series, and the motivic sum over $\m\in S^{\n}C$. 
 
 The Poisson summation formula (see section \ref{Poisson.families}, and especially \ref{us.poisson} and \ref{reversesummation}; see also remark \ref{dropkappaD}), applied to the uniformly smooth constructible family of Schwartz-Bruhat functions $(\1_{H(\m,\be)})_{\m\in S^{\n}C}$ (see the end of section \ref{twoconstructiblefamilies}) gives, for any $\n\in\N^{\scr{A}}$:
 \begin{eqnarray*} \sum_{\m\in S^{\n}C}\  [H_F(\m,\be)] & = & \sum_{\m\in S^{\n}C}\ \ \sum_{x\in k(C)^n} \1_{H(\m,\be)}(x)\\
 & = & \sum_{\m\in S^{\n}C}\ \ \left(\LL^{(1-g)n}\sum_{\xi\in k(C)^n} \four(\1_{H(\m,\be)})(\xi)\right)\\
 & = & \LL^{(1-g)n}\sum_{\xi\in k(C)^n}\ \ \sum_{\m\in S^{\n}C}\four(\1_{H(\m,\be)})(\xi)
 \end{eqnarray*}
 Then $Z(\T)$ may me rewritten in the form
 $$Z(\T) = \LL^{(1-g)n}\sum_{\xi\in k(C)^n}\sum_{\be\in\prod_{v}\scr{B}_{0,v}}\T^{||\be||}\sum_{\n\in\N^{\scr{A}}} \ \sum_{\m\in S^{\n}C}\four(\1_{H(\m,\be)})(\xi)\T^{\n}.$$
 By the definition of Euler products, we have
 \begin{equation}\label{zetaeulerproduct}\sum_{\n\in\N^{\scr{A}}}\sum_{\m\in S^{\n}C}\four(\1_{H(\m,\be)})(\xi)\T^{\n}  =  \prod_{v\in C}\left(\sum_{\m_v\in \N^{\scr{A}}}\four(\1_{H(\m_v,\be_v)})(\xi_v)\T^{\m_v}\right).\end{equation}
 Indeed, we are dealing here with the uniformly compactly supported family $$(\four (\1_{H(\m,\be)}))_{\m\in S^{\n}C}\in \expp_{\scr{A}_{\n}(\nu-s,\nu-s', 0, N_{\i})}$$ (see end of section \ref{twoconstructiblefamilies}). The summation over $k(C)^n$ is therefore in fact a summation on the $n$-th power of the Riemann-Roch space associated to the divisor $-\sum_v(\nu_v-s_v)[v]$. For every point $\xi$ of this space, as explained in the beginning of the proof of lemma \ref{prop.uniformsummationuc}, $\xi$ induces constructible morphisms $$\phi_{\xi,\i}:C\to \A_C^{(\nu-s,\nu-s'+N_{\i})}$$ for every $\i\in \N^{\scr{A}}$, which after taking symmetric products, give constructible morphisms
  $$S^{\n}\phi_{\xi}: S^{\n}C\to \scr{A}_{\n}(s',s,0,N)$$
such that   $$(\four (\1_{H(\m,\be)})(\xi))_{\m\in S^{\n}C} = (S^{\n}\phi_{\xi})^*(\four (\1_{H(\m,\be)}))_{\m\in S^{\n}C}.$$
  By section \ref{sect.symprodfourtransform} (especially proposition \ref{prop.symproductfourtransform}), we have the equality $$(\four (\1_{H(\m,\be)}))_{\m\in S^{\n}C}=S^{\n}(  (\four\1_{H_{\i,\be}})_{\i\in\N^{\scr{A}}})$$ in $\expp_{\scr{A}_{\n}(\nu-s,\nu-s', 0, N)}$. Pulling back via $S^{\n}\phi_{\xi} = S^{\n}((\phi_{\xi,\i})_{\i\in\N^{\scr{A}}})$ gives %by lemma \ref{symmetricfibreproduct} a justifier
   the equality
 $$(\four (\1_{H(\m,\be)})(\xi))_{\m\in S^{\n}C}=S^{\n}( ( \phi_{\xi,\i}^*\four(\1_{H_{\i,\be}})_{\i\in\N^{\scr{A}}})$$
 in $\expp_{S^{\n}C}.$
 The pullback to any $v\in C$ of the $\i$-th element $\phi_{\xi,\i}^*\four\1_{H_{\i,\be}}\in \expp_C$ of the family on the right-hand side is exactly $\four(\1_{H(\i,\be_v)})(\xi_v)$, by the definition of $\phi_{\xi,\i}$ and by remark \ref{familyrestrictiontov}, so this should be the $\i$-th coefficient of the local factor of the Euler product corresponding to $v$. 

 \begin{notation}\label{notationfromCL} We use the notation from \cite{CL}, 3.6: for every $v\in C$, $\al\in\scr{A}$, $\m_v\in\N^{\scr{A}}$ and $\be\in \scr{B}_v$, we define
 $$||\m_v,\be_v||_{\al} := m_{\al,v} + e_{\al,\be_v}$$
 and $\T^{||\m_v,\be_v||}:= \prod_{\al\in\scr{A}}T_{\al}^{||\m_v,\be_{v}||}.$
 Note that for $g\in G(F_v)\cap G(\m_v,\be_v)_v$, the local intersection degree $(g,\scr{L}_{\al})_v$ is exactly
 $$(g,\scr{L}_{\al})_v = (g,\scr{D}_{\al})_v + \sum_{\be\in\scr{B}_v}e_{\al,\be}(g,E_{\be})_v = m_{\al,v} + e_{\al,\be_v} = ||\m_v,\be_v||_{\al}.$$
 \end{notation}
 Finally, we have
 \begin{eqnarray*}Z(\T) &= &\LL^{(1-g)n}\sum_{\xi\in k(C)^n}\sum_{\be\in\prod_{v}\scr{B}_{0,v}}\prod_{v\in C}\prod_{\al\in \scr{A}}T_{\al}^{e_{\al,\be_v}}\prod_{v\in C}\left(\sum_{\m_v\in \N^{\scr{A}}}\four(\1_{H(\m_v,\be_v)})(\xi_v)\T^{\m_v}\right)\\
  & = & \LL^{(1-g)n}\sum_{\xi\in k(C)^n} \prod_{v\in C}\left(\sum_{\be_v\in\scr{B}_{0,v}}\prod_{\al\in\scr{A}}T_{\al}^{e_{\al,\be_v}}\sum_{\m_v\in \N^{\scr{A}}}\four(\1_{H(\m_v,\be_v)})(\xi_v)\T^{\m_v}\right)\\
  & = & \LL^{(1-g)n}\sum_{\xi\in k(C)^n} \prod_{v\in C}\left(\sum_{\substack{\m_v\in \N^{\scr{A}}\\ \be_v\in\scr{B}_{0,v}}}\four(\1_{H(\m_v,\be_v)})(\xi_v)\T^{||\m_v,\be_v|| }\right)\end{eqnarray*}
 Thus, we have written $Z(\T)$ in the form 
 \begin{equation}\label{summation.zeta}Z(\T) = \LL^{(1-g)n}\sum_{\xi\in k(C)^n}Z(\T,\xi)\end{equation} \index{ZTxi@$Z(\T,\xi)$}
 where $Z(\T,\xi)$ has an Euler product decomposition with local factors 
 $$Z_v(\T,\xi) := \sum_{\substack{\m_v\in \N^{\scr{A}}\\ \be_v\in\scr{B}_{0,v}}}\four(\1_{H(\m_v,\be_v)})(\xi_v)\T^{||\m_v,\be_v||} = \sum_{(\m_v,\be_v)\ v-\text{integral}} \four(\1_{G(\m_v,\be_v)})(\xi_v)\T^{||\m_v,\be_v||}.$$
More precisely, it is the product of the finite number of factors corresponding to $v\in C\setminus C_1$, and of the Euler product of the series
 $$Z_{C_1}(\T,\xi):= \sum_{\m\in\N^{\scr{A}}}\four(\1_{H_{\m,\be}\times_CC_1})(\xi)\T^{\m}\in \expp_{C_1}[[T]]$$
                                  where $\T^{\m} = \prod_{\al\in \scr{A}}T_{\al}^{m_{\al}}$. In what follows, we will study these local factors to be able to apply lemma \ref{convergence} to the series $Z_{C_1}(\T,\xi)$. Combined with estimates for the remaining local factors, this will give us the convergence properties of the series $Z(\T)$. 

\begin{remark}[Restriction of the summation domain]\label{summationdomain}
As noted in the end of section \ref{twoconstructiblefamilies}, the Fourier transforms of the families $(\1_{H(\m,\be)})_{\m\in S^{\m}C}$ and $(\1_{G(\m,\be)})_{\m\in S^{\m}C}$ are uniformly compactly supported. We may conclude, as in \cite{CL}, 4.2, that there is a finite-dimensional $k$-vector space $V$ (given by an appropriate Riemann-Roch space), and a linear $F$-morphism $\a:V_F\to G_F$ such that the summation (\ref{summation.zeta}) restricts in fact to 
$$Z(\T) = \LL^{(1-g)n}\sum_{\xi\in \a(V)}Z(\T,\xi).$$
\end{remark}

%\subsection{The summation domain}\label{summationdomain}
%With the notation of proposition \ref{gm}, for every $v\in C$ there is an integer $r_v$ such that for every $\m\in\N^{\scr{A}}$ and for every $\be\in\scr{B}_v$, the characteristic function of the set $G(\m,\be)$ is invariant under $G(\mathfrak{m}_v^{r_v})$. Therefore, its Fourier transform vanishes outside $G(\mathfrak{m}_v^{-r_v + \nu_v})$, and the summation in (\ref{summation.zeta}) is over the finite-dimensional $k$-vector space $E$ of $\xi\in k(C)^n$ such that for all $v$,~$\xi_v\in G(\mathfrak{m}_v^{-r_v+\nu_v})$. 

%[A réécrire plus clairement pour faire le lien avec la partie sur la sommation sur $k(C)^n$]

%As in \cite{CL}, we denote by $\a$ the morphism $a:E_{F_v}\to G_{F_v}$. 
\section{Analysis of local factors and convergence}\label{sectLocal}
The aim of this section is to study the local factors 
 \begin{eqnarray*}Z_v(\T,\xi) &= &\sum_{\substack{\m_v\in \N^{\scr{A}}\\ \be\in\scr{B}_{0,v}}}\four(\1_{H(\m_v,\be_v)})(\xi_v)\T^{||\m_v,\be_v||}\\ &=& \sum_{\substack{\m_v\in \N^{\scr{A}}\\ \be\in\scr{B}_{0,v}}} \int_{H(\m_v,\be_v)}\prod_{\al\in\scr{A}}T_{\al}^{(g,\scr{L}_{\al})_v}\mathrm{e}(\langle g,\xi\rangle)\dx g.\end{eqnarray*}
 (We use here the notation $\mathrm{e}$ from \cite{CL}, which is a substitute for the notation $\psi\circ r$ from section \ref{sect.localfouriertransform}.)
 We follow \cite{CL} and rewrite this integral as an integral with respect to the motivic measure on the arc space~$\scr{L}(\scr{X}).$ We then give estimates for it, first in the case when $\xi$ is the trivial character, then in the case when it is non-trivial. This allows us to study convergence of the Euler product $Z(T,\xi)$. 
 
In this section, we will often omit the index $v$ once the place $v$ is fixed.

 \subsection{Motivic functions and integrals}
 
   In this section, we consider a field $k$ of characteristic zero, $R = k[[t]]$ and $K=k((t))$. For this and the next section only, let $\scr{X}$ \index{X@$\scr{X}$}be a smooth and flat $R$-scheme of finite type and of pure relative dimension $n$. We are going to use the notion of \textit{motivic residual function} \index{motivic residual function} on the arc space $\scr{L}(\scr{X})$, \index{arc space} \index{LX@$\scr{L}(\scr{X})$} introduced in section 2.4 of~\cite{CL}. 
 
 Recall that the spaces of $m$-jets $\scr{L}_m(\scr{X})$ \index{LmX@$\scr{L}_m(\scr{X})$} for $m\geq 0$ come with natural affine morphisms $p^{m+1}_m:\scr{L}_{m+1}(\scr{X})\to \scr{L}_{m}(\scr{X})$ which turn the collection of relative Grothendieck rings $\expp_{\scr{L}_m(\scr{X})}$ into an inductive system via the induced ring morphisms
 $$ \left(p_m^{m+1}\right)^*: \expp_{ \scr{L}_{m}(\scr{X})} \to \expp_{\scr{L}_{m+1}(\scr{X})} $$
  sending the class of a variety $H\to \scr{L}_m(\scr{X})$ to the class of $H\times_{\scr{L}_m(\scr{X})}\scr{L}_{m+1}(\scr{X})\to \scr{L}_{m+1}(\scr{X})$ with the appropriate operation on exponentials.
 
 \begin{definition} The ring of motivic residual functions on $\scr{L}(\scr{X})$ is the inductive limit of all Grothendieck rings $\expp_{\scr{L}_m(\scr{X})}$, $m\geq 0$. 
 \end{definition}
 
 For example, take a constructible subset $W$ of $\scr{L}(\scr{X})$, that is, a subset of the form $p_m^{-1}(W_m)$ where $W_m$ is a constructible subset of $\scr{L}_m(\scr{X})$ and $p_m:\scr{L}(\scr{X})\to \scr{L}_m(\scr{X})$ is the projection morphism. Then the characteristic function of $W$ may be seen as a motivic residual function.% These are in fact the only examples we are going to use.

% We normalise the motivic measure on the arc space so that $\mathrm{vol}\scr{L}(\scr{X}) = 1

Let $h$ be a motivic residual function. Assume it to be of the form $[H,f]$ where $H$ is a variety over $\scr{L}_{m}(\scr{X})$ for some $m$, and $f:H\to\A^1$ a morphism.  Then the integral of $h$ over the arc space $\scr{L}(\scr{X})$ is defined to be
$$\int_{\scr{L}(\scr{X})}h(x)\dx x = \int_{\scr{L}(\scr{X})}H_x\psi(r(f(x)))\dx x := \LL^{-(m+1)n}[H,f]_k\in\expp_{k}.$$\index{motivic residual function!integration}
This does not depend on $m$ because for any $m'\geq m$, the projection morphism $\scr{L}_{m'}(\scr{X})\to \scr{L}_m(\scr{X})$ is a locally trivial fibration with fibre $\A_k^{(m'-m)n}.$ More generally, one can consider integrals $\int_{W}$ over constructible subsets of $\scr{L}(\scr{X})$ by multiplying the integrand by the characteristic function $\1_W$.

\begin{example}\label{volume.constset} If $W = p_m^{-1}(W_m)$ is a constructible subset of $\scr{L}(\scr{X})$ for some $m\geq 0$, then one may define the volume of $W$ \index{volume} to be
$$\mathrm{vol}(W) := \int_{\scr{L}(\scr{X})}\1_{W}(x) \dx x = \LL^{-(m+1)n}[W_m,0]\in\expp_{k}.$$\index{vol@$\mathrm{vol}$}
In particular, if $W = \scr{L}(\scr{X}) = p_0^{-1}(\scr{X}_k)$ where $\scr{X}_k$ is the special fibre of $\scr{X}$, then
\begin{equation}\label{volumearcspace}\mathrm{vol}(\scr{L}(\scr{X})) = \LL^{-n}[\scr{X}_k,0].\end{equation}
%More generally, if $W = \scr{L}_{m}(\scr{X})$ for some $m\geq 0$, we have
Another useful special case is the volume of the subspace $\scr{L}(\A^1,0)$ of $\scr{L}(\A^1)$ of arcs with origin in $0\in \A^1$. We have:
\begin{equation}\label{volumeaffineline}\mathrm{vol}(\scr{L}(\A^1,0)) = \LL^{-1}.\end{equation}
Combining (\ref{volumearcspace}) and (\ref{volumeaffineline}), we also get that the volume of the subspace of arcs of order 0 in $\scr{L}(\A^1)$ is
\begin{equation}\label{volumeaffinelineordzero}\mathrm{vol}(\{x\in\scr{L}(\A^1),\ \ord(x) = 0\}) = \mathrm{vol}(\scr{L}(\A^1)) - \mathrm{vol}(\scr{L}(\A^1,0)) = 1-\LL^{-1}\end{equation}
\end{example}

\begin{remark}\label{volume.integrals}  Let $W= p^{-1}(W_m)$ be a constructible subset of $\scr{L}(\scr{X})$ for some $W_m\subset \scr{L}_m(\scr{X})$ and some $m\geq 0$ and let $h = [\scr{L}_{m}(\scr{X}),f]$. Then by the triangular inequality for weights (chapter \ref{hodgemodules}, lemma \ref{triangularwt}), as well as lemma \ref{weightdimension}:
\begin{eqnarray*}w\left(\int_{W}h(x)\dx x \right)& = & w\left(\int_{\scr{L}(\scr{X})}\mathbf{1}_{W}(x)\mathrm{e}(f(x))\dx x \right)\\
& =& w(\LL^{-(m+1)n}[W,f_{|W}])\\ 
&\leq& -2(m+1)n + w([W,0])\\
& = & w(\mathrm{vol}(W))\end{eqnarray*}
We are going to use this property repeatedly.
\end{remark}

%More generally, let $X$ be a smooth projective $K$-scheme. By a \emph{motivic function} on  $X$ we mean the choice of a proper flat $R$-model $\scr{X}$ of $X$, an integer $m$, and a class $\phi\in\expp_{\scr{L}_m(\scr{X})},$ modulo identifications $(\scr{X},m,\phi)\sim (\scr{X},m+1,(p_m^{m+1})^*\phi)$ and 

\subsection{Some computations of motivic integrals}
We keep the notations from the previous section. Let $r:K\to k$ be the linear map defined by $r(a) = \res_0(a\dx t),$ so that $r(t^{-1}) =1$ and $r(t^n) = 0$ for any $n\neq -1$. 
 \begin{lemma} Let $d$ be a non-zero integer and let  $\xi\in K$ be such that $\ord(\xi) = 0$. Let $Q=a_0 + a_1 x + \ldots + a_dx^d \in k((t))[x]$ be a non-zero polynomial of degree $\leq d$ such that for all $i\in\{1,\ldots,d\}$, $\ord (a_i)> \ord (a_0)$.  Then for any $n\in\N$ such that $-2n \leq \ord(a_0) < -n$, one has
 $$\int_{\xi + t^nk[[t]]}\psi(r(Q(x)x^{-d}))\dx x =0.$$
 \end{lemma}
 
 \begin{proof} The proof goes along the same lines as the proof of lemma 5.1.1 in \cite{CL}. The condition on the orders of the coefficients of $Q$ implies that:
 \begin{itemize}
\item $\ord (Q(\xi)) = \ord (a_0)$.
 \item For all $i\in\{1,\ldots,d\}$, 
 \begin{equation}\label{Qderivatives}\ord (Q^{(i)}(\xi))\geq \min_{i\leq j\leq d}\ \ord (a_j) > \ord (Q(\xi))\geq -2n.\end{equation}
\end{itemize}
Write $x = \xi(1 + t^n u)$ for some $u$ with $\ord (u) \geq 0$, so that we have
$$\int_{\xi + t^nk[[t]]}\psi(r(Q(x)x^{-d}))\dx x = \LL^{-n}\int_{k[[t]]}\psi(r(Q(\xi(1 + t^nu))(\xi(1 + t^nu))^{-d})\dx u.$$
 We may expand
$$Q(\xi(1 + t^n u)) = Q(\xi) + Q'(x)\xi t^n u + Q''(x)(\xi t^{n} u)^2 + \ldots$$
and
$$(\xi(1 + t^n u))^{-d} = \xi^{-d}\left(1 - dt^n u + {-d\choose 2} t^{2n} u^2 + \ldots\right)$$
Taking the product and using (\ref{Qderivatives}), we have
$$r\left(Q(\xi(1 + t^nu))\xi^{-d}(1 + t^n u)\right) = r(Q(\xi)\xi^{-d}) + r\left(Q'(\xi)\xi^{1-d} t^n u -Q(\xi)\xi^{-d} d t^nu\right),$$
since all the other terms belong to the maximal ideal $tk[[t]]$. We therefore have
$$\int_{k[[t]]}\psi(r(Q(\xi(1 + t^nu))(\xi(1 + t^nu))^{-d})\dx u$$
$$ = [\spec k, r(Q(\xi)\xi^{-d})]\int_{k[[t]]}\psi(r\left(Q'(\xi)\xi^{1-d} t^n u -Q(\xi)\xi^{-d} d t^nu\right))\dx u.$$
Now, $\ord( Q'(\xi)\xi^{1-d} t^n-Q(\xi)d\xi^{-d} t^n )  =\ord (Q(\xi)t^n) < 0$, and therefore the integral in the right-hand side is zero.
 \end{proof}
 
 For $m\in\Z$, let $C_m$ be the annulus defined by $\ord(x) = m$. 
 \begin{lemma}\label{motivic.integral} Let $m$ and $d$ be positive integers and $P\in k[[t]][x]$ a non-zero polynomial such that $\ord(P(0)) = 0$. Then
 $$\int_{C_m}\psi(r(P(x)x^{-d}))\dx x = \left\{\begin{array}{lc} -\LL^{-2} & \text{if}\ m=d=1\\
                                                                 0     & otherwise 
                                                                 \end{array}\right.$$
 \end{lemma}
 
 \begin{proof} Let $I(m,d,P)$ be the above integral. A change of variable allows us to write
 $$I(m,d,P) = \LL^{-m}\int_{C_0}\psi(r(P(t^mu)t^{-md}u^{-d}))\dx u.$$
 Thus, we have $I(m,d,P) = I(0,d,Q)$ where $Q(u) = a_0 + a_1 u + \ldots \in k((t))[u]$ is the polynomial given by $Q(u) = P(t^mu)t^{-md}$. Since $P$ has integral coefficients, and $\ord(P(0)) = 0$, we have 
 \begin{itemize}\item $\ord (a_0) = \ord (P(0)t^{-md}) = -md$
 \item for all $i\geq 1$, $\ord (a_i) \geq mi - md > -md = \ord(a_0).$
 \end{itemize}
 If $md> 1$, then there exists a positive integer $n$ such that $-2n\leq -md < -n$, and choosing such an~$n$, the previous lemma tells us that
 $$I(m,d,P) = \LL^{-m}\int_{C_0}\int_{k[[t]]}\psi(r(Q(u)(u+t^ny)^{-d})\dx y\, \dx u = 0.$$ 
 Assume now $m=d=1$. Then, writing $P(0) = a$, we have
 \begin{eqnarray*}I(1,1,P) &= &\LL^{-1}\int_{C_0}\psi(r(P(tu)t^{-1}u^{-1}))\dx u  = \LL^{-1}\int_{C_0}\psi(r(at^{-1}u^{-1}))\dx u\\
                         &=&\LL^{-1}\int_{C_0}\psi(r(at^{-1}u))\dx u\\
                         &=& \LL^{-1}\left(\int_{k[[t]]}\psi(r(at^{-1}u))\dx u -  \int_{tk[[t]]}\psi(r(at^{-1}u))\dx u \right)\\
                         \end{eqnarray*}
                         which gives the result, since the first term in the parenthesis is zero (using again that $\ord(a)=0$), and the second one is equal to~$\LL^{-1}$. 
 \end{proof}
 
 \subsection{An integral over the arc space}\label{sect.integralarcspace}
 We now go back to the setting and notations of section \ref{sectGeometry}. In this section we recall briefly, following~\cite{CL}, how integrals of motivic Schwartz-Bruhat functions can be rewritten as integrals on arc spaces. 
 
% Let $\Phi\in\Sch(F_v^n)$ be a local Schwartz-Bruhat function on $F_v^n$. 
 
 \begin{lemma}(\cite{CL}, lemma 6.1.1) Let $\Phi\in\Sch(F_v^n)$ be a local motivic Schwartz-Bruhat function. Then the integral $\int_{G(F_v)}\Phi(g)\dx g$ can be rewritten as 
 $$\int_{\scr{L}(\scr{X})}\Phi(x)\LL^{-\mathrm{ord}_{\omega}(x)}\dx x$$
 where $\dx x$ denotes the motivic measure on the arc space $\scr{L}(\scr{X}).$ %Moreover, if the smooth locus $\scr{X}_1$ of~$\scr{X}$ is a weak Néron model of $X$, this integral may be restricted to the arc space $\scr{L}(\scr{X}_1)$. 
 \end{lemma}

 %The model $\scr{X}$ of $X$ is a scheme over $\spec(R)$. Consider the scheme $\scr{X}_1$ obtained from $\scr{X}$ by removing the vertical components with multiplicity strictly greater than one, as well as the intersections of distinct vertical components. Lemma 3.2.1. of \cite{CL} asserts that $\scr{X}_1$ is a weak Néron model of $X$ over $\spec(R)$. The set of remaining vertical components is denoted by $\scr{B}_1.$
 
 %Every point $x$ of the special fibre $\scr{X}_v$ of $\scr{X}$ reduces to a unique point $\tilde{x}\in X(k)$. 
For every subset $A\subset \scr{A}$ and every $\be\in\scr{B}_{1,v}$ for some $v$, we denote by $\Delta(A,\be)$ the locally closed subset of the special fibre $\scr{X}_v:= \scr{X}\times_R \spec(k)$ of points belonging to the divisors $\scr{D}_{\scr{\al}},\ \al\in A$ and no other, as well as to the vertical divisor $E_{\be}$, and no other. 

The special fibre $\scr{X}_v$ identifies with the jet scheme $\scr{L}_0(\scr{X})$ of order 0, and therefore there is a specialisation morphism $\scr{L}(\scr{X})\arr\scr{X}_v$. We denote by $\Omega(A,\beta)$ the preimage in $\scr{L}(\scr{X})$ of $\Delta(A,\beta).$ Lemma 5.2.6 of \cite{CL} then states the following:\index{omegaa@$\Omega(A,\beta)$}\index{Delta@$\Delta(A,\beta)$}

\begin{lemma}\label{measure} Let $A$ be a subset of $\scr{A}$ and let $B$ be a set of cardinality $n-\card(A)$. There exists a mesure-preserving definable isomorphism $\theta$ from $\Delta(A,\beta)\times \scr{L}(\A^1,0)^{A}\times \scr{L}(\A^1,0)^{B}$ with coordinates $x_{\al},\ \al\in A$ and $y_{\be},\ \be\in B$, to $\Omega(A,\be)$, such that $\ord_{\scr{D}_{\al}}(\theta(x)) = \ord(x_{\al})$ for $\al\in A$, and $\ord_{\scr{D}_\al}(\theta(x)) = 0$ for $\al\not \in A$. 
\end{lemma}
\begin{remark} We changed the normalisation slightly with respect to \cite{CL}. Note that this is consistent with example \ref{volume.constset}: the volume of $\Omega(A,\be)$ is given by $[\Delta(A,\be)]\LL^{-n}$. 
\end{remark}
In what follows, we therefore identify a point of $\Omega(A,\be)$ with a triple 
$$(w,x,y)\in \Delta(A,\beta)\times \scr{L}(\A^1,0)^{A}\times \scr{L}(\A^1,0)^{B}.$$

We also recall Lemma 6.2.6 from \cite{CL}, which uses this isomorphism to rewrite the motivic Fourier transforms $Z_v(\T,\xi)$ as sums of motivic integrals over arc spaces:

\begin{lemma}\label{arcspace} For every motivic residual function $h$ on $\scr{L}(\scr{X})$ and every $\xi\in G(F_v)$, one has
$$\int_{G(F_v)}\prod_{\al\in\scr{A}}T_{\al}^{(g,\scr{L}_{\al})_v}h(g)\mathrm{e}(\langle g,\xi\rangle)\dx g $$
$$= 
\sum_{\substack{A\subset \scr{A}\\ \beta\in\scr{B}_1}}\prod_{\al\in\scr{A}}T_{\al}^{e_{\al,\be}}\LL^{\rho_{\be}}\int_{\Omega(A,\be)}\prod_{\al\in A}(\LL^{\rho_{\al}}T_{\al})^{\ord(x_{\al})}h(x)\mathrm{e}(\langle x,\xi\rangle)\dx x. $$\end{lemma}

\subsection{A few words on convergence of Euler products in our setting}
If in our general convergence result, proposition \ref{convergence}, we take $X$ to be a curve (so that $w(X) = 2$), $\epsilon = \frac12$ and $\beta =0$ (and replacing $\alpha$ by the letter $c$ to avoid confusion with the indices of the components $D_{\al}$), the obtain the following particular case which will be the one we are going to use:
\begin{lemma}\label{convergencetoapply} Let $X$ be a curve over $\C$. Assume $$F(T) = 1 + \sum_{i\geq 1}X_it^i\in\expp_{X}[[T]]$$ is such that there exist an integer $M\geq 1$ and a real number $c<1$  such that
\begin{itemize}\item for all $i\in\{1,\ldots,M\}$, $w_X(X_i)\leq 2i-2$
\item for all $i\geq M$, $w_X(X_i)\leq 2c i -1$.
\end{itemize}
Then there exists $\delta>0$ such that the Euler product $\prod_{v\in X}F_v(T)\in\expp_k[[T]]$ converges for $|T|<\LL^{-1 + \delta}$.% and takes non-zero values for $|T|\leq \LL
\end{lemma}\index{Euler product!convergence}
In practice, $X$ will be some dense open subset of our original curve $C$, and the convergence of the remaining factors will be checked separately. 

\begin{remark}\label{bounddimension} Note that by lemma \ref{weightdimensionnoneffective}, to bound $w_X(X_i)$, it suffices to bound $$2\dim_X(X_i) + \dim X.$$ We are going to use this remark for almost all terms except the very few first ones where a bound on weights rather than dimensions is crucial. 
\end{remark}

\begin{remark}\label{apply.polynomial}Whenever the series $F(T)$ is in fact a polynomial, we may take $M$ to be its degree and then it suffices to check the bound $w_X(X_i)\leq 2i-2$ for all $i$, taking $c$ to be zero in the statement of the lemma. In the case when we want to check it on dimensions, this boils down to the inequality $\dim_X(X_i)\leq i-2$. 
\end{remark}
This remark motivates the following terminology which we will use to discard terms that won't obstruct convergence:
\begin{definition}\label{admissible} Let $X$ be a $k$-variety and $i\geq 0$ be an integer.
 A polynomial $F = \sum a_iT^i\in\expp_X[T]$ is said to be admissible if $\dim_X(a_i) \leq i-2$ for all $i\geq 0$. A polynomial $F = \sum_{\m\in \N^{\scr{A}}}a_{\m}\T^{\m}\in \expp_X[\T]$ is said to be $\rho'$-admissible if $F((T^{\rho'_{\al}})_{\al\in\scr{A}})$ is admissible.
\end{definition}\index{admissible}\index{rhopa@$\rho'$-admissible}
Note that $F\in\expp_{X}[T]$ is admissible if and only if for all $v\in X$, $F_v = \sum a_{i,v}T^{i}\in \expp_{k}[T]$ is admissible, so admissibility may be checked locally.

In what follows, we are going to use the weight function from chapter \ref{hodgemodules}. Therefore, unless explicitly stated, in all what follows, the base field $k$ will be the field of complex numbers~$\C$.

\subsection{Places in $C_0$}
Let $v$ be a place in $C_0$. In this case, for any character $\xi$, $Z_v(\T,\xi)$ is given by lemma \ref{arcspace}, taking~$h$ to be the characteristic function of the set~$\scr{U}(\OO_v)$ inside $G(F_v)$. In other words, one has $h=0$ on $\Omega(A,\be)$ whenever $A\cap \scr{A}_D \neq \varnothing$ or $\be\not\in\scr{B}_0$, and $h=1$ otherwise. Therefore,
\begin{equation}\label{zeta.formula} Z_v(\T,\xi) = \sum_{\substack{A\subset \scr{A}\setminus\scr{A}_D\\\be\in\scr{B}_{0,v}}}\prod_{\al\in\scr{A}}T_{\al}^{e_{\al,\be}}\LL^{\rho_{\be}}\int_{\Omega(A,\be)}\prod_{\al\in A}(\LL^{\rho_{\al}}T_{\al})^{\ord(x_{\al})}\mathrm{e}(\langle x,\xi\rangle)\dx x.\end{equation}

We are going to study the local factors in this form. For all but a finite number of places in $C_0$, a precise analysis is required to prove meromorphic continuation of the Euler product (specialised for $\T = (T^{\rho'_{\al}})_{\al\in\scr{A}}$) for $|T|<\LL^{-1 + \delta}$, with a pole at $\LL^{-1}$. For the finite number of remaining places, a coarser estimate suffices, given by the following lemma:

\begin{lemma}\label{finitenumberfactors} Let $v\in C_0$. The local factor $Z_v((T^{\rho'_{\al}})_{\al},\xi)$ converges for $|T|<\LL^{-1 + \delta}$ for some $\delta >0$.
\end{lemma}
\begin{proof} Write $$Z_v(\T,\xi) = \sum_{\beta\in\scr{B}_{0,v}}\prod_{\al\in\scr{A}}T_{\al}^{e_{\al,\be}}\LL^{\rho_{\be}}\sum_{A\subset \scr{A}\setminus \scr{A}_D}Z_{A,\be}(\T,\xi),$$
with $$Z_{A,\be}(\T,\xi) = \int_{\Omega(A,\be)}\prod_{\al\in A}(\LL^{\rho_{\al}}T_{\al})^{\ord(x_{\al})}\mathrm{e}(\langle x,\xi\rangle)\dx x.$$
It suffices to check convergence of each $Z_{A,\be}(\xi)$. Using the notation $f_{\xi}$ for the linear form $\langle \cdot, \xi \rangle$ on $G_F$ and the rational function it induces on $X$, we have
\begin{eqnarray*}Z_{A,\be}(\T,\xi)& =& \sum_{\m\in\N_{>0}^{A}}\prod_{\al\in A}(\LL^{\rho_{\al}}T_{\al})^{m_{\al}}\int_{\substack{\Delta(A,\be)\times \scr{L}(\A^1,0)^{n- |A|}\times\scr{L}(\A^1,0)^{|A|}\\ \ord(x_{\al}) = m_{\al}}}\mathrm{e}(f_{\xi}(x,y))\dx x \dx y\\
& = & \sum_{\m\in\N_{>0}^{A}}\prod_{\al\in A}(\LL^{\rho_{\al}-1}T_{\al})^{m_{\al}}\int_{\substack{\Delta(A,\be)\times \scr{L}(\A^1)^{n-|A|}\times\scr{L}(\A^1,0)^{|A|}\\ \ord(x_{\al}) = 0}}\mathrm{e}(f_{\xi}(t^{\m}x,y))\dx x \dx y.\end{eqnarray*}
By remark \ref{volume.integrals}, each integral is of weight at most the weight of the volume of the domain of integration, which, by (\ref{volumeaffineline}) and (\ref{volumeaffinelineordzero}) is 
\begin{eqnarray*} w( [\Delta(A,\be)] (1-\LL^{-1})^{|A|} \LL^{-n+ |A|})&\leq& 2\dim [\Delta(A,\be)] + 2(- n + |A|)\\
& \leq& 0
\end{eqnarray*}
because the divisors $D_{\al}$ intersect transversely. Thus, the series converges if the series
$$\sum_{\m \in \N_{>0}^{A}}\prod_{\al\in A}(\LL^{\rho_{\al}-1}T_{\al})^{m_{\al}}.$$
converges. This series is equal to 
$$\prod_{\al\in A}\frac{\LL^{\rho_{\al}-1}T_{\al}}{1-\LL^{\rho_{\al}-1}T_{\al}}.$$
Specialising to $\T = (T^{\rho_{\al}})_{\al\in A}$ we get the result. 
\end{proof}

For $\xi = 0$, we need to be more precise and to check the exact order of the pole, therefore we are going to use this lemma only for $\xi\neq 0$. 

\subsubsection{Trivial character}\label{trivchar}
Assume $\xi = 0$. Then the factor $\mathrm{e}(\langle x,\xi\rangle)$ equals 1, and we may compute directly (see \cite{CL}, 6.4):
$$Z_v(\T,0) = \sum_{\be\in\scr{B}_{0,v}}\prod_{\alpha\in \scr{A}}T^{e_{\al,\be}}_{\al}\LL^{\rho_\be}\sum_{A\subset\scr{A}\setminus\scr{A}_D}[\Delta(A,\be)]\LL^{-n + |A|}(1-\LL^{-1})^{|A|}\prod_{\al\in A}\frac{\LL^{\rho_{\al}-1}T_{\al}}{1-\LL^{\rho_{\al}-1}T_{\al}}.$$
Therefore, $Z_v(\T,0)$ is a rational function, and $Z_v(\T,0) \prod_{\al\in \scr{A}\setminus\scr{A}_D}(1-\LL^{\rho_{\al}-1}T_{\al}) $ is a Laurent polynomial (and a polynomial for almost all $v$). 

\begin{prop}\label{trivialcharconv} There is a real number $\delta>0$ such that the product
$$\prod_{v\in C_0}\left(Z_v((T^{\rho'_{\al}})_{\al},0)\prod_{\al\in \scr{A}\setminus\scr{A}_D}(1-\LL^{\rho_{\al}-1} T^{\rho_{\al}}) \right)$$
converges for $|T|< \LL^{-1+\delta}$ and takes a non-zero effective value in $\widehat{\M_k}$ at $T = \LL^{-1}$.
% $$\prod_{\al\in\scr{A}\setminus\scr{A}_D}Z_{C_0}((\LL T)^{\rho_{\al}-1})\prod_{v\in C_0}Z_v((T^{\rho_{\al}})_{\al},0)$$
\end{prop}
\begin{proof}  Since all factors are Laurent polynomials, and by the properties of Euler products, it suffices to check convergence of the product over $v$ inside the dense open subset~$C_1$ of~$C_0$ (see notation \ref{C1}). Thus, we may assume that $\scr{B}_v = \{\be\}$ (and denote $\Delta(A,\be)$ simply by $\Delta(A)$) and that the integers $e_{\al,\be}$ and $\rho_{\be}$ are zero. Then the above formula simplifies to 
$$Z_v(\T,0) = \sum_{A\subset \scr{A}\setminus\scr{A}_D}[\Delta(A)]\LL^{-n + |A|}(1-\LL^{-1})^{|A|}\prod_{\al\in A}\frac{\LL^{\rho_{\al}-1}T_{\al}}{1-\LL^{\rho_{\al}-1}T_{\al}}.$$
Put $F_v(\T,0):= Z_{v}(\T,0)\prod_{\al\in\scr{A}\setminus \scr{A}_D}(1-\LL^{\rho_{\al}-1}T_{\al})$. It is a polynomial.
 For $A$ such that $|A|\geq 2$ we have $\dim\Delta(A)\leq n-|A|$ so that 
\begin{eqnarray*}F_v(\T,0) & = & 1 - \sum_{\alpha\in\scr{A}\setminus\scr{A}_D}\LL^{\rho_{\al}-1}T_{\al} + \sum_{\alpha\in \scr{A}\setminus\scr{A}_D}[\Delta(\{\al\})]\LL^{1-n}\LL^{\rho_{\al}-1}T_{\al} + P_v(\T)\\
& = & 1-\sum_{\al\in \scr{A}\setminus \scr{A}_D}([\scr{D}_{\al,v}]-\LL^{n-1})\LL^{-n}\LL^{\rho_{\al}}T_{\al} + P_v(\T)\end{eqnarray*}
where $P_v(\T)\in\M_k[\T]$ is a $\rho'$-admissible polynomial (see definition \ref{admissible}). This computation is uniform in $v\in C_1$, meaning that there is a polynomial $F(\T,0)\in\M_{C_1}[\T]$ and a $\rho'$-admissible polynomial $P(\T)\in\M_{C_1}[\T]$ such that, denoting by $v:\spec k\to C_1$ the morphism defining the point $v$, we have
$v^*F = F_v$, $v^*P = P_v$, and
$$F(\T,0) = 1 - \sum_{\al\in\scr{A}\setminus \scr{A}_D}([\scr{D}_{\al}] - \LL^{n-1})\LL^{-n}\LL^{\rho_{\al}}T_{\al} + P(\T)$$
in $\M_{C_1}[\T]$. By lemma \ref{weightcancellation}, we have
$$w_{C_1}(([\scr{D}_{\al}] - \LL^{n-1})\LL^{-n}\LL^{\rho_{\al}}) \leq 2(n-1) -2n + 2\rho_{\al} \leq 2(\rho_{\al} -1).$$
Specialising $T_{\al}$ to $T^{\rho_{\al}}$, we see that we can apply lemma \ref{convergencetoapply} with $X = C_1$, as explained in remark~\ref{apply.polynomial}.

Thus, there is a real number $\delta>0$ such that the infinite product $\prod_{v\in C_1}F_v((T^{\rho'_{\al}})_{\al},0)$ converges for $|T|<\LL^{-1 + \delta}$ and has a non-zero effective value at~$\LL^{-1}$ in $\widehat{\M_{k}}$, which is of the form $1 + a$ with $w(a) < 0$. Multiplying by the finite number of factors $v\in (C_0\setminus C_1)(k)$ does not change convergence. Moreover, for $v\in C_0\setminus C_1(k)$ the value of $F_v((T^{\rho'_\al})_{\al},0)$ at $T = \LL^{-1}$ is exactly:
\begin{equation}\label{valueatL}(1-\LL^{-1})^{|\scr{A}\setminus \scr{A}_D|}\sum_{\substack{A\subset\scr{A}\setminus\scr{A}_D\\ \beta\in \scr{B}_0}} \LL^{\rho_{\be}-\sum_{\al\in\scr{A}}\rho'_{\al}e_{\al,\be}}[\Delta(A,\be)]\LL^{-n},\end{equation}
which is clearly effective. It is non-zero unless $\Delta(A,\beta) = \varnothing$ for all $A\subset \scr{A}\setminus\scr{A}_D$ and all $\beta\in\scr{B}_{0,v}$, which would mean that $\scr{U}(\OO_v) = \varnothing$. The latter was ruled out by our assumption of existence of local sections, see \ref{sect.goodmodelchoice}. We may conclude that the product
$$\prod_{v\in C_0}F_v((T^{\rho'_\al})_{\al},0)$$
converges for $|T| < \LL^{-1 + \delta}$ and has a non-zero effective value at~$\LL^{-1}$ in $\widehat{\M_{\C}}$. Moreover, we may give a formula for this value. For all $v\in C_1$, the expression in (\ref{valueatL}) simplifies to 
$$(1-\LL^{-1})^{|\scr{A}\setminus\scr{A}_D|}\sum_{A\subset\scr{A}\setminus\scr{A}_D}[\Delta(A)]\LL^{-n}$$
$$ = (1-\LL^{-1})^{|\scr{A}\setminus\scr{A}_D|}\left ( 1 + \LL^{-n}\sum_{\varnothing\neq A\subset \scr{A}\setminus\scr{A}_D}[\Delta(A)]\right)$$
because $[\Delta(\varnothing)]= [G(k)] = \LL^{n}$, so that the total value at $\LL^{-1}$ is:
$$\prod_{v\in C_1}(1-\LL^{-1})^{|\scr{A}\setminus \scr{A}_D|}\left(1 + \LL^{-n}\sum_{\varnothing\neq A\subset \scr{A}\setminus\scr{A}_D}[\Delta(A)]\right)
$$
$$\times\prod_{v\in C_0\setminus C_1}(1-\LL^{-1})^{|\scr{A}\setminus \scr{A}_D|}\left(\sum_{\substack{A\subset\scr{A}\setminus\scr{A}_D\\ \beta\in \scr{B}_0}} \LL^{\rho_{\be}-\sum_{\al\in\scr{A}}\rho'_{\al}e_{\al,\be}}[\Delta(A,\be)]\LL^{-n}\right).$$\end{proof}

This result implies that the infinite product $\prod_{v\in C_0}Z_v(0,(T^{\rho'_{\al}})_{\al})$  has a meromorphic continuation for $|T|<\LL^{-1 + \delta}$, its only pole being a pole of order $\card(\scr{A}\setminus\scr{A}_D) = \mathrm{rk\ Pic}(U)$ at $T = \LL^{-1}$. 
\subsubsection{Non-trivial characters} \label{nontrivchar}

For every $\xi\in G(F_v)$, the linear form $x\mapsto \langle x,\xi\rangle$ on $G_{F_v}$ defines a meromorphic function $f_{\xi}$ on  $X$. The support of its divisor of poles is contained in $\bigcup_{\alpha} D_{\alpha}$, so that we can write
$$\div f_{\xi} = E(\xi) - \sum_{\alpha\in \scr{A}}d_{\alpha}(\xi) D_{\alpha},$$
where $E(\xi)$ is an effective divisor, and the integers $d_{\alpha}(\xi)$ \index{dalph@$d_{\alpha}(\xi)$, $d_{\al}$} are non-negative. Once $\xi$ is fixed, we will simply write $E$ and $d_{\alpha}$, omitting the mention of $\xi$. In the following analysis, the place $v$ is fixed, so that $Z_v(\T,\xi)$ will be simply denoted $Z(\T,\xi)$. Moreover, using lemma \ref{finitenumberfactors}, we may assume that $v$ belongs to the open subset $C_1$ of $C$ (see notation \ref{C1}). %$C'_0$, the open dense subset of places $v\in C_0$ such that $\be_{0,v} = \{\beta\}$, $e_{\al,\be} = 0$ for all $\al\in\scr{A}$, $\nu_v = 0$, and the fibre of $\scr{D}_{\al}$ above $v$ is irreducible.
 We write $\Omega(A)$ \index{omegaa@$\Omega(A)$} for $\Omega(A,\be)$. Recall that because of the restriction of the domain of summation (section \ref{summationdomain}), $\xi$ is an element of $(t^{\nu_v}k[[t]])^n = (k[[t]])^n$. 

We have
$$Z(\T,\xi) = \sum_{A\subset \scr{A}\setminus \scr{A}_D} Z_A(\T,\xi)$$ where
$$Z_{A}(\T,\xi) = \int_{\Omega(A)}\prod_{\alpha\in A}\left(\LL^{\rho_{\al}}T_{\al}\right)^{\ord x_{\al}}\mathrm{e}(f_{\xi}(x))\dx x.$$\index{ZA@$Z_A(\T,\xi)$}

\paragraph{The case $A = \varnothing$} The set $\Omega(\varnothing)$ corresponds to arcs with origin contained in none of the $D_{\al}$, that is, contained in $G$.
Then we get 
$$Z_{\varnothing}(\T,\xi) = \int_{\scr{L}(\mathbf{G}^n_{\mathrm{a}})}\mathrm{e}(\langle x,\xi\rangle)\dx x $$

Since $\xi$ is an element of $(t^{\nu_v}k[[t]])^n$, for all $x\in\scr{L}(\mathbf{G}^n_{\mathrm{a}})(k) = k[[t]]^n$, we have 
$$\ord(\langle x,\xi\rangle) = \ord(x_1\xi_1+\ldots + x_n\xi_n) \geq \nu_v.$$
Thus, in fact $r(\langle x,\xi\rangle) = 0$, and the integral is equal to 1. 
 
 \paragraph{The case $A = \{\al\}$} We are going to cut the integral into two pieces: arcs with origin outside or inside the divisor $E$:
 \begin{eqnarray*} Z_{v,\{\al\}}(\T,\xi) &= &\int_{\scr{L}(\scr{X},D_{\al}^{\circ}\backslash E)} (\LL^{\rho_{\al}}T_{\al})^{\ord x_{\al}} \mathrm{e}(f_{\xi}(x_{\al},\y))\dx x_{\al} \dx \y \\
 & &+ \int_{\scr{L}(\scr{X},D_{\al}^{\circ}\cap E)} (\LL^{\rho_{\al}}T_{\al})^{\ord x_{\al}} \mathrm{e}(f_{\xi}(x_{\al},\y)) \dx x_{\al} \dx \y.\end{eqnarray*}
 
By the equality $X(k((t)))= \scr{X}(k[[t]])$, the rational function $f_{\xi}$ on $X$ induces a rational function $f_{\xi}$ on $\scr{L}(\scr{X})$. On the subspace $\scr{L}(\scr{X},D_{\al}^{\circ}\setminus E)$, $f_{\xi}$ is of the form $f_{\xi}(x,\y) = g_{\xi}(x,\y)x^{-d}$, where $g_{\xi}$ is a regular function on $\scr{L}(\scr{X},D_{\al}^{\circ}\setminus E)$. We may expand the function $g_{\xi}$ as a converging and non-vanishing power series in $x\in\scr{L}^{1}(\A^1,0)$ and $\y\in(\scr{L}^{1}(\A^1,0))^{B}$ where $B = A\backslash\{\al\}$:
\begin{equation}\label{g.series}g_{\xi}(x,\y) = \sum_{\substack{p\geq 0 \\ \mathbf{q}\in\N^B}}g_{p,q}x^{p}\y^{\mathbf{q}},\end{equation}
with $g_{p,q}\in\OO(D_{\al}^{\circ})[[t]].$ Since we consider only arcs with origin outside $E$, we have that $\ord(g_0)= 0$, and more generally, $\ord (g_{\xi}(x,\y)) = 0$ for all $x,\y$. 

There are several cases to consider, depending on the order $d =d_{\al}$ of the pole of $f_{\xi}$ at~$D_{\al}$. Define
$$\scr{A}_0(\xi)^D =\{\al\in\scr{A}\setminus \scr{A}_D,\ d_{\al} =0\}$$
and 
$$\scr{A}_1(\xi)^D =\{\al\in\scr{A}\setminus \scr{A}_D,\ d_{\al} =1\}.$$

\paragraph{The order of the pole at $D_{\al}$ is zero}
 Here we assume that $\alpha\in\scr{A}_0^{D}(\xi)$, so that $d=0$. Since $\ord (g) \geq 0$, we have $r\circ g = 0$. Therefore 

\begin{eqnarray*} & & \int_{\scr{L}(\scr{X},D_{\al}^{\circ}\backslash E)} (\LL^{\rho_{\al}}T_{\al})^{\ord x_{\al}} \mathrm{e}(f_{\xi}(x,\y))\dx x \dx \y\\
&= &\int_{(D_{\al}^{\circ}\setminus E) \times \scr{L}(\A^1,0)^{n-1}}\sum_{m\geq 1}(\LL^{\rho_{\al}}T_{\al})^{m}\left(\int_{\substack{\scr{L}(\A^1,0) \\ \ord x = m}}\psi(r(g_{\xi}(x,\y)))\dx x \right)\dx \y\\
& = & \int_{(D_{\al}^{\circ}\setminus E) \times \scr{L}(\A^1,0)^{n-1}}\sum_{m\geq 1}(\LL^{\rho_{\al}}T_{\al})^{m} \LL^{-m}(1-\LL^{-1})\\
& = & (1-\LL^{-1})\frac{\LL^{\rho_{\al}-1}T_{\al}}{1-\LL^{\rho_{\al}-1}T_{\al}}[D^{\circ}_{\alpha}\setminus E]\LL^{-n+1}
\end{eqnarray*}

\paragraph{The order of the pole at $D_{\al}$ is positive} Assume now $d\geq 1$. Then:

$$\int_{\scr{L}(\scr{X},D_{\al}^{\circ}\setminus E)} (\LL^{\rho_{\al}}T_{\al})^{\ord x_{\al}} \mathrm{e}(f_{\xi}(x,\y))\dx x \dx \y $$
\begin{eqnarray*} &=&\sum_{m\geq 1}(\LL^{\rho_{\al}}T_{\al})^{m}\int_{\substack{\scr{L}(\scr{X},D_{\al}^{\circ}\setminus E)\\ \ord x = m}}\mathrm{e}(g_{\xi}(x,\y)x^{-d})\dx x \dx \y\\
&=&\int_{(D_{\al}^{\circ}\setminus E)\times\scr{L}(\A^1,0)^{n-1}}\sum_{m\geq 1}(\LL^{\rho_{\al}}T_{\al})^{m}\left(\int_{\substack{\scr{L}(\A^1,0)\\ \ord x = m}}\psi(r(g_{\xi}(x,\y)x^{-d}))\dx x \right)\dx \y.\end{eqnarray*}
Fixing $\y$ and viewing $g_{\xi}(x,\y)$ as a power series in~$x$, we may apply lemma \ref{motivic.integral}: indeed, first of all we may drop all the terms of degree in $x$ greater than $d$ because of the invariance of $r$, so that we get a polynomial in $x$ with coefficients in $k[[t]]$. Moreover, its constant term is $g_{\xi}(0,\y)$ which by a remark above is of order $0$.  This shows that this expression is zero when $d>1$. When $d=1$, only the term for $m=1$ remains, and it is equal to 
$$[D_{\al}^{\circ}\setminus E]\times \LL^{1-n} \LL^{\rho_{\al}}T_{\al}(-\LL^{-2}) ,$$
which is $\rho'_{\al}$-admissible, as $\dim [D_{\al}^{\circ}\setminus E] = n-1$. 

\paragraph{Arcs with origin in $E$}
The term corresponding to arcs with origin in $E$ may be rewritten as
\begin{eqnarray*}&&\sum_{m\geq 1}(\LL^{\rho_{\al}}T_{\al})^{m}\int_{\substack{\scr{L}(\scr{X},D_{\al}^{\circ}\cap E)\\ \ord x = m}}\mathrm{e}(f_{\xi}(x,\y)) \dx x \dx \y\\
& = & \sum_{m\geq 1}(\LL^{\rho_{\al}}T_{\al})^{m}\int_{\substack{D_{\al}^{\circ}\cap E\times\scr{L}(\A^1) \scr{L}(\A^1,0)^{n-1}\\ \ord x = m}}\mathrm{e}(f_{\xi}(t^mx,\y)) \dx x \dx \y\\
\end{eqnarray*}
 Using remark \ref{volume.integrals} again, the weight of the coefficient of degree $m$  is bounded by the the weight of $\LL^{m\rho_{\al}}[D_{\al}^{\circ}\cap E]\times \LL^{-m}(1-\LL^{-1})\LL^{-n+1}$. The dimension of the latter is smaller than 
 $$m\rho_{\al}+ (n-2) -m -(n-1) = m\rho_{\al} -m -1 $$
 Since $m\geq 1$, we see that all these terms are $\rho'$-admissible. It remains to bound all terms of sufficiently large degree as in lemma \ref{convergencetoapply}. For this, put
 \begin{equation}\label{defofc}c = \max_{\al\in\scr{A}\setminus\scr{A}_D}\left(1- \frac{1}{2\rho_{\al}}\right) < 1,\end{equation}
 so that for all $\alpha\in\scr{A}\setminus \scr{A}_D$, $\rho_{\al}-\frac{1}{2}\leq c\rho_{\al}.$
 Using remark \ref{bounddimension}, we see that
  $$2(m\rho_{\al} - m-1) + \dim C_1 \leq 2(mc\rho_{\al} -1) + 1 = 2c (m\rho_{\al}) -1$$
  so if $T_{\al}$ is specialised to $T^{\rho_{\al}}$ for all $\alpha\in\scr{A}\setminus\scr{A}_D$, we are in the situation of lemma \ref{convergencetoapply}. Note that here we in fact could have taken a smaller $c$, namely $c = \max_{\al\in \scr{A}\setminus\scr{A}_D}(1 -\rho_{\al}^{-1})$. The definition we chose will be important in the next case.

\paragraph{The case $\#A > 2$} Then $Z_{A}$ can be rewritten as:
\begin{eqnarray*}Z_{v,A}(\T,\xi) &=& \sum_{\m\in\N_{>0}^{A}} \prod_{\alpha\in A}(\LL^{\rho_{\al}}T_{\al})^{m_{\al}}\int_{\ord x_{\al} = m_{\al}}\mathrm{e}(f_{\xi}(x,y))\dx x \dx y \\
%&\substack{\displaystyle{=} \\ x'_{\al} = t^{m_{\al}}x_{\al}}& \sum_{\m\in\N^{A}} \prod_{\alpha\in A}(\LL^{\rho_{\al}-1}T_{\al})^{m_{\al}}\int_{\ord x'_{\al} =0}\mathrm{e}(f_{\xi}(t^{\m}x',y))\dx x' \dx y \\
\end{eqnarray*}
We proceed as in the previous case. The integral in each term is over the constructible subset of $\Omega(A)$ % viewed as $D^{\circ}_{A}\scr{L}(\A^{1};0)^{A}\scr{L}(\A^{1})^{B}$, 
of points satisfying $\ord(x_{\al}) = m_{\al}$, which has motivic volume $$[D_A^{\circ}]\prod_{\al\in A}\left( \LL^{-m_{\al}}(1-\LL^{-1})\right) \LL^{-n + \#A}.$$ The constructible set $D_A^{\circ}$ is given as an open subset of an intersection of the $|A|$ divisors $D_{\al}$, $\al\in A$, and therefore is of dimension at most $n-|A|$. We may conclude that the dimension of the element of $\expp_k$ given by the integral is at most $-\sum_{\alpha\in A}m_{\al} \leq -2.$ Thus, all terms are $\rho'$-admissible, and it remains to give a stronger bound for all terms of sufficiently large degree.
Using remark \ref{bounddimension} again, with $c$ given by (\ref{defofc}), we have
\begin{eqnarray*}2\sum_{\al\in A}m_{\al}(\rho_{\al}-1) + \dim C_1 &\leq& 2 \sum_{\al\in A}m_{\al}\left(\rho_{\al}-\frac{1}{2}\right)  - \sum_{\al\in A} m_{\al}+1\\
&\leq & 2c \sum_{\al\in A} m_{\al}\rho_{\al} - 1.\end{eqnarray*}
Thus, here again we are in the situation of lemma \ref{convergencetoapply}.

\paragraph{Putting everything together}
We decomposed the zeta function at $v$ in the following manner:
\begin{eqnarray*}Z_{v}(\T,\xi)  &= & 1 + \sum_{\al\in\scr{A}_0}Z_{v,\al}(\T,\xi) + \sum_{\al\in\scr{A}_1}Z_{v,\al}(\T,\xi) + \sum_{\# A>2} Z_{v,A}(\T,\xi)\\
 &= & 1 + \LL^{-n}\sum_{\al\in \scr{A}_0^{D}(\xi)}[D_{\al}^{\circ}\setminus E] \frac{\LL^{\rho_{\al}}T_{\al}}{1-\LL^{\rho_{\al}-1} T_{\al}} + \substack{\text{terms satisfying the bounds}\\\text{of lemma \ref{convergencetoapply} for $c$ given by (\ref{defofc})}}
\end{eqnarray*}
Since $\rho_{\al}-1 \leq c\rho_{\al}$, multiplication by $\LL^{\rho_{\al}-1}T^{\rho_{\al}}$ preserves the bounds of lemma \ref{convergencetoapply}. Thus, multiplying everything by $$\prod_{\al\in \scr{A}^D_0(\xi)}(1-\LL^{\rho_{\al}-1}T_{\al})$$ and keeping only potentially non-admissible terms, we get:

\begin{eqnarray}\prod_{\al\in \scr{A}_0^D(\xi)}(1-\LL^{\rho_{\al}-1}T_{\al})Z_{v}(\T,\xi)& = & 1 - \sum_{\al\in\scr{A}_0^D(\xi)}\LL^{\rho_{\al}-1}T_{\al} \label{A0line}\\
&& +\ \LL^{1-n}\sum_{\al\in\scr{A}^D_0(\xi)}[D_{\al}^{\circ}\setminus E](\LL^{\rho_{\al}-1}T_{\al})\label{A1line}\\
 &  &  +\ \ \substack{\text{terms satisfying the bounds}\\\text{of lemma \ref{convergencetoapply} for $c$ given by (\ref{defofc})}}  \label{A2line}
\end{eqnarray}

\begin{prop}\label{nontrivcharconv} There is a real number $\delta>0$ such that the product
$$\prod_{v\in C_0}\left(Z_v((T^{\rho'_{\al}})_{\al},\xi)\prod_{\al\in \scr{A}^D_0(\xi)}(1-\LL^{\rho_{\al}-1} T^{\rho_{\al}}) \right)$$
converges for $|T|< \LL^{-1+\delta}$. %and takes a non-zero effective value in $\widehat{\expm}$ for $T = \LL^{-1}$.
% $$\prod_{\al\in\scr{A}\setminus\scr{A}_D}Z_{C_0}((\LL T)^{\rho_{\al}-1})\prod_{v\in C_0}Z_v((T^{\rho_{\al}})_{\al},0)$$
\end{prop}

\begin{proof} The above calculations give explicit formulas for the main terms of all $v\in C_1$. Using lemma \ref{finitenumberfactors}, it suffices to check convergence for these places, by showing that they satisfy the conditions of lemma \ref{convergencetoapply} for $c$ given by (\ref{defofc}) and some sufficiently large $M$. According to the above, it suffices to bound the weights of the terms in (\ref{A0line}) and (\ref{A1line}), which is done exactly as in the proof of
 proposition \ref{trivialcharconv}. The conclusion follows.
\end{proof}

%\paragraph{Uniformity} The analogue of this computation in the arithmetic case (see \cite{CLTi}, Proposition 3.3.6.) automatically gives a bound that is uniform in $\xi$. In the motivic case to get uniformity we need to use the arguments from \cite{CL}, section 6.5 (see also proposition 4.3.4). According to the latter, for every $v\in C_0$, the summation domain $V\setminus \{0\}$ may be partitioned into constructible strata $(U_{v,i})$ such that there exist elements $P_{v,i}\in \expp_{U_{v,i}}[\T,\T^{-1}]$ such that the  
 
%$$\Z_v(\T,\xi}

\subsection{Places in $S$}
Let $v$ be a place in $S = C\setminus C_0$. In this case, for any character $\xi$, $Z(\T,\xi)$ is given by lemma \ref{arcspace}, taking $h=1$:
$$Z(\T,\xi) = \sum_{\substack{A\subset \scr{A}\\\be\in\scr{B}_0}}\prod_{\al\in\scr{A}}T_{\al}^{e_{\al,\be}}\LL^{\rho_{\be}}\int_{\Omega(A,\be)}\prod_{\al\in A}(\LL^{\rho_{\al}}T_{\al})^{\ord(x_{\al})}\mathrm{e}(\langle x,\xi\rangle))\dx x.$$

In this section we are going to use Clemens complexes: see \cite{CL}, section 2.3, for the definition.

\subsubsection{Trivial character}\label{trivcharS}
Assume $\xi = 0$, so that, as in section 6.4 of \cite{CL},
\begin{equation}\label{zetatrivialcharS}Z_{v}(\T,0) = \sum_{\substack{A\subset \scr{A}\\ \beta\in\scr{B}_{1,v}}}\left(\prod_{\al\in\scr{A}}T_{\al}^{e_{\al,\be}}\LL^{\rho_{\be}}\right) [\Delta(A,\be)](1-\LL^{-1})^{|A|}\LL^{-n + |A|}\prod_{\al\in A}\frac{\LL^{\rho_{\al}-1}T_{\al}}{1-\LL^{\rho_{\al}-1} T_{\al}}.\end{equation}
This case has essentially been treated in section 6.4 and proposition 4.3.2 of \cite{CL}. The only modification we have to make is to adapt to the case where $D$ is not all of $X\setminus G$, so that the $\alpha\not\in \scr{A}_D$ do not contribute to the pole. For this, we proceed as in \cite{CL}, fixing for every pair $(A,\beta)$ a maximal subset $A_0$ of $\scr{A}_{D}$ such that $A\cap \scr{A}_D\subset A_0$ and $\Delta(A_0,\be)\neq \emptyset$. We collect the terms in equation \ref{zetatrivialcharS} corresponding to pairs $(A,\be)$ associated with any given $A_0$:
$$Z_{v}(\T,0) = \sum_{\substack{A_0\in \Cl_{v}^{\an,\max}(X,D)}}\sum_{\substack{A\subset \scr{A}, \beta\in\scr{B}_{1,v}\\
(A,\be) \mapsto A_0}}\left(\prod_{\al\in\scr{A}}T_{\al}^{e_{\al,\be}}\LL^{\rho_{\be}}\right) [\Delta(A,\be)]$$
$$\times(1-\LL^{-1})^{|A|}\LL^{-n + |A|}\prod_{\al\in A\setminus\scr{A}_D}\frac{\LL^{\rho_{\al}-1}T_{\al}}{1-\LL^{\rho_{\al}-1} T_{\al}} \prod_{\al\in A\cap \scr{A}_D}\frac{\LL^{\rho_{\al}-1}T_{\al}}{1-\LL^{\rho_{\al}-1} T_{\al}}.$$
 Thus, there exists a family of Laurent polynomials $(P_{v,A})$ with coefficients in $\M_k$ indexed by the set of maximal faces $A$ of $\Cl_v^{\an}(X,D)$ such that
$$Z_v(\T,0) = \sum_{A\in \Cl_{v}^{\an,\max}(X,D)}\frac{P_{v,A}(\T)}{\prod_{\al\in \scr{A}\setminus \scr{A}_D}(1-\LL^{\rho_{\al}-1}T_{\al})}\prod_{\al\in A}\frac{1}{1-\LL^{\rho_{\al}-1}T_{\al}},$$
with $P_{A,v}( \T)$ congruent to some non-zero effective element of $\M_k$ modulo the ideal generated by the polynomials $1-\LL^{\rho_{\al}-1}T_{\al}$ for $\al\in \scr{A}$. Putting $T_{\al} = T^{\rho'_{\al}}$ for every $\al\in\scr{A}$, we may deduce from this that there is a family of Laurent series 
  $(F_{v,A})$ with coefficients in $\M_k$, indexed by the set of maximal faces $A$ of $\Cl_v^{\an}(X,D)$, converging for $|T| < \LL^{-1 + \min_{\al\in\scr{A}\setminus \scr{A}_D} \frac{1}{\rho_{\al}}}$, taking a non-zero effective value in $\widehat{\M}_k$ at $\LL^{-1}$, and such that
$$Z_v(\T,0) = \sum_{A\in \Cl_{v}^{\an,\max}(X,D)}F_{v,A}(T)\prod_{\al\in A}\frac{1}{1-(\LL T)^{\rho_{\al}-1}}.$$
 In particular, setting $d_v = 1 + \dim \mathrm{Cl}_v^{\mathrm{an}}(X,D)$, we may deduce the following result:

\begin{prop} There is a real number $\delta >0$ such that for every non-zero common multiple $a$ of the integers $\rho_{\al}-1$, $\al\in\scr{A}_D$, the Laurent series $(1-\LL^{a}T^{a})^{d_v}(Z_v(T,0))$ converges for $|T| <\LL^{-1 + \delta}$ and takes a non-zero effective value at~$\LL^{-1}$.
\end{prop}
\begin{remark} According to the above calculations, one may take $\delta = \min_{\al\in\scr{A}\setminus\scr{A}_D}\frac{1}{\rho_{\al}}.$
\end{remark}

\subsubsection{Non-trivial characters} %\label{nontrivcharS}
 For this case we refer to proposition 4.3.4 and section~6.5 in~\cite{CL}. Recall from \ref{summationdomain} that we restricted the summation domain to a finite-dimensional $k$-vector space $V$. For any $v$, denote by $\a:V_{F_v}\to G_{F_v}$ the corresponding $F_v$-linear inclusion. For every $v$, we then have a Laurent series $Z_{v}(T,\a(\cdot))\in \expp_{V}[[T]][T^{-1}]$, and we ask for its convergence properties.  Section 6.5 of \cite{CL} gives an argument, based on ideas from section 3.4 of \cite{CLTi}, to describe the convergence properties of $Z_v(T,\a(\cdot))$, uniformly on the strata of a constructible partition of $V$. First of all, Chambert-Loir and Loeser show lemma~6.5.1, which allows them to resolve indeterminacies of the function~$f_{\xi}$ uniformly on each stratum $P$ of such a partition. Then they apply change of variables to show that one can compute the integral giving $Z_v(\T,\xi)$ on the fibre above $\xi$ of such a resolution. On the other hand, they have a general result, proposition 5.3.1, giving a formula for motivic Igusa zeta functions without indeterminacies. It states that the poles of such an Igusa zeta function are controlled by the set of maximal faces of the subcomplex $\mathrm{Cl}^{\an}_v(X,D)_{\xi}$ of the complex $\mathrm{Cl}^{\an}_v(X,D)$ where we only keep vertices $\al\in \scr{A}_{D}$ such that $d_{\al}(\xi)=0$. This argument works in exactly the same way in our setting, the only difference being that in Chambert-Loir and Loeser's paper, the set $\scr{A}_D$ is equal to $\scr{A}$, which is not necessarily the case here. Therefore, proposition 4.3.4 from \cite{CL} adapts to our setting in the following form:
 
 \begin{prop}\label{nontrivcharS} Let $v\in S$ and let $d_v(\xi) = 1 + \dim \mathrm{Cl}_v^{\mathrm{an}}(X,D)_{\xi}$. There exists a constructible partition $(U_{v,i})$ of $V\setminus\{0\}$ on each stratum of which $\xi\mapsto d_v(\xi)$ is constant equal to some integer $d_{v,i}$, and, for every $i$, an element $P_{v,i}\in\expp_{U_{v,i}}[\T,\T^{-1}]$ and finite families $(a_{v,i,j})$, $(b_{v,i,j})$ where $a_{v,i,j}\in\N$, $b_{v,i,j}\in\N^{\scr{A}}$, such that the restriction of $Z_{v}(\T,\a(\cdot))$ to $U_{v,i}$ equals 
 $$\prod_{j}(1-\LL^{a_{v,i,j}}\T^{b_{v,i,j}})^{-1} P_{v,i}(\T;\cdot).$$
 Moreover, there exist integers $a_{v,i}\geq 1$ and a real number $\delta>0$ such that the restriction to~$U_{v,i}$ of $(1-(\LL T)^{a_{v,i}})^{d_{v,i}}Z_v(T,\a(\cdot))$ converges for $|T|<\LL^{-1 + \delta}$. 
 \end{prop}

\section{Proof of the main theorem and its corollary}
According to (\ref{summation.zeta}), we may write the multivariate zeta function $Z(\T)$ in the form
$$Z(\T) = \LL^{(1-g)n}Z(\T,0) + \LL^{(1-g)n}\sum_{\xi\in V\setminus\{0\}}Z(\T,\xi)$$
where $V$ is a finite-dimensional $k$-vector space contained in $k(C)^n$ and $g$ is the genus of the smooth projective curve $C$. We are interested in the convergence properties of  
$$Z(T) = \LL^{(1-g)n}Z(T,0) + \LL^{(1-g)n}\sum_{\xi\in V\setminus\{0\}}Z(T,\xi)$$
where for all $\xi$, we have $Z(T,\xi) = Z((T^{\rho'_{\al}})_{\al},\xi).$ Recall we still assume $k=\C$.
\subsection{The function $Z(T,0)$}\label{zt0}
In section \ref{trivchar}  we showed the convergence beyond~$\LL^{-1}$ of the product of local factors $Z_v(T,0)$  over $v\in C_0$ after multiplication by the polynomial
$$\prod_{\al\in\scr{A}\setminus\scr{A}_D} (1-\LL^{\rho_{\al}-1}T^{\rho_{\al}})$$
at each place $v$. To derive a meromorphic continuation for $\prod_{v\in C_0}Z_v(T,0)$, it therefore suffices to describe the convergence of the product
$$\prod_{\al\in\scr{A}\setminus\scr{A}_D}\prod_{v\in C_0} (1-\LL^{\rho_{\al}-1}T^{\rho_{\al}})^{-1}.$$
In the latter, we recognise for each $\al\in \scr{A}\setminus \scr{A}_D$ the Euler product decomposition of the motivic zeta function of $C_0$ at $\LL^{\rho_{\al}-1}T^{\rho_{\al}}$. 

Denote by $Z_C(T) = \sum_{m\geq 0}[S^m C] T^m\in\M_k[[T]]$  the motivic zeta function of the smooth projective curve $C$. Since $k=\C$ is algebraically closed, we have, by theorem 1.1.9 in \cite{Kapr}, that $$Z_C(T) = \frac{P_C(T)}{(1-T)(1-\LL T)}$$ where $P_C(T)\in\M_k[T]$ is a polynomial of degree $2g$ such that $P_C(\LL^{-1}) = \LL^{-g}[J(C)]$, with $J(C)$ being the Jacobian of $C$. %In particular, the rational series
%$$(1-\LL T)Z_C(T)$$
%takes a non-zero effective value $(P_C(\LL^{-1}))\sum_{i\geq 0}\LL^{-i}\in\hat{M}_k[[T]]$ at $\LL^{-1}$.
Moreover, the zeta function $Z_{C_0}(T)$ of the open dense subset $C_0\subset C$ is given by
$$Z_{C_0}(T) = \prod_{v\in C_0}(1-T)^{-1} = Z_C(T)\prod_{v\in C\setminus C_0}(1-T).$$
We have
\begin{eqnarray*}\prod_{v\in C} \prod_{\al\in\scr{A}\setminus \scr{A}_D}(1-\LL^{\rho_{\al}-1}T^{\rho_{\al}})^{-1} &=& \prod_{\al\in\scr{A}\setminus \scr{A}_D}Z_C(\LL^{\rho_{\al}-1}T^{\rho_{\al}})\\
& =& \prod_{\al\in\scr{A}\setminus\scr{A}_D}\frac{P_C(\LL^{\rho_{\al}-1}T^{\rho_{\al}})}{(1-\LL ^{\rho_{\al}-1}T^{\rho_{\al}})(1-(\LL T)^{\rho_{\al}})}.\end{eqnarray*}
Since $1-\LL^{\rho_{\al}-1}T^{\rho_{\al}}$ evaluated at $\LL^{-1}$ is $1-\LL^{-1} = \LL^{-1}[\A^1\setminus\{0\}]$ which is effective, we may conclude that the product over places in $C_0$
$$\prod_{v\in C_0}\prod_{\al\in \scr{A}\setminus\scr{A}_D}(1-\LL^{\rho_{\al}-1}T^{\rho_{\al}})^{-1} $$
 is of the form
$$\frac{F(T)}{\prod_{\al\in\scr{A}\setminus\scr{A}_D}(1-(\LL T)^{\rho_{\al}})}$$
where $F\in \M_k[[T]]$ is a rational power series converging for $|T| < \LL^{-1+\delta}$ for some $\delta>0$ (more precisely, any $\delta\leq \min_{\al\in\scr{A}\setminus \scr{A}_D} \frac{1}{\rho_{\al}}$ works) and taking a non-zero effective value at~$\LL^{-1}$. Thus, by proposition \ref{trivialcharconv}, for a common multiple $a$ of the $\rho_{\al}, \al\in \scr{A}\setminus\scr{A}_D$, we may write the product $\prod_{v\in C_0}Z(T,0)$ in the form
$$\frac{F_1(T)}{(1-(\LL T)^{a})^{|\scr{A}\setminus \scr{A}_{D}|}}$$
for $F_1(T)\in\M_k[[T]]$ converging for $|T|<\LL^{-1 + \delta}$ for some $\delta$ and taking a non-zero effective value at $\LL^{-1}$. Thus, the product $\prod_{v\in C_0}Z_v(T,0)$ has a pole at $\LL^{-1}$ of order $|\scr{A}\setminus \scr{A}_D| = \mathrm{rk}\ \mathrm{Pic} U$, and a meromorphic continuation beyond  $\LL^{-1}$. 

%Taking the product over $v\in C_0$ of the local factors estimated in section \ref{trivchar}, we see that %the product of Laurent series
%$$\left(\prod_{v\in C_0}Z_(T,0)\right)\prod_{\al\in\scr{A}\setminus \scr{A}_D}Z_{C_0}(\LL^{\rho_{\al} -1}%T^{\rho_{\al}})^{-1}$$
%converges for $|T|< \LL^{-1 + \delta}$, and takes a non-zero effective value at $\LL^{-1}$. Therefore, the 
%the product $\prod_{v\in C_0}Z(T,0)$ has a pole of exact order $\#(\scr{A}\setminus \scr{A}_D) = \mathrm{rk}(U)$ at $\LL^{-1}$. 
Combining this with the result of section \ref{trivcharS} we see each place $v\in C\setminus C_0$ gives an additional contribution to the pole at $\LL^{-1}$, of order exactly $d_v = 1+ \Cl_v^{\an}(X,D)$. Finally, we can conclude  that there is a real number $\delta >0$,  a Laurent series $G_0(T)\in \M_k[[T]][T^{-1}]$ converging for $|T|<\LL^{-1 + \delta}$, and taking a non-zero effective value at $\LL^{-1}$, and an integer $a\geq 1$ (we may take $a$ to be any common multiple of the $\rho'_{\al},\ \al\in\scr{A}$) such that the product $Z(T,0) = \prod_{v\in C} Z(T,0)$ may be written in the form 
$$Z(T,0) = \frac{G_0(T)}{(1-(\LL T)^a)^{r}},$$
where $$r = \mathrm{rk}\ \mathrm{Pic}(U) +\sum_{v\in C\setminus C_0}(1 + \dim\Cl_v^{\an}(X,D)).$$

\begin{example}\begin{enumerate}\item Assume $U = G$, so that $\mathrm{Pic}(U) = 0$. One then recovers the result from \cite{CL}.
\item Assume $U=X$. Then one may take $C_0 = C$, and the order of the pole is exactly $\mathrm{rk}\ \mathrm{Pic}(X)$. 
\end{enumerate}
\end{example}
\subsection{The function $Z(T,\xi)$}\label{ztxi}
We proceed as in section \ref{zt0}: according to proposition \ref{nontrivcharconv}, for every $\xi\in V\setminus\{0\}$ the product $$\left(\prod_{v\in C_0}Z_v(T,\xi)\right)\prod_{\al\in \scr{A}_0^{D}(\xi)}Z_C(\LL^{\rho_{\al}-1}T^{\rho_{\al}})\in\expp_k[[T]][T^{-1}]$$ converges for $|T| < \LL^{-1 + \delta}$ for some $\delta >0$. We may apply this to the generic point of $V\setminus\{0\}$ and use spreading-out and induction on the dimension to show that there exists a finite constructible partition of $V\setminus\{0\}$ on which the functions $\xi\mapsto d_{\al}(\xi)$ are constant, and such that this convergence holds uniformly in $\xi$ on each piece of the partition. Proposition \ref{nontrivcharS} on the other hand tells us that the order of the pole of $\prod_{v\in C\setminus C_0} Z(T,\xi)$ at $\LL^{-1}$ is at most $\sum_{v} d'_v$ where $d'_v = 1 + \dim \mathrm{Cl}^{\an}_v(X,D)_0$, again uniformly on the pieces of a constructible partition of $V\setminus\{0\}$. Using a partition of $V\setminus\{0\}$ refining the aforementioned two partitions, we may conclude that for any stratum $P$ of this partition, the order  of the pole of 
$$\LL^{(1-g)n}\sum_{\xi\in \a(P)}Z(T,\xi)$$ is at most
$$|\scr{A}_0^D(\xi))| + \sum_{v\in C\setminus C_0}d_v(\xi)$$
for any $\xi\in \a(P)$ (recall $\xi\mapsto d_{\al}(\xi)$, and therefore also $\xi \mapsto |\scr{A}_0^D(\xi))|$ and $\xi \mapsto d_v(\xi)$ are constant on $P$). Lemma 3.5.4 in \cite{CLTi} shows that this is strictly less than the order~$r$ of the pole of $Z(T,0)$ at $\LL^{-1}.$
%There are now two cases to consider: if $C\setminus C_0$ is nonempty, then this is strictly smaller than 
%$|\scr{A}_0^D(\xi))| + d \leq r$. If on the other hand $C = C_0$, then $\scr{A}_0^D(\xi))$ must be a strict subset of $\scr{A}\setminus \scr{A}_D$. Indeed, [Comment conclure ?]%if not, then $\xi$ induces on $X$ a function $f_{\xi}$ which has no poles on $U$. 
%%strictly smaller than .  

\subsection{Conclusion of the proof of theorem \ref{main}}
Taking $a$ to be the least common multiple of the integers $\rho'_{\al}$, $\al\in \scr{A}$ and of the integers $a_{v,i}$ appearing in proposition \ref{nontrivcharS}, we have shown that $(1-(\LL T)^{a})^{r}Z(T)$ converges for $T = \LL^{-1}$, and takes a non-zero effective value at $\LL^{-1}$, which concludes the proof of theorem~\ref{main}.

\subsection{Proof of corollary \ref{maincor}} Applying lemma \ref{coefgrowth}, we get corollary \ref{maincor} in the case where $k = \C$. We now explain how we may deduce from this the general case.

All the geometric data in our counting problem involves a finite number of equations over the field $k$: we may therefore assume that everything is defined over a finitely generated subfield of $\C$. Moreover, the assumption $\scr{U}(\OO_v) \neq \varnothing$ for all $v\in C_0$ may be reformulated more geometrically by saying that the volume of the arc space $\scr{L}(\scr{U}_v)$ at the place $v$ should be non-zero. With the notations of section \ref{sect.integralarcspace}, this volume may be expressed by the formula:
\begin{eqnarray*} \mathrm{vol}(\scr{L}(\scr{U}_v)) &= &\sum_{\substack{ A \subset \scr{A}\setminus \scr{A}_D\\ \be\in \scr{B}_{0,v}}}\mathrm{vol}(\Omega(A,\be))\\
& = & \LL^{-n}\sum_{\substack{ A \subset \scr{A}\setminus \scr{A}_D\\ \be\in \scr{B}_{0,v}}}[\Delta(A,\be)].\end{eqnarray*}
Thus, at least one of the sets $\Delta(A,\be)$ for $A \subset \scr{A}\setminus \scr{A}_D$, $\be\in \scr{B}_{0,v}$ has a $k$-point. Since it is defined over $\C$, it also has a $\C$-point.

Thus, we have shown that without loss of generality, even if the original problem was stated over some algebraically closed field $k$ of characteristic zero, we may in fact assume everything is defined over $\C$, and apply corollary \ref{maincor} in this setting. By functoriality of Hilbert schemes, we may then deduce corollary \ref{maincor} over $k$.

%\bibliography{references}{}

\begin{thebibliography}{99}
\bibitem[BS]{BS}{P. Balmer, M. Schlichting, Idempotent completion of triangulated categories,
Journal of Alg. 236 (2001), 819--834.}
\bibitem[BM]{BM}{V. Batyrev, Yu. Manin, Sur le nombre des points rationnels de hauteur bornée des variétés algébriques, Math. Ann. 286 (1990), 27-43}
\bibitem[BT95]{BT95}{V. Batyrev, Yu. Tschinkel, Rational points of bounded height on compactifications of anisotropic tori, Internat. Math. Res. Notices 12 (1995), 591-635}
\bibitem[BT96]{BT96}{V. Batyrev, Yu. Tschinkel, Rational points on some Fano cubic bundles, C.R. Acad. Sci. Paris Sér. I Math. 323 (1996), 41-46}
\bibitem[BT98]{BT98}{V. Batyrev, Yu. Tschinkel, Manin's conjecture for toric varieties, Journ. of Alg. Geom.,~7, no. 1, 15-53, (1998)}
\bibitem[BBD]{BBD}{A. A. Beilinson, J. Bernstein, P. Deligne, Faisceaux pervers, in \textit{Analysis and topology on singular spaces, 1 (Luminy 1981)}, 5-171, Astérisque \textbf{100}, Soc. Math. France, Paris, 1982}
\bibitem[Bilu]{Bilu}{M. Bilu, Produits eulériens motiviques, PhD thesis, \url{https://tel.archives-ouvertes.fr/tel-01662414}, Université Paris-Saclay, 2017}
\bibitem[Bit]{Bittner}{F. Bittner, The universal Euler characteristic for varieties of characteristic zero, Compos.
Math. $\mathbf{140}$ (2004), no. 4, 1011--1032.}
\bibitem[Bor]{Borisov}{L. Borisov, Class of the affine line is a zero divisor in the Grothendieck ring, arXiv:1412.6194}
\bibitem[Bou02]{Bou02}{D. Bourqui, Fonction zêta des hauteurs des surfaces de Hirzebruch dans le cas
fonctionnel, J. of Number Theory \textbf{94}, 2002, p. 343--358}
\bibitem[Bou03]{Bou03}{D. Bourqui, Fonction zêta des hauteurs des variétés toriques déployées dans le cas fonctionnel, J. Reine Angew. Math. \textbf{562}, 2003, p. 171-199.}
\bibitem[Bou09]{Bou09}{D. Bourqui, Produit eulérien motivique et courbes rationnelles sur les variétés toriques 
Comp. Math. \textbf{45}, No. 6, 2009, p. 1360-1400}
\bibitem[Bou10]{Bou10}{D. Bourqui, Fonctions $L$ d'Artin et nombre de Tamagawa motiviques, New York J. Math. 16 (2010), 179--233}
\bibitem[Bou11]{Bou11}{D. Bourqui, Fonctions zêta des hauteurs des variétés toriques non déployées, 
Memoirs of the AMS \textbf{211}, 2011.}
\bibitem[BBDJS]{BBDJS}{C. Brav, V. Bussi, D. Dupont, D. Joyce, B. Szendr\H{o}i, Symmetries and stabilization for sheaves of vanishing cycles, with an appendix by J. Sch\"{u}rmann, J. Singul. \textbf{11} (2015), 85--151}
\bibitem[Bry]{Brylinski}{J.-L. Brylinski, Transformations canoniques, dualité projective, théorie de Lefschetz, transformations de Fourier et sommes trigonométriques, \textit{Géométrie et analyse microlocales},  
Astérisque No. 140-141 (1986), 3--134. }
\bibitem[CMSSY]{CMSSY}{S. E. Cappell, L. Maxim, J. Schürmann, J. L. Shaneson, S. Yokura,  Characteristic classes of symmetric products of complex quasi-projective varieties, J. Reine Angew. Math. 728 (2017), 35--63}
\bibitem[CLT02]{CLT}{A. Chambert-Loir, Yu. Tschinkel,  On the distribution of points of bounded height on equivariant compactifications of vector groups, Invent. Math. \textbf{148} (2002), 421-452 }
\bibitem[CLT12]{CLTi}{A. Chambert-Loir, Yu. Tschinkel, Integral points of bounded height on partial equivariant compactifications of vector groups, Duke Math. Journal \textbf{161} (15), p. 2799--2836 (2012).}
\bibitem[ChL]{CL}{A. Chambert-Loir, F. Loeser, Motivic height zeta functions, Amer. J. Math. \textbf{138} (1), 1-59 (2016)}
\bibitem[CNS]{CNS}{A. Chambert-Loir, J. Nicaise, J. Sebag, \textit{Motivic integration}, Prog. in Math., Birkhäuser, to appear.}
\bibitem[ClL08]{CluckLos}{R. Cluckers, F. Loeser, Constructible motivic functions and
motivic integration, Invent. Math., $\mathbf{173}$ (1), pp. 23--121 (2008)}
\bibitem[ClL10]{CluckLosexp}{R. Cluckers, F. Loeser,  Constructible exponential functions, motivic Fourier transform and transfer principle, Annals of Math. \textbf{171}, no. 2, 1011-1065 (2010)}
\bibitem[CS]{ColSer}{P. Colmez, J.-P. Serre (editors) (2001), \textit{Correspondance Grothendieck-
Serre}, Documents Mathématiques (Paris), 2, Société Mathématique de France, Paris.}
\bibitem[Del73]{DeligneSGA}{P. Deligne, Comparaison avec la théorie transcendante, in SGA 7, Groupes de Monodromie en Géométrie Algébrique, Part II, Lect. Notes in Math. \textbf{340}, pp.116-164, Springer, Berlin, 1973.}
\bibitem[Del77]{DeligneSGA45}{P. Deligne, Théorèmes de finitude en cohomologie $\ell$-adique, in SGA 4 1/2, Cohomologie Etale, Lect. Notes in Math. \textbf{569}, pp.233-251, Springer, 1977}
\bibitem[Del]{Deligne}{P. Deligne, Théorie de Hodge II, Publ. Math. IHÉS \textbf{40} (1971) , p. 5-57}
\bibitem[DL98]{DL98}{J. Denef, F. Loeser, Motivic Igusa zeta functions , J. Algebraic Geom. \textbf{7} (1998), no. 3, 505-537.}
\bibitem[DL99a]{DL99mot}{J. Denef, F. Loeser, Germs of arcs on singular algebraic varieties
and motivic integration, Invent. Math., 135 (1), pp. 201--232 (1999)}
\bibitem[DL99b]{DL99}{J. Denef, F. Loeser, Motivic exponential integrals and a motivic Thom-Sebastiani theorem, Duke Math. Journal \textbf{99} (1999), no. 2, 285-309}
\bibitem[DL01]{DL01}{J. Denef, F. Loeser, Geometry on arc spaces of algebraic varieties, European Congress of Mathematics, Vol. I (Barcelona 2000), Progr. Math. vol. 201, Birkhäuser, Basel, 2001, pp. 327-348}
\bibitem[DL02]{DL02}{J. Denef, F. Loeser, Lefschetz numbers of iterates of the monodromy and truncated arcs, Topology 41, 1031-1040 (2002)} 
\bibitem[Ekedahl]{Ekedahl}{T. Ekedahl, The Grothendieck group of algebraic stacks, \href{https://arxiv.org/abs/0903.3143}{arXiv:0903.3143v2}, 2009.}
\bibitem[FMT]{FManT}{J. Franke, Yu. I. Manin, Yu. Tschinkel, Rational points of bounded height on Fano varieties, Invent. Math. 95 (1989), 421-435}
\bibitem[GLM]{GLM}{G. Guibert, F. Loeser, M. Merle,  Iterated vanishing cycles, convolution, and a motivic analogue of a conjecture of Steenbrink, Duke Math. Journal 132 (3), p. 409-457 (2006).}
\bibitem[GZLM]{gusein}{S. M. Gusein-Zade, I. Luengo, A Melle, A power structure over the Grothendieck ring of varieties, Mathematical Research Letters 11, 49-57 (2004)}
\bibitem[Howe]{Howe}{S. Howe, Motivic random variables and representation stability II: Hypersurface sections, \href{https://arxiv.org/abs/1610.05720}{arXiv:1610.05720}, 2016}
\bibitem[HK06]{HKint}{E. Hrushovski, D.Kazhdan, Integration in valued fields, in
\textit{Algebraic geometry and number theory}, Progr. Math. 253, pp. 261--405,
Birkhäuser Boston, Boston, MA (2006)}
\bibitem[HK09]{HK}{E. Hrushovski, D. Kazhdan, Motivic Poisson summation, Mosc. Math. Journal 9 (3), p.569-623 (2009)}
\bibitem[HT]{HT}{B. Hassett, Yu. Tschinkel, Geometry of equivariant compactifications of
$\G^n_a$, Internat. Math. Res. Notices \textbf{22} (1999), p. 1211--1230.}
\bibitem[Kapr]{Kapr}{M. Kapranov, The elliptic curve in the S-duality theory and Eisenstein series for
Kac-Moody groups, \href{https://arxiv.org/abs/math/0001005}{arXiv:math/0001005}, 2000}
\bibitem[KS]{KS}{M. Kashiwara and P. Schapira, Sheaves on Manifolds, Grund. der math. Wiss. 292, Springer, 1990}
\bibitem[Kol]{Kollar}{J. Koll\'{a}r, \textit{Lectures on resolution of singularities}, Annals of Mathematics Studies, vol. \textbf{166}, Princeton University Press, Princeton, NJ, 2007}
\bibitem[LY]{LY}{K. F. Lai and K. M. Yeung, Rational points in flag varieties over function fields, J. Number Theory 95 (2002), 142--149.}
\bibitem[LL]{LL}{M. Larsen, V.A. Lunts, Motivic measures and stable birational geometry, Mosc. Math. J., 3 (1), pp. 85--95 (2003)}
\bibitem[Loo]{Loo}{E. Looijenga, Motivic measures, Astérisque \textbf{276} (2002), 267-297, Séminaire Bourbaki 1999/2000, no 874}
\bibitem[LS16a]{LSmat}{V.A. Lunts, O.M. Schnürer, Matrix factorizations and motivic measures, J. Noncommut. Geom. \textbf{10} (2016), no. 3, 981--1042.}
\bibitem[LS16b]{LS}{V.A. Lunts, O.M. Schnürer, Motivic vanishing cycles as a motivic measure, Pure Appl. Math. Q. \textbf{12} (2016), no. 1, 33--74.}
\bibitem[LW]{LW}{S. Lang and A. Weil, Number of points of varieties in finite fields, Amer. J. Math. \textbf{76} (1954),
819--827.}
\bibitem[Mas]{Massey}{D. Massey, The Sebastiani-Thom isomorphism in the
derived category, Comp. Math. \textbf{125}, 353-362, 2001}
\bibitem[MSS]{MSS}{L. Maxim, M. Saito, J. Schürmann, Symmetric products of mixed Hodge modules, J. Math. Pures Appl. \textbf{96} (2011), no. 5, 462--483.}
\bibitem[Milne]{Milne}{J. S. Milne, Abelian varieties, in \textit{Arithmetic Geometry}, G. Cornell, J. H. Silverman eds., Springer, 1986}
\bibitem[MFK]{GIT}{D. Mumford, J. Fogarty, F. Kirwan, \textit{Geometric invariant theory},
Third edition, Ergebnisse der Mathematik und ihrer Grenzgebiete (2), \textbf{34},
Springer-Verlag, Berlin, 1994}
\bibitem[Mus]{Mustata}{M. Musta\unichar{539}\unichar{259}, \textit{Zeta functions in algebraic geometry}, Lecture notes available at \url{http://www.math.lsa.umich.edu/~mmustata/zeta_book.pdf}}
\bibitem[Pey95]{Peyre}{E. Peyre, Hauteurs et mesures de Tamagawa sur les variétés de Fano, Duke Math. J. 79 (1995), No 1, 101-218}
\bibitem[Pey01]{PeyreBourbaki}{E. Peyre, Points de hauteur bornée et géométrie des variétés (d'après Y. Manin et al.), Séminaire Bourbaki, Vol. 2000/2001, Astérisque No. 282 (2002), Exp. No. 891, ix, 323–344.}
\bibitem[Pey12]{Peyre12}{E. Peyre, Points de hauteur bornée sur les variétés de drapeaux en caractéristique finie, Acta Arith. 152 (2012), 185—216}
\bibitem[PS]{PS}{C.A.M. Peters, J.H.M. Steenbrink, \textit{Mixed Hodge structures},
Ergebnisse der Mathematik und ihrer Grenzgebiete. 3. Folge: A Series of Modern Surveys in Mathematics \textbf{52}, Springer-Verlag, Berlin, 2008.}
\bibitem[S88]{Saito88}{M. Saito, Modules de Hodge polarisables, Publ. Res. Inst. Math. Sci. \textbf{24}, no. 6, 849-995 (1990)}
\bibitem[S89]{Saito89}{M. Saito, Introduction to mixed Hodge modules, pp. 145--162 in Actes du Colloque
de Th\'eorie de Hodge (Luminy, 1987), Ast\'erisque Vol. 179--180 (1989).}
\bibitem[S90]{Saito90}{M. Saito, Mixed Hodge Modules, Publ. Res. Inst. Math. Sci. \textbf{26}, no. 2, 221-333 (1990)}
\bibitem[S]{SaitoTS}{M. Saito, Thom-Sebastiani Theorem for Hodge Modules, preprint}
\bibitem[Sch]{Schanuel}{S. H. Schanuel, Heights in number fields, Bull. Soc. Math. France \textbf{107}
(1979), 433--449.}
\bibitem[Schn]{Schnell}{Ch. Schnell, An overview of Morihiko Saito's theory of mixed Hodge modules, \url{https://www.math.stonybrook.edu/~cschnell/pdf/notes/sanya.pdf}}
\bibitem[ST]{TS}{M. Sebastiani,  R. Thom, Un résultat sur la monodromie, Invent. Math.
\textbf{13} (1971), 90--96.}
\bibitem[ShTBT03]{ShTTB03}{J. Shalika, R. Takloo-Bighash, Yu. Tschinkel, Rational points on compactifications of semi-simple groups of rank 1, \textit{Arithmetic of higher-dimensional algebraic varieties}, 205-233, Progress in Math. 226, Birkhäuser, 2003.}
\bibitem[ShTBT07]{ShTTB07}{J. Shalika, R. Takloo-Bighash, Yu. Tschinkel, Rational points on compactifications of semi-simple groups, Journ. of the AMS, 20, 1135--1186, 2007.}
\bibitem[ShT04]{STH}{J. Shalika, Yu. Tschinkel, Height zeta functions of equivariant compactifications of the Heisenberg group, \textit{Contributions to Automorphic Forms, Geometry, and Number Theory}, (H. Hida, D. Ramakrishnan, F. Shahidi eds.), 743--771, Johns Hopkins University Press, 2004.}
\bibitem[ShT16]{STU}{J. Shalika, Yu. Tschinkel, Height zeta functions of equivariant compactifications of unipotent groups, Comm. Pure and Applied Math., 69, no. 4, 693--733, 2016.}
\bibitem[TBT]{TTB}{R. Takloo-Bighash, S. Tanimoto, Distribution of rational points of bounded height on equivariant compactifications of $PGL_2$ I, Res. number theory (2016) 2: 6. }
\bibitem[TT]{TT}{S. Tanimoto, Yu. Tschinkel, Height zeta functions of equivariant compactifications of semi-direct products of algebraic groups, \textit{Zeta functions in algebra and geometry}, 119--157, 
Contemp. Math., 566, Amer. Math. Soc., Providence, RI, 2012.}
\bibitem[VW]{VW}{R. Vakil, M. Wood, Discriminants in the Grothendieck ring, Duke Math. J. 164 (2015), no. 6, 1139--1185.}
\bibitem[Verdier]{Verdier}{J.-L. Verdier, Stratifications de Whitney et théorème de Bertini-Sard, Invent. Math. 36 (1976), 295--312.}
\end{thebibliography}
%\bibliographystyle{alpha}
\backmatter

\cleardoublepage
%\addcontentsline{toc}{chapter}{Bibliography}

\cleardoublepage
%\addcontentsline{toc}{chapter}{Index}
\printindex

\end{document}